\documentclass[11pt,a4paper]{article}
\usepackage[T1]{fontenc}
\usepackage{amsmath,amsfonts,amssymb,amsthm}
\usepackage[numbers]{natbib}
\usepackage{mathtools}
\usepackage{bm}
\usepackage{microtype}
\usepackage{graphicx}
\usepackage{tikz}
\usetikzlibrary{arrows.meta,positioning}
\usepackage{xcolor}
\usepackage[a4paper,margin=1in]{geometry}
\usepackage[hidelinks,pdftitle={Local asymptotic normality for spectral-tail likelihoods under fixed-domain sampling},pdfauthor={Alexandre B. Simas and Jonas Wallin},pdfsubject={},pdfkeywords={fixed-domain asymptotics; Gaussian likelihood; generalized locally Toeplitz sequence; local asymptotic normality; spectral-tail parameters; Gaussian random fields; maximum likelihood estimation; covariance misspecification; Matérn covariance; generalized Wendland covariance}]{hyperref}
\numberwithin{equation}{section}
\theoremstyle{plain}
\newtheorem{theorem}{Theorem}[section]
\newtheorem{lemma}[theorem]{Lemma}
\newtheorem{proposition}[theorem]{Proposition}
\newtheorem{corollary}[theorem]{Corollary}
\theoremstyle{definition}
\newtheorem{definition}[theorem]{Definition}
\newtheorem{assumption}[theorem]{Assumption}
\newtheorem{example}[theorem]{Example}
\newtheorem{remark}[theorem]{Remark}
\newcommand{\calD}{\mathcal{D}}

\DeclareMathOperator*{\argmax}{arg\,max}

\newcommand{\tr}{\operatorname{tr}}
\newcommand{\GLT}{\mathrm{GLT}}
\newcommand{\C}{\mathbb{C}}
\newcommand{\R}{\mathbb{R}}
\newcommand{\N}{\mathbb{N}}
\newcommand{\Z}{\mathbb{Z}}
\newcommand{\bn}{\bm{n}}
\newcommand{\bh}{\bm{h}}
\newcommand{\bth}{\bm{\vartheta}}
\newcommand{\bk}{\bm{k}}
\newcommand{\br}{\boldsymbol{r}}
\newcommand{\bs}{\boldsymbol{s}}
\newcommand{\bi}{\boldsymbol{i}}
\newcommand{\bomega}{\bm{\omega}}
\newcommand{\bxi}{\bm{\xi}}
\newcommand{\bmm}{\bm{m}}
\newcommand{\bx}{\boldsymbol{x}}

\DeclareMathOperator{\rank}{rank}
\title{Local asymptotic normality for spectral-tail likelihoods under fixed-domain sampling}
\author{Alexandre B. Simas\thanks{The authors are listed in alphabetical order.}\textsuperscript{,}\thanks{Statistics program, King Abdullah University of Science and Technology, SA 23955 Thuwal, Saudi Arabia} \and Jonas Wallin\thanks{Department of Statistics, Lund University, Tycho Brahes väg 1, SE 223 63, Sweden}}
\date{}
\begin{document}
\maketitle
\begin{abstract}

We establish local asymptotic normality (LAN) for joint inference on scale and roughness in stationary Gaussian random fields observed on increasingly dense regular grids in a fixed domain.
 Scale and roughness are represented by the coefficient and decay exponent of the leading spectral tail. Under spectral regularity conditions,
  we identify the joint normalization governing the LAN expansion and prove consistency and asymptotic normality of the maximum likelihood estimators. 
  For \(N\) observations, the estimation errors are of order \(N^{-1/2}\) for roughness and \((\log N)/\sqrt N\) for scale. 
  After componentwise normalization, the errors converge to a perfectly correlated Gaussian pair.
  Knowing either parameter improves the convergence rate for estimating the other by a logarithmic factor.
   The analysis uses generalized locally Toeplitz (GLT) theory, for which we develop a truncation principle accommodating nonintegrable limiting symbols and quantitative bounds on the contribution from low frequencies. 
   We verify the assumptions for Matérn and generalized Wendland covariance families, establishing uniform spectral bounds for the latter.
    We also show that consistency and the limiting distributions persist under covariance misspecification when the true and fitted spectra share the same leading tail and their difference decays sufficiently rapidly. 
    This includes misspecification between Matérn and generalized Wendland models under explicit parameter restrictions.
\end{abstract}
\noindent\textbf{Keywords:} fixed-domain asymptotics; Gaussian likelihood;
generalized locally Toeplitz sequence; local asymptotic normality;
spectral-tail parameters; Gaussian random fields; maximum likelihood estimation;
covariance misspecification; Matérn covariance; generalized Wendland covariance.

\noindent\textbf{MSC 2020:} Primary 62M30; Secondary 62F12, 60G60.

\section{Introduction}
\label{sec:introduction}

Asymptotic statistics for stationary Gaussian random fields is often studied in two regimes \cite{Stein1999}. 
In increasing-domain asymptotics, we collect more data over an expanding domain. 
This is similar to observing a longer time series: more data usually provides more information and allows more precise estimation.
In fixed-domain asymptotics, the domain stays fixed, and we collect more data within it. Here,
 more observations do not necessarily provide more information about every parameter. 
 Some parameters can be estimated increasingly precisely, with estimates showing the usual asymptotically normal behavior, 
 while information about others remains limited. This second regime is particularly relevant in spatial statistics, where the region of interest often remains fixed as more observations are collected.

In the fixed-domain regime, denser sampling reveals increasingly fine spatial detail. 
This brings the field’s behavior at small spatial scales into focus. Such behavior is closely related to the spectral density at high frequencies.
 For the models considered here, two parameters describe this behavior:
  one controls how quickly the spectral density decays and is related to the roughness of the field, 
  while the other controls its magnitude and acts as a scale parameter. 
  We denote these the tail parameters. 
  In their review of the Mat\'ern model,
\cite[Section~3.1.2]{PorcuEtAl2024} identify efficient estimation of the
microergodic parameter when smoothness is unknown as an open problem
under fixed-domain asymptotics. We study this question through joint
maximum likelihood estimation of the these tail parameters. 
We show that the maximum likelihood estimator (MLE) is consistent and that the likelihood is locally asymptotically normal (LAN), in the sense of \cite{LeCamYang2000}, under weak spectral assumptions. 
Our framework covers the Matérn \cite{Matern1986} and generalized Wendland \cite{Wendland1995,ChernihHubbert2014} covariance families.
 We also establish consistency under certain forms of model misspecification, 
 linking parameter estimation across these two families.

We observe a stationary Gaussian random field at \(N\) equally spaced locations on a regular grid within a fixed domain in \(\mathbb R^d\). As \(N\to\infty\), the observations become increasingly dense within this domain. We assume that the spectral density satisfies
\[
S(\omega)\sim c_\infty\|\omega\|^{-2\alpha},
\qquad \|\omega\|\to\infty,
\]
where \(\omega\) denotes spatial frequency, \(c_\infty>0\), and \(\alpha>d/2\). The coefficient \(c_\infty\) sets the scale of high-frequency variation, while the exponent \(\alpha\) controls the roughness of the field.
 We estimate these two parameters jointly, holding the remaining covariance parameters fixed.

The Gaussian likelihood depends on the determinant and inverse of the covariance matrix of the observations.
 Understanding its asymptotic behavior therefore requires controlling these matrices as their dimension increases. 
 Regular sampling gives the covariance matrices a Toeplitz structure \cite{GrenanderSzego1958}.
  We use generalized locally Toeplitz (GLT) theory \cite{GaroniSerraCapizzanoI,GaroniSerraCapizzanoII} to study their limiting eigenvalue distribution after rescaling.

A standard way to establish GLT convergence is to show that the functions defining the Toeplitz matrices converge in \(L^1\): the integral of their absolute difference from the limiting function must tend to zero. For our rescaled covariance matrices, however, the limiting function grows too rapidly near zero for this argument to apply. We prove a truncation principle that establishes GLT convergence in this setting, even though the limiting function has an infinite integral.

For the sampled covariance matrices, we show that the contribution from frequencies near zero consists of a part whose rank is small relative to the number of observations and a remainder that is small in operator norm. These bounds retain quantitative information needed for the likelihood analysis. Together with bounds on inverses and covariance derivatives, they yield the trace limits governing the likelihood score and information matrix.

Joint estimation presents a further difficulty: changing the exponent also changes the magnitude of the spectrum at high frequencies, so its effect can be partly offset by changing the tail coefficient. 
Under the stated regularity conditions, we establish consistency and local asymptotic normality. 
The joint limiting distribution of the normalized estimation errors is singular: both components converge to the same Gaussian variable.
Writing \((c_{\infty,0},\alpha_0)\) for the true tail parameters, 
the joint maximum likelihood estimator satisfies
\[
\begin{pmatrix}
\sqrt N\,(\widehat\alpha-\alpha_0)\\[4pt]
\displaystyle
\frac{d\sqrt N}{2\log N}\,
\frac{\widehat c_\infty-c_{\infty,0}}{c_{\infty,0}}
\end{pmatrix}
\;\Rightarrow\;
\begin{pmatrix}
Z\\[2pt]
Z
\end{pmatrix},
\qquad Z\sim\mathcal N(0,2/v_{\alpha_0}),
\]
where \(v_{\alpha_0}>0\) is defined in Section~\ref{sec:maximum-likelihood}.
Estimating roughness and scale jointly incurs a logarithmic loss in convergence rate for each parameter.
When \(c_\infty\) is known,  \(\alpha\) can be estimated with an error of order \((\sqrt N\log N)^{-1}\), instead of \(N^{-1/2}\) under joint estimation.
When  \(\alpha\) is known, \(c_\infty\) can be estimated with an error of order \(N^{-1/2}\), instead of \((\log N)/\sqrt N\) under joint estimation.

We also consider covariance misspecification under spectral conditions ensuring equivalence of the true and fitted Gaussian measures on the fixed observation domain. 
Under these conditions, consistency and the limiting distributions remain valid even though the covariance functions differ. In particular, the results connect Matérn and generalized Wendland models that share the same tail parameters and induce equivalent Gaussian measures on the observation domain.
For generalized Wendland models, we derive uniform spectral bounds that are, to our knowledge, new.

\subsection*{Outline of the paper}
Section~\ref{sec:related-literature} reviews related work on covariance estimation under infill asymptotics, focusing on joint estimation of scale and roughness, likelihood-based inference, and covariance misspecification. Section~\ref{sec:sampled-glt} introduces the stationary random field, spectral assumptions, and regular sampling scheme. 
It develops the truncation principle and matrix bounds used to establish GLT convergence of the rescaled covariance matrices.
 Section~\ref{sec:maximum-likelihood} introduces the Gaussian likelihood for the tail parameters and derives the score and information matrix.
  It presents the limiting distributions and convergence rates for joint estimation and for estimation of either parameter when the other is known.
Section~\ref{sec:tail-estimator-consistency} develops the global and local likelihood arguments underlying consistency and asymptotic normality of the MLE. These arguments control the likelihood away from the true parameters, its quadratic approximation near the truth, and the location of its maximizer.
Section~\ref{sec:common-tail-misspecification} gives spectral conditions under which consistency and the limiting distributions remain valid when the covariance model is misspecified.
Section~\ref{sec:standard-covariance-families} verifies the assumptions for the Mat\'ern and generalized Wendland families, with the remaining covariance parameters held fixed. It also identifies pairs of models to which the misspecification results apply.
Section~\ref{sec:simulation-study} examines the finite-sample behavior of the estimators through simulations. It compares joint estimation with estimation of either parameter when the other is known, using both a Mat\'ern likelihood and a generalized Wendland likelihood fitted to Mat\'ern data.
Section~\ref{sec:discussion} summarizes the statistical and analytical contributions and their role in the likelihood analysis. The appendices contain detailed proofs and additional supporting results.

% Previous review retained in likelihood_asymptotics_literature.tex.
\section{Related literature}
\label{sec:related-literature}
We review likelihood asymptotics and related estimation methods under
fixed-domain sampling, beginning with one-dimensional models.
\subsection*{One-dimensional models}

Earlier one-dimensional fixed-domain results concern estimation of the tail
scale with roughness fixed. \cite{Ying1991} establishes consistency and
asymptotic normality for the exponential model, while
\cite{DuZhangMandrekar2009} derives limiting distributions for exact and
tapered maximum likelihood estimators in Mat\'ern models and gives
conditions under which tapering preserves asymptotic efficiency.

\cite{BrousteFukasawa2018} proves LAN for joint scale and roughness
estimation from fractional Gaussian noise under high-frequency sampling,
using a normalization that combines the two estimation errors.

\cite{Petit2025} proves strong consistency of
the smoothness estimator and joint asymptotic normality of the amplitude and
smoothness estimators for  the exact likelihood of a periodized Matérn process
on an increasingly dense regular grid on the circle.

\cite{ChongMies2025} establish LAN and the existence
of an asymptotically efficient sequence of likelihood score roots for
triangular arrays of one-dimensional stationary Gaussian time series,
allowing a multidimensional unknown parameter. Their work is closely related,
but the spectral assumptions differ. Their conditions concern the
sample-size-dependent spectral densities of the sampled time series and
their reciprocals. They allow different power-law bounds across frequency
ranges, while requiring parameter derivatives through order three and
additional frequency-derivative bounds. Under fixed-domain infill, their
conditions include at least some Matérn models with very low spatial
smoothness. Their paper does not provide a general application to Matérn
or generalized Wendland covariance families in this regime.

\subsection*{Spatial covariance estimation}

\cite{Ying1993} studies fixed-domain likelihood estimation for a
multiplicative Ornstein--Uhlenbeck field on a spatial lattice.

\cite[Section~6.7]{Stein1999} considers joint scale and roughness estimation
for periodic Mat\'ern fields on regular grids. Exact Fourier diagonalization
gives the Fisher information and motivates a conjecture for the joint MLE
limit in dimensions at most three.

\cite[Theorem~2.7]{ChenOwhadiStuart2021} proves consistency of a
likelihood-based regularity estimator for a Mat\'ern-like model sampled on a
regular grid on the \(d\)-dimensional torus, with the remaining parameters
fixed.

For Mat\'ern models with roughness fixed, \cite{Zhang2004} identifies the
estimable variance--range combination, corresponding to the tail scale,
through equivalence of Gaussian measures in dimensions at most three.
\cite{WangLoh2011} studies covariance tapering under fixed-domain
asymptotics, while \cite{KaufmanShaby2013} allows joint maximization over
variance and range. \cite{LiuFengXiao2026} constructs closed-form covariance parameter
estimators for anisotropic tensor-product Ornstein--Uhlenbeck fields and
establishes consistency and asymptotic normality.

\subsection*{Spectral and increment-based estimation}

\cite{LimStein2008} establishes relative asymptotic unbiasedness,
consistency, and asymptotic normality for filtered, smoothed cross-periodograms
under high-frequency assumptions. \cite{Loh2015} estimates roughness from irregularly spaced observations
using higher-order quadratic variations. \cite{WuLim2016} estimates
roughness from regular-grid observations using a bias-reduced periodogram.
\cite{LohSunWen2021} extends estimation by quadratic variations to
Mat\'ern covariance parameters and derives rate bounds. In dimensions
one to three, their procedures consistently estimate both smoothness and
the variance--range microergodic parameter under stratified, randomized
and deformed lattice sampling. \cite{KorteStapffKarvonenMoulines2025}  proves consistency of Gaussian
likelihood smoothness estimators  for Whittle--Mat\'ern processes on closed
Riemannian manifolds under quasi-uniform sampling, including certain
non-Gaussian and misspecified settings. 

The work of \cite{KotharkarStein2025} is closely related the present paper,
establishing joint Gaussian limits for moment estimators of scale and
roughness in two-dimensional Gaussian random fields with power-law
generalized covariances under increasing-domain and fixed-domain
asymptotics. Under fixed-domain sampling, their estimators have the
same convergence-rate orders as our MLEs and, after componentwise
normalization, likewise converge to a singular Gaussian pair with
perfect positive correlation. Their fixed-domain limits follow from
exact self-similarity. They also transfer the moment method to
two-dimensional Mat\'ern random fields with roughness strictly between
zero and one, provided the estimators are constructed using sufficiently
high-order finite differences.

\subsection*{Covariance misspecification}

\cite{bevilacqua2019gw} gives sufficient conditions for equivalence of
Gaussian measures with Mat\'ern and generalized Wendland covariances.
\cite{BevilacquaCaamanoPorcu2022} gives a reparameterization under which
generalized Wendland covariances converge to Mat\'ern covariances and
discusses likelihood results with roughness fixed.

\cite{Karvonen2023} considers noiseless observations of
a fixed function and fits a Mat\'ern Gaussian process with variance and range
held fixed. The question is whether maximum likelihood learns the function's
smoothness. Under
quasi-uniform sampling, the estimated
smoothness cannot asymptotically fall below the function's Sobolev smoothness.

\cite{KarvonenBachoc2025} studies likelihood and cross-validation scale
estimation with oversmoothed Sobolev kernels. It also constructs a consistent
roughness estimator from the growth of scale estimates with sample size.
\cite{Pillai2026} shows that, in dimension four, Mat\'ern range differences
can yield orthogonal Gaussian measures even when the leading spectral tails
agree.

\section{Stationary random fields and spectral assumptions}
\label{sec:sampled-glt}
This section introduces the stationary random field, its covariance function,
and the spectral assumption used throughout the paper.

We consider a real-valued Gaussian random field \(\{Y(\bx):\bx\in\R^d\}\). The field is centered and second-order stationary. Its covariance function \(\rho:\R^d\to\R\) is \(\rho(\bh):=\operatorname{Cov}(Y(\bx+\bh),Y(\bx))\), which is independent of \(\bx\) by stationarity.
We work with covariance functions having a spectral density \(S:\R^d\to[0,\infty)\), using the Fourier convention
\begin{equation}
\label{eq:rho-spectral-representation}
        \rho(\bh)
        =
        \frac{1}{(2\pi)^d}
        \int_{\R^d} e^{\mathrm i\bomega\cdot\bh}S(\bomega)\,d\bomega,
        \qquad \bh\in\R^d .
\end{equation}
The spectral density is even, and hence \(\rho(-\bh)=\rho(\bh)\).
For background on stationary covariance functions and their spectral
representations, see \cite[Chapter~2]{Stein1999}.

The following condition specifies the high-frequency behavior used throughout the asymptotic analysis.
\begin{assumption}
\label{ass:block-tail-detailed}
The spectral density \(S\) is continuous on \(\R^d\), and there exist constants \(\alpha>d/2\),
\(c_\infty>0\), and \(c_2>0\) such that, for every \(\bomega\in\R^d\),
\begin{equation}
\label{eq:block-envelope-detailed}
        0 < S(\bomega)
        \le
        c_2(1+\|\bomega\|^2)^{-\alpha},
\end{equation}
and
\begin{equation}
\label{eq:block-tail-limit-detailed}
        \|\bomega\|^{2\alpha}S(\bomega)
        \longrightarrow
        c_\infty,
        \qquad \|\bomega\|\to\infty .
\end{equation}
\end{assumption}

Since \(\alpha>d/2\), Equation~\eqref{eq:block-envelope-detailed} implies \(S\in L^1(\R^d)\). Consequently, the Fourier representation in Equation~\eqref{eq:rho-spectral-representation} shows that \(\rho\) is bounded and continuous. More generally, for any finite collection of locations \(\bx_1,\ldots,\bx_N\), the matrix \((\rho(\bx_i-\bx_j))_{i,j=1}^N\) is a well-defined covariance matrix.

The Mat\'ern and generalized Wendland covariance families studied later both satisfy Assumption~\ref{ass:block-tail-detailed} under the parameter restrictions stated in Section~\ref{sec:standard-covariance-families}.

\subsection{Regular-grid sampling of a stationary random field}
\label{sec:Rgrid}

Throughout the article, we observe the field on regular grids in the fixed cube \([0,L]^d\), where \(L>0\).  For \(\bn=(n_1,\ldots,n_d)\in\N^d\), with \(n_k\ge2\), let \(\mathcal I_{\bn}:=\{\bi=(i_1,\ldots,i_d)\in\Z^d:0\le i_k\le n_k-1\}\) and \(N_{\bn}:=\prod_{k=1}^d n_k\).  The grid points are
\begin{equation}
\label{eq:grid-detailed}
        \bx_{\bi}^{(\bn)}
        :=
        \left(
        \frac{L}{n_1-1}i_1,\ldots,
        \frac{L}{n_d-1}i_d
        \right),
        \qquad \bi\in\mathcal I_{\bn}.
\end{equation}
Fix a lexicographic ordering of \(\mathcal I_{\bn}\), and let \(Y_{\bn}:=(Y(\bx_{\bi}^{(\bn)}))_{\bi\in\mathcal I_{\bn}}\in\R^{N_{\bn}}\) denote the vector of observations. Because the field is centered and Gaussian,
\begin{equation}
\label{eq:sampled-covariance-matrix}
        Y_{\bn}\sim\mathcal N(0,\Sigma_{\bn}),
        \qquad
        \Sigma_{\bn}
        :=
        \bigl(
        \rho(\bx_{\br}^{(\bn)}-\bx_{\bs}^{(\bn)})
        \bigr)_{\br,\bs\in\mathcal I_{\bn}}
        \in\R^{N_{\bn}\times N_{\bn}} .
\end{equation}
We let \(\bn\to\infty\) mean \(\min_k n_k\to\infty\), so the sample size \(N_{\bn}\) tends to infinity. For  all asymptotic results below, 
we impose isotropic refinement: \(h_{1,\bn}=\cdots=h_{d,\bn}=:h(\bn)\), where \(h_{k,\bn}:=L/(n_k-1)\). 
Because the side length is the same in every coordinate, this condition forces \(n_1=\cdots=n_d=:n\); hence the asymptotic theory concerns cubic grids \(\bn=(n,\ldots,n)\), with \(N_{\bn}=n^d\).
 Consequently  \(\bx_{\br}^{(\bn)}-\bx_{\bs}^{(\bn)}=h(\bn)(\br-\bs)\). We write \(h=h(\bn)\); then \(h\to0\) as the sample size increases.

Following the classical Toeplitz construction in \cite{GrenanderSzego1958} and its multilevel extension in \cite{TyrtyshnikovZamarashkin1998}, we fix the following convention. Let \(Q:=[-\pi,\pi]^d\) and \(\nu_Q(d\bth):=(2\pi)^{-d}\,d\bth\). For \(\phi\in L^1(Q)\), define its Fourier coefficients by \(\widehat\phi_{\boldsymbol{\ell}}:=\int_Q e^{-\mathrm i\boldsymbol{\ell}\cdot\bth}\phi(\bth)\,d\nu_Q(\bth)\), \(\boldsymbol{\ell}\in\Z^d\). The multilevel Toeplitz matrix generated by \(\phi\) is
\[
        T_{\bn}(\phi)
        :=\bigl(\widehat\phi_{\br-\bs}\bigr)_{\br,\bs\in\mathcal I_{\bn}}
        \in\C^{N_{\bn}\times N_{\bn}}.
\]
Under isotropic refinement, the spectral density produces the aliased generating function
\begin{equation}
\label{eq:aliased-symbol-def2}
        \phi_h(\bth)
        :=
        \frac{1}{h^d}\sum_{\bk\in\Z^d}
        S\!\left(\frac{\bth+2\pi\bk}{h}\right),
        \qquad \bth\in Q.
\end{equation}
Under Assumption~\ref{ass:block-tail-detailed}, \(S\in L^1(\R^d)\), and hence \(\phi_h\in L^1(Q)\). Appendix~\ref{app:aliased-symbols} shows that
\[
        \widehat{\phi_h}_{\boldsymbol{\ell}}
        =\rho(-h\boldsymbol{\ell})
        =\rho(h\boldsymbol{\ell}),
        \qquad \boldsymbol{\ell}\in\Z^d.
\]
Together with Equation~\eqref{eq:sampled-covariance-matrix}, this identity gives the Toeplitz representation \(\Sigma_{\bn}=T_{\bn}(\phi_h)\).
The entrywise definition of \(T_{\bn}(\phi)\) also yields the identity used below for quadratic forms. For \(x\in\C^{N_{\bn}}\), write \(p_x(\bth):=\sum_{\br\in\mathcal I_{\bn}}x_{\br}e^{\mathrm i\br\cdot\bth}\). Then, for \(x,y\in\C^{N_{\bn}}\),
\begin{equation}
        x^*T_{\bn}(\phi)y
        =\int_Q\phi(\bth)\overline{p_x(\bth)}p_y(\bth)\,d\nu_Q(\bth).
        \label{eq:toeplitz-quadratic-form}
\end{equation}
The corresponding renormalized generating function is
\begin{equation}
\label{eq:renormalized-aliased-symbol-def}
        \widetilde\phi_h(\bth)
        :=h^{d-2\alpha}\phi_h(\bth)
        =h^{-2\alpha}\sum_{\bk\in\Z^d}
        S\!\left(\frac{\bth+2\pi\bk}{h}\right).
\end{equation}
For every \(\bth\in Q\setminus\{0\}\), Assumption~\ref{ass:block-tail-detailed} yields
\begin{equation}
\label{eq:renormalized-symbol-pointwise-limit}
        \widetilde\phi_h(\bth)
        \longrightarrow
        \widetilde\phi(\bth)
        :=c_\infty\sum_{\bk\in\Z^d}
        \|\bth+2\pi\bk\|^{-2\alpha}.
\end{equation}
{For \(\bth\ne\boldsymbol0\), every \(\bth+2\pi\bk\) is nonzero, and each term \(h^{-2\alpha}S((\bth+2\pi\bk)/h)\) converges to \(c_\infty\|\bth+2\pi\bk\|^{-2\alpha}\).  The spectral envelope bounds it by \(c_2\|\bth+2\pi\bk\|^{-2\alpha}\).  Because \(2\alpha>d\), the series of these bounds over \(\bk\in\mathbb Z^d\) is summable, so dominated convergence with respect to counting measure proves \eqref{eq:renormalized-symbol-pointwise-limit}.}
The candidate limit behaves as \(\|\bth\|^{-2\alpha}\) at the origin.  Because \(2\alpha>d\), it is not in \(L^1(Q)\).
Thus, we pass from classical \(L^1\)-based Toeplitz theory to the broader GLT
framework, where the singular limiting symbol can be identified through
approximating classes of sequences.
\subsection{Generalized Locally Toeplitz sequences}
Let \(\{A_{\bn}\}_{\bn}\) be a matrix sequence, let \(N_{\bn}\) denote
the dimension of \(A_{\bn}\), and let \(a:Q\to\C\) be measurable and finite
almost everywhere. We write
\[
        \{A_{\bn}\}_{\bn}\sim_{\GLT}a
\]
if the sequence belongs to the multilevel GLT class with full symbol
\(\kappa(x,\bth)=a(\bth)\) on \([0,1]^d\times Q\). Thus, throughout this
paper, GLT symbols depend only on frequency. We use the multilevel convention
of \cite{GaroniSerraCapizzanoII}; see also
\cite{GaroniSerraCapizzanoI} for the unilevel theory.

In particular, for every \(F\in C_c(\R)\),
\[
        \frac{1}{N_{\bn}}\sum_{j=1}^{N_{\bn}}
        F\!\left(\sigma_j(A_{\bn})\right)
        \longrightarrow
        \int_Q F\!\left(|a(\bth)|\right)\,d\nu_Q(\bth).
\]
If the matrices are Hermitian, then \(a\) is real-valued
\(\nu_Q\)-almost everywhere and
\[
        \frac{1}{N_{\bn}}\sum_{j=1}^{N_{\bn}}
        F\!\left(\lambda_j(A_{\bn})\right)
        \longrightarrow
        \int_Q F\!\left(a(\bth)\right)\,d\nu_Q(\bth).
\]
GLT symbols are identified up to \(\nu_Q\)-almost-everywhere equality; in
particular, the value assigned at \(\bth=0\) is immaterial. GLT membership is
stronger than either distributional identity: we also use its algebra and its
closure under approximating classes of sequences.

\begin{definition}
\label{def:scaled-glt-symbol}
Let \(s_{\bn}>0\). We say that \(\{A_{\bn}\}_{\bn}\) has scaled GLT
symbol \(a\) under the normalization \(s_{\bn}\), and write
\[
        \{A_{\bn}\}_{\bn}\sim_{\GLT,s_{\bn}}a,
\]
if
$
        \{s_{\bn}A_{\bn}\}_{\bn}\sim_{\GLT}a.
$
\end{definition}

The definition makes the desired conclusion concrete: the covariance matrices
are studied under the normalization \(h^{d-2\alpha}\), with candidate limiting
symbol \(\widetilde\phi\). Since \(\widetilde\phi\notin L^1(Q)\), the GLT
conclusion cannot be obtained from global \(L^1(Q)\) convergence of the
generating functions. The next subsection supplies the required truncation
argument.
\subsection{A truncation principle for singular Toeplitz symbols}
\label{sec:triangular-array-truncation}

The renormalized covariance symbols form a triangular array: at each sample size there is an \(L^1(Q)\) Toeplitz generating function, but the candidate limit need not belong to \(L^1(Q)\).  Thus the fixed-\(L^1(Q)\) Toeplitz GLT identification, applied through global \(L^1(Q)\) convergence of generating functions, does not identify the limiting symbol.

For each annulus \(Q\setminus B_r(0)\), the singularity is absent and the fixed-\(L^1(Q)\) Toeplitz GLT identification applies \cite[Chapter~6, GLT~3, p.~118]{GaroniSerraCapizzanoII}.  The contribution supported in \(B_r(0)\) must instead be controlled directly at the matrix level.  We use the following approximating-class-of-sequences notion; see \cite{SerraCapizzano2003,GaroniSerraCapizzanoI,GaroniSerraCapizzanoII}.

\begin{definition}
\label{def:acs-detailed}
{Throughout this definition and the following truncation argument,
\(\|\cdot\|_2\) denotes the Euclidean norm for vectors and the induced
operator (spectral) norm for matrices; \(\|\cdot\|_1\) and
\(\|\cdot\|_{\mathrm{HS}}\) denote the trace and Hilbert--Schmidt norms,
respectively.}
Let \(B_{\bn}^{(r)}\in\C^{N_{\bn}\times N_{\bn}}\), \(0<r<\pi\).  We write
\[
        \{B_{\bn}^{(r)}\}_{\bn}
        \xrightarrow[\;r\downarrow0\;]{\mathrm{a.c.s.}}0
\]
if, for every \(\varepsilon>0\), there exists \(r_\varepsilon>0\) such that for every \(0<r<r_\varepsilon\) there are matrices \(R_{\bn}^{(r)}\) and \(E_{\bn}^{(r)}\), defined for all sufficiently large \(\bn\), with
\[
        B_{\bn}^{(r)}=R_{\bn}^{(r)}+E_{\bn}^{(r)},
        \qquad
        \frac{\rank R_{\bn}^{(r)}}{N_{\bn}}\le \varepsilon,
        \qquad
        \|E_{\bn}^{(r)}\|_2\le \varepsilon .
\]
If \(A_{\bn}^{(r)},A_{\bn}\in\C^{N_{\bn}\times N_{\bn}}\) and \(\{A_{\bn}-A_{\bn}^{(r)}\}_{\bn}\xrightarrow[r\downarrow0]{\mathrm{a.c.s.}}0\), then \(\{\{A_{\bn}^{(r)}\}_{\bn}\}_{r\downarrow0}\) is called an approximating class of sequences, or a.c.s., for \(\{A_{\bn}\}_{\bn}\).
\end{definition}

We use the standard GLT approximation principle in the following form: if \(\{A_{\bn}^{(r)}\}_{\bn}\sim_{\GLT} a_r\) for every fixed \(r>0\), if \(\{\{A_{\bn}^{(r)}\}_{\bn}\}_{r\downarrow0}\) is an a.c.s. for \(\{A_{\bn}\}_{\bn}\), and if \(a_r\to a\) in measure, then \(\{A_{\bn}\}_{\bn}\sim_{\GLT}a\).  This gives the interface between the annular symbol calculation and the low-frequency matrix estimate.

\begin{theorem}
\label{thm:triangular-array-toeplitz-truncation-detailed}
Let \(f_{\bn}\in L^1(Q)\), and let \(f:Q\to\C\) be measurable and finite
almost everywhere.  For \(r\in(0,\pi)\), set
\[
        f_{\bn}^{(r)}:=f_{\bn}\mathbb I_{Q\setminus B_r(0)},
        \qquad
        f^{(r)}:=f\mathbb I_{Q\setminus B_r(0)} .
\]
Assume that the following two conditions hold.
\begin{enumerate}
\item[\textup{(i)}]
For every fixed \(r\in(0,\pi)\), one has \(f^{(r)}\in L^1(Q)\) and
\[
        f_{\bn}^{(r)}\to f^{(r)}
        \qquad\text{in }L^1(Q).
\]

\item[\textup{(ii)}]
The low-frequency part is negligible in the a.c.s. sense:
\[
        \{T_{\bn}(f_{\bn}-f_{\bn}^{(r)})\}_{\bn}
        \xrightarrow[\;r\downarrow0\;]{\mathrm{a.c.s.}}0 .
\]
\end{enumerate}
Then
\[
        \{T_{\bn}(f_{\bn})\}_{\bn}\sim_{\GLT} f .
\]
\end{theorem}
\begin{proof}
Fix \(r\in(0,\pi)\).  The \(p=1\) trace-norm estimate for multilevel Toeplitz
matrices \cite[Theorem~3.2, p.~47]{GaroniSerraCapizzanoII} gives
\[
 \frac{1}{N_{\bn}}\|T_{\bn}(f_{\bn}^{(r)})-T_{\bn}(f^{(r)})\|_1
 \le \frac{1}{(2\pi)^d}\|f_{\bn}^{(r)}-f^{(r)}\|_{L^1(Q)}\longrightarrow0.
\]
Here \(\|\cdot\|_1\) denotes the trace norm.
By \cite[Theorem~2.3 and Chapter~6, GLT~3, pp.~28 and~118]{GaroniSerraCapizzanoII}, the trace-norm-small difference is zero-distributed and hence GLT-zero, while the fixed \(L^1\) Toeplitz sequence
\(\{T_{\bn}(f^{(r)})\}_{\bn}\) has GLT symbol \(f^{(r)}\).  Therefore
\(\{T_{\bn}(f_{\bn}^{(r)})\}_{\bn}\sim_{\GLT}f^{(r)}\).

Set \(A_{\bn}:=T_{\bn}(f_{\bn})\) and
\(A_{\bn}^{(r)}:=T_{\bn}(f_{\bn}^{(r)})\).  Then
\(A_{\bn}-A_{\bn}^{(r)}=T_{\bn}(f_{\bn}-f_{\bn}^{(r)})\), so
assumption~\textup{(ii)} has exactly the a.c.s. quantifier order required
for \(\{\{A_{\bn}^{(r)}\}_{\bn}\}_{r\downarrow0}\) to approximate
\(\{A_{\bn}\}_{\bn}\).  Finally, \(f^{(r)}\to f\) almost everywhere,
and hence in measure on \(Q\), because \(f\) is finite almost everywhere.
The GLT approximation principle now yields
\(\{A_{\bn}\}_{\bn}\sim_{\GLT}f\).
\end{proof}
The theorem closes the abstract GLT approximation step.  For a spectral
density \(S\) satisfying Assumption~\ref{ass:block-tail-detailed}, the
localized construction below supplies the concrete a.c.s. decomposition
needed to show
\[
        \{\Sigma_{\bn}\}_{\bn}
        \sim_{\GLT,h^{d-2\alpha}}
        \widetilde\phi .
\]
However this conclusion alone does not control the
normalized traces needed for convergence of the likelihood. For Hermitian sequences, ordinary GLT
convergence describes the limiting eigenvalue distribution without controlling
exceptional eigenvalues. For example, let
\(A_N=I_N+Ne_1e_1^*\). Then
\(\{A_N\}_N\sim_{\GLT}1\), whereas
\[
N^{-1}\operatorname{tr}(A_N)=2,
\qquad
N^{-1}\operatorname{tr}(A_N^2)=N+3.
\]
Thus, \(o(N)\) eigenvalues may have magnitudes that are invisible to the GLT
symbol. The trace arguments therefore require quantitative rank and
operator-norm bounds, which the localized compression in the next subsection
provides.

\subsection{Localized Fourier-kernel compression}
\label{sec:localized-fourier-compression}
The function $h^{-2\alpha}S(\,\cdot\,/h)\mathbb I_{B_r(0)}$ is not uniformly
bounded as $h\to0$, and its Toeplitz matrix need not be small in operator
norm.  We instead write its Toeplitz matrix as a low-rank matrix plus
a residual that is small in operator norm; this is the information ordinary
GLT convergence misses when exceptional eigenvalues affect traces.
For \(\bomega\in\R^d\), define the sampled Fourier mode
\begin{equation}
        [v_{\bn}(\bomega)]_{\bi}
        :=
        e^{-\mathrm i\bx_{\bi}^{(\bn)}\cdot\bomega},
        \qquad \bi\in\mathcal I_{\bn}.
        \label{eq:localized-fourier-sampled-mode}
\end{equation}
These modes evaluate the spectral representation
\eqref{eq:rho-spectral-representation} on the grid.

For a real-valued Fourier weight \(w\in L^1(\R^d)\) and \(R>0\), define
the localized Fourier matrix
\begin{equation}
        \mathcal P_{\bn}^{(R)}(w)
        :=
        \frac1{(2\pi)^d}
        \int_{\|\bomega\|<R}
        w(\bomega)v_{\bn}(\bomega)v_{\bn}(\bomega)^*
        \,d\bomega.
        \label{eq:localized-fourier-matrix}
\end{equation}
For the principal alias, the change of variables \(\bth=h\bomega\) gives
the exact identity
\[
        T_{\bn}\!\left(
        h^{-2\alpha}S(\,\cdot\,/h)\mathbb I_{B_r(0)}
        \right)
        =
        h^{d-2\alpha}\mathcal P_{\bn}^{(r/h)}(S).
\]
The ball \(B_r(0)\) therefore corresponds to \(\|\bomega\|<r/h\).
The following interface is enough here; its explicit Taylor construction is
given in Appendix~\ref{app:preliminary-main-paper-proofs}.

\begin{lemma}
\label{lem:localized-fourier-kernel-compression}
There is a constant \(C=C(d,L)\) such that, for every \(R>0\),
\(0<\varepsilon<1\), and sampling index \(\bn\), there is a subspace
\(\mathcal U_{\bn}(R,\varepsilon)\), independent of \(w\), with
\[
 \dim\mathcal U_{\bn}(R,\varepsilon)
 \le C\{1+R+\log(\varepsilon^{-1})\}^d.
\]
For every real \(w\in L^1(\R^d)\), there are Hermitian matrices
\(H_{\bn}^{(R)}(w)\) and \(F_{\bn}^{(R)}(w)\) such that
\begin{align*}
 \mathcal P_{\bn}^{(R)}(w)&=H_{\bn}^{(R)}(w)+F_{\bn}^{(R)}(w),\\
 \operatorname{ran}H_{\bn}^{(R)}(w)&\subseteq\mathcal U_{\bn}(R,\varepsilon),
 &\|F_{\bn}^{(R)}(w)\|_2&\le CN_{\bn}\varepsilon\|w\|_{L^1(\R^d)}.
\end{align*}
\end{lemma}

The compression lemma applies directly to the localized principal alias.  The following theorem records the resulting decomposition and bounds for every spectral density covered by the assumption.

\begin{theorem}
\label{thm:quantitative-sampled-covariance-glt}
Let \(S\) be a spectral density satisfying
Assumption~\ref{ass:block-tail-detailed}, and let
\(\Sigma_{\bn}\) be the sampled covariance matrix in
\eqref{eq:sampled-covariance-matrix} on the isotropic grid described in
Section~\ref{sec:Rgrid}.  Let \(\widetilde\phi\) be the a.e. finite
symbol defined in \eqref{eq:renormalized-symbol-pointwise-limit}, with
an arbitrary finite value assigned at \(\bth=0\).  Then
\[
        \{\Sigma_{\bn}\}_{\bn}
        \sim_{\GLT,h^{d-2\alpha}}
        \widetilde\phi .
\]
{Moreover, there exist constants \(m_{\mathrm{reg}}>0\), \(C>0\), and \(h_0>0\), depending only on \(S\), \(d\), and \(L\), such that, for every fixed \(r\in(0,\pi)\) and every sampling index with \(h\le h_0\), there are matrices}
\(\Sigma_{\bn}^{\mathrm{reg},r}\), \(H_{\Sigma,\bn}^{(r)}\), and
\(F_{\Sigma,\bn}^{(r)}\) such that
\[
h^{d-2\alpha}\Sigma_{\bn}
=\Sigma_{\bn}^{\mathrm{reg},r}+H_{\Sigma,\bn}^{(r)}+F_{\Sigma,\bn}^{(r)}.
\]
The regular core has GLT symbol \(\widetilde\phi^{(r)}\), where
\(\widetilde\phi^{(r)}\to\widetilde\phi\) in measure, and satisfies
\begin{align*}
 m_{\mathrm{reg}}I\preceq\Sigma_{\bn}^{\mathrm{reg},r},\qquad
 \|\Sigma_{\bn}^{\mathrm{reg},r}\|_2&\le C(1+r^{-2\alpha}),\\
 \frac{\rank H_{\Sigma,\bn}^{(r)}}{N_{\bn}}
 &\le C\{r^d+h^d\log^d(h^{-1})\},
 &\|F_{\Sigma,\bn}^{(r)}\|_2&\le Ch.
\end{align*}
\end{theorem}
{The proof is given in Appendix~\ref{app:preliminary-main-paper-proofs}.}
The decomposition should be read as an exact separation followed by two
limiting approximations. For fixed \(r\), the regular core removes the
principal contribution from frequencies in \(B_r(0)\), leaving a matrix that
is uniformly positive and bounded in operator norm. The matrix
\(H_{\Sigma,\bn}^{(r)}\) recovers this low-frequency contribution up to a
compression: it may contain large eigenvalues, but it acts on an
asymptotically negligible fraction of directions as \(N_{\bn}\to\infty\) and
then \(r\downarrow0\). The remainder \(F_{\Sigma,\bn}^{(r)}\) may act in every
direction, but its operator norm vanishes as the sample size increases. At the
symbol level, \(\widetilde\phi^{(r)}\to\widetilde\phi\) in measure as
\(r\downarrow0\). Thus the singular low-frequency behavior is confined to few
directions, while the error outside those directions is uniformly small; these
are precisely the quantitative controls that ordinary GLT convergence does not
retain.

We next bound
\(\|(h^{d-2\alpha}\Sigma_{\bn})^{-1}\|_2\), which is needed for the likelihood analysis.

\begin{lemma}
\label{lem:renorm-glt-inversion}
Under Assumption~\ref{ass:block-tail-detailed} the renormalized sampled
covariance is uniformly invertible after rescaling: with
\(m:=\tfrac12 c_\infty\allowbreak(\pi\sqrt d)^{-2\alpha}>0\),
\[
        \lambda_{\min}\bigl(h^{d-2\alpha}\Sigma_{\bn}\bigr)\ge m
        \qquad\text{for all sufficiently small \(h\).}
\]
\end{lemma}
{The proof is given in Appendix~\ref{app:preliminary-main-paper-proofs}.}
Together, Theorem~\ref{thm:quantitative-sampled-covariance-glt} and
Lemma~\ref{lem:renorm-glt-inversion} provide the matrix estimates needed for
the likelihood analysis. Under the differentiated spectral-tail assumption
introduced in the next section, corresponding estimates for the covariance
derivatives yield the normalized trace and quadratic-trace limits that enter
the likelihood score and the limiting information matrix.

\section{Maximum likelihood in the fitted model}
\label{sec:maximum-likelihood}

We now study maximum likelihood estimation of the tail scale \(c_\infty\)
and tail exponent \(\alpha\) from Assumption~\ref{ass:block-tail-detailed}
within a parametric Gaussian family. We use the parameterization
\[
        \varphi=(\tau,\alpha),
        \qquad
        \tau=\log c_\infty,
        \qquad
        \varphi\in\R\times(d/2,\infty).
\]

We observe the field on the isotropically refined grid introduced in
Section~\ref{sec:Rgrid}. For \(\varphi=(\tau,\alpha)\), let
\(S_\varphi\) denote the fitted spectral density and let \(\rho_\varphi\)
be its covariance function under the Fourier convention in
Equation~\eqref{eq:rho-spectral-representation}. The corresponding sampled
covariance matrix is
\[
        \Sigma_{\varphi,\bn}
        :=
        \bigl(
        \rho_\varphi(
        \bx_{\br}^{(\bn)}-\bx_{\bs}^{(\bn)})
        \bigr)_{\br,\bs\in\mathcal I_{\bn}},
\]
and we write
\(P_{\varphi,\bn}:=\mathcal N(0,\Sigma_{\varphi,\bn})\) for the associated
Gaussian law.

We assume that the model is correctly specified. Thus, for some true
parameter \(\varphi_0=(\tau_0,\alpha_0)\), the sampled data vector satisfies
\(Y_{\bn}\sim P_{\varphi_0,\bn}\). Up to an additive term independent of
\(\varphi\), its Gaussian log-likelihood is
\begin{equation}
        \ell_{\bn}(\varphi)
        =
        -\frac12\log\det\Sigma_{\varphi,\bn}
        -\frac12Y_{\bn}^{T}
        \Sigma_{\varphi,\bn}^{-1}Y_{\bn}.
        \label{eq:working-gaussian-loglik}
\end{equation}

Likelihood-based inference is governed by the local likelihood
geometry around \(\varphi_0\). Following the LAN framework of
\cite[Chapter~6, Section~6.2]{LeCamYang2000}, we seek a
sample-size-dependent local chart whose fitted likelihood ratio has a uniform
nondegenerate quadratic expansion, with a Gaussian linear part and
deterministic positive-definite curvature. The compensated chart, score, and
information are derived below. The score limit established below identifies the Gaussian linear term.
Section~\ref{sec:tail-estimator-consistency} establishes global likelihood
separation and compact LAN, while Appendix~\ref{app:likelihood-separation-and-lan} supplies the additional
localization argument needed to prove that
{\(\Delta_{\bn}^{-1}(\widehat\varphi_{\bn}-\varphi_0)\) is tight.} Together, these ingredients yield the asymptotic
normality of the maximum likelihood estimator stated below.

To establish LAN, we must show that these log-likelihood ratios admit a
nondegenerate quadratic expansion, uniformly for \(u\) in bounded subsets of
\(\R^2\), with an asymptotically Gaussian linear term and deterministic
curvature. Since \(\varphi\) enters the likelihood through
\(\Sigma_{\varphi,\bn}\), whose entries are determined by \(S_\varphi\), we
impose the required differentiability and remainder control directly on the
{spectral family.} The following assumption makes the dependence on
\(\tau\) explicit and controls differentiation with respect to \(\alpha\).
\begin{assumption}
\label{ass:differentiated-spectral-tail-family}
There is an open neighborhood \(\mathcal U\) of \(\varphi_0\) such that the
fitted spectral density has the scale-family representation
\[
        S_{\tau,\alpha}(\bomega)
        =e^\tau\overline S_{\alpha}(\bomega),
\]
where \(\overline S_{\alpha}\) is even, strictly positive, and satisfies
Assumption~\ref{ass:block-tail-detailed} uniformly on \(\mathcal U\), with
leading constant one. Moreover, for some \(\varepsilon>0\),
\[
        \overline S_{\alpha}(\bomega)
        =(1+\|\bomega\|^2)^{-\alpha}
        +S^{\mathrm{rem}}_{\alpha}(\bomega),
\]
where, for each \(\bomega\), the remainder is continuously differentiable
in \(\alpha\), and
the map
\((\alpha,\bomega)\mapsto S^{\mathrm{rem}}_{\alpha}(\bomega)\)
is jointly continuous.
Uniformly over the \(\alpha\)-projection of \(\mathcal U\),
\begin{align}
        |S^{\mathrm{rem}}_{\alpha}(\bomega)|
        &\le C(1+\|\bomega\|^2)^{-\alpha-\varepsilon},
        \label{eq:spectral-tail-remainder-bound}\\
        |\partial_\alpha S^{\mathrm{rem}}_{\alpha}(\bomega)|
        &\le C(1+\|\bomega\|^2)^{-\alpha-\varepsilon}
        \{1+\log(2+\|\bomega\|)\}.
        \label{eq:spectral-tail-alpha-remainder-bound}
\end{align}
\end{assumption}
We now restrict likelihood optimization to a compact rectangle contained in
\(\mathcal U\):
\begin{equation}
        \Theta_{\mathrm{tail}}
        :=[\tau_-,\tau_+]\times[\alpha_-,\alpha_+]
        \Subset\mathcal U.
        \label{eq:working-tail-chart}
\end{equation}
We assume that
\(\varphi_0=(\tau_0,\alpha_0)\in
\operatorname{int}(\Theta_{\mathrm{tail}})\).
Compactness of \(\Theta_{\mathrm{tail}}\) and continuity of
\(\ell_{\bn}\) ensure that its argmax is nonempty. The measurable maximum
theorem provides a measurable selection \cite{BrownPurves1973}, which we
denote by
\begin{equation}
        \widehat\varphi_{\bn}
        =(\widehat\tau_{\bn},\widehat\alpha_{\bn})
        \in\argmax_{\varphi\in\Theta_{\mathrm{tail}}}
        \ell_{\bn}(\varphi).
        \label{eq:working-tail-mle}
\end{equation}

The following result provides the global localization needed for the
subsequent LAN analysis.

\begin{theorem}
\label{thm:tail-chart-consistency}
Suppose Assumption~\ref{ass:differentiated-spectral-tail-family} holds.
Then, under \(P_{\varphi_0,\bn}\), every measurable maximizer in
Equation~\eqref{eq:working-tail-mle} satisfies
\[
        \widehat\varphi_{\bn}
        \longrightarrow\varphi_0
        \qquad\text{in probability}.
\]
\end{theorem}
\subsection{LAN and asymptotic normality}
Consistency localizes the maximizer near \(\varphi_0\), but does not determine
its rate. We now derive the compensated local parameterization and information
matrix that govern the LAN expansion and the limiting distribution of
\(\widehat\varphi_{\bn}\).

We will now use the GLT symbols to derive the asymptotic behavior of the likelihood score and information matrix.
For each fixed
\(\varphi=(\tau,\alpha)\in\Theta_{\mathrm{tail}}\),
Theorem~\ref{thm:quantitative-sampled-covariance-glt}, applied to
\(S_\varphi\), gives
\[
        \{\Sigma_{\varphi,\bn}\}_{\bn}
        \sim_{\GLT,h^{d-2\alpha}}
        e^\tau\psi_\alpha,
\]
where
$
        \psi_\alpha(\bth)
        :=
        \sum_{\bk\in\mathbb Z^d}
        \|\bth+2\pi\bk\|^{-2\alpha}
$
for $\bth\in Q\setminus\{0\}$. 

For \(j\in\{\tau,\alpha\}\), differentiation of
Equation~\eqref{eq:working-gaussian-loglik} gives
\[
        \partial_j\ell_{\bn}(\varphi)
        =
        -\frac12
        \operatorname{tr}\!\left(
        \Sigma_{\varphi,\bn}^{-1}
        \partial_j\Sigma_{\varphi,\bn}
        \right)
        +
        \frac12Y_{\bn}^{T}
        \Sigma_{\varphi,\bn}^{-1}
        (\partial_j\Sigma_{\varphi,\bn})
        \Sigma_{\varphi,\bn}^{-1}Y_{\bn}.
\]
At \(\varphi_0\), set
$
        Z_{\bn}:=\Sigma_{\varphi_0,\bn}^{-1/2}Y_{\bn},$
        and $
        G_{j,\bn}
        :=
        \Sigma_{\varphi_0,\bn}^{-1/2}
        \left.
        (\partial_j\Sigma_{\varphi,\bn})
        \right|_{\varphi=\varphi_0}
        \Sigma_{\varphi_0,\bn}^{-1/2}.
$
Then
\[
        \partial_j\ell_{\bn}(\varphi_0)
        =
        \frac12
        \left\{
        Z_{\bn}^{T}G_{j,\bn}Z_{\bn}
        -\operatorname{tr}G_{j,\bn}
        \right\}.
\]
Under \(P_{\varphi_0,\bn}\),
\(Z_{\bn}\sim\mathcal N(0,I_{N_{\bn}})\). Thus each score component is a
centered Gaussian quadratic form. Its limiting behavior is controlled by
normalized trace and operator-norm estimates for \(G_{j,\bn}\). Writing \(a_{\bn}:=2\log(1/h)\), the scale representation and the
differentiated GLT analysis gives
\begin{align*}
        G_{\tau,\bn}
        =I_{N_{\bn}}, \quad
        G_{\alpha,\bn} =-a_{\bn}I_{N_{\bn}}+R_{\alpha,\bn},
\end{align*}
where
$
        R_{\alpha,\bn}
        :=G_{\alpha,\bn}+a_{\bn}I_{N_{\bn}},
$
which Toeplitz representation is given in Appendix~\ref{app:moving-tail-spectral-proofs}
. The GLT yields
$
        \{R_{\alpha,\bn}\}_{\bn}
        \sim_{\GLT}r_{\alpha_0}
$
where
\begin{equation}
        r_\alpha(\bth)
        :=
        \partial_\alpha\log\psi_\alpha(\bth)
        =
        -
        \frac{
        \displaystyle\sum_{\bk\in\mathbb Z^d}
        \log\!\bigl(\|\bth+2\pi\bk\|^2\bigr)
        \|\bth+2\pi\bk\|^{-2\alpha}}
        {\displaystyle\sum_{\bk\in\mathbb Z^d}
        \|\bth+2\pi\bk\|^{-2\alpha}},
        \qquad
        \bth\in Q\setminus\{0\}.
        \label{eq:ralpha}
\end{equation}

The term \(-a_{\bn}I_{N_{\bn}}\) in \(G_{\alpha,\bn}\) is proportional to
the scale direction \(G_{\tau,\bn}\). It is therefore removed by the
compensated score
\[
        \partial_\alpha\ell_{\bn}(\varphi_0)
        +a_{\bn}\partial_\tau\ell_{\bn}(\varphi_0)
        =
        \frac12
        \left\{
        Z_{\bn}^{T}R_{\alpha,\bn}Z_{\bn}
        -\operatorname{tr}R_{\alpha,\bn}
        \right\}.
\]

To encode this compensation in the local parameterization, write
\(u=(u_\tau,u_\alpha)^T\), where \(u_\tau\) indexes a perturbation of
\(\tau\) alone, whereas \(u_\alpha\) indexes a perturbation of \(\alpha\)
accompanied by a change in \(\tau\) at relative rate \(a_{\bn}\).
The corresponding directions in the \((\tau,\alpha)\)-coordinates are
\((1,0)^T\) and \((a_{\bn},1)^T\), respectively, and the latter produces
the compensated score above. We therefore define the sample-size-dependent local parameter chart
\begin{equation}
        \Delta_{\bn}
        :=
        \frac{1}{\sqrt{N_{\bn}}}
        \begin{pmatrix}
                1 & a_{\bn}\\
                0 & 1
        \end{pmatrix}.
        \label{eq:tail-rate-matrix}
\end{equation}
{For local coordinates \(u=(u_\tau,u_\alpha)^T\), this chart gives}
\[
        \varphi_{\bn}(u)
        =
        \varphi_0+\Delta_{\bn}u
        =
        \begin{pmatrix}
        \tau_0+N_{\bn}^{-1/2}(u_\tau+a_{\bn}u_\alpha)\\
        \alpha_0+N_{\bn}^{-1/2}u_\alpha
        \end{pmatrix},
        \qquad
        u=(u_\tau,u_\alpha)^T.
\]

The score associated with this local chart is
\[
        \mathcal S_{\bn}
        :=
        \Delta_{\bn}^{T}\nabla_\varphi
        \ell_{\bn}(\varphi_0)
        =
        \frac{1}{2\sqrt{N_{\bn}}}
        \begin{pmatrix}
        Z_{\bn}^{T}Z_{\bn}-N_{\bn}\\
        Z_{\bn}^{T}R_{\alpha,\bn}Z_{\bn}
        -\operatorname{tr}R_{\alpha,\bn}
        \end{pmatrix}.
\]
Let
\(m_1:=\int_Q r_{\alpha_0}\,d\nu_Q\),
\(m_2:=\int_Q r_{\alpha_0}^2\,d\nu_Q\), and
\(v_{\alpha_0}:=m_2-m_1^2\), and define
\begin{equation}
        J:=\frac12
        \begin{pmatrix}
        1&m_1\\
        m_1&m_2
        \end{pmatrix}.
        \label{eq:tail-information-matrix}
\end{equation}
{The principal alias gives
\(r_{\alpha_0}(\bth)+\log\|\bth\|^2\to0\) near the origin, whereas the
shifted aliases remain bounded; thus the logarithmic singularity belongs to
\(L^2(\nu_Q)\), and strict Cauchy--Schwarz gives
\(v_{\alpha_0}>0\) and \(\det(J)=v_{\alpha_0}/4>0\).}
Under \(P_{\varphi_0,\bn}\), the Gaussian quadratic-form covariance identity
and the GLT trace limits give
\[
        \operatorname{Cov}_{\varphi_0}(\mathcal S_{\bn})
        =
        \frac{1}{2N_{\bn}}
        \begin{pmatrix}
        N_{\bn} & \operatorname{tr}R_{\alpha,\bn}\\
        \operatorname{tr}R_{\alpha,\bn}
        & \operatorname{tr}(R_{\alpha,\bn}^2)
        \end{pmatrix}
        \longrightarrow J.
\]

\begin{theorem}
\label{thm:tail-exponent-gradient-information}
Suppose Assumption~\ref{ass:differentiated-spectral-tail-family} holds.
Then, under \(P_{\varphi_0,\bn}\),
\[
        \mathcal S_{\bn}\Rightarrow\mathcal N(0,J).
\]
\end{theorem}

Thus \(J\) is positive definite.
The score limit above and the uniform LAN expansion in
Section~\ref{sec:tail-estimator-consistency}, together with consistency and
the Appendix~\ref{app:likelihood-separation-and-lan} localization argument, yield the main theorem.

\begin{theorem}
\label{thm:working-tail-mle-normality}
Suppose Assumption~\ref{ass:differentiated-spectral-tail-family} holds.
Then, under \(P_{\varphi_0,\bn}\), every measurable maximizer in
Equation~\eqref{eq:working-tail-mle} satisfies
\begin{equation}
        \Delta_{\bn}^{-1}
        (\widehat\varphi_{\bn}-\varphi_0)
        \Rightarrow
        \mathcal N(0,J^{-1}).
        \label{eq:working-tail-mle-normality}
\end{equation}
\end{theorem}

\begin{corollary}
\label{cor:cubic-grid-componentwise-rates}
Under the assumptions of Theorem~\ref{thm:working-tail-mle-normality}, suppose
\(\bn=(n,\ldots,n)\), so that \(N_{\bn}=n^d\) and
\(h=L/(n-1)\). Define
\(\widehat c_{\infty,\bn}:=\exp(\widehat\tau_{\bn})\). Then 
\[
\begin{pmatrix}
\sqrt{N_{\bn}}(\widehat\alpha_{\bn}-\alpha_0)\\[4pt]
\displaystyle
\frac{d\sqrt{N_{\bn}}}{2\log N_{\bn}}
\frac{\widehat c_{\infty,\bn}-c_{\infty,0}}{c_{\infty,0}}
\end{pmatrix}
\Rightarrow
\begin{pmatrix}
Z_\alpha\\ Z_\alpha
\end{pmatrix},
\qquad
Z_\alpha\sim\mathcal N(0,\sigma_{\alpha_0}^2),
\]
where
\(
\sigma_{\alpha_0}^2=(J^{-1})_{22}=2/v_{\alpha_0}
\).
\end{corollary}

 The same likelihood analysis gives sharper rates when either tail coordinate
is known. Define the corresponding restricted estimators by
\[
        \widehat\alpha_{\bn}^{(\tau_0)}
        \in
        \argmax_{\alpha\in[\alpha_-,\alpha_+]}
        \ell_{\bn}(\tau_0,\alpha),
        \qquad
        \widehat\tau_{\bn}^{(\alpha_0)}
        \in
        \argmax_{\tau\in[\tau_-,\tau_+]}
        \ell_{\bn}(\tau,\alpha_0).
\]

\begin{corollary}
\label{cor:known-tail-parameter-mles}
\label{cor:known-tail-coefficient-alpha-mle}
\label{cor:known-tail-exponent-coefficient-mle}
Suppose Assumption~\ref{ass:differentiated-spectral-tail-family} holds.
Then, under \(P_{\varphi_0,\bn}\), 
\begin{align}
        \sqrt{N_{\bn}}a_{\bn}
        \bigl(\widehat\alpha_{\bn}^{(\tau_0)}-\alpha_0\bigr)
        &\Rightarrow \mathcal N(0,2),
        \label{eq:known-tail-coefficient-alpha-mle}\\
        \sqrt{N_{\bn}}
        \bigl(\widehat\tau_{\bn}^{(\alpha_0)}-\tau_0\bigr)
        &\Rightarrow \mathcal N(0,2).
        \label{eq:known-tail-exponent-coefficient-mle}
\end{align}
\end{corollary}

Both restricted assertions are proved together in Appendix~\ref{app:likelihood-separation-and-lan}.
Corollary~\ref{cor:cubic-grid-componentwise-rates} gives the joint componentwise limit in arbitrary dimension. 
The following example records its two-dimensional specialization together with the restricted rates.
\begin{example}
\label{ex:regular-square-grid-tail-rates}
Consider the regular square grid \(\bn=(n,n)\) in dimension \(d=2\), with
fixed side length \(L>0\). Then \(N_{\bn}=n^2\) and \(h=L/(n-1)\), so
\[
        a_{\bn}
        =2\log\!\left(\frac{n-1}{L}\right)
        =\log N_{\bn}-2\log L+o(1).
\]
In particular, \(a_{\bn}/\log N_{\bn}\to1\). Define
\(\widehat c_{\infty,\bn}:=\exp(\widehat\tau_{\bn})\).
Theorem~\ref{thm:working-tail-mle-normality} and the delta method give
\begin{equation}
        \begin{pmatrix}
        \dfrac{\sqrt{N_{\bn}}}{\log N_{\bn}}
        \dfrac{\widehat c_{\infty,\bn}-c_{\infty,0}}
              {c_{\infty,0}}
        \\[6pt]
        \sqrt{N_{\bn}}
        (\widehat\alpha_{\bn}-\alpha_0)
        \end{pmatrix}
        \Rightarrow
        \begin{pmatrix}
        Z_\alpha\\
        Z_\alpha
        \end{pmatrix},
        \qquad
        Z_\alpha\sim
        \mathcal N\!\left(0,\bigl(J^{-1}\bigr)_{22}\right).
        \label{eq:working-tail-mle-readable-limit}
\end{equation}
Thus, under joint estimation,
\[
        \widehat\alpha_{\bn}-\alpha_0
        =O_{\mathbb P}(N_{\bn}^{-1/2}),
        \qquad
        \widehat\tau_{\bn}-\tau_0
        =O_{\mathbb P}\!\left(
        \frac{\log N_{\bn}}{\sqrt{N_{\bn}}}
        \right).
\]
The latter is also the order of the relative error in
\(\widehat c_{\infty,\bn}\).

For the restricted estimators,
Corollary~\ref{cor:known-tail-parameter-mles} gives
\[
        \sqrt{N_{\bn}}\log N_{\bn}
        \bigl(\widehat\alpha_{\bn}^{(\tau_0)}-\alpha_0\bigr)
        \Rightarrow\mathcal N(0,2),
        \qquad
        \sqrt{N_{\bn}}
        \bigl(\widehat\tau_{\bn}^{(\alpha_0)}-\tau_0\bigr)
        \Rightarrow\mathcal N(0,2).
\]
Consequently, knowing the tail coefficient improves the exponent error from
\(O_{\mathbb P}(N_{\bn}^{-1/2})\) to
\(O_{\mathbb P}\{(\sqrt{N_{\bn}}\log N_{\bn})^{-1}\}\). Similarly, knowing
the tail exponent improves the log-coefficient error from
\(O_{\mathbb P}\{(\log N_{\bn})/\sqrt{N_{\bn}}\}\) to
\(O_{\mathbb P}(N_{\bn}^{-1/2})\).
\end{example}

\section{Fitted likelihood: separation, LAN, and localization}
\label{sec:tail-estimator-consistency}

Theorem~\ref{thm:working-tail-mle-normality} and
Corollary~\ref{cor:known-tail-parameter-mles} state the limiting
distributions of the maximum likelihood estimators under the fitted model.
The score limit supplies the Gaussian term in these distributions. By itself,
however, it neither controls the likelihood away from \(\varphi_0\) nor
ensures that the maximizer remains on the local scale.

The argument therefore has a global and a local stage. The global stage
establishes uniform separation outside every fixed neighborhood of
\(\varphi_0\). The local stage combines compact LAN with a separate
localization argument to prove that
{\(\Delta_{\bn}^{-1}(\widehat\varphi_{\bn}-\varphi_0)\) is tight.} Throughout this section, the fitted spectrum satisfies
Assumption~\ref{ass:differentiated-spectral-tail-family}, and the compact tail
chart and parameter conditions in
Equation~\eqref{eq:working-tail-chart} remain in force.

\subsection{Global likelihood separation}
We begin with the global stage. For each fixed \(\varepsilon>0\), separation
must hold uniformly over all parameters satisfying
\(\|\varphi-\varphi_0\|\ge\varepsilon\). Since these alternatives may vary
with \(\bn\), pointwise separation is not sufficient. We obtain the required
uniformity by expressing the Gaussian affinity in terms of a relative
covariance matrix and separating changes in scale from changes in spectral
shape.

For \(\varphi\in\Theta_{\mathrm{tail}}\), write
\(P_{\varphi,\bn}:=\mathcal N(0,\Sigma_{\varphi,\bn})\). We compare this law
with \(P_{\varphi_0,\bn}\) through the relative covariance matrix
\[
 V_{\varphi,\bn}
 :=
 \Sigma_{\varphi_0,\bn}^{-1/2}
 \Sigma_{\varphi,\bn}
 \Sigma_{\varphi_0,\bn}^{-1/2}.
\]
The scalar
$
 g_{\bn}(\varphi)
 :=
 (\tau-\tau_0)-a_{\bn}(\alpha-\alpha_0)
 =
 \log\left(
 \frac{e^\tau h^{2\alpha-d}}
      {e^{\tau_0}h^{2\alpha_0-d}}
 \right)
$
records the change in the leading covariance factor. Thus
\(g_{\bn}(\varphi)=0\) precisely when
\[
 e^\tau h^{2\alpha-d}
 =
 e^{\tau_0}h^{2\alpha_0-d}.
 \]
The tangent direction to this level set is \((a_{\bn},1)^T\), which is the
compensated direction introduced through {\(\Delta_{\bn}\)} in
Section~\ref{sec:maximum-likelihood}.

To describe the remaining change in spectral shape, let
\(\mathsf K_{\delta,\bn}\), for \(\delta\ne0\), denote the covariance secant
defined in Appendix~\ref{app:moving-tail-spectral-proofs} after whitening by
\(\Sigma_{\varphi_0,\bn}^{-1/2}\), and set
\(\mathsf K_{0,\bn}:=R_{\alpha,\bn}\). Taking
\(\delta=\alpha-\alpha_0\) gives the exact factorization
\begin{equation}
 V_{\varphi,\bn}
 =
 e^{g_{\bn}(\varphi)}
 \{I+(\alpha-\alpha_0)\mathsf K_{\alpha-\alpha_0,\bn}\}.
 \label{eq:working-full-relative-covariance-secant}
\end{equation}
The scalar factor \(e^{g_{\bn}(\varphi)}\) describes the change in overall
covariance magnitude, whereas the matrix factor describes the change in
spectral shape. The separation argument must control both components.

For centered Gaussian laws, the standard determinant formula for the
Hellinger affinity \cite{Kailath1967} gives
\begin{equation}
 \mathcal A(P_{\varphi_0,\bn},P_{\varphi,\bn})
 =
 \frac{\det(V_{\varphi,\bn})^{1/4}}
 {\det\{(I+V_{\varphi,\bn})/2\}^{1/2}}.
 \label{eq:gaussian-affinity-relative-covariance}
\end{equation}
The GLT limits for the two covariance components yield the required uniform
bound.

\begin{proposition}
\label{prop:working-glt-hellinger-separation}
There is \(c>0\) such that, for all sufficiently large sample sizes and all
\(\varphi\in\Theta_{\mathrm{tail}}\),
\[
 -\log\mathcal A(P_{\varphi_0,\bn},P_{\varphi,\bn})
 \ge cN_{\bn}
 \frac{g_{\bn}(\varphi)^2+(\alpha-\alpha_0)^2}
 {1+g_{\bn}(\varphi)^2+(\alpha-\alpha_0)^2}.
\]
\end{proposition}

In particular, for every \(\varepsilon>0\), there is
\(c_\varepsilon>0\) such that
\[
 \inf_{\substack{\varphi\in\Theta_{\mathrm{tail}}\\
                  \|\varphi-\varphi_0\|\ge\varepsilon}}
 \{-\log\mathcal A(P_{\varphi_0,\bn},P_{\varphi,\bn})\}
 \ge c_\varepsilon\frac{N_{\bn}}{a_{\bn}^2}\longrightarrow\infty.
\]
Let \(U\) be a fixed neighborhood of \(\varphi_0\) relative to
\(\Theta_{\mathrm{tail}}\). The preceding bound separates \(\varphi_0\)
uniformly from \(\Theta_{\mathrm{tail}}\setminus U\) at rate
\(N_{\bn}/a_{\bn}^2\). Although \(U\) is fixed, a potential maximizer in its
complement may vary with \(\bn\). The proof of
Theorem~\ref{thm:tail-chart-consistency} in Appendix~\ref{app:likelihood-separation-and-lan} therefore covers the
entire complement by a finite parameter grid \(\mathcal T_{\bn}\) satisfying
\[
 \log|\mathcal T_{\bn}|=O(a_{\bn}),
 \qquad
 \frac{N_{\bn}}{a_{\bn}^2}\gg\log|\mathcal T_{\bn}|.
\]
A modulus bound for \(\ell_{\bn}\) ensures that any likelihood improvement
outside \(U\) is reproduced, up to a negligible error, at a point of
\(\mathcal T_{\bn}\). At each grid point, the pairwise affinity inequality
controls the probability of such an improvement. The displayed comparison
then allows a union bound over \(\mathcal T_{\bn}\).

Thus, with probability tending to one, no maximizer lies outside \(U\).
Because \(U\) is arbitrary, this gives the global consistency step underlying
Theorem~\ref{thm:tail-chart-consistency}. Determining the rate of the
maximizer requires the local analysis developed next.

\subsection{Local expansion and estimator localization}

We now return to the local chart \(\varphi_{\bn}(u)\) defined in
Section~\ref{sec:maximum-likelihood}. Its coordinates satisfy
\[
 u_\tau
 =
 \sqrt{N_{\bn}}\,g_{\bn}\{\varphi_{\bn}(u)\},
 \qquad
 u_\alpha
 =
 \sqrt{N_{\bn}}(\alpha-\alpha_0).
\]
Thus, \(u_\tau\) measures displacement across the level set
\(g_{\bn}=0\), whereas \(u_\alpha\) is the coordinate in its compensated
tangent direction.

The feasible domain of the local chart is
$
 U_{\bn}
 :=
 \{u:\varphi_{\bn}(u)\in\Theta_{\mathrm{tail}}\}.
$
Because \(\varphi_0\) is interior and
\(\|\Delta_{\bn}\|_2=O(a_{\bn}/\sqrt{N_{\bn}})\to0\), every fixed \(u\)-ball
is contained in \(U_{\bn}\) for all sufficiently large sample sizes. For
\(u\in U_{\bn}\), define the local log-likelihood ratio by
\[
 \Lambda_{\bn}(u)
 :=
 \ell_{\bn}\{\varphi_{\bn}(u)\}-\ell_{\bn}(\varphi_0).
\]
The next theorem gives a uniform quadratic expansion of
\(\Lambda_{\bn}\) on each fixed \(u\)-ball.

\begin{theorem}
\label{thm:compact-lan}
Suppose Assumption~\ref{ass:differentiated-spectral-tail-family} holds.
Then, under \(P_{\varphi_0,\bn}\), for every \(M<\infty\),
\[
 \sup_{\|u\|\le M}
 \left|
 \Lambda_{\bn}(u)
 -u^T\mathcal S_{\bn}
 +\frac12u^TJu
 \right|
 \xrightarrow{\mathbb P}0.
\]
\end{theorem}

Completing the quadratic form in
Theorem~\ref{thm:compact-lan} identifies its center. Write
\[
 u_{\bn}^{\mathrm q}:=J^{-1}\mathcal S_{\bn},
 \qquad
 \varphi_{\bn}^{\mathrm q}
 :=
 \varphi_0+\Delta_{\bn}u_{\bn}^{\mathrm q}
 =
 (\tau_{\bn}^{\mathrm q},\alpha_{\bn}^{\mathrm q})^T.
\]

\begin{corollary}
\label{cor:local-likelihood-geometry}
Under the assumptions of Theorem~\ref{thm:compact-lan}, for every
\(M<\infty\),
\[
 \sup_{\|u\|\le M}
 \left|
 \Lambda_{\bn}(u)
 -\frac12\mathcal S_{\bn}^{T}J^{-1}\mathcal S_{\bn}
 +\frac12
 (u-u_{\bn}^{\mathrm q})^{T}
 J
 (u-u_{\bn}^{\mathrm q})
 \right|
 \xrightarrow{\mathbb P}0.
\]
Consequently, any locally tight sequence
$
 \widehat u_{\bn}
 \in
 \argmax_{u\in U_{\bn}}\Lambda_{\bn}(u)
$
satisfies
$
 \widehat u_{\bn}
 =
 u_{\bn}^{\mathrm q}+o_{\mathbb P}(1).
$
In the original coordinates, the quadratic loss is
\[
 \frac12
 (u-u_{\bn}^{\mathrm q})^{T}
 J
 (u-u_{\bn}^{\mathrm q})
 =
 \frac{N_{\bn}}4
 \begin{pmatrix}
  \tau-\tau_{\bn}^{\mathrm q}\\
  \alpha-\alpha_{\bn}^{\mathrm q}
 \end{pmatrix}^{T}
 \begin{pmatrix}
  1 & -(a_{\bn}-m_1)\\
  -(a_{\bn}-m_1) & (a_{\bn}-m_1)^2+v_{\alpha_0}
 \end{pmatrix}
 \begin{pmatrix}
  \tau-\tau_{\bn}^{\mathrm q}\\
  \alpha-\alpha_{\bn}^{\mathrm q}
 \end{pmatrix}.
\]
For each fixed \(\alpha\), this loss is minimized over \(\tau\), and hence
the quadratic approximation to \(\Lambda_{\bn}\) is maximized, when
\[
 (\tau-\tau_{\bn}^{\mathrm q})
 -(a_{\bn}-m_1)(\alpha-\alpha_{\bn}^{\mathrm q})
 =
 0.
\]
\end{corollary}

{
\begin{proof}
Choose \(M<\infty\) sufficiently large, and write
\[
 r_{\bn,M}(u)
 :=\Lambda_{\bn}(u)-u^T\mathcal S_{\bn}
 +\frac12u^TJu,
 \qquad \|u\|\le M.
\]
By Theorem~\ref{thm:compact-lan},
\(\sup_{\|u\|\le M}|r_{\bn,M}(u)|=o_{\mathbb P}(1)\).
The local tightness of \(\widehat u_{\bn}\) and the tightness of
\(u_{\bn}^{\mathrm q}=J^{-1}\mathcal S_{\bn}\) ensure that both points
belong to the ball with probability arbitrarily close to one.  Because
\(\varphi_0\) is an interior point of
\(\Theta_{\mathrm{tail}}\), that fixed ball is contained in \(U_{\bn}\)
for all sufficiently large sample sizes.  Thus
\(u_{\bn}^{\mathrm q}\in U_{\bn}\) on this high-probability event, and
the maximality of \(\widehat u_{\bn}\) gives
\[
 \frac12\lambda_{\min}(J)
 \|\widehat u_{\bn}-u_{\bn}^{\mathrm q}\|^2
 \le
 \frac12(\widehat u_{\bn}-u_{\bn}^{\mathrm q})^TJ
 (\widehat u_{\bn}-u_{\bn}^{\mathrm q})
 \le2\sup_{\|u\|\le M}|r_{\bn,M}(u)|
 =o_{\mathbb P}(1).
\]
The asserted argmax approximation follows.
\end{proof}
}

The argmax conclusion in
Corollary~\ref{cor:local-likelihood-geometry} is conditional on local
tightness. This condition does not follow from consistency alone: although
\(\Theta_{\mathrm{tail}}\) is compact, \(U_{\bn}\) expands as
\(\Delta_{\bn}\to0\), so
\[
 \widehat\varphi_{\bn}\to\varphi_0
 \quad\not\Longrightarrow\quad
 \widehat u_{\bn}
 =
 \Delta_{\bn}^{-1}
 (\widehat\varphi_{\bn}-\varphi_0)
 =
 O_{\mathbb P}(1).
\]

Appendix~\ref{app:likelihood-separation-and-lan} supplies this final localization step and proves
\(\widehat u_{\bn}=O_{\mathbb P}(1)\). Combining this result with
Corollary~\ref{cor:local-likelihood-geometry} and
Theorem~\ref{thm:tail-exponent-gradient-information} yields the estimator
limits stated in Section~\ref{sec:maximum-likelihood}.
Section~\ref{sec:common-tail-misspecification} shows that the same limits hold
when the data follow the misspecified law \(P^\star_{\bn}\).

\section{Misspecification with a common leading tail}
\label{sec:common-tail-misspecification}

This section shows that the maximum likelihood estimator under the fitted
model remains consistent and asymptotically normal when the covariance model
is mildly misspecified. Suppose that
\(Y_{\bn}\sim P^\star_{\bn}:=\mathcal N(0,\Sigma^\star_{\bn})\), where
\(\Sigma^\star_{\bn}\) is obtained by sampling the covariance function
associated with a spectral density \(S^\star\). The fitted model, however, is
based on the spectral density \(S_\varphi\). {The following
assumption identifies \(\varphi_0\) as the reference parameter within the
spectral family under misspecification.}

\begingroup
\begin{assumption}
\label{ass:common-leading-tail-misspecification}
Let \(S^\star\) be a continuous, even, and strictly positive spectral
density. Suppose that the spectral family
\(\{S_\varphi:\varphi\in\mathcal U\}\) satisfies
Assumption~\ref{ass:differentiated-spectral-tail-family} on an open
neighborhood \(\mathcal U\) of a reference parameter
\(\varphi_0=(\tau_0,\alpha_0)\). Write \(S_{\varphi_0}\) for the reference
spectrum.
Assume that there are \(\eta>0\) and constants \(0<c\le C<\infty\) such that
\(4\eta>d\) and
\[
 |S^\star(\bomega)-S_{\varphi_0}(\bomega)|
 \le C(1+\|\bomega\|^2)^{-\alpha_0-\eta}
 \qquad\text{for all \(\bomega\in\R^d\).}
\]
Moreover,
\[
 cS_{\varphi_0}(\bomega)\le S^\star(\bomega)\le CS_{\varphi_0}(\bomega)
 \qquad\text{for all \(\bomega\in\R^d\).}
\]
\end{assumption}
\endgroup

\begingroup

To explain the role of
Assumption~\ref{ass:common-leading-tail-misspecification}, let
$
 C_{\bn}:=\Sigma_{\varphi_0,\bn}^{-1/2}\Sigma^\star_{\bn}
 \Sigma_{\varphi_0,\bn}^{-1/2},
$
Under \(P^\star_{\bn}\), this is the covariance matrix after whitening by
\(\Sigma_{\varphi_0,\bn}\). The assumption ensures that the eigenvalues of
\(C_{\bn}\) are bounded above and away from zero, uniformly in the sample
size, and that \(C_{\bn}-I_{N_{\bn}}\) remains bounded in Hilbert--Schmidt
norm. These conclusions imply mutual contiguity between
\(P^\star_{\bn}\) and \(P_{\varphi_0,\bn}\), which transfers events whose
probabilities tend to zero.
Contiguity alone does not preserve a nondegenerate centered-score limit.
For example, consider the Gaussian laws on \(\R^{N_{\bn}}\)
\[
 \widetilde P^0_{\bn}:=\mathcal N(0,I_{N_{\bn}}),
 \qquad \widetilde P^1_{\bn}:=\mathcal N(0,\widetilde C_{\bn}),
 \qquad \widetilde C_{\bn}:=
 \bigl(1+N_{\bn}^{-1/2}\bigr)I_{N_{\bn}}.
\]
These laws are mutually contiguous. Nevertheless,
\[
 \frac{\widetilde Z_{\bn}^T\widetilde Z_{\bn}-N_{\bn}}{2\sqrt{N_{\bn}}}
 \Rightarrow
 \begin{cases}
  \mathcal N(0,1/2),
  &\text{under \(\widetilde P^0_{\bn}\),}\\
  \mathcal N(1/2,1/2),
  &\text{under \(\widetilde P^1_{\bn}\).}
\end{cases}
\]
This example shows that contiguity does not suffice to transfer the score
limit. 

Under Assumption~\ref{ass:common-leading-tail-misspecification}, all the
preceding results remain valid under \(P^\star_{\bn}\). The following theorem
summarizes the main conclusions.

\begin{theorem}
\label{thm:misspecified-tail-mle-transfer}
Under Assumption~\ref{ass:common-leading-tail-misspecification}, the
consistency conclusion of Theorem~\ref{thm:tail-chart-consistency}, the
compact working-likelihood (quasi-LAN) expansion of
Theorem~\ref{thm:compact-lan}, the joint MLE normality conclusion of
Theorem~\ref{thm:working-tail-mle-normality}, and both restricted-estimator
conclusions of Corollary~\ref{cor:known-tail-parameter-mles} hold under
\(P^\star_{\bn}\). In particular,
\[
 \Delta_{\bn}^{-1}(\widehat\varphi_{\bn}-\varphi_0)
 \Rightarrow\mathcal N(0,J^{-1}).
\]
\end{theorem}
Our proof uses polynomial approximation in
Appendix~\ref{app:working-likelihood-proofs} to control frequencies near the origin. This yields explicit rates
that may be useful for deriving further results, such as an Edgeworth
expansion. The following result gives a more immediate application.

\begin{lemma}
\label{lem:expected-scale-loglikelihood-misspecification}
Suppose Assumption~\ref{ass:common-leading-tail-misspecification} holds, fix
\(\alpha=\alpha_0\), and define \(\beta:=2\eta/d>1/2\). Let
\(\overline\ell^\star_{\bn}(\tau):=\mathbb E_{P^\star_{\bn}}
\ell_{\bn}(\tau,\alpha_0)\) for \(\tau\in[\tau_-,\tau_+]\). For all
sufficiently large \(N_{\bn}\), its unique maximizer is
\[
\tau^\dagger_{\bn}=\tau_0+\log\left\{
\frac{\operatorname{tr}C_{\bn}}{N_{\bn}}\right\}.
\]
Moreover,
\[
 \tau^\dagger_{\bn}-\tau_0
 =
 \begin{cases}
 O(N_{\bn}^{-\beta}),
     &1/2<\beta<1,\\
 O(N_{\bn}^{-1}\log N_{\bn}),
     &\beta=1,\\
 O(N_{\bn}^{-1}),
     &\beta>1.
 \end{cases}
\]
\end{lemma}

\begin{proof}
At \(\alpha=\alpha_0\),
$
 \Sigma_{(\tau,\alpha_0),\bn}
 =
 e^{\tau-\tau_0}\Sigma_{\varphi_0,\bn}.
$
Taking expectation of the corresponding exact likelihood identity gives
\[
\overline\ell^\star_{\bn}(\tau)-\overline\ell^\star_{\bn}(\tau_0)
=-\frac{N_{\bn}}2(\tau-\tau_0)
-\frac{\operatorname{tr}C_{\bn}}2\{e^{-(\tau-\tau_0)}-1\}.
\]
Differentiation gives
\[
 \partial_\tau\overline\ell^\star_{\bn}(\tau)
 =
 -\frac{N_{\bn}}2
 +\frac{\operatorname{tr}C_{\bn}}2e^{-(\tau-\tau_0)},
 \qquad
 \partial_\tau^2\overline\ell^\star_{\bn}(\tau)<0.
\]
Thus its unique critical point is \(\tau^\dagger_{\bn}\). Since
\(\tau^\dagger_{\bn}\to\tau_0\) and
\(\tau_0\in(\tau_-,\tau_+)\), this point belongs to the optimization
interval for all sufficiently large \(N_{\bn}\). Substitution of
\(\tau^\dagger_{\bn}\) into the first identity gives the centered
expected-log-likelihood formula.

Finally,  we have \(C_{\bn}-I_{N_{\bn}}\) with eigenvalues
\(\lambda_j(C_{\bn}-I_{N_{\bn}})=\lambda_j(C_{\bn})-1\) and singular values
\(s_j(C_{\bn}-I_{N_{\bn}})=|\lambda_j(C_{\bn})-1|\). By
{Lemma~\ref{lem:relative-covariance-control}}, these singular values are
bounded by \(Kj^{-\beta}\) for some \(K<\infty\). Hence
\[
 \left|
 \operatorname{tr}(C_{\bn}-I_{N_{\bn}})
 \right|
 \le
 \sum_{j=1}^{N_{\bn}}
 s_j(C_{\bn}-I_{N_{\bn}})
 \le
 K\sum_{j=1}^{N_{\bn}}j^{-\beta},
\]
and
$
 \left|
 \frac{\operatorname{tr}C_{\bn}}{N_{\bn}}-1
 \right|
 \le
 \frac{K}{N_{\bn}}
 \sum_{j=1}^{N_{\bn}}j^{-\beta}.
$
The standard bounds for this sum give the three stated rates, and applying
the logarithm preserves them.
\end{proof}
\endgroup
\section{Standard covariance families}
\label{sec:standard-covariance-families}
{We now apply the preceding likelihood theory to the Mat\'ern and generalized
Wendland families. The first two propositions put each fitted family in the
tail chart \(\varphi=(\tau,\alpha)\) and verify
Assumption~\ref{ass:differentiated-spectral-tail-family}; the final
proposition checks when an external spectrum \(S^\star\) and a fitted
reference spectrum \(S_{\varphi_0}\) satisfy
Assumption~\ref{ass:common-leading-tail-misspecification}. Together, these
checks make the fitted likelihood conclusions, and their transfer under a
common leading tail, available for the covariance families below. Here
\(\tau\) is the logarithm of the tail coefficient and \(\alpha\) is the tail
exponent. All remaining family-specific parameters are held fixed and are not
estimated.}

{For the Mat\'ern family, let \(\nu,\kappa,\sigma^2>0\) denote the
smoothness, range, and variance parameters, respectively. Its covariance
function is}
\[
 \rho_{\mathrm M}(\bx)
 =\sigma^2\frac{2^{1-\nu}}{\Gamma(\nu)}
 (\kappa\|\bx\|)^\nu K_\nu(\kappa\|\bx\|).
\]
{Here \(K_\nu\) is the modified Bessel function of the second kind.
We keep \(\kappa\) fixed and set \(\alpha=\nu+d/2\). The spectral density is
then}
\[
 S_{\mathrm M}(\bomega)
 =\sigma^2
 \frac{2^d\pi^{d/2}\Gamma(\alpha)}
      {\Gamma(\alpha-d/2)}
 \kappa^{2\alpha-d}
 (\kappa^2+\|\bomega\|^2)^{-\alpha}.
\]
Its tail coefficient is
\[
 c_{\infty,\mathrm M}
 =\sigma^2
 \frac{2^d\pi^{d/2}\Gamma(\alpha)}
      {\Gamma(\alpha-d/2)}
 \kappa^{2\alpha-d}.
\]
{The following proposition verifies the differentiated-tail
conditions for this fixed-range chart.}
\begin{proposition}
\label{prop:matern-tail-family-verification}
Fix \(\kappa>0\) and a reference point
\(\varphi_0=(\tau_0,\alpha_0)\in\R\times(d/2,\infty)\).  Under the tail chart
\(\alpha=\nu+d/2\) and \(\tau=\log c_{\infty,\mathrm M}\), there exists an
open neighborhood \(\mathcal U\) of \(\varphi_0\), with compact closure in
\(\R\times(d/2,\infty)\), on which the resulting two-parameter Mat\'ern
family satisfies Assumption~\ref{ass:differentiated-spectral-tail-family}
with \(\varepsilon=1\).  Moreover, the constants in that assumption may be
chosen uniformly whenever \(\overline{\mathcal U}\) is restricted to a
prescribed compact subset of \(\R\times(d/2,\infty)\).
\end{proposition}

{The generalized Wendland family differs because compact support contributes
a term from the support boundary. In the fitted family, fixing the boundary
parameter \(\mu_0\) and logarithmic support radius \(\ell_0\) leaves the same
\((\tau,\alpha)\) chart. The gap \(\mu_0-\alpha\) determines the available
remainder margin, and neither the boundary nor the support parameter is
differentiated. Fix \(\mu_0>(d+1)/2\), \(\ell_0\in\R\), and
\(\beta_0=e^{\ell_0}\). For \(k>0\), its covariance function is}
\[
        \rho_{\mathrm W}(\bx)
        =\frac{\sigma^2\mathbb I_{\{\|\bx\|<\beta_0\}}}
        {B(2k,\mu_0+1)}
        \int_{\|\bx\|/\beta_0}^1
        u\left\{u^2-\frac{\|\bx\|^2}{\beta_0^2}\right\}^{k-1}
        (1-u)^{\mu_0}\,du.
\]
{Here \(B\) is the beta function. Writing
\(\alpha=k+(d+1)/2\), this covariance is positive definite on \(\R^d\) when
\(\alpha\le\mu_0\); see \cite[Theorem~2.2]{ChernihHubbert2014} and
\cite{bevilacqua2019gw}.}

The corresponding tail coefficient is
\[
        c_{\infty,\mathrm W}
        =2^d\pi^{d/2}\sigma^2
        \frac{\Gamma(\alpha)\Gamma(\mu_0+2\alpha-d)}
        {\Gamma(\alpha-d/2)\Gamma(\mu_0)}
        e^{-(2\alpha-d)\ell_0}.
\]
{Thus \(\tau=\log c_{\infty,\mathrm W}\) replaces the raw variance
coordinate, while \(\mu_0\) and \(\ell_0\) remain fixed. The resulting
covariance parameter is \(\varphi=(\tau,\alpha)\).}
\begin{proposition}
\label{prop:wendland-tail-family-verification}
{Fix \(\mu_0>(d+1)/2\) and \(\ell_0\in\R\), and let
\(\varphi_0=(\tau_0,\alpha_0)\in
\R\times((d+1)/2,\mu_0)\).  There are compact intervals
\(\mathcal K_\tau\Subset\R\) and
\(\mathcal K_\alpha\Subset((d+1)/2,\mu_0)\) whose interiors contain
\(\tau_0\) and \(\alpha_0\), respectively, and an open neighborhood
\(\mathcal U\) of \(\varphi_0\) with
\(\overline{\mathcal U}\subset
\mathcal K_\tau\times\mathcal K_\alpha\), such that, for every
\(\varepsilon\) satisfying
\[
 0<\varepsilon<
 \min\left\{1,\frac12
 \inf_{\alpha\in\mathcal K_\alpha}(\mu_0-\alpha)\right\},
\]
the generalized Wendland family \(\{S_\varphi:\varphi\in\mathcal U\}\)
satisfies Assumption~\ref{ass:differentiated-spectral-tail-family}.}
\end{proposition}

{The preceding propositions verify
Assumption~\ref{ass:differentiated-spectral-tail-family} for the fitted
Mat\'ern and generalized Wendland families. We next identify the direction
needed for the misspecification transfer: \(S^\star\) is the external
spectrum, whereas \(S_{\varphi_0}\) is the reference spectrum in the fitted
family. When their leading tails agree with the stated margin, the likelihood
conclusions for \(S_{\varphi_0}\) transfer to data generated under
\(S^\star\).

Let \(\{S_\varphi:\varphi\in\mathcal U\}\) be one of the fitted families
above, and fix \(\varphi_0=(\tau_0,\alpha_0)\in\mathcal U\). Thus
\(S_{\varphi_0}\) is either a Mat\'ern spectrum with fixed range
\(\kappa_0\), or a generalized Wendland spectrum in which \(\mu_0\) is the
fixed boundary parameter and \(\ell_0\) is the logarithmic support radius. Let
\(S^\star\) be a Mat\'ern or generalized Wendland external spectrum with the
same tail coefficient \(e^{\tau_0}\) and tail exponent \(\alpha_0\). When
\(S^\star\) is Mat\'ern, denote its range by \(\kappa^\star\); when it is
generalized Wendland, denote its boundary parameter by \(\mu^\star\).}

\begin{proposition}
\label{prop:common-leading-tail-covariance-pairs}
Suppose that \(d\le3\) and \(S^\star\ne S_{\varphi_0}\). Then
\(S^\star\) and the family
\(\{S_\varphi:\varphi\in\mathcal U\}\) satisfy
Assumption~\ref{ass:common-leading-tail-misspecification} in each of the
following cases:
\begin{enumerate}
\item \(S^\star\) and \(S_{\varphi_0}\) are both Mat\'ern spectra;
\item \(S^\star\) is generalized Wendland and \(S_{\varphi_0}\) is
Mat\'ern, provided
\[
 \mu^\star-\alpha_0>\frac d2;
\]
\item \(S^\star\) is Mat\'ern and \(S_{\varphi_0}\) is generalized
Wendland, provided
\[
 \mu_0-\alpha_0>\frac d2;
\]
\item \(S^\star\) and \(S_{\varphi_0}\) are both generalized Wendland,
provided
\[
 \min\{\mu^\star-\alpha_0,\mu_0-\alpha_0\}>\frac d2.
\]
\end{enumerate}
\end{proposition}
Thus the likelihood conclusions for the fitted family transfer in each listed
pair.  The simulation below uses a two-dimensional Mat\'ern--Wendland pair
from the third case and compares it with the matched Mat\'ern fit.
\section{Simulation study}
\label{sec:simulation-study}

We compare estimation under Mat\'ern likelihood
and a generalized-Wendland likelihood fitted to data generated from a Mat\'ern distribution.  For
each \(m\in\{9,13,16,25,31,41,61,81\}\), we observe 2,000 independent
centered Gaussian fields, shared across grid widths and fitted families, on
the \(m\times m\)
regular grid in \([0,0.125]^2\).    Thus \(\bn=(m,m)\), \(N=m^2\),
\(h=0.125/(m-1)\), and \(a_{\bn}=2\log(1/h)\).  The generating Mat\'ern
field has unit variance, \(\kappa=120^{1/3}\), \(\nu_0=1.5\),
\(\alpha_0=2.5\), and thus \(\tau_0=\log(720\pi)\).  With \(r\) denoting
Euclidean distance, the Mat\'ern covariane is
\[
        \rho_{\mathrm M}(r)=(1+\kappa r)\exp(-\kappa r).
\]
Figure~\ref{fig:covariance-comparison} displays this covariance and the 
tail-matched generalized-Wendland covariance.

\begin{figure}[t]
\centering
\includegraphics[width=0.60\textwidth]{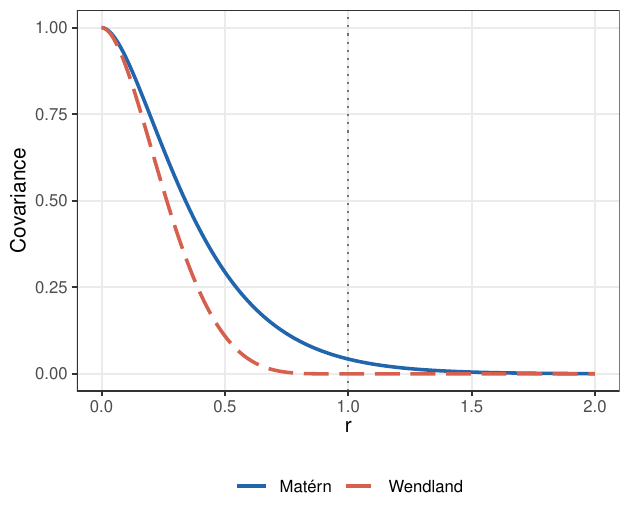}
\caption{Mat\'ern and tail-matched generalized-Wendland covariance.  The
Wendland covariance vanishes at its support radius one.}
\label{fig:covariance-comparison}
\end{figure}

For each fit, we maximize the likelihood jointly over
\(\alpha\in[1.75,3.25]\) and \(\tau\in[\tau_0-4,\tau_0+4]\), and also fit
\(\alpha\) with \(\tau_0\) fixed and \(\tau\) with \(\alpha_0\) fixed.

\subsection{Correctly specified Mat\'ern likelihood}
\label{sec:matern-simulation}

For the joint Mat\'ern fit, define
\[
\begin{aligned}
        \widehat Z_\alpha&=m(\widehat\alpha-\alpha_0),
        &\widehat Z_\tau&=\frac{m}{a_{\bn}}
        (\widehat\tau-\tau_0),\\
        (\widehat Z_\alpha,\widehat Z_\tau)
        &\Rightarrow (Z_\alpha,Z_\alpha),
        &Z_\alpha&\sim\mathcal N(0,2/v_{\alpha_0}).
\end{aligned}
\]
Theorem~\ref{thm:working-tail-mle-normality} gives this joint limit,
where \(v_{\alpha_0}\approx1.209674\).  Corollary~\ref{cor:known-tail-parameter-mles} gives, for the two restricted fits,
\[
 m a_{\bn}(\widehat\alpha^{(\tau_0)}-\alpha_0)
 \Rightarrow\mathcal N(0,2),
 \qquad
 m(\widehat\tau^{(\alpha_0)}-\tau_0)
 \Rightarrow\mathcal N(0,2).
\]

Figure~\ref{fig:matern-mle-simulation} summarizes joint estimation of \(\alpha\) and \(\tau\), 
together with estimation of either parameter when the other is fixed at its true value.
In panels (a) and (b), the solid curves show the empirical
2.5\%, 50\%, and 97.5\% quantiles of the normalized estimation
errors across the 2,000 simulations. The dashed curves show
the 2.5\% and 97.5\% quantiles of the corresponding limiting
Gaussian distributions. In the joint fit, the empirical correlation between
\(\widehat Z_\alpha\) and \(\widehat Z_\tau\) is about \(.99\)
at \(m=81\), as shown in panel (c). This is consistent with
their joint limit \((Z_\alpha,Z_\alpha)\), in which the two
components are equal. Fixing either parameter at its true
value improves the convergence rate for the other by a
factor \(a_{\bn}\). The simulation results therefore support the theoretical
convergence rates and limiting Gaussian distributions.

\begin{figure}[t]
\centering
\begin{minipage}[t]{0.32\textwidth}\centering\textbf{(a)}\par
\includegraphics[width=\textwidth]{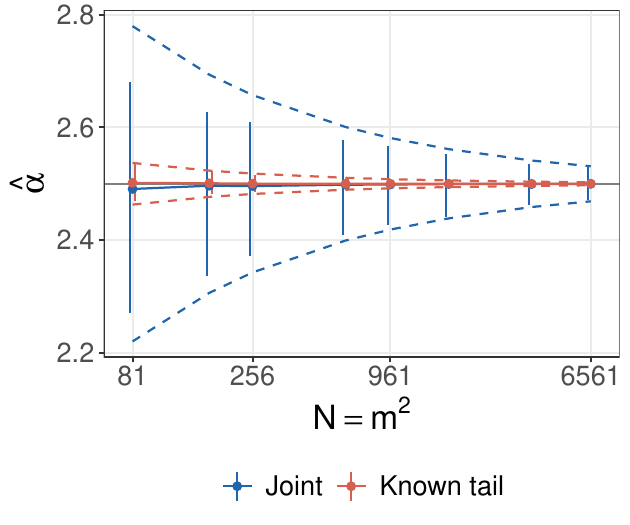}\end{minipage}
\hfill
\begin{minipage}[t]{0.32\textwidth}\centering\textbf{(b)}\par
\includegraphics[width=\textwidth]{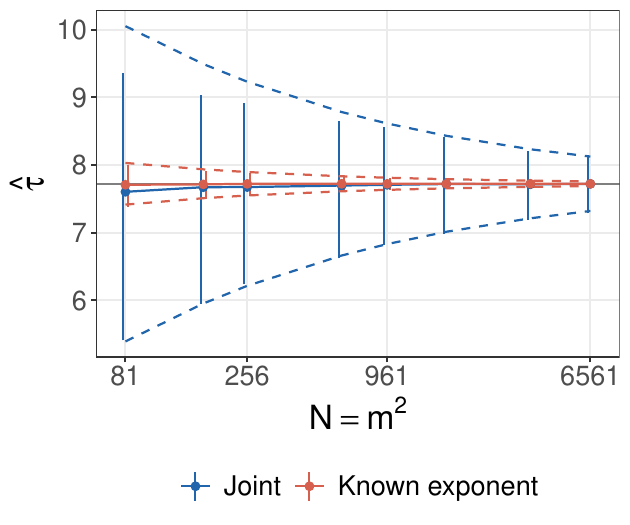}\end{minipage}
\hfill
\begin{minipage}[t]{0.32\textwidth}\centering\textbf{(c)}\par
\includegraphics[width=\textwidth]{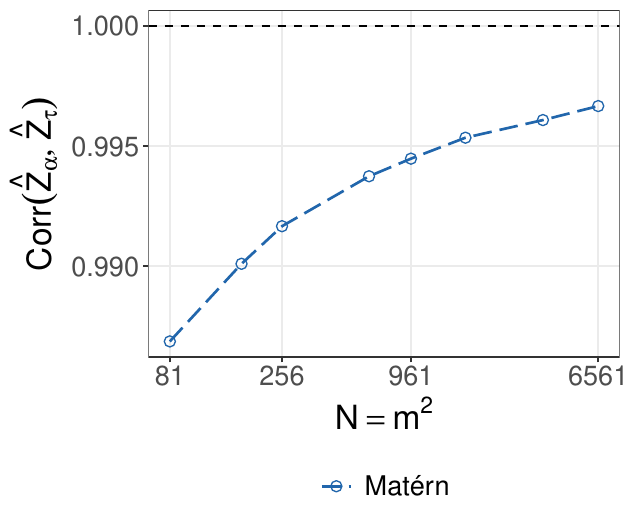}\end{minipage}
\caption{Matched Mat\'ern likelihood: empirical medians and central 95\%
ranges for the exponent and log-tail estimators, and the correlation of the
joint normalized errors.  Dashed bands are leading Gaussian 95\% reference
ranges.}
\label{fig:matern-mle-simulation}
\end{figure}

Corollary~\ref{cor:local-likelihood-geometry} shows that the
quadratic approximation to the likelihood has a ridge with
slope \(a_{\bn}-m_1\). We investigate this behavior in
Figure~\ref{fig:matern-likelihood-surfaces}, which shows the
Mat\'ern log likelihood for the same simulated field observed
on \(9\times9\) and \(81\times81\) grids in \([0,0.125]^2\).
Both surfaces exhibit the predicted ridge, which is narrower
at \(m=81\).
\begin{figure}[t]
\centering
\begin{minipage}[t]{0.48\textwidth}\centering\textbf{(a)}\par
\includegraphics[width=\textwidth]{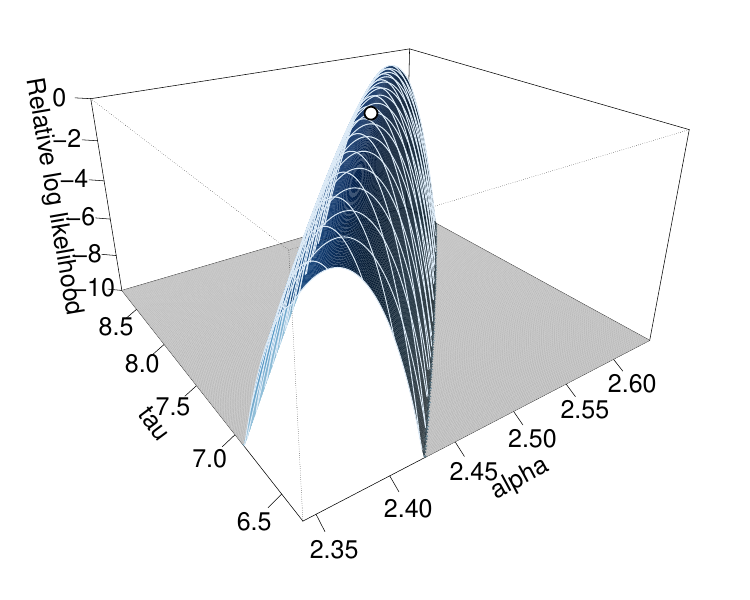}\end{minipage}
\hfill
\begin{minipage}[t]{0.48\textwidth}\centering\textbf{(b)}\par
\includegraphics[width=\textwidth]{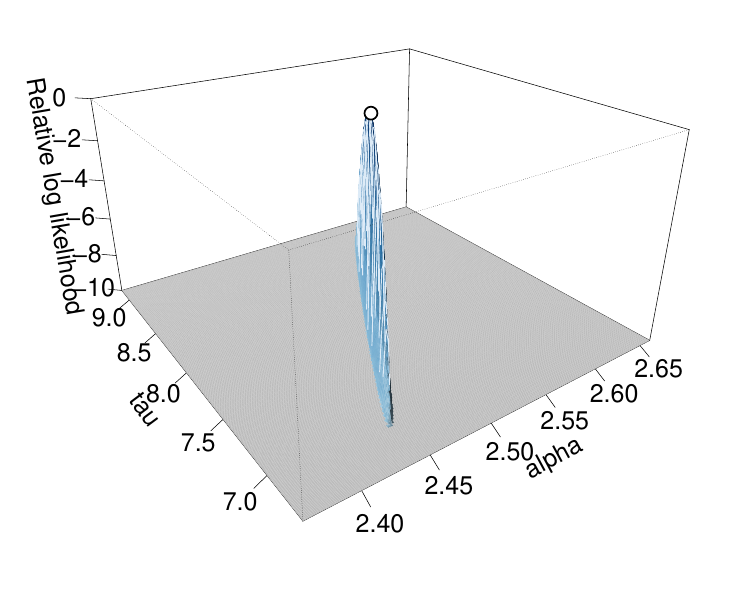}\end{minipage}
\caption{Relative exact Mat\'ern log likelihood over \((\alpha,\tau)\) in
the selected shared-study realization on \([0,0.125]^2\), at \(m=9\) and
\(m=81\) (left and right, respectively).  The surfaces are floored at \(-10\),
and circles mark the MLEs.}
\label{fig:matern-likelihood-surfaces}
\end{figure}

\subsection{Generalized Wendland likelihood fitted to Mat\'ern data}
\label{sec:wendland-matern-simulation}

We fix \(\mu_0=4\) and \(\beta_0=e^{\ell_0}=1\), as in the full fitted
family in the preceding section.  At \(\alpha_0=2.5\), the specialization is
\[
        \rho_{\mathrm W}(r)=(1-r)_+^5(1+5r).
\]
The generalized-Wendland model matches the variance and leading
spectral tail of the Mat\'ern model. Since
\(\mu_0-\alpha_0=1.5>d/2=1\),
Figure~\ref{fig:wendland-matern-simulation} compares joint
estimation of \(\alpha\) and \(\tau\) with estimation of either
parameter when the other is fixed at its true value. 

The joint estimators show a downward bias at smaller sample
sizes, particularly for \(\tau\). Its empirical bias decreases
in magnitude from about \(-0.374\) at \(m=9\) to \(-0.024\)
at \(m=81\). Fixing either parameter at its true value
substantially reduces the bias in estimating the other.

These results provide empirical support for the convergence
rates and Gaussian limits predicted under covariance
misspecification, and show how the finite-sample bias diminishes
as the sample size increases.

\begin{figure}[t]
\centering
\begin{minipage}[t]{0.32\textwidth}\centering\textbf{(a)}\par
\includegraphics[width=\textwidth]{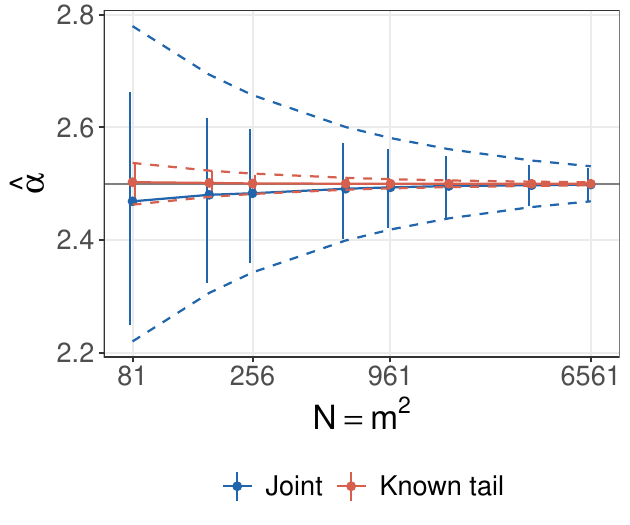}\end{minipage}
\hfill
\begin{minipage}[t]{0.32\textwidth}\centering\textbf{(b)}\par
\includegraphics[width=\textwidth]{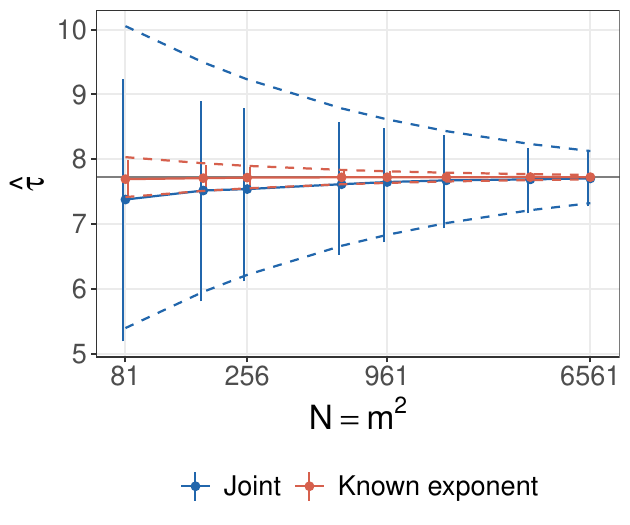}\end{minipage}
\hfill
\begin{minipage}[t]{0.32\textwidth}\centering\textbf{(c)}\par
\includegraphics[width=\textwidth]{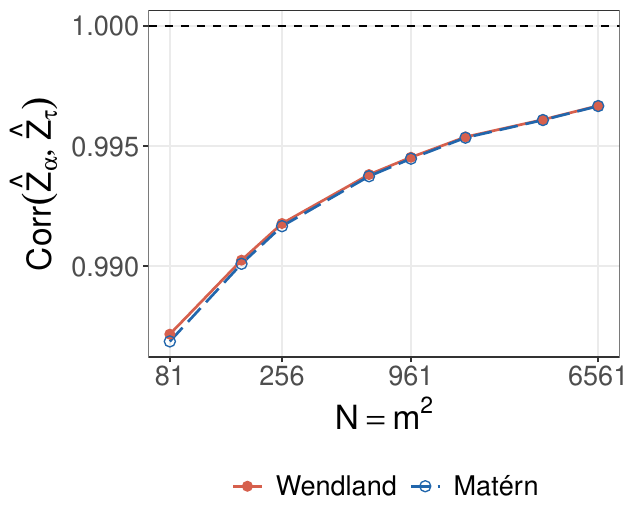}\end{minipage}
\caption{Generalized-Wendland likelihood fitted to Mat\'ern data: empirical
medians and central 95\% ranges for the exponent and log-tail estimators, and
the correlation of the joint normalized errors.  Dashed bands are leading
Gaussian 95\% reference ranges.}
\label{fig:wendland-matern-simulation}
\end{figure}

\section{Discussion}
\label{sec:discussion}

We have established consistency and local asymptotic normality 
for joint estimation of scale and roughness under fixed-domain sampling on regular grids.
 The resulting limits describe the interaction between the two parameters 
 and the logarithmic loss in convergence rates caused by estimating them jointly. 
 GLT theory and quantitative control of low frequencies provide the analytical framework,
  while uniform spectral bounds extend its applicability to generalized Wendland models alongside the Matérn family. 
  The estimator limits also persist under covariance misspecification satisfying the stated spectral conditions.
The GLT framework provides a natural route to nonstationary Gaussian processes, 
since its symbols can depend on both spatial location and frequency.
 Extending the present LAN theory to models with spatially varying scale or roughness is therefore an interesting direction. 
 Another question is whether the same limits remain valid when the remaining covariance parameters are estimated,
  particularly when information about these parameters stays limited under fixed-domain sampling. 
  Extensions to irregular observation locations would further broaden the scope of the theory.
   Finally, the quantitative bounds developed here may support higher-order likelihood approximations and a more precise account of finite-sample bias under covariance misspecification.

\clearpage
\appendix
\part*{Appendices}

\section*{Standing setup}
We work on the fixed domain \([0,L]^d\) with the regular grid
\(\bx_{\bi}^{(\bn)}=(Li_1/(n_1-1),\ldots,Li_d/(n_d-1))\),
\(\bi\in\mathcal I_{\bn}:=\prod_{j=1}^d\{0,\ldots,n_j-1\}\), and
\(N_{\bn}:=\prod_{j=1}^d n_j\).  Under isotropic refinement,
\(h:=L/(n_j-1)\) is independent of \(j\) and \(h\to0\) as
\(N_{\bn}\to\infty\).  Our Fourier convention is
\(\rho(\bh)=(2\pi)^{-d}\int_{\mathbb R^d}e^{\mathrm i\bomega\cdot\bh}S(\bomega)\,d\bomega\). 
{The spectral family \((S_{\varphi},\rho_{\varphi},\Sigma_{\varphi,\bn},P_{\varphi,\bn})\) is indexed by \(\varphi\), 
whereas the fixed external data-generating objects are \((S^\star,\rho^\star,\Sigma^\star_{\bn},P^\star_{\bn})\). 
The tail chart is \(\varphi=(\tau,\alpha)\), with reference parameter \(\varphi_0=(\tau_0,\alpha_0)\) 
and \(c_{\infty,0}=e^{\tau_0}\). Exact spectral specification, \(S^\star=S_{\varphi_0}\),
 implies \(P^\star_{\bn}=P_{\varphi_0,\bn}\). The scale-family interface is \(S_{\tau,\alpha}=e^\tau\overline S_\alpha\), where \(\overline S_\alpha=(1+\|\bomega\|^2)^{-\alpha}+S^{\mathrm{rem}}_\alpha\).} We set \(\Sigma_{\varphi,\bn}=(\rho_{\varphi}(\bx_{\br}^{(\bn)}-\bx_{\bs}^{(\bn)}))_{\br,\bs\in\mathcal I_{\bn}}\), \(P_{\varphi,\bn}:=\mathcal N(0,\Sigma_{\varphi,\bn})\), and analogously \(\Sigma^\star_{\bn}=(\rho^\star(\bx_{\br}^{(\bn)}-\bx_{\bs}^{(\bn)}))_{\br,\bs\in\mathcal I_{\bn}}\) and \(P^\star_{\bn}:=\mathcal N(0,\Sigma^\star_{\bn})\). For the observed sampled vector \(Y_{\bn}\), we use, up to an additive term independent of \(\varphi\),
\[
        \ell_{\bn}(\varphi)
        =-\frac{1}{2}\log\det \Sigma_{\varphi,\bn}
        -\frac{1}{2}Y_{\bn}^T(\Sigma_{\varphi,\bn})^{-1}Y_{\bn}.
\]
{Unless a norm is explicitly subscripted otherwise, \(\|\cdot\|_2\)
denotes the Euclidean norm for vectors and the induced operator (spectral)
norm for matrices; \(\|\cdot\|_1\) and \(\|\cdot\|_{\mathrm{HS}}\) denote
the trace and Hilbert--Schmidt norms, respectively.}
{The compact parameter set \(\Theta_{\mathrm{tail}}\) contains
\(\varphi_0\) in its interior. We write \(a_{\bn}:=2\log(1/h)\).}

\section*{Roadmap of the proofs}

The reader route follows the five stages A--E below.  Appendix~\ref{app:aliased-symbols}
starts from the sampled-covariance Toeplitz representation.  Appendix~\ref{app:preliminary-main-paper-proofs}
then develops the reusable controls, which lead to the moving-tail interfaces
in Appendix~\ref{app:moving-tail-spectral-proofs} and to the comparison and
transfer arguments in Appendix~\ref{app:working-likelihood-proofs}.
{Appendix~\ref{app:moving-tail-spectral-proofs} 
establishes the score and secant interfaces for the spectral family. 
Appendix~\ref{app:working-likelihood-proofs} 
then supplies whitening, polynomial compression, 
and transfer to \(P^\star_{\bn}\) under misspecification. 
Appendix~\ref{app:likelihood-separation-and-lan} proves working-model separation and LAN; the proof below combines its compact working-likelihood quasi-LAN expansion with the score transfer under the external law. Assumption~\ref{ass:common-leading-tail-misspecification} and Theorem~\ref{thm:misspecified-tail-mle-transfer} identify that setting and result.}

This route organizes the reading; it is not a literal list of proof imports.
Appendix~\ref{app:wendland-differentiated-tail} is a separate lane for
verifying the model families.  It checks the framework assumptions for the
Mat\'ern and generalized-Wendland examples with fixed boundary and support
parameters.  Assumption~\ref{ass:differentiated-spectral-tail-family} is an entry
condition cited in Appendices~\ref{app:moving-tail-spectral-proofs} and
\ref{app:working-likelihood-proofs}; the verification establishes the
examples against that condition.

\begin{center}
\begin{tikzpicture}[
  stage/.style={
    draw=black!65,
    rounded corners=2pt,
    text width=7.0cm,
    minimum height=11mm,
    inner sep=3pt,
    align=center,
    font=\footnotesize
  },
  spine/.style={
    stage,
    draw=black!85,
    fill=black!8,
    line width=0.75pt
  },
  verification/.style={
    stage,
    text width=4.6cm,
    fill=black!3,
    line width=0.6pt
  },
  route arrow/.style={
    -{Stealth[length=2mm,width=1.4mm]},
    draw=black!75,
    line width=0.65pt
  },
  entry link/.style={
    dashed,
    draw=black!60,
    line width=0.55pt
  }
]
\node[spine] (A) {
  \textbf{Appendix~\ref{app:aliased-symbols}}\\
  Sampled-covariance Toeplitz representation
};
\node[spine, below=4.5mm of A] (B) {
  \textbf{Appendix~\ref{app:preliminary-main-paper-proofs}}\\
  Reusable truncation, compression, and inversion controls
};
\node[spine, below=4.5mm of B] (C) {
  \textbf{Appendix~\ref{app:moving-tail-spectral-proofs}}\\
  Moving-tail covariance, secant, whitening, and score interfaces
};
\node[spine, below=4.5mm of C] (D) {
  \textbf{Appendix~\ref{app:working-likelihood-proofs}}\\
  Fitted--external comparison, relative covariance, contiguity, and transfer
};
\node[spine, below=4.5mm of D] (E) {
  \textbf{Appendix~\ref{app:likelihood-separation-and-lan}}\\
  Hellinger separation, consistency, LAN, and MLE limits
};
\node[verification, right=8mm of C, yshift=-7mm] (F) {
  \textbf{Appendix~\ref{app:wendland-differentiated-tail}}\\
  Model-family verification for the Mat\'ern and generalized-Wendland examples with fixed boundary and support parameters\\[-1pt]
  {\scriptsize Dashed links: entry condition cited in C and D}
};

\draw[route arrow] (A) -- (B);
\draw[route arrow] (B) -- (C);
\draw[route arrow] (C) -- (D);
\draw[route arrow] (D) -- (E);
\draw[entry link] (F.west) -- (C.east);
\draw[entry link] (F.south west) |- (D.east);
\end{tikzpicture}
\end{center}

The route above organizes the appendices.  The map below records the proof
locations and direct interfaces used at the final proof stage for each
article result; it suppresses upstream dependencies already contained in
those interfaces.

\begingroup
\makeatletter
\newcommand{\roadmapresult}[2]{#1~\ref{#2}}
\newcommand{\roadmapsupport}[2]{\hyperref[#1]{#2}}
\begin{center}
\begin{tikzpicture}[
  support/.style={
    draw=black!65,
    rounded corners=2pt,
    text width=6.25cm,
    minimum height=8mm,
    inner sep=2.5pt,
    align=left,
    font=\scriptsize
  },
  result/.style={
    support,
    text width=4.35cm,
    align=center,
    fill=black!8,
    draw=black!85
  },
  map arrow/.style={
    -{Stealth[length=1.7mm,width=1.2mm]},
    draw=black!75,
    line width=0.55pt
  },
  applicability arrow/.style={
    map arrow,
    dashed,
    draw=black!60
  }
]
\node[support] (s1) {\textbf{Truncation proof in the article; explicit compression and inversion proofs in Appendix~\ref{app:preliminary-main-paper-proofs}}};
\node[result, right=6mm of s1] (r1) {\roadmapresult{Theorem}{thm:triangular-array-toeplitz-truncation-detailed}; \roadmapresult{Lemma}{lem:localized-fourier-kernel-compression}; \roadmapresult{Lemma}{lem:renorm-glt-inversion}};
\node[support, below=2.5mm of s1] (s2) {Lemma~\ref{lem:explicit-localized-fourier-kernel-compression} and Corollary~\ref{cor:localized-fourier-kernel-compression-interface}};
\node[result, right=6mm of s2] (r2) {\roadmapresult{Theorem}{thm:quantitative-sampled-covariance-glt} (proof in Appendix~\ref{app:preliminary-main-paper-proofs})};
\node[support, below=2.5mm of s2] (s3) {\roadmapsupport{lem:centered-quadratic-form-transfer}{Centered quadratic-form transfer}};
\node[result, right=6mm of s3] (r3) {\roadmapresult{Theorem}{thm:tail-exponent-gradient-information}; \roadmapresult{Corollary}{cor:known-tail-parameter-mles}};
\node[support, below=2.5mm of s3] (s4) {\roadmapsupport{lem:working-cayley-hellinger}{Cayley affinity}, \roadmapsupport{lem:working-local-cayley-coercivity}{local coercivity}, and \roadmapsupport{lem:working-noncritical-cayley-separation}{noncritical separation}};
\node[result, right=6mm of s4] (r4) {\roadmapresult{Proposition}{prop:working-glt-hellinger-separation}};
\node[support, below=2.5mm of s4] (s5) {\roadmapsupport{cor:working-glt-fixed-neighborhood-separation}{Fixed-neighborhood separation}, \roadmapsupport{lem:working-likelihood-exclusion}{likelihood exclusion}, and \roadmapsupport{lem:common-leading-tail-mutual-contiguity}{mutual contiguity}};
\node[result, right=6mm of s5] (r5) {\roadmapresult{Theorem}{thm:tail-chart-consistency}};
\node[support, below=2.5mm of s5] (s6) {\roadmapsupport{cor:compact-relative-covariance-expansion}{Compact relative-covariance expansion}, \roadmapsupport{cor:relative-covariance-trace-comparison}{trace comparison}, \roadmapsupport{cor:smooth-transform-vanishing-gaussian-quadratic-process}{control after a smooth spectral transform}, \roadmapsupport{cor:vanishing-gaussian-quadratic-process}{control of the original Gaussian process}, and \roadmapsupport{lem:common-leading-tail-mutual-contiguity}{mutual contiguity}};
\node[result, right=6mm of s6] (r6) {\roadmapresult{Theorem}{thm:compact-lan}};
\node[support, below=2.5mm of s6] (s7) {\roadmapresult{Theorem}{thm:compact-lan}};
\node[result, right=6mm of s7] (r7) {\roadmapresult{Corollary}{cor:local-likelihood-geometry} {(direct proof in the article)}};
\node[support, below=2.5mm of s7] (s8) {\roadmapresult{Theorem}{thm:tail-chart-consistency}, \roadmapsupport{lem:spectral-comparison-localization}{spectral-comparison localization}, \roadmapresult{Theorem}{thm:compact-lan}, \roadmapresult{Theorem}{thm:tail-exponent-gradient-information}, and \roadmapresult{Corollary}{cor:local-likelihood-geometry}};
\node[result, right=6mm of s8] (r8) {\roadmapresult{Theorem}{thm:working-tail-mle-normality}};
\node[support, below=2.5mm of s8] (s9) {the restricted-model separation argument in the proof of \roadmapresult{Theorem}{thm:tail-chart-consistency}, \roadmapsupport{lem:spectral-comparison-localization}{spectral-comparison localization}, \roadmapresult{Theorem}{thm:compact-lan}, and \roadmapresult{Theorem}{thm:tail-exponent-gradient-information}};
\node[result, right=6mm of s9] (r9) {\roadmapresult{Corollary}{cor:known-tail-parameter-mles}};
\node[support, below=2.5mm of s9] (s10) {Assumption~\ref{ass:common-leading-tail-misspecification}; Appendix~\ref{app:working-likelihood-proofs} transfer and contiguity; fitted results in Appendix~\ref{app:likelihood-separation-and-lan}};
\node[result, right=6mm of s10] (r10) {Theorem~\ref{thm:misspecified-tail-mle-transfer}};
\node[support, below=2.5mm of s10] (s11) {\textbf{Verification in Appendix~\ref{app:wendland-differentiated-tail}}};
\node[result, right=6mm of s11] (r11) {\roadmapresult{Proposition}{prop:matern-tail-family-verification}; \roadmapresult{Proposition}{prop:wendland-tail-family-verification}; \roadmapresult{Proposition}{prop:common-leading-tail-covariance-pairs}};

\draw[map arrow] (s1) -- (r1);
\draw[map arrow] (s2) -- (r2);
\draw[map arrow] (s3) -- (r3);
\draw[map arrow] (s4) -- (r4);
\draw[map arrow] (s5) -- (r5);
\draw[map arrow] (s6) -- (r6);
\draw[map arrow] (s7) -- (r7);
\draw[map arrow] (s8) -- (r8);
\draw[map arrow] (s9) -- (r9);
\draw[map arrow] (s10) -- (r10);
\draw[applicability arrow] (s11) -- (r11);
\end{tikzpicture}
\end{center}
\makeatother
\endgroup

The two tables below are a reader key, not a replacement glossary: the
defining equations remain at their native sites.  They collect the recurring
interfaces in the order in which the proof route uses them.

\begin{table}[!ht]
\centering
\footnotesize
\setlength{\tabcolsep}{3pt}
\renewcommand{\arraystretch}{1.12}
\begin{tabular}{p{0.19\textwidth}p{0.18\textwidth}p{0.35\textwidth}p{0.18\textwidth}}
\hline
\textbf{Object or family} & \textbf{Defined in} & \textbf{Reader meaning} & \textbf{Carried forward to} \\
\hline
\(Q\), \(T_{\bn}(f)\), \(\Sigma_{\bn}\), and \(\phi_h\)
& Appendix~\ref{app:aliased-symbols}, Equation~\eqref{eq:aliased-symbol-def-app}
& The torus, Toeplitz construction, sampled covariance, and its aliased symbol; the sampled covariance is represented by the Toeplitz matrix with symbol \(\phi_h\).
& Appendices~\ref{app:preliminary-main-paper-proofs} and \ref{app:moving-tail-spectral-proofs} \\

\(a_{\bn}\), \(\beta_{\bk,h,\alpha}\), \(b_{h,\alpha}\), and \(\psi_{h,\alpha}\)
& Appendix~\ref{app:moving-tail-spectral-proofs}, Equations~\eqref{eq:supp-aliased-likelihood-factorization} and \eqref{eq:tail-alias-term-and-sum}
& \(a_{\bn}\) is the mesh-log rate isolating the common exponent-induced scale drift; \(\beta_{\bk,h,\alpha}\) is the \(\bk\)-th term in the sum defining \(b_{h,\alpha}\), and \(\psi_{h,\alpha}\) is the finite-mesh alias series used to compare with \(b_{h,\alpha}\).
& Appendix~\ref{app:moving-tail-spectral-proofs} \\

\(\psi_\alpha\) and \(r_{\alpha_0}\)
& Equation~\eqref{eq:tail-limit-symbols}
& The limiting alias shape and its relative exponent score; their moments determine the information used in the local likelihood analysis.
& Appendices~\ref{app:working-likelihood-proofs} and \ref{app:likelihood-separation-and-lan} \\

\(R_{\alpha,\bn}\) and \(\mathsf K_{\delta,\bn}\)
& Equations~\eqref{eq:Ralpha}, \eqref{eq:supp-normalized-tail-secant}, and \eqref{eq:fixed-whitening-compensated-secant-path}
& The fixed-whitened exponent-score matrix and its compensated finite secant.  The path relation identifies the secant with the averaged compensated relative score.
& Appendices~\ref{app:working-likelihood-proofs} and \ref{app:likelihood-separation-and-lan} \\
\hline
\end{tabular}
\caption{Reader key for the representation and compensated local geometry developed in Appendices~\ref{app:aliased-symbols}--\ref{app:moving-tail-spectral-proofs}.}
\label{tab:reader-key-representation}

\vspace{1.25em}
\begin{tabular}{p{0.19\textwidth}p{0.18\textwidth}p{0.35\textwidth}p{0.18\textwidth}}
\hline
\textbf{Object or family} & \textbf{Defined in} & \textbf{Reader meaning} & \textbf{Carried forward to} \\
\hline
\(P_{\varphi,\bn}\), \(P^\star_{\bn}\), \(\mathcal W_{\bn}\), and \(C_{\bn}\)
& Appendix~\ref{app:working-likelihood-proofs}
& The fitted model, external data-generating model, fitted-model whitening matrix, and data covariance after that whitening.
& Transfer, contiguity, and the LAN argument in Appendix~\ref{app:likelihood-separation-and-lan} \\

\(B_{\bn}(\alpha)\)
& Appendix~\ref{app:working-likelihood-proofs}
& The compensated relative covariance along the exponent direction in the fitted model; its derivative supplies the quadratic score form whose centering and variance are transferred to \(P^\star_{\bn}\).
& Score transfer and the local likelihood construction \\

\(V_{\varphi,\bn}\), \(g_{\bn}(\varphi)\), and \(q_{\bn}(\varphi)\)
& Equation~\eqref{eq:supp-full-relative-covariance-secant} and Appendix~\ref{app:likelihood-separation-and-lan}
& The relative covariance, compensated scale--shape coordinate, and its distance.  Their secant representation turns covariance departure into the separation scale used by the affinity bounds.
& Hellinger separation, consistency, and local-band control \\

\(D_{\bn}(u)\), \(L_{\bn}(u)\), \(\mathcal S_{\bn}\), and \(\Lambda_{\bn}(u)\)
& Equations~\eqref{eq:relative-covariance-increment}, \eqref{eq:relative-covariance-linear-part}, \eqref{eq:Dpath}, and \eqref{eq:supp-compensated-score-and-likelihood}
& The local relative-covariance increment, its first-order term, normalized score, and likelihood increment.  Linearizing \(D_{\bn}(u)\) by \(L_{\bn}(u)\) identifies the score term in the LAN expansion.
& The compact LAN theorem and MLE limits \\
\hline
\end{tabular}
\caption{Reader key for the fitted--external comparison and inferential geometry developed in Appendices~\ref{app:working-likelihood-proofs} and \ref{app:likelihood-separation-and-lan}.}
\label{tab:reader-key-inference}
\end{table}

\clearpage
\section{Aliased symbols for sampled covariance sequences}
\label{app:aliased-symbols}

This appendix records the precise relation between a spectral density and the
Toeplitz symbol of its sampled covariance sequence.

Let \(Q:=[-\pi,\pi]^d\) and let
\(\nu_Q(d\bth):=(2\pi)^{-d}\,d\bth\) be normalized Lebesgue measure on
\(Q\).  For a matrix sequence \(\{A_{\bn}\}_{\bn}\), write \(N_{\bn}\) for
the dimension of \(A_{\bn}\).  The relation \(\{A_{\bn}\}_{\bn}\sim_{\GLT}a\)
means that the sequence belongs to the
standard multilevel GLT class with full symbol
\(\kappa(x,\bth)=a(\bth)\) on \([0,1]^d\times Q\); here the symbols depend
only on frequency.  We use the convention of
\cite{GaroniSerraCapizzanoI,GaroniSerraCapizzanoII}.  Thus, for
\(F\in C_c(\R)\),
\[
        \frac{1}{N_{\bn}}\sum_{j=1}^{N_{\bn}}
        F\!\left(\sigma_j(A_{\bn})\right)
        \longrightarrow
        \int_Q F\!\left(|a(\bth)|\right)\,d\nu_Q(\bth).
\]
If the matrices are Hermitian, then \(a\) is real-valued
\(\nu_Q\)-almost everywhere, and the eigenvalue identity also holds.
\[
        \frac{1}{N_{\bn}}\sum_{j=1}^{N_{\bn}}
        F\!\left(\lambda_j(A_{\bn})\right)
        \longrightarrow
        \int_Q F\!\left(a(\bth)\right)\,d\nu_Q(\bth).
\]
Symbols are identified \(\nu_Q\)-almost everywhere, so their
value at \(\bth=0\) is immaterial.  The scaled convention is
\(\{A_{\bn}\}_{\bn}\sim_{\GLT,s_{\bn}}a\) if
\(\{s_{\bn}A_{\bn}\}_{\bn}\sim_{\GLT}a\).

We use the multilevel Toeplitz convention
\[
[T_{\bn}(f)]_{\br,\bs}=\widehat f_{\br-\bs},
\qquad
\widehat f_{\boldsymbol{\ell}}
:=
\frac{1}{(2\pi)^d}\int_Q e^{-\mathrm i\boldsymbol{\ell}\cdot\bth}f(\bth)\,d\bth,
{\qquad \boldsymbol{\ell}\in\Z^d.}
\]

Recall the grid in Equation~\eqref{eq:grid-detailed}.  The next lemma
identifies the aliased symbol associated with the sampled covariance matrix
\begin{equation*}
        \Sigma_{\bn}
        :=
        \bigl(
        \rho(\bx_{\br}^{(\bn)}-\bx_{\bs}^{(\bn)})
        \bigr)_{\br,\bs\in\mathcal I_{\bn}}
        \in\R^{N_{\bn}\times N_{\bn}} .
\end{equation*}

\begin{lemma}
\label{lem:aliased-symbol-rigorous}
Let the scalar spectral density \(S\) be even and belong to \(L^1(\R^d)\),
and let \(h>0\).  Define
\begin{equation}
\label{eq:aliased-symbol-def-app}
\phi_h(\bth)
:=
\frac{1}{h^d}\sum_{\bk\in\Z^d}
S\!\left(\frac{\bth+2\pi\bk}{h}\right),
\qquad \bth\in Q.
\end{equation}
Then \(\phi_h\in L^1(Q)\), and its Fourier coefficients satisfy
\begin{equation}
\label{eq:aliased-symbol-fourier-coeff}
\widehat{\phi_h}_{\boldsymbol{\ell}}=\rho(h\boldsymbol{\ell}),
\qquad \boldsymbol{\ell}\in\Z^d.
\end{equation}
Moreover,
$
\Sigma_{\bn}
=
T_{\bn}\!\bigl(\phi_h\bigr).
$
\end{lemma}

\begin{proof}
We first show that \(\phi_h\) is well defined as an \(L^1(Q)\) function.
Since \(S\in L^1(\R^d)\), Tonelli's theorem gives
\begin{align*}
\int_Q \sum_{\bk\in\Z^d}
\frac{1}{h^d}\left|S\!\left(\frac{\bth+2\pi\bk}{h}\right)\right|\,d\bth
&=
\sum_{\bk\in\Z^d}\int_Q
\frac{1}{h^d}\left|S\!\left(\frac{\bth+2\pi\bk}{h}\right)\right|\,d\bth \\
&=
\sum_{\bk\in\Z^d}\int_{(Q+2\pi\bk)/h}|S(\bomega)|\,d\bomega \\
&=
\int_{\R^d}|S(\bomega)|\,d\bomega
<\infty,
\end{align*}
because the sets \((Q+2\pi\bk)/h\), \(\bk\in\Z^d\), form a partition of \(\R^d\) up to boundaries of Lebesgue measure zero. Hence the series in \eqref{eq:aliased-symbol-def-app} converges absolutely for almost every \(\bth\in Q\), and \(\phi_h\in L^1(Q)\).

We now compute its Fourier coefficients. Fix \(\boldsymbol{\ell}\in\Z^d\). Since the defining series is absolutely integrable on \(Q\), {Fubini's theorem} again permits interchange of sum and integral:
\begin{align*}
\widehat{\phi_h}_{\boldsymbol{\ell}}
&=
\frac{1}{(2\pi)^d}\int_Q e^{-\mathrm i\boldsymbol{\ell}\cdot\bth}\phi_h(\bth)\,d\bth \\
&=
\frac{1}{(2\pi)^d}
\sum_{\bk\in\Z^d}
\int_Q
e^{-\mathrm i\boldsymbol{\ell}\cdot\bth}
\frac{1}{h^d}
S\!\left(\frac{\bth+2\pi\bk}{h}\right)\,d\bth.
\end{align*}
In the \(\bk\)-th integral, we perform the change of variables
$
        \bomega=\frac{\bth+2\pi\bk}{h},
        \bth=h\bomega-2\pi\bk,
$
which gives \(d\bth=h^d\,d\bomega\).
Since \(\boldsymbol{\ell}\in\Z^d\), we have \(e^{\mathrm i\,2\pi\,\boldsymbol{\ell}\cdot\bk}=1\), and therefore
\begin{align*}
\widehat{\phi_h}_{\boldsymbol{\ell}}
&=
\frac{1}{(2\pi)^d}
\sum_{\bk\in\Z^d}
\int_{(Q+2\pi\bk)/h}
e^{-\mathrm i\boldsymbol{\ell}\cdot(h\bomega-2\pi\bk)}
S(\bomega)\,d\bomega \\
&=
\frac{1}{(2\pi)^d}
\sum_{\bk\in\Z^d}
\int_{(Q+2\pi\bk)/h}
e^{-\mathrm i h\boldsymbol{\ell}\cdot\bomega}
S(\bomega)\,d\bomega \\
&=
\frac{1}{(2\pi)^d}
\int_{\R^d}
e^{-\mathrm i h\boldsymbol{\ell}\cdot\bomega}
S(\bomega)\,d\bomega.
\end{align*}
Using the evenness \(S(\bomega)=S(-\bomega)\),
we obtain, after the change of variables \(\bxi=-\bomega\),
\begin{align*}
\widehat{\phi_h}_{\boldsymbol{\ell}}
&=
\frac{1}{(2\pi)^d}
\int_{\R^d}
e^{-\mathrm i h\boldsymbol{\ell}\cdot\bomega}
S(-\bomega)\,d\bomega =
\frac{1}{(2\pi)^d}
\int_{\R^d}
e^{\mathrm i h\boldsymbol{\ell}\cdot\bxi}
 S(\bxi)\,d\bxi =\rho(h\boldsymbol{\ell}).
\end{align*}
This proves \eqref{eq:aliased-symbol-fourier-coeff}.

Finally, recall \(\bx_{\br}^{(\bn)}-\bx_{\bs}^{(\bn)}=h(\br-\bs)\), and thus
\[
[T_{\bn}(\phi_h)]_{\br,\bs}
=
\widehat{\phi_h}_{\br-\bs}
=
\rho(h(\br-\bs))
=
\rho(\bx_{\br}^{(\bn)}-\bx_{\bs}^{(\bn)}),
\]
which is exactly the desired Toeplitz representation.
\end{proof}

\begin{remark}
\label{rem:aliased-symbol-equivalent-form}
The formula \eqref{eq:aliased-symbol-def-app} may equally be written as
\[
\phi_h(\bth)
=
\frac{1}{h^d}\sum_{\bk\in\Z^d}
S\!\left(\frac{2\pi\bk-\bth}{h}\right),
\]
where \(\bth\in Q\),
since \(S(\bomega)=S(-\bomega)\) and the index set
\(\Z^d\) is invariant under \(\bk\mapsto-\bk\).
\end{remark}

\section{Explicit compression, quantitative GLT, and uniform comparison}
\label{app:preliminary-main-paper-proofs}

This appendix gives the explicit localized compression and the complete proof
of the quantitative GLT theorem.  The truncation theorem is proved in the
article immediately after its statement; the inversion proof follows the
quantitative result below.

For \(\bomega\in\R^d\), let
\begin{equation}\label{eq:supp-localized-fourier-mode}
 [v_{\bn}(\bomega)]_{\bi}:=e^{-\mathrm i\bx_{\bi}^{(\bn)}\cdot\bomega},
 \qquad \bi\in\mathcal I_{\bn},
\end{equation}
and, for a real \(w\in L^1(\R^d)\), set
\begin{equation}\label{eq:supp-localized-fourier-matrix}
 \mathcal P_{\bn}^{(R)}(w):=(2\pi)^{-d}\int_{\|\bomega\|<R}
 w(\bomega)v_{\bn}(\bomega)v_{\bn}(\bomega)^*\,d\bomega.
\end{equation}
For the principal alias, the change of variables \(\bth=h\bomega\) gives
\begin{equation}\label{eq:supp-localized-fourier-principal-alias}
 T_{\bn}\!\left(h^{-2\alpha}S(\,\cdot\,/h)\mathbb I_{B_r(0)}\right)
 =h^{d-2\alpha}\mathcal P_{\bn}^{(r/h)}(S).
\end{equation}
Thus the fixed ball in the \(\bth\)-variable becomes an expanding ball in the
spectral variable.  The matrix in \eqref{eq:supp-localized-fourier-principal-alias}
need not be small in norm.  What enters the a.c.s. argument is instead a
common range for its sampled modes, whose dimension is small relative to
\(N_{\bn}\) when \(r\) is small.

For an integer \(M\ge1\), write
\begin{equation}\label{eq:supp-localized-fourier-taylor-mode}
 t_M(y):=\sum_{m=0}^{M-1}\frac{(-\mathrm i y)^m}{m!},
 \qquad [\widetilde v_{\bn,M}(\bomega)]_{\bi}:=
 t_M(\bx_{\bi}^{(\bn)}\cdot\bomega).
\end{equation}
Define
\begin{equation}\label{eq:supp-localized-fourier-compression}
 L_{\bn,M}^{(R)}(w):=(2\pi)^{-d}\int_{\|\bomega\|<R}
 w(\bomega)\widetilde v_{\bn,M}(\bomega)\widetilde v_{\bn,M}(\bomega)^*\,d\bomega,
\end{equation}
and, with \(|\bmm|=m_1+\cdots+m_d\) and
\(\bx^{\bmm}=\prod_{j=1}^d x_j^{m_j}\), let
\begin{equation}\label{eq:supp-localized-fourier-space}
 \mathcal U_{\bn,M}:=\operatorname{span}\left\{((\bx_{\bi}^{(\bn)})^{\bmm})_{\bi\in\mathcal I_{\bn}}:
 \bmm\in\mathbb N_0^d,\ |\bmm|<M\right\}.
\end{equation}

\begin{lemma}
\label{lem:explicit-localized-fourier-kernel-compression}
Let \(R>0\), \(0<\varepsilon_{\mathrm{comp}}<1\), and let \(w\in L^1(\R^d)\) be real.
\begin{equation}\label{eq:supp-localized-fourier-order}
 M:=\left\lceil1+\max\{2eL\sqrt d\,R,\log_2(2\varepsilon_{\mathrm{comp}}^{-1})\}\right\rceil.
\end{equation}
Then \(L_{\bn,M}^{(R)}(w)\) and
\(\mathcal P_{\bn}^{(R)}(w)-L_{\bn,M}^{(R)}(w)\) are Hermitian, and
\begin{equation}\label{eq:supp-localized-fourier-rank}
 \operatorname{ran}L_{\bn,M}^{(R)}(w)\subseteq \mathcal U_{\bn,M},\qquad
 \rank L_{\bn,M}^{(R)}(w)\le\dim \mathcal U_{\bn,M}\le\binom{M+d-1}{d}.
\end{equation}
Moreover,
\begin{equation}\label{eq:supp-localized-fourier-error}
 \|\mathcal P_{\bn}^{(R)}(w)-L_{\bn,M}^{(R)}(w)\|_2
 \le\frac{3}{(2\pi)^d}N_{\bn}\varepsilon_{\mathrm{comp}}\|w\|_{L^1(\R^d)}.
\end{equation}
The binomial bound and the choice of \(M\) imply
\(\dim\mathcal U_{\bn,M}\le C\{1+R+\log(\varepsilon_{\mathrm{comp}}^{-1})\}^d\).
\end{lemma}
\begin{proof}
Because \(w\) is real, the defining integrals express both matrices as
integrals of Hermitian rank-one matrices with real weights.  For
\(0\le m<M\), the multinomial formula gives
\[
 (\bx\cdot\bomega)^m=\sum_{\substack{\bmm\in\mathbb N_0^d\\|\bmm|=m}}
 \frac{m!}{m_1!\cdots m_d!}\bx^{\bmm}\bomega^{\bmm}.
\]
Thus \(\widetilde v_{\bn,M}(\bomega)\in\mathcal U_{\bn,M}\); since this
space is finite-dimensional and closed, the integral defining \(L_{\bn,M}^{(R)}(w)\)
also has range in it.  The number of multi-indices with \(|\bmm|<M\) is
\(\binom{M+d-1}{d}\), which proves the range and dimension assertions.

For \(\bx\in[0,L]^d\) and \(\|\bomega\|<R\), Cauchy--Schwarz gives
\[
 |\bx\cdot\bomega|\le\|\bx\|\,\|\bomega\|\le L\sqrt d\,R=:A_R.
\]
The Taylor estimate in \cite[Lemma~1]{CandesDemanetYing2007}, with this
bound and the chosen \(M\), gives
\[
 \sup_{|y|\le A_R}|e^{-\mathrm i y}-t_M(y)|\le\varepsilon_{\mathrm{comp}}.
\]
Consequently,
\[
 \|v_{\bn}(\bomega)-\widetilde v_{\bn,M}(\bomega)\|_2^2
 =\sum_{\bi\in\mathcal I_{\bn}}|e^{-\mathrm i\bx_{\bi}^{(\bn)}\cdot\bomega}
 -t_M(\bx_{\bi}^{(\bn)}\cdot\bomega)|^2
 \le N_{\bn}\varepsilon_{\mathrm{comp}}^2.
\]
Every entry of \(v_{\bn}(\bomega)\) has modulus one, so
\[
 \begin{aligned}
 \|v_{\bn}(\bomega)-\widetilde v_{\bn,M}(\bomega)\|_2&\le\sqrt{N_{\bn}}\varepsilon_{\mathrm{comp}},
 &\|v_{\bn}(\bomega)\|_2&=\sqrt{N_{\bn}},\\
 \|\widetilde v_{\bn,M}(\bomega)\|_2&\le\sqrt{N_{\bn}}(1+\varepsilon_{\mathrm{comp}}).
 \end{aligned}
\]
For vectors \(u\) and \(\widetilde u\),
\[
 uu^*-\widetilde u\widetilde u^*=(u-\widetilde u)u^*+
 \widetilde u(u-\widetilde u)^*,
\]
and hence, uniformly for \(\|\bomega\|<R\),
\[
 \|v_{\bn}(\bomega)v_{\bn}(\bomega)^*-
 \widetilde v_{\bn,M}(\bomega)\widetilde v_{\bn,M}(\bomega)^*\|_2
 \le3N_{\bn}\varepsilon_{\mathrm{comp}}.
\]
Therefore
\begin{align*}
 \|\mathcal P_{\bn}^{(R)}(w)-L_{\bn,M}^{(R)}(w)\|_2
 &\le(2\pi)^{-d}\int_{\|\bomega\|<R}|w(\bomega)|
 3N_{\bn}\varepsilon_{\mathrm{comp}}\,d\bomega\\
 &\le\frac{3}{(2\pi)^d}N_{\bn}\varepsilon_{\mathrm{comp}}\|w\|_{L^1(\R^d)},
\end{align*}
which proves \eqref{eq:supp-localized-fourier-error}.
\end{proof}

\begin{corollary}
\label{cor:localized-fourier-kernel-compression-interface}
There is a constant \(C=C(d,L)\) such that, for every \(R>0\),
\(0<\varepsilon<1\), and sampling index \(\bn\), there is a subspace
\(\mathcal U_{\bn}(R,\varepsilon)\), independent of \(w\), with
\[
 \dim\mathcal U_{\bn}(R,\varepsilon)\le C\{1+R+\log(\varepsilon^{-1})\}^d.
\]
For every real \(w\in L^1(\R^d)\), there are Hermitian matrices
\(H_{\bn}^{(R)}(w)\) and \(F_{\bn}^{(R)}(w)\) such that
\begin{align*}
 \mathcal P_{\bn}^{(R)}(w)&=H_{\bn}^{(R)}(w)+F_{\bn}^{(R)}(w),\\
 \operatorname{ran}H_{\bn}^{(R)}(w)&\subseteq\mathcal U_{\bn}(R,\varepsilon),
 &\|F_{\bn}^{(R)}(w)\|_2&\le CN_{\bn}\varepsilon\|w\|_{L^1(\R^d)}.
\end{align*}
\end{corollary}
\begin{proof}
Set \(\mathcal U_{\bn}(R,\varepsilon)=\mathcal U_{\bn,M}\),
\(H_{\bn}^{(R)}(w)=L_{\bn,M}^{(R)}(w)\), and
\(F_{\bn}^{(R)}(w)=\mathcal P_{\bn}^{(R)}(w)-L_{\bn,M}^{(R)}(w)\),
where Lemma~\ref{lem:explicit-localized-fourier-kernel-compression} is
applied with \(\varepsilon_{\mathrm{comp}}=\varepsilon\).
\end{proof}

For the quantitative argument, define the renormalized aliased symbol and its
limit by
\begin{align}
 \widetilde\phi_h(\bth)&:=h^{-2\alpha}\sum_{\bk\in\mathbb Z^d}
 S\!\left(\frac{\bth+2\pi\bk}{h}\right),
 \label{eq:supp-renormalized-aliased-symbol}\\
 \widetilde\phi(\bth)&:=c_\infty\sum_{\bk\in\mathbb Z^d}
 \|\bth+2\pi\bk\|^{-2\alpha},\qquad \bth\in Q\setminus\{0\},
 \label{eq:supp-renormalized-symbol-limit}
\end{align}
with an arbitrary finite value at the origin.  Thus
\begin{equation}\label{eq:supp-renormalized-covariance-toeplitz}
 h^{d-2\alpha}\Sigma_{\bn}=T_{\bn}(\widetilde\phi_h).
\end{equation}
For \(x\in\mathbb C^{N_{\bn}}\), write
\(p_x(\bth):=\sum_{\bi\in\mathcal I_{\bn}}x_{\bi}e^{\mathrm i\bi\cdot\bth}\).
The Toeplitz quadratic-form identity is
\begin{equation}\label{eq:supp-toeplitz-quadratic-form}
 x^*T_{\bn}(a)x=\int_Qa(\bth)|p_x(\bth)|^2\,d\nu_Q(\bth),
 \qquad a\in L^1(Q).
\end{equation}

\begin{proof}[Proof of Theorem~\ref{thm:quantitative-sampled-covariance-glt}]
Fix \(r\in(0,\pi)\), and write
\[
 p_h^{(r)}(\bth):=h^{-2\alpha}S(\bth/h)\mathbb I_{B_r(0)}(\bth),\quad
 \Sigma_{\bn}^{\mathrm{reg},r}:=T_{\bn}(\widetilde\phi_h-p_h^{(r)}).
\]
Apply Corollary~\ref{cor:localized-fourier-kernel-compression-interface}
with \((R,\varepsilon,w)=(r/h,h^{2\alpha+1},S)\), and set
\[
 H_{\Sigma,\bn}^{(r)}:=h^{d-2\alpha}H_{\bn}^{(r/h)}(S),\qquad
 F_{\Sigma,\bn}^{(r)}:=h^{d-2\alpha}F_{\bn}^{(r/h)}(S).
\]
Equations~\eqref{eq:supp-localized-fourier-principal-alias} and
\eqref{eq:supp-renormalized-covariance-toeplitz} give the stated
decomposition.  Since \(N_{\bn}=(L/h+1)^d\), the corollary gives, for all
sufficiently fine grids and constants independent of \(r\) and \(\bn\),
\[
 \frac{\rank H_{\Sigma,\bn}^{(r)}}{N_{\bn}}\le C\{r^d+h^d\log^d(h^{-1})\},\qquad \|F_{\Sigma,\bn}^{(r)}\|_2\le Ch.
\]
The tail envelope gives
\(\|\widetilde\phi_h-p_h^{(r)}\|_{L^\infty(Q)}\le C(1+r^{-2\alpha})\), so
Equation~\eqref{eq:supp-toeplitz-quadratic-form} gives
\(\|\Sigma_{\bn}^{\mathrm{reg},r}\|_2\le C(1+r^{-2\alpha})\).
{
For the shifted alias \(\bk=\boldsymbol1\), every coordinate of
\(\bth+2\pi\boldsymbol1\) lies in \([\pi,3\pi]\) when \(\bth\in Q\).
Hence
\[
 \pi\sqrt d\le\|\bth+2\pi\boldsymbol1\|\le3\pi\sqrt d,
 \qquad \bth\in Q.
\]
In particular,
\(\inf_{\bth\in Q}\|(\bth+2\pi\boldsymbol1)/h\|
\ge\pi\sqrt d/h\to\infty\), so the tail limit applies uniformly to this
shifted alias.  Therefore, for all sufficiently fine grids,
\[
 h^{-2\alpha}S\!\left(\frac{\bth+2\pi\boldsymbol1}{h}\right)
 \ge \tfrac12c_\infty
 \|\bth+2\pi\boldsymbol1\|^{-2\alpha}
 \ge \tfrac12c_\infty(3\pi\sqrt d)^{-2\alpha}
 =:m_{\mathrm{reg}},
\]
uniformly on \(Q\).  Since the principal contribution outside
\(B_r(0)\) and all remaining aliases are nonnegative, this gives
\(\widetilde\phi_h-p_h^{(r)}\ge m_{\mathrm{reg}}\) uniformly on \(Q\).}  Equation~\eqref{eq:supp-toeplitz-quadratic-form}
therefore gives \(\Sigma_{\bn}^{\mathrm{reg},r}\succeq m_{\mathrm{reg}}I\).

For fixed \(r\), define
\[
 \widetilde\phi^{(r)}(\bth):=
 \begin{cases}
 \widetilde\phi(\bth)-c_\infty\|\bth\|^{-2\alpha}\mathbb I_{B_r(0)}(\bth),&\bth\in Q\setminus\{0\},\\
 c_\infty\displaystyle\sum_{\bk\ne\boldsymbol0}\|2\pi\bk\|^{-2\alpha},&\bth=0.
 \end{cases}
\]
For fixed \(r\), the tail envelope is an integrable majorant, so dominated
convergence gives \(\widetilde\phi_h-p_h^{(r)}\to\widetilde\phi^{(r)}\)
in \(L^1(Q)\).  {Consequently,
\[
 \frac{1}{N_{\bn}}\bigl\|T_{\bn}(\widetilde\phi_h-p_h^{(r)})
 -T_{\bn}(\widetilde\phi^{(r)})\bigr\|_1
 \le(2\pi)^{-d}\|\widetilde\phi_h-p_h^{(r)}-\widetilde\phi^{(r)}\|_{L^1(Q)}
 \longrightarrow0.
\]
By \cite[Theorem~2.3 and Chapter~6, GLT~3]{GaroniSerraCapizzanoII}, the normalized trace-norm estimate makes the difference zero-distributed, and the fixed multilevel Toeplitz sequence \(T_{\bn}(\widetilde\phi^{(r)})\) has GLT symbol \(\widetilde\phi^{(r)}\).  Hence the sequence generated by \(\widetilde\phi_h-p_h^{(r)}\) has symbol \(\widetilde\phi^{(r)}\).}
Finally \(\widetilde\phi^{(r)}\to\widetilde\phi\) in measure; the rank
and residual bounds give the a.c.s. conclusion by taking \(h\to0\) first
and then \(r\downarrow0\).
\end{proof}

\begin{proof}[Proof of Lemma~\ref{lem:renorm-glt-inversion}]
The matrix \(h^{d-2\alpha}\Sigma_{\bn}\) is the Hermitian
multilevel Toeplitz matrix generated by the real, nonnegative renormalized
aliased symbol \(\widetilde\phi_h\) of
\eqref{eq:supp-renormalized-aliased-symbol}.  By
Equation~\eqref{eq:supp-toeplitz-quadratic-form}, for any
\(x\in\mathbb C^{N_{\bn}}\),
\[
        x^*h^{d-2\alpha}\Sigma_{\bn}x
        =\int_Q\widetilde\phi_h(\bth)\,|p_x(\bth)|^2\,d\nu_Q(\bth)
        \ \ge\ \Bigl(\operatorname*{ess\,inf}_{\bth\in Q}\widetilde\phi_h\Bigr)\,\|x\|^2,
        \qquad Q=[-\pi,\pi]^d,
\]
so it suffices to bound \(\operatorname*{ess\,inf}_Q\widetilde\phi_h\) below
uniformly in \(\bn\).  Every summand of
\eqref{eq:supp-renormalized-aliased-symbol} is nonnegative, so discarding
all of them except the one indexed by \(\bk=\boldsymbol0\) leaves the
pointwise lower
bound
\[
        \widetilde\phi_h(\bth)\ \ge\ h^{-2\alpha}S(\bth/h),
        \qquad \bth\in Q .
\]
We bound this \(\bk=\boldsymbol0\) term \(h^{-2\alpha}S(\bth/h)\) on two
regions.  By
\eqref{eq:block-tail-limit-detailed} there is \(R_0\) with
\(S(\bomega)\ge\tfrac12 c_\infty\|\bomega\|^{-2\alpha}\) whenever
\(\|\bomega\|\ge R_0\).  If \(\|\bth\|\ge R_0 h\), then
\(\|\bth/h\|\ge R_0\), and since \(\|\bth\|\le\pi\sqrt d\) on \(Q\),
\[
        h^{-2\alpha}S(\bth/h)
        \ \ge\ \tfrac12 c_\infty\,h^{-2\alpha}\|\bth/h\|^{-2\alpha}
        =\tfrac12 c_\infty\|\bth\|^{-2\alpha}
        \ \ge\ \tfrac12 c_\infty(\pi\sqrt d)^{-2\alpha}=m .
\]
If instead \(\|\bth\|<R_0 h\), then \(\bth/h\) lies in the ball
\(\{\|\bomega\|\le R_0\}\), on which \(S\) attains a positive minimum
\(s_0:=\min_{\|\bomega\|\le R_0}S>0\), because \(S\) is continuous and
strictly positive; hence the same \(\bk=\boldsymbol0\) term satisfies
\(h^{-2\alpha}S(\bth/h)\ge h^{-2\alpha}s_0\), which exceeds \(m\) once \(h\)
is small enough.  In both regions \(\widetilde\phi_h(\bth)\ge m\), so
\(\operatorname*{ess\,inf}_Q\widetilde\phi_h\ge m\) and hence
\(\lambda_{\min}(h^{d-2\alpha}\Sigma_{\bn})\ge m\) for all sufficiently small \(h\).
Uniform invertibility follows.
\end{proof}

\section{Uniform moving-tail spectral estimates}
\label{app:moving-tail-spectral-proofs}

Moving \(\alpha\) produces the common scale drift \(-a_{\bn}\) in the exponent derivative.  The exact relation
\[
h^{-2(\alpha-\alpha_0)}e^{\tau_0}h^{2\alpha-d}
=e^{\tau_0}h^{2\alpha_0-d}
\]
removes the changing leading scale.  The factorization below applies this relation to the sampled covariance, and the compensated direction isolates the remaining spectral shape.

We study
\(e^{-\tau_0}h^{d-2\alpha}\Sigma_{(\tau_0,\alpha),\bn}\).
Throughout this section, \(h=h(\bn)\) is the isotropic mesh scale,
\(a_{\bn}:=2\log(1/h)\), \(\bth\in Q\) is a torus frequency, and
\(\bk\in\mathbb Z^d\) is an alias index.  Let \(\tau_0\) denote the base
value and let
\(\mathcal I_\alpha=[\alpha_-,\alpha_+]\) be a compact interval with
\(\alpha_->d/2\)
and \(\alpha_0\) in its interior, contained in the
\(\tau_0\) exponent slice \(\{\alpha:(\tau_0,\alpha)\in\mathcal U\}\) of
Assumption~\ref{ass:differentiated-spectral-tail-family}.  For
\(\varphi=(\tau,\alpha)\)  write
\begin{equation}
\begin{aligned}
 \Sigma_{\varphi,\bn}
 &=T_{\bn}(\phi_{h,\varphi})
 =e^\tau h^{2\alpha-d}T_{\bn}(b_{h,\alpha}),\\
 e^{-\tau_0}h^{d-2\alpha}\Sigma_{(\tau_0,\alpha),\bn}
 &=T_{\bn}(b_{h,\alpha}).
\end{aligned}
 \label{eq:supp-aliased-likelihood-factorization}
\end{equation}
Thus \(b_{h,\alpha}=e^{-\tau}h^{d-2\alpha}\phi_{h,\varphi}\).  Write $\xi_{\bk}(\bth):=\bth+2\pi\bk,$
\begin{equation}
\begin{aligned}
 \beta_{\bk,h,\alpha}(\bth)&:=h^{-2\alpha}
 \overline S_{\alpha}(\xi_{\bk}(\bth)/h),\\
 b_{h,\alpha}(\bth)&:=\sum_{\bk\in\mathbb Z^d}
 \beta_{\bk,h,\alpha}(\bth).
\end{aligned}
 \label{eq:tail-alias-term-and-sum}
\end{equation}
Below, \(\bth\) is omitted from the notation.

For \(\alpha\in\mathcal I_\alpha\), write
$
 \psi_{h,\alpha}(\bth):=\sum_{\bk\in\mathbb Z^d}
 \{h^2+\|\xi_{\bk}(\bth)\|^2\}^{-\alpha}.
$
The following lemma compares \(b_{h,\alpha}\) with \(\psi_{h,\alpha}\) and
bounds \(\partial_\alpha b_{h,\alpha}\).
\begin{lemma}
\label{lem:uniform-tail-family-comparison}
There are \(0<c_*\le C_*<\infty\), uniformly for
\(\alpha\in\mathcal I_\alpha\), such that
\begin{align}
 c_*(1+\|\bomega\|^2)^{-\alpha}
 &\le\overline S_{\alpha}(\bomega)
 \le C_*(1+\|\bomega\|^2)^{-\alpha},
 \label{eq:uniform-tail-family-comparison}\\
 |\partial_\alpha\overline S_{\alpha}(\bomega)|
 &\le C_*\{1+\log(2+\|\bomega\|)\}(1+\|\bomega\|^2)^{-\alpha},
 \label{eq:uniform-tail-family-derivative}\\
 \intertext{{For all sufficiently fine grids,}}
 c_*\psi_{h,\alpha}\le b_{h,\alpha}\le C_*\psi_{h,\alpha},\qquad
 |\partial_\alpha b_{h,\alpha}|&\le Ca_{\bn}b_{h,\alpha},
 \label{eq:tail-alias-global-comparison}
\end{align}
\end{lemma}
\begin{proof}
Let \(\mathcal I_\alpha=[\alpha_-,\alpha_+]\) be the compact
\(\alpha\)-interval under consideration.  By
Equation~\eqref{eq:spectral-tail-remainder-bound}, uniformly in
\(\alpha\in\mathcal I_\alpha\), the ratio
\(|S^{\mathrm{rem}}_\alpha(\bomega)|/(1+\|\bomega\|^2)^{-\alpha}\)
is bounded by \(C(1+\|\bomega\|^2)^{-\varepsilon}\).  Hence we may choose
\(R<\infty\) so that, whenever \(\|\bomega\|\ge R\),
\[
 \frac12(1+\|\bomega\|^2)^{-\alpha}
 \le \overline S_\alpha(\bomega)
 \le \frac32(1+\|\bomega\|^2)^{-\alpha}.
\]
On the complementary compact-frequency region, the decomposition
\(\overline S_\alpha(\bomega)=(1+\|\bomega\|^2)^{-\alpha}
+S^{\mathrm{rem}}_\alpha(\bomega)\) shows that
\((\alpha,\bomega)\mapsto\overline S_\alpha(\bomega)\) is jointly
continuous: the principal term is jointly continuous and the remainder is
jointly continuous by assumption.  Since \(\overline S_\alpha\) is strictly
positive, the function
\(\overline S_\alpha(\bomega)(1+\|\bomega\|^2)^\alpha\) has a positive
minimum and a finite maximum on
\(\mathcal I_\alpha\times\{\|\bomega\|\le R\}\).  Combining the compact and
tail regions proves Equation~\eqref{eq:uniform-tail-family-comparison}.

For the derivative bound, differentiating the principal term gives
\(\partial_\alpha(1+\|\bomega\|^2)^{-\alpha}
=-\log(1+\|\bomega\|^2)(1+\|\bomega\|^2)^{-\alpha}\).  Together with
Equation~\eqref{eq:spectral-tail-alpha-remainder-bound}, this directly yields
Equation~\eqref{eq:uniform-tail-family-derivative}, uniformly for
\(\alpha\in\mathcal I_\alpha\).  No joint continuity of
\(\partial_\alpha S^{\mathrm{rem}}_\alpha\) is needed here.

Evaluating Equation~\eqref{eq:uniform-tail-family-comparison} at
\(\bomega=\xi_{\bk}(\bth)/h\) and multiplying by \(h^{-2\alpha}\) shows that
\(\beta_{\bk,h,\alpha}(\bth)\) is bounded above and below, with uniform
constants, by \(\{h^2+\|\xi_{\bk}(\bth)\|^2\}^{-\alpha}\).  Summing over
\(\bk\) gives the comparison between \(b_{h,\alpha}\) and
\(\psi_{h,\alpha}\).

It remains to control the \(\alpha\)-derivative of the aliased symbol.
Since \(a_{\bn}=2\log(1/h)\), differentiation gives
\(\partial_\alpha\beta_{\bk,h,\alpha}
=a_{\bn}\beta_{\bk,h,\alpha}
+h^{-2\alpha}\partial_\alpha\overline S_\alpha(\xi_{\bk}/h)\).
Using Equation~\eqref{eq:uniform-tail-family-derivative} and
\(\log(2+\|\xi_{\bk}(\bth)\|/h)
\le a_{\bn}/2+\log(2+\|\xi_{\bk}(\bth)\|)\) for \(h\le1\), we obtain a
summable majorant proportional to
\(\{1+a_{\bn}+\log(2+\|\xi_{\bk}(\bth)\|)\}
\{h^2+\|\xi_{\bk}(\bth)\|^2\}^{-\alpha}\).  Because
\(2\alpha_->d\), the standard lattice estimate
\[
 \sum_{\bk\in\mathbb Z^d}
 \log(2+\|\xi_{\bk}(\bth)\|)
 \{h^2+\|\xi_{\bk}(\bth)\|^2\}^{-\alpha}
 \le C\psi_{h,\alpha}(\bth)
\]
holds uniformly for \(\alpha\in\mathcal I_\alpha\), \(\bth\in Q\), and
sufficiently small \(h\).  Indeed, for large \(\|\bk\|\),
\(\|\xi_{\bk}(\bth)\|\asymp\|\bk\|\) uniformly in \(\bth\), and
\(\sum_{\bk}\log(2+\|\bk\|)(1+\|\bk\|)^{-2\alpha_-}<\infty\); the finitely
many remaining indices are absorbed into the constant.  The same majorant
justifies termwise differentiation of the alias series.  Consequently
\(|\partial_\alpha b_{h,\alpha}(\bth)|
\le C(1+a_{\bn})\psi_{h,\alpha}(\bth)\le C a_{\bn}b_{h,\alpha}(\bth)\) for
all sufficiently fine grids, because \(a_{\bn}\to\infty\).  This proves
Equation~\eqref{eq:tail-alias-global-comparison}.
\end{proof}

We first prove the uniform alias estimates and their two regular-region
consequences.  We then prove the principal-alias weight bounds and the common
compression proposition below.
These results supply the regular-region and low-frequency
estimates used in the proof of
Theorem~\ref{thm:tail-exponent-gradient-information}.

\subsection{Alias estimates and affine-secant consequences}
\label{app:affine-secant-alias-estimates}

Choose \(\varepsilon_{\mathrm{tail}}>0\) for which the bounds in
Assumption~\ref{ass:differentiated-spectral-tail-family} hold and set
\(\varepsilon:=\varepsilon_{\mathrm{tail}}/2\).  Define
\[
        A_{\bk,h}(\bth):=h^2+\|\xi_{\bk}(\bth)\|^2,
        \qquad
        \eta_{\bk,h,\alpha}(\bth)
        :=h^{-2\alpha}S^{\mathrm{rem}}_{\alpha}
        \left(\frac{\xi_{\bk}(\bth)}h\right).
\]

\begin{lemma}
\label{lem:uniform-moving-alias-estimates}
Uniformly over \(\alpha\in\mathcal I_\alpha\), \(\bth\in Q\), and
\(\bk\in\mathbb Z^d\),
\begin{equation}
        \beta_{\bk,h,\alpha}
        =A_{\bk,h}^{-\alpha}+\eta_{\bk,h,\alpha}.
        \label{eq:uniform-moving-alias-decomposition}
\end{equation}
There is a constant independent of \(h,\alpha,\bth,\bk\) such that, on all
sufficiently fine grids,
\begin{align}
        |\eta_{\bk,h,\alpha}|
        &\le Ch^{2\varepsilon}A_{\bk,h}^{-\alpha-\varepsilon},
        &
        |\partial_\alpha\eta_{\bk,h,\alpha}|
        &\le Ch^{2\varepsilon}(1+a_{\bn})
        A_{\bk,h}^{-\alpha-\varepsilon},
        \label{eq:uniform-moving-alias-remainder-bounds}\\
        |\beta_{\bk,h,\alpha}|
        &\le CA_{\bk,h}^{-\alpha},
        &
        |\partial_\alpha\beta_{\bk,h,\alpha}|
        &\le C\{1+|\log A_{\bk,h}|\}A_{\bk,h}^{-\alpha}
        +Ch^{2\varepsilon}(1+a_{\bn})A_{\bk,h}^{-\alpha-\varepsilon}.
        \label{eq:uniform-moving-alias-majorants}
\end{align}
\end{lemma}

\begin{proof}[Proof of Lemma~\ref{lem:uniform-moving-alias-estimates}]
Substituting the expansion in
Assumption~\ref{ass:differentiated-spectral-tail-family} into
Equation~\eqref{eq:tail-alias-term-and-sum} gives
Equation~\eqref{eq:uniform-moving-alias-decomposition}.
The exact remainder bounds
\begin{align*}
 |\eta_{\bk,h,\alpha}|&\le
 Ch^{4\varepsilon}A_{\bk,h}^{-\alpha-2\varepsilon},\\
 h^{-2\alpha}\left|\partial_\alpha
 S^{\mathrm{rem}}_{\alpha}(\xi_{\bk}/h)\right|
 &\le Ch^{2\varepsilon}A_{\bk,h}^{-\alpha-\varepsilon}
\end{align*}
follow respectively from
Equations~\eqref{eq:spectral-tail-remainder-bound} and
\eqref{eq:spectral-tail-alpha-remainder-bound}; in the second, one
\(\varepsilon\)-power absorbs
\(1+\log(2+\|\xi_{\bk}/h\|)\).  Since
\(A_{\bk,h}\ge h^2\), the first bound implies the first remainder bound in
Equation~\eqref{eq:uniform-moving-alias-remainder-bounds}:
\[
 h^{4\varepsilon}A_{\bk,h}^{-\alpha-2\varepsilon}
 =\{h^{2\varepsilon}A_{\bk,h}^{-\varepsilon}\}
 h^{2\varepsilon}A_{\bk,h}^{-\alpha-\varepsilon}
 \le h^{2\varepsilon}A_{\bk,h}^{-\alpha-\varepsilon}.
\]
Differentiating \(h^{-2\alpha}S^{\mathrm{rem}}_{\alpha}
(\xi_{\bk}/h)\) gives the differentiated remainder and
\(a_{\bn}\eta_{\bk,h,\alpha}\).  The two displayed bounds give its
derivative bound.  Finally,
\(\partial_\alpha\beta_{\bk,h,\alpha}=-\log(A_{\bk,h})A_{\bk,h}^{-\alpha}
+\partial_\alpha\eta_{\bk,h,\alpha}\); this identity, the decomposition,
and the two remainder bounds prove
Equation~\eqref{eq:uniform-moving-alias-majorants}.
\end{proof}

The affine path parametrizes the moving exponent secant used in
the proof of Theorem~\ref{thm:tail-exponent-gradient-information}, and the
fundamental theorem of calculus turns its endpoint secant into a path-average
derivative.  The maximal displacement \(\rho_{\bn}\) tracks the cutoff loss.
Since the principal alias is the only alias whose limiting term is singular at
the origin, we subtract only
\(\beta_{\boldsymbol0,h,\alpha}\mathbb I_{B_r(0)}\) and leave the shifted
aliases in the regular symbol.  After integration, the removed contribution
is compressed into a common range independent of the path.

Fix \(r_0\in(0,1)\).  Given \(\delta_{\bn}\to0\), define
\[
        \alpha_{\bn}(t):=\alpha_0+t\delta_{\bn},
        \qquad \rho_{\bn}:=|\delta_{\bn}|.
\]

For \(0<r\le r_0\), define
\[
        b_{h,\alpha}^{r}
        :=b_{h,\alpha}-\beta_{\boldsymbol0,h,\alpha}\mathbb I_{B_r(0)},
        \qquad
        \dot b_{h,\alpha}^{r}
        :=\partial_\alpha b_{h,\alpha}
        -\partial_\alpha \beta_{\boldsymbol0,h,\alpha}\mathbb I_{B_r(0)}.
\]
Along the path, define
\[
        \dot b_{h,\bn}^{\mathrm{sec}}
        :=\int_0^1\partial_\alpha b_{h,\alpha_{\bn}(t)}\,dt,
        \qquad
        \dot b_{h,\bn}^{\mathrm{sec},r}
        :=\int_0^1\dot b_{h,\alpha_{\bn}(t)}^{r}\,dt.
\]
When \(\delta_{\bn}\ne0\), the superscript \(\mathrm{sec}\) denotes the
corresponding difference quotient.

The previously defined finite-mesh alias series
\(\psi_{h,\alpha}\), used to compare \(b_{h,\alpha}\), has
\(\psi_\alpha\) as its limit away from the singular origin; the origin is
excluded because the principal alias is singular.

Recall from Section~\ref{sec:maximum-likelihood} that
\(\psi_\alpha\) is the tail-normalized limiting alias shape and
\(r_{\alpha_0}\) is its relative exponent derivative.  In the present alias
notation, their explicit formulas are
\begin{equation}
\begin{aligned}
\psi_\alpha(\bth)&:=\sum_{\bk\in\mathbb Z^d}
 \|\xi_{\bk}(\bth)\|^{-2\alpha},\\
 \dot\psi_\alpha(\bth)&:=-\sum_{\bk\in\mathbb Z^d}
 \log\!\bigl(\|\xi_{\bk}(\bth)\|^2\bigr)
 \|\xi_{\bk}(\bth)\|^{-2\alpha},\qquad
 r_{\alpha_0}:=\dot\psi_{\alpha_0}/\psi_{\alpha_0}.
\end{aligned}
 \label{eq:tail-limit-symbols}
\end{equation}

{
\begin{lemma}
\label{lem:local-affine-secant-regular-limits}
For \(\alpha_*\in\mathcal I_\alpha\), \(0<r\le r_0\), and
\(\delta_{\bn}\to0\) such that
\(\alpha_{*,\bn}(t):=\alpha_*+t\delta_{\bn}\) belongs to
\(\mathcal I_\alpha\) for \(0\le t\le1\) and all sufficiently large
\(\bn\), let
\[
        \psi_{\alpha_*}^{r}
        :=\psi_{\alpha_*}-\|\bth\|^{-2\alpha_*}\mathbb I_{B_r(0)},
        \qquad
        \dot\psi_{\alpha_*}^{r}
        :=\dot\psi_{\alpha_*}
        +\log(\|\bth\|^2)\|\bth\|^{-2\alpha_*}\mathbb I_{B_r(0)}.
\]
Then
\begin{equation}
        \sup_{0\le t\le1}
        \|b_{h,\alpha_{*,\bn}(t)}^{r}-\psi_{\alpha_*}^{r}\|_{L^1(Q)}
        \longrightarrow0,\quad
        \sup_{0\le t\le1}
        \|\dot b_{h,\alpha_{*,\bn}(t)}^{r}
        -\dot\psi_{\alpha_*}^{r}\|_{L^1(Q)}
        \longrightarrow0,
        \label{eq:local-affine-secant-regular-limits}
\end{equation}
and
$
        \int_0^1\dot b_{h,\alpha_{*,\bn}(t)}^r\,dt
        \longrightarrow\dot\psi_{\alpha_*}^{r}
        \quad\text{in \(L^1(Q)\)}.
$
\end{lemma}

\begin{proof}[Proof of Lemma~\ref{lem:local-affine-secant-regular-limits}]
For \(q\in\{0,1\}\), let \(D_0=\mathrm{id}\),
\(D_1=\partial_\alpha\), and, away from
\(\xi_{\bk}(\bth)=\boldsymbol0\), write
\[
 g_{0,\bk}(\bth)=\|\xi_{\bk}(\bth)\|^{-2\alpha_*},\qquad
 g_{1,\bk}(\bth)=-\log\!\bigl(\|\xi_{\bk}(\bth)\|^2\bigr)
 \|\xi_{\bk}(\bth)\|^{-2\alpha_*}.
\]
The decomposition in Equation~\eqref{eq:uniform-moving-alias-decomposition},
the derivative identity
\(\partial_\alpha\beta_{\bk,h,\alpha}=-\log(A_{\bk,h})A_{\bk,h}^{-\alpha}
+\partial_\alpha\eta_{\bk,h,\alpha}\), and the two bounds in
Equation~\eqref{eq:uniform-moving-alias-remainder-bounds} imply
\begin{equation}
 \sup_{0\le t\le1}|D_q\beta_{\bk,h,\alpha_{*,\bn}(t)}(\bth)-g_{q,\bk}(\bth)|
 \longrightarrow0,\qquad q\in\{0,1\}.
 \label{eq:common-moving-alias-pointwise-limits}
\end{equation}
Indeed, \(A_{\bk,h}\to\|\xi_{\bk}\|^2>0\),
\(\alpha_{*,\bn}(t)\to\alpha_*\) uniformly in \(t\), and the derivative
remainder has the additional factor \(h^{2\varepsilon}(1+a_{\bn})\to0\).
For \(q=1\), this additional factor is the only distinction from the
zeroth-order limit: differentiating the reference term gives
\(-\log(A_{\bk,h})A_{\bk,h}^{-\alpha}\), whose limit is
\(g_{1,\bk}\).  Thus the same pointwise statement covers the two
branches without identifying their remainder bounds.

Put \(Q_{r,\boldsymbol0}=Q\setminus B_r(0)\) and
\(Q_{r,\bk}=Q\) for \(\bk\ne\boldsymbol0\).  The majorants in
Equation~\eqref{eq:uniform-moving-alias-majorants}, together with
\(A_{\boldsymbol0,h}\ge r^2\) on \(Q_{r,\boldsymbol0}\) and uniform
separation of the shifted aliases, give summable sequences
\begin{align*}
 M_{0,r}(\bk)&=C_r\mathbb I_{\{\bk=\boldsymbol0\}}+
 C(1+\|\bk\|)^{-2\alpha_-}\mathbb I_{\{\bk\ne\boldsymbol0\}},\\
 M_{1,r}(\bk)&=C_r\mathbb I_{\{\bk=\boldsymbol0\}}+
 C\{1+\log(2+\|\bk\|)\}(1+\|\bk\|)^{-2\alpha_-}
 \mathbb I_{\{\bk\ne\boldsymbol0\}},
\end{align*}
such that, for \(q\in\{0,1\}\) and \(\bth\in Q_{r,\bk}\), both
\(\sup_t|D_q\beta_{\bk,h,\alpha_{*,\bn}(t)}(\bth)|\) and
\(|g_{q,\bk}(\bth)|\) are bounded by \(M_{q,r}(\bk)\).  Dominated
convergence in the product of counting measure on \(\mathbb Z^d\) and
\(\nu_Q\), applied to Equation~\eqref{eq:common-moving-alias-pointwise-limits},
therefore yields, for \(q\in\{0,1\}\),
\[
 \sum_{\bk\in\mathbb Z^d}\int_{Q_{r,\bk}}
 \sup_{0\le t\le1}|D_q\beta_{\bk,h,\alpha_{*,\bn}(t)}-g_{q,\bk}|
 \,d\nu_Q\longrightarrow0.
\]
For \(q=0\) and \(q=1\), the triangle inequality identifies these sums
with bounds for the \(b\) and \(\dot b\) limits in
Equation~\eqref{eq:local-affine-secant-regular-limits}.  Finally,
\[
 \left\|\int_0^1\dot b_{h,\alpha_{*,\bn}(t)}^r\,dt
 -\dot\psi_{\alpha_*}^{r}\right\|_{L^1(Q)}
 \le\sup_{0\le t\le1}\|\dot b_{h,\alpha_{*,\bn}(t)}^r-
 \dot\psi_{\alpha_*}^r\|_{L^1(Q)}\longrightarrow0,
\]
which proves the last assertion.
\end{proof}
}

\begin{lemma}
\label{lem:local-fixed-alpha-relative-bound}
{
There exist constants \(C<\infty\) and \(h_0>0\) such that, for every
\(\alpha\in\mathcal I_\alpha\), every \(0<r\le r_0\), and every grid with
\(0<h\le h_0\),
\begin{equation}
        |\dot b_{h,\alpha}^{r}|
        \le C\{1+\log(1/r)\}r^{-2|\alpha-\alpha_0|}b_{h,\alpha_0}
        \quad\text{a.e. on \(Q\)}.
        \label{eq:local-fixed-alpha-relative-bound}
\end{equation}
The constants \(C\) and \(h_0\) do not depend on \(\alpha\) or \(r\).
}
\end{lemma}

\begin{proof}[Proof of Lemma~\ref{lem:local-fixed-alpha-relative-bound}]
{
Fix \(\alpha\in\mathcal I_\alpha\).  By definition of the truncated
symbol, the principal contribution to \(\dot b_{h,\alpha}^{r}\) is
\(\partial_\alpha\beta_{\boldsymbol0,h,\alpha}\) on
\(Q\setminus B_r(0)\), while all shifted derivative aliases are retained.
On \(Q\setminus B_r(0)\),
\[
 r^2\le A_{\boldsymbol0,h}\le1+d\pi^2.
\]
Consequently,
\[
 |\log(A_{\boldsymbol0,h})|A_{\boldsymbol0,h}^{-\alpha}
 \le C\{1+\log(1/r)\}r^{-2|\alpha-\alpha_0|}
 A_{\boldsymbol0,h}^{-\alpha_0}.
\]
Indeed, if \(\alpha\ge\alpha_0\), then
\(A_{\boldsymbol0,h}^{-(\alpha-\alpha_0)}
\le r^{-2(\alpha-\alpha_0)}\).  If \(\alpha<\alpha_0\), compactness of
\(\mathcal I_\alpha\) and the upper bound
\(A_{\boldsymbol0,h}\le1+d\pi^2\) give
\(A_{\boldsymbol0,h}^{\alpha_0-\alpha}\le C\).  Thus, in both cases,
\[
 A_{\boldsymbol0,h}^{-(\alpha-\alpha_0)}
 \le C r^{-2|\alpha-\alpha_0|},
\]
while \(|\log A_{\boldsymbol0,h}|\le C\{1+\log(1/r)\}\) on this set.

For the derivative remainder,
Equation~\eqref{eq:uniform-moving-alias-remainder-bounds} gives
\[
 |\partial_\alpha\eta_{\boldsymbol0,h,\alpha}|
 \mathbb I_{Q\setminus B_r(0)}
 \le C h^{2\varepsilon}(1+a_{\bn})
 A_{\boldsymbol0,h}^{-\varepsilon}A_{\boldsymbol0,h}^{-\alpha}
 \mathbb I_{Q\setminus B_r(0)}.
\]
If \(r\le h\), then \(A_{\boldsymbol0,h}\ge h^2\), so
\(h^{2\varepsilon}A_{\boldsymbol0,h}^{-\varepsilon}\le1\), and
\(a_{\bn}=2\log(1/h)\le2\log(1/r)\).  If \(h<r\), then
\(A_{\boldsymbol0,h}\ge r^2\), hence
\[
 h^{2\varepsilon}(1+a_{\bn})A_{\boldsymbol0,h}^{-\varepsilon}
 \le
 (h/r)^{2\varepsilon}
 \{1+2\log(1/r)+2\log(r/h)\}
 \le C\{1+\log(1/r)\},
\]
because \(x^{-2\varepsilon}\log x\) is bounded on \([1,\infty)\).
In both cases, the same comparison gives
\[
 A_{\boldsymbol0,h}^{-\alpha}
 \le C r^{-2|\alpha-\alpha_0|}A_{\boldsymbol0,h}^{-\alpha_0},
\]
and therefore
\[
 |\partial_\alpha\eta_{\boldsymbol0,h,\alpha}|
 \mathbb I_{Q\setminus B_r(0)}
 \le C\{1+\log(1/r)\}r^{-2|\alpha-\alpha_0|}
 A_{\boldsymbol0,h}^{-\alpha_0}.
\]
The lower comparison in Equation~\eqref{eq:tail-alias-global-comparison}
converts both principal-alias estimates into the corresponding multiple of
\(b_{h,\alpha_0}\).

For \(\bk\ne\boldsymbol0\),
Equation~\eqref{eq:uniform-moving-alias-majorants}, the bound
\(h^{2\varepsilon}(1+a_{\bn})\le C_\varepsilon\), and the positive
separation of the shifted aliases from the origin, with a lower bound that
does not depend on \(h\), give
\[
 \sup_{\substack{\alpha\in\mathcal I_\alpha\\\bth\in Q}}
 |\partial_\alpha\beta_{\bk,h,\alpha}(\bth)|
 \le
 C\{1+\log(2+\|\bk\|)\}
 (1+\|\bk\|)^{-2\alpha_-}.
\]
The right-hand side is summable in \(\bk\).  Since the principal term in
\(\psi_{h,\alpha_0}\) is bounded below on \(Q\) by a positive constant
independent of sufficiently small \(h\), the sum of the shifted aliases is
bounded by \(Cb_{h,\alpha_0}\).  Combining the principal and shifted
contributions proves Equation~\eqref{eq:local-fixed-alpha-relative-bound}
with the same constants \(C\) and \(h_0\) for every
\(\alpha\in\mathcal I_\alpha\) and \(0<r\le r_0\).
}
\end{proof}

{
For the affine path \(\alpha_{\bn}(t)=\alpha_0+t\delta_{\bn}\),
\(0\le t\le1\), put \(\rho_{\bn}=|\delta_{\bn}|\).  The fundamental
theorem of calculus and Equation~\eqref{eq:local-fixed-alpha-relative-bound}
give
\[
\begin{aligned}
 |\dot b_{h,\bn}^{\mathrm{sec},r}|
 &\le \int_0^1 |\dot b_{h,\alpha_{\bn}(t)}^r|\,dt\\
 &\le C\{1+\log(1/r)\}
      \int_0^1 r^{-2|\alpha_{\bn}(t)-\alpha_0|}\,dt\;b_{h,\alpha_0}\\
 &\le C\{1+\log(1/r)\}r^{-2\rho_{\bn}}b_{h,\alpha_0},
\end{aligned}
\]
because \(|\alpha_{\bn}(t)-\alpha_0|=t|\delta_{\bn}|\le\rho_{\bn}\)
and \(0<r\le1\).  Thus
\begin{equation}
        |\dot b_{h,\bn}^{\mathrm{sec},r}|
        \le C\{1+\log(1/r)\}r^{-2\rho_{\bn}}b_{h,\alpha_0}
        \quad\text{a.e. on \(Q\)}.
        \label{eq:local-affine-secant-relative-bound}
\end{equation}
The same constant \(C\) as above may be used; in particular it does not
depend on \(\delta_{\bn}\) or \(r\).
}

\begin{proposition}
\label{prop:local-affine-secant-low-frequency-decomposition}
For the affine segment \(\alpha_{\bn}(t)=\alpha_0+t\delta_{\bn}\) and
\(0<r\le r_0\), set
\(R_{h,r}:=r/h\), \(\varepsilon_{\mathrm{comp},h}:=h^{2\alpha_++1}\), and
\[
        M_{h,r}:=\left\lceil 1+\max\left\{
        2eL\sqrt d\,R_{h,r},
        \log_2(2\varepsilon_{\mathrm{comp},h}^{-1})
        \right\}\right\rceil.
\]
There are Hermitian decompositions
\begin{align*}
        T_{\bn}(b_{h,\alpha_{\bn}(1)}-b_{h,\alpha_{\bn}(1)}^r)
        &=L_{b,\bn}^{(r)}+E_{b,\bn}^{(r)},\\
        T_{\bn}(\dot b_{h,\bn}^{\mathrm{sec}}
        -\dot b_{h,\bn}^{\mathrm{sec},r})
        &=L_{\dot b,\bn}^{(r)}+E_{\dot b,\bn}^{(r)},
\end{align*}
whose low-rank terms have the common range
\(\operatorname{ran}L_{s,\bn}^{(r)}\subseteq
\mathcal U_{\bn,M_{h,r}}\), \(s\in\{b,\dot b\}\), and satisfy
\begin{equation}
\begin{aligned}
        \max_{s\in\{b,\dot b\}}
        \frac{\rank L_{s,\bn}^{(r)}}{N_{\bn}}
        &\le C\left(r^d+\frac{a_{\bn}^d}{N_{\bn}}\right),\\
        \|E_{b,\bn}^{(r)}\|_2&\le Ch,
        &\|E_{\dot b,\bn}^{(r)}\|_2&\le Ch(1+a_{\bn}).
\end{aligned}
        \label{eq:local-affine-secant-low-frequency-decomposition}
\end{equation}
The constant \(C\) does not depend on \(\delta_{\bn}\) or \(r\), once the grid is
sufficiently fine.
\end{proposition}

\begin{proof}[Proof of
Proposition~\ref{prop:local-affine-secant-low-frequency-decomposition}]
The two differences in the statement equal, respectively,
\(\beta_{\boldsymbol0,h,\alpha_{\bn}(1)}\mathbb I_{B_r(0)}\) and
\(\{\int_0^1\partial_\alpha\beta_{\boldsymbol0,h,\alpha_{\bn}(t)}dt\}
\mathbb I_{B_r(0)}\).  Write
\begin{align*}
 w_{b,\bn}(\bomega)&:=h^{d-2\alpha_{\bn}(1)}
 \overline S_{\alpha_{\bn}(1)}(\bomega),\\
 w_{\dot b,\bn}(\bomega)&:=\int_0^1h^{d-2\alpha_{\bn}(t)}
 \left\{a_{\bn}\overline S_{\alpha_{\bn}(t)}(\bomega)+
 \partial_\alpha\overline S_{\alpha_{\bn}(t)}(\bomega)\right\}
 \,dt.
\end{align*}
Assumption~\ref{ass:differentiated-spectral-tail-family} gives
real-valuedness and evenness of \(w_{b,\bn}\) and \(w_{\dot b,\bn}\), while
Equations~\eqref{eq:uniform-tail-family-comparison}
and~\eqref{eq:uniform-tail-family-derivative} give
\begin{align*}
 \|w_{b,\bn}\|_{L^1(\mathbb R^d)}
 &\le Ch^{d-2\alpha_+},\\
 \|w_{\dot b,\bn}\|_{L^1(\mathbb R^d)}
 &\le C(1+a_{\bn})h^{d-2\alpha_+}.
\end{align*}
The integrand defining \(w_{\dot b,\bn}\) is absolutely integrable on
\([0,1]\times\mathbb R^d\), so Fubini applies.  The change of variables
\(\bth=h\bomega\) therefore gives the exact identities
\begin{align*}
 T_{\bn}\!\left(\beta_{\boldsymbol0,h,\alpha_{\bn}(1)}\mathbb I_{B_r(0)}\right)
 &=\mathcal P_{\bn}^{(R_{h,r})}(w_{b,\bn}),\\
 T_{\bn}\!\left(\left\{\int_0^1\partial_\alpha
 \beta_{\boldsymbol0,h,\alpha_{\bn}(t)}\,dt\right\}\mathbb I_{B_r(0)}\right)
 &=\mathcal P_{\bn}^{(R_{h,r})}(w_{\dot b,\bn}).
\end{align*}
Apply Lemma~\ref{lem:explicit-localized-fourier-kernel-compression} to
\(w_{b,\bn}\) and \(w_{\dot b,\bn}\), in both cases with \(R=R_{h,r}\) and
\(\varepsilon_{\mathrm{comp}}=h^{2\alpha_++1}\).  It gives the Hermitian
decompositions in the statement and places both low-rank terms in
\(\mathcal U_{\bn,M_{h,r}}\).  Since \(M_{h,r}\le C(r/h+a_{\bn})\) and
\(h^dN_{\bn}=(L+h)^d\), its rank conclusion gives the rank bound in
Equation~\eqref{eq:local-affine-secant-low-frequency-decomposition}.  Its
error conclusion and the \(L^1\) bounds for \(w_{b,\bn}\) and
\(w_{\dot b,\bn}\) give
\[
 \|E_{b,\bn}^{(r)}\|_2\le
 CN_{\bn}h^{2\alpha_++1}h^{d-2\alpha_+}\le Ch,
\qquad \|E_{\dot b,\bn}^{(r)}\|_2\le Ch(1+a_{\bn}).
\]
\end{proof}

\begin{lemma}
\label{lem:working-moving-rescaled-glt}
If \(\alpha_{\bn}\to\alpha_*\) in \(\mathcal I_\alpha\), then
\[\{\Sigma_{(\tau_0,\alpha_{\bn}),\bn}\}_{\bn}
\sim_{\GLT,\,e^{-\tau_0}h^{d-2\alpha_{\bn}}}\psi_{\alpha_*}.\]
There is \(m>0\) such that, uniformly in \(\alpha\in\mathcal I_\alpha\),
\[e^{-\tau_0}h^{d-2\alpha}\Sigma_{(\tau_0,\alpha),\bn}\succeq mI.\]
\end{lemma}
\begin{proof}
The factorization in
Equation~\eqref{eq:supp-aliased-likelihood-factorization} gives
\[
e^{-\tau_0}h^{d-2\alpha_{\bn}}
\Sigma_{(\tau_0,\alpha_{\bn}),\bn}
=
T_{\bn}(b_{h,\alpha_{\bn}}).
\]
Fix \(r\in(0,r_0]\).  The truncated symbol
\(b_{h,\alpha_{\bn}}^r\) removes only the principal alias on \(B_r(0)\).
Lemma~\ref{lem:local-affine-secant-regular-limits}, applied with base point
\(\alpha_*\) and \(\delta_{\bn}=\alpha_{\bn}-\alpha_*\), gives
\[
b_{h,\alpha_{\bn}}^r\longrightarrow\psi_{\alpha_*}^r
\qquad\text{in }L^1(Q).
\]
Set
$
Z_{\bn}^{(r)}
:=
T_{\bn}(b_{h,\alpha_{\bn}}^r)
-
T_{\bn}(\psi_{\alpha_*}^r).
$
The \(p=1\) multilevel Toeplitz norm bound in
\cite[Theorem~3.2, p.~47]{GaroniSerraCapizzanoII} therefore yields
\[
\frac{1}{N_{\bn}}\|Z_{\bn}^{(r)}\|_1\longrightarrow0.
\]
It follows from
\cite[Theorem~2.3, p.~28]{GaroniSerraCapizzanoII}
that \(\{Z_{\bn}^{(r)}\}_{\bn}\) is zero-distributed.  Using the GLT rules
\cite[GLT~3--4, p.~118]{GaroniSerraCapizzanoII} therefore gives
\[
\{T_{\bn}(b_{h,\alpha_{\bn}}^r)\}_{\bn}
=
\{T_{\bn}(\psi_{\alpha_*}^r)+Z_{\bn}^{(r)}\}_{\bn}
\sim_{\GLT}\psi_{\alpha_*}^r.
\]

Corollary~\ref{cor:localized-fourier-kernel-compression-interface}, applied
to the removed principal alias, shows that
$
\bigl\{\{T_{\bn}(b_{h,\alpha_{\bn}}^r)\}_{\bn}\bigr\}_{r\downarrow0}
$
is an a.c.s.\ for \(\{T_{\bn}(b_{h,\alpha_{\bn}})\}_{\bn}\).
Since
\(\psi_{\alpha_*}^r\to\psi_{\alpha_*}\) in measure as \(r\downarrow0\),
the GLT approximation theorem
\cite[Theorem~5.4, p.~95]{GaroniSerraCapizzanoII} yields
\[
\{T_{\bn}(b_{h,\alpha_{\bn}})\}_{\bn}
\sim_{\GLT}\psi_{\alpha_*},
\]
where \(\bn\to\infty\) is taken before \(r\downarrow0\).  Together with the
factorization above, this proves the stated GLT limit.
For all sufficiently fine grids, Equation~\eqref{eq:tail-alias-global-comparison}
and the definition of \(\psi_{h,\alpha}\) give, uniformly in
\(\alpha\in\mathcal I_\alpha\) and \(\bth\in Q\),
\[
b_{h,\alpha}(\bth)
\ge c_*\bigl(h^2+\|\bth\|^2\bigr)^{-\alpha}
\ge c_*(1+d\pi^2)^{-\alpha_+}
=:m.
\]
The Toeplitz quadratic-form identity
\eqref{eq:supp-toeplitz-quadratic-form} therefore yields
\[
e^{-\tau_0}h^{d-2\alpha}
\Sigma_{(\tau_0,\alpha),\bn}
=
T_{\bn}(b_{h,\alpha})
\succeq mI.
\]
\end{proof}
\subsection{Local fixed-whitening and secant controls}
\label{app:local-fixed-whitening-secant-controls}

At the parameter level, \(\tau=\tau_0+a_{\bn}(\alpha-\alpha_0)\) keeps \(e^\tau h^{2\alpha-d}\) constant.  The affine and fixed-whitening constructions below instead compare \(\Sigma_{(\tau_0,\alpha_0+\delta),\bn}\) with \(\Sigma_{\varphi_0,\bn}\) after multiplication by \(h^{-2\delta}\), so their secants follow the same compensated relation on local exponent bands without requiring the parameter pair to remain in the compact chart.

At the true exponent, write $b_{h,0}:=b_{h,\alpha_0}$ and
\begin{align}
        \label{eq:Ralpha}
        R_{\alpha,\bn}:=T_{\bn}(b_{h,0})^{-1/2}
        T_{\bn}\!\left(\left.\partial_\alpha b_{h,\alpha}
        \right|_{\alpha=\alpha_0}\right)T_{\bn}(b_{h,0})^{-1/2}.
\end{align}
The subscript \(\alpha\) labels the exponent-score coordinate, whereas the
derivative is evaluated at \(\alpha_0\).

For \(\alpha\in\mathcal I_\alpha\), write
\begin{align}
        \label{eq:Rmath_alpha}
 \mathcal R_{\bn}(\alpha):=\Sigma_{\varphi_0,\bn}^{-1/2}
 \{\partial_\alpha\Sigma_{(\tau_0,\alpha),\bn}
 +a_{\bn}\Sigma_{(\tau_0,\alpha),\bn}\}\Sigma_{\varphi_0,\bn}^{-1/2}.
\end{align}
For \(\delta\ne0\), set
\begin{equation}
 \mathsf K_{\delta,\bn}:=\frac{h^{-2\delta}\Sigma_{\varphi_0,\bn}^{-1/2}
 \Sigma_{(\tau_0,\alpha_0+\delta),\bn}\Sigma_{\varphi_0,\bn}^{-1/2}-I
 }{\delta},\qquad \mathsf K_{0,\bn}:=R_{\alpha,\bn}.
 \label{eq:supp-normalized-tail-secant}
\end{equation}
Fixed whitening compares the covariance at \(\alpha_0+\delta\) with the base
covariance.  The term \(a_{\bn}\Sigma_{(\tau_0,\alpha),\bn}\) compensates
the scale direction, and the factor \(h^{-2\delta}\) makes the
finite secant in \eqref{eq:supp-normalized-tail-secant} follow the
compensated direction.  With
\(W_{\bn}=T_{\bn}(b_{h,\alpha_0})^{-1/2}\) and
\(\alpha_{\bn}(t)=\alpha_0+t\delta_{\bn}\), applying the fundamental theorem
of calculus along this affine path gives the first equality below, and the
definition of \(\mathcal R_{\bn}\) gives the second:
\begin{equation}
 \begin{aligned}
 \mathsf K_{\delta_{\bn},\bn}
 &=W_{\bn}T_{\bn}\!\left(\int_0^1\partial_\alpha
 b_{h,\alpha_{\bn}(t)}\,dt\right)W_{\bn}\\
 &=\int_0^1 h^{-2t\delta_{\bn}}
 \mathcal R_{\bn}(\alpha_{\bn}(t))\,dt.
 \end{aligned}
 \label{eq:fixed-whitening-compensated-secant-path}
\end{equation}
Both equalities hold with the derivative at \(\alpha_0\) when
\(\delta_{\bn}=0\).  The first equality is unweighted in
\(b_{h,\alpha}\), while the second rewrites the same matrix in terms of the
compensated relative score.  Assumption~\ref{ass:differentiated-spectral-tail-family} supplies
the \(C^1\) dependence in \(\alpha\), and
Equation~\eqref{eq:uniform-tail-family-derivative} supplies summable
derivative majorants; the standing path assumption keeps
\(\alpha_{\bn}([0,1])\) inside \(\mathcal I_\alpha\).

The following stages give path control, the fixed-whitening secant
decomposition and GLT limit, score continuity, and the passage from secant
tails to trace and operator-norm control.
For the rest of this subsection, let \(\delta_{\bn}\to0\),  suppose that
\(\alpha_{\bn}([0,1])\subseteq\mathcal I_\alpha\) for all sufficiently large
\(\bn\).  Theorem~\ref{thm:quantitative-sampled-covariance-glt} and
Lemma~\ref{lem:renorm-glt-inversion} give
\(W_{\bn}\sim_{\GLT}\psi_{\alpha_0}^{-1/2}\) and
\(\sup_{\bn}\|W_{\bn}\|_2<\infty\).

\begin{lemma}
\label{lem:local-fixed-whitening-path-control}
Let \(V_{(\tau_0,\alpha),\bn}:=\Sigma_{\varphi_0,\bn}^{-1/2}
\Sigma_{(\tau_0,\alpha),\bn}\Sigma_{\varphi_0,\bn}^{-1/2}\).  If
\(|\delta_{\bn}|a_{\bn}\le\rho\), then {for all sufficiently
fine grids,}
\begin{equation}
 \sup_{0\le t\le1}\|V_{(\tau_0,\alpha_{\bn}(t)),\bn}-I\|_2
 \le e^{C\rho}-1.
 \label{eq:moving-fixed-whitening-covariance-band}
\end{equation}
Moreover,
{for those grids,}
\begin{equation}
 \|\mathsf K_{\delta_{\bn},\bn}\|_2\le Ca_{\bn}e^{C\rho}.
 \label{eq:fixed-whitening-secant-global-bound}
\end{equation}
\end{lemma}
\begin{proof}
Set \(q_\alpha=h^{2(\alpha-\alpha_0)}b_{h,\alpha}\).  Then
\(V_{(\tau_0,\alpha),\bn}=W_{\bn}T_{\bn}(q_\alpha)W_{\bn}\), and
Equation~\eqref{eq:tail-alias-global-comparison} gives
\(|\partial_\alpha\log q_\alpha|\le Ca_{\bn}\).  Integrating along the
path yields \(e^{-C\rho}b_{h,\alpha_0}\le q_{\alpha_{\bn}(t)}
\le e^{C\rho}b_{h,\alpha_0}\), proving
Equation~\eqref{eq:moving-fixed-whitening-covariance-band}.

By the second equality in
Equation~\eqref{eq:fixed-whitening-compensated-secant-path},
\(\mathsf K_{\delta_{\bn},\bn}\) is the integral of
\(h^{-2t\delta_{\bn}}\mathcal R_{\bn}(\alpha_{\bn}(t))\).  Since
\(h^{-2t\delta_{\bn}}\le\exp(|\delta_{\bn}|a_{\bn})\), the Toeplitz
quadratic-form bound and
Equation~\eqref{eq:moving-fixed-whitening-covariance-band} give
\(-Ca_{\bn}e^{C\rho}I\preceq\mathcal R_{\bn}(\alpha_{\bn}(t))
\preceq Ca_{\bn}e^{C\rho}I\), after increasing \(C\).  Integrating
the preceding integrand bound proves
Equation~\eqref{eq:fixed-whitening-secant-global-bound}.
\end{proof}

\begin{lemma}
\label{lem:local-fixed-whitening-secant-spectral-control}
For every fixed \(r\in(0,r_0]\), set
\(\mathcal V_{\bn}^{(r)}:=W_{\bn}\mathcal U_{\bn,M_{h,r}}\).  Define the
Hermitian matrices
\[
 \begin{aligned}
 K_{\bn}^{\mathrm{reg},r}
 &:=W_{\bn}T_{\bn}(\dot b_{h,\bn}^{\mathrm{sec},r})W_{\bn},\\
 L_{\bn}^{(r)}&:=W_{\bn}L_{\dot b,\bn}^{(r)}W_{\bn},\\
 E_{\bn}^{(r)}&:=W_{\bn}E_{\dot b,\bn}^{(r)}W_{\bn}.
 \end{aligned}
\]
The first equality in
Equation~\eqref{eq:fixed-whitening-compensated-secant-path} gives
\begin{equation}
 \mathsf K_{\delta_{\bn},\bn}=K_{\bn}^{\mathrm{reg},r}
 +L_{\bn}^{(r)}+E_{\bn}^{(r)}.
 \label{eq:fixed-whitening-secant-decomposition}
\end{equation}
The range of \(L_{\bn}^{(r)}\) is contained in \(\mathcal V_{\bn}^{(r)}\), and
\[
 \{K_{\bn}^{\mathrm{reg},r}\}_{\bn}\sim_{\GLT}
 \dot\psi_{\alpha_0}^r/\psi_{\alpha_0},\qquad
 \|K_{\bn}^{\mathrm{reg},r}\|_2
 \le C\{1+\log(1/r)\}r^{-2|\delta_{\bn}|},
\]
while
\[
 \frac{\rank L_{\bn}^{(r)}}{N_{\bn}},\
 \frac{\dim\mathcal V_{\bn}^{(r)}}{N_{\bn}}
 \le C\left(r^d+\frac{a_{\bn}^d}{N_{\bn}}\right),\qquad
 \|E_{\bn}^{(r)}\|_2\le Ch(1+a_{\bn}).
\]
Consequently, \(\{\mathsf K_{\delta_{\bn},\bn}\}_{\bn}
\sim_{\GLT}r_{\alpha_0}\).
\end{lemma}
\begin{proof}
Proposition~\ref{prop:local-affine-secant-low-frequency-decomposition}
gives the common range, rank bound, and error bound before whitening.
Boundedness of \(W_{\bn}\) preserves the rank and error conclusions and
maps the common range to \(\mathcal V_{\bn}^{(r)}\).
Lemma~\ref{lem:local-affine-secant-regular-limits} gives the regular GLT
symbol.  For the operator-norm bound,
Equation~\eqref{eq:local-affine-secant-relative-bound} gives
\[
 |\dot b_{h,\bn}^{\mathrm{sec},r}|
 \le C\{1+\log(1/r)\}r^{-2|\delta_{\bn}|}b_{h,\alpha_0}
 \quad\text{a.e. on \(Q\)}.
\]
Hence the Toeplitz quadratic-form inequality gives
\[
 -c_{\bn,r}T_{\bn}(b_{h,\alpha_0})
 \preceq T_{\bn}(\dot b_{h,\bn}^{\mathrm{sec},r})
 \preceq c_{\bn,r}T_{\bn}(b_{h,\alpha_0}),
 \qquad
 c_{\bn,r}:=C\{1+\log(1/r)\}r^{-2|\delta_{\bn}|}.
\]
Congruence by \(W_{\bn}\), together with
\(W_{\bn}T_{\bn}(b_{h,\alpha_0})W_{\bn}=I\), yields
\(\|K_{\bn}^{\mathrm{reg},r}\|_2\le c_{\bn,r}\).
The GLT approximation principle then gives the claimed GLT limit.
\end{proof}
{
\begin{lemma}
\label{lem:local-fixed-whitening-score-spectral-tails}
For every fixed \(\rho_0>0\), there exist constants \(c,C>0\) and
\(h_0>0\) such that, for every grid with \(0<h\le h_0\), every
\(\alpha\in\mathcal I_\alpha\) satisfying
\(|\alpha-\alpha_0|a_{\bn}\le\rho_0\), and every \(t\ge0\),
\begin{equation}
 \frac{1}{N_{\bn}}
 \#\{j:|\lambda_j(\mathcal R_{\bn}(\alpha))|>t\}
 \le Ce^{-ct}+C\frac{a_{\bn}^d}{N_{\bn}}.
 \label{eq:fixed-whitening-eigenvalue-tail}
\end{equation}
The constants \(c,C,h_0\) do not depend on \(\alpha\) or \(t\).
Consequently,
\[
 \sup_{\substack{\alpha\in\mathcal I_\alpha\\
 |\alpha-\alpha_0|a_{\bn}\le\rho_0}}
 \frac{\tr|\mathcal R_{\bn}(\alpha)|^4}{N_{\bn}}
 \le C+C\frac{a_{\bn}^{d+4}}{N_{\bn}}=O(1).
\]
\end{lemma}

\begin{proof}
Put \(\delta=\alpha-\alpha_0\).  The aliased likelihood factorization
and the definition of \(\mathcal R_{\bn}(\alpha)\) give
\[
 \mathcal R_{\bn}(\alpha)
 =h^{2\delta}W_{\bn}
 T_{\bn}(\partial_\alpha b_{h,\alpha})W_{\bn}.
\]
For \(0<r\le r_0\), the substitution \(\bth=h\bomega\) identifies the
principal derivative alias on \(B_r(0)\) with the localized Fourier
matrix having weight
\[
 w_{\alpha,\bn}^{\mathrm{lf}}(\bomega)
 :=h^{d-2\alpha_0}
 \{a_{\bn}\overline S_\alpha(\bomega)
   +\partial_\alpha\overline S_\alpha(\bomega)\}.
\]
Apply Lemma~\ref{lem:explicit-localized-fourier-kernel-compression}
with \(R=r/h\), compression tolerance \(h^{2\alpha_++1}\), and the
polynomial order \(M_{h,r}\) from
Proposition~\ref{prop:local-affine-secant-low-frequency-decomposition}.
Together with boundedness of \(W_{\bn}\), this gives a Hermitian
decomposition
\[
 \mathcal R_{\bn}(\alpha)
 =G_{\alpha,\bn}^{(r)}
  +L_{\alpha,\bn}^{(r)}+E_{\alpha,\bn}^{(r)},
 \qquad
 G_{\alpha,\bn}^{(r)}
 :=h^{2\delta}W_{\bn}
 T_{\bn}(\dot b_{h,\alpha}^{r})W_{\bn},
\]
where the constants in
\[
 \frac{\rank L_{\alpha,\bn}^{(r)}}{N_{\bn}}
 \le C\left(r^d+\frac{a_{\bn}^d}{N_{\bn}}\right),
 \qquad
 \|E_{\alpha,\bn}^{(r)}\|_2\le Ch(1+a_{\bn})
\]
do not depend on \(r\) or \(\alpha\).  Indeed, the unwhitened compression
error is bounded by a constant multiple of
\[
 N_{\bn}h^{2\alpha_++1}(1+a_{\bn})h^{d-2\alpha_0}
 =(1+a_{\bn})(L+h)^d h^{2(\alpha_+-\alpha_0)+1}
 \le Ch(1+a_{\bn}).
\]

Equation~\eqref{eq:local-fixed-alpha-relative-bound} gives
\[
 |\dot b_{h,\alpha}^{r}|
 \le C\{1+\log(1/r)\}r^{-2|\delta|}b_{h,\alpha_0}.
\]
Since \(|\delta|a_{\bn}\le\rho_0\),
\(h^{2\delta}=\exp(-\delta a_{\bn})\le e^{\rho_0}\).

The Toeplitz
quadratic-form identity and
\(W_{\bn}T_{\bn}(b_{h,\alpha_0})W_{\bn}=I\) therefore imply
\begin{equation}
 \|G_{\alpha,\bn}^{(r)}\|_2
 \le Ce^{\rho_0}\{1+\log(1/r)\}r^{-2|\delta|}.
 \label{eq:fixed-whitening-score-regular-op}
\end{equation}

Equation~\eqref{eq:tail-alias-global-comparison} gives
\[
 -Ca_{\bn}V_{(\tau_0,\alpha),\bn}
 \preceq \mathcal R_{\bn}(\alpha)
 \preceq Ca_{\bn}V_{(\tau_0,\alpha),\bn}.
\]
Lemma~\ref{lem:local-fixed-whitening-path-control}, applied along the
affine segment from \(\alpha_0\) to \(\alpha\), yields
\(\|V_{(\tau_0,\alpha),\bn}\|_2\le e^{C\rho_0}\).  Hence there is a
constant \(C_0\), depending on \(\rho_0\) but not on \(\alpha\), such that
\[
 \|\mathcal R_{\bn}(\alpha)\|_2\le C_0a_{\bn}.
\]
Thus the eigenvalue count in the statement is zero for \(t>C_0a_{\bn}\).

For the remaining range choose
\(r=e^{-\kappa t}\).  We shall apply this choice only after selecting a
fixed \(t_0\) such that \(e^{-\kappa t_0}\le r_0\).  When
\(t\le C_0a_{\bn}\),
\[
 t|\delta|\le C_0a_{\bn}|\delta|\le C_0\rho_0,
\]
and Equation~\eqref{eq:fixed-whitening-score-regular-op} gives
\[
 \|G_{\alpha,\bn}^{(r)}\|_2
 \le Ce^{\rho_0}e^{2\kappa C_0\rho_0}(1+\kappa t).
\]
Set
\[
 A_\kappa:=Ce^{\rho_0}e^{2\kappa C_0\rho_0}.
\]
Then
\[
 \|G_{\alpha,\bn}^{(r)}\|_2
 \le A_\kappa+\kappa A_\kappa t.
\]
Since
\[
 \kappa A_\kappa
 =Ce^{\rho_0}\kappa e^{2\kappa C_0\rho_0}\longrightarrow0
 \qquad(\kappa\downarrow0),
\]
choose \(\kappa>0\) so that \(\kappa A_\kappa\le1/8\).  Having fixed
\(\kappa\), choose \(t_0\) so large that
\[
 t_0\ge8A_\kappa,
 \qquad
 e^{-\kappa t_0}\le r_0.
\]
For every \(t\ge t_0\),
\[
 A_\kappa\le\frac{t}{8},
 \qquad
 \kappa A_\kappa t\le\frac{t}{8},
\]
and therefore
\[
 \|G_{\alpha,\bn}^{(r)}\|_2\le\frac{t}{4},
 \qquad t_0\le t\le C_0a_{\bn}.
\]
Because \(h(1+a_{\bn})\to0\), by decreasing \(h_0\) if necessary we
also have
\[
 \|E_{\alpha,\bn}^{(r)}\|_2\le\frac{t_0}{4}\le\frac{t}{4}
\]
for every \(t\ge t_0\) and every grid with \(0<h\le h_0\).  Hence
\[
 \|G_{\alpha,\bn}^{(r)}+E_{\alpha,\bn}^{(r)}\|_2\le\frac{t}{2}.
\]
Let \(m=\rank L_{\alpha,\bn}^{(r)}\).  The singular-value rank
inequality gives
\[
 s_{m+1}\bigl(\mathcal R_{\bn}(\alpha)\bigr)
 \le \|G_{\alpha,\bn}^{(r)}+E_{\alpha,\bn}^{(r)}\|_2
 \le \frac{t}{2}.
\]
Since \(\mathcal R_{\bn}(\alpha)\) is Hermitian, its singular values are
the absolute values of its eigenvalues.  Thus
\[
 \begin{aligned}
 \frac{1}{N_{\bn}}
 \#\{j:|\lambda_j(\mathcal R_{\bn}(\alpha))|>t\}
 &\le \frac{\rank L_{\alpha,\bn}^{(r)}}{N_{\bn}}\\
 &\le Ce^{-\kappa dt}+C\frac{a_{\bn}^d}{N_{\bn}},
 \end{aligned}
\]
for \(t_0\le t\le C_0a_{\bn}\).

It remains to cover \(0\le t<t_0\).  Let \(C_1\) be a constant for which
the preceding estimate holds on \([t_0,C_0a_{\bn}]\), set
\[
 c:=\kappa d,
 \qquad
 C_2:=\max\{C_1,e^{ct_0}\}.
\]
For \(0\le t<t_0\), the normalized eigenvalue count is at most one and
\(e^{-ct}\ge e^{-ct_0}\), so
\[
 \frac{1}{N_{\bn}}
 \#\{j:|\lambda_j(\mathcal R_{\bn}(\alpha))|>t\}
 \le1\le C_2e^{-ct}
 \le C_2e^{-ct}+C_2\frac{a_{\bn}^d}{N_{\bn}}.
\]
For \(t>C_0a_{\bn}\) the count is zero.  Replacing \(C_1\) by \(C_2\)
therefore proves Equation~\eqref{eq:fixed-whitening-eigenvalue-tail} for
all \(t\ge0\).

Finally, the layer-cake identity gives
\[
 \begin{aligned}
 \frac{\tr|\mathcal R_{\bn}(\alpha)|^4}{N_{\bn}}
 &=4\int_0^{C_0a_{\bn}}t^3
 \frac{\#\{j:|\lambda_j(\mathcal R_{\bn}(\alpha))|>t\}}{N_{\bn}}\,dt\\
 &\le C+C\frac{a_{\bn}^{d+4}}{N_{\bn}}.
 \end{aligned}
\]
The last term tends to zero because
\(a_{\bn}=2\log(1/h)\) and \(N_{\bn}\asymp h^{-d}\).
\end{proof}
}

\begin{lemma}
\label{lem:local-fixed-whitening-score-continuity}
In normalized Hilbert--Schmidt norm, \(\mathcal R_{\bn}(\alpha)\) converges
to \(R_{\alpha,\bn}\) uniformly over \(a_{\bn}^{-1}\)-neighborhoods of
\(\alpha_0\):
\[
 \lim_{\rho\downarrow0}\limsup_{\bn}
 \sup_{\substack{\alpha\in\mathcal I_\alpha\\
 |\alpha-\alpha_0|a_{\bn}\le\rho}}
 \frac{\|\mathcal R_{\bn}(\alpha)-R_{\alpha,\bn}\|_{\mathrm{HS}}}{\sqrt{N_{\bn}}}=0.
\]
\end{lemma}

\begin{proof}
Put \(D_{\alpha,\bn}:=\mathcal R_{\bn}(\alpha)-R_{\alpha,\bn}\).
Since \(R_{\alpha,\bn}=\mathcal R_{\bn}(\alpha_0)\), subtracting the two
decompositions in the preceding proof, with the same \(r\), compression
tolerance, and polynomial order, gives
\[
 D_{\alpha,\bn}
 =\widetilde G_{\alpha,\bn}^{(r)}
  +\widetilde L_{\alpha,\bn}^{(r)}
  +\widetilde E_{\alpha,\bn}^{(r)},
\]
where
\[
 \begin{aligned}
 \widetilde G_{\alpha,\bn}^{(r)}
 &:=W_{\bn}T_{\bn}\!\left(
 h^{2(\alpha-\alpha_0)}\dot b_{h,\alpha}^{r}
 -\dot b_{h,\alpha_0}^{r}\right)W_{\bn},\\
 \widetilde L_{\alpha,\bn}^{(r)}
 &:=L_{\alpha,\bn}^{(r)}-L_{\alpha_0,\bn}^{(r)},\\
 \widetilde E_{\alpha,\bn}^{(r)}
 &:=E_{\alpha,\bn}^{(r)}-E_{\alpha_0,\bn}^{(r)}.
 \end{aligned}
\]
{
We spell out the fixed-\(r\) estimate used here.  Since the desired result
takes the ordered limit \(\rho\downarrow0\), it is enough to consider
\(0<\rho\le1\).  Put \(\delta:=\alpha-\alpha_0\).  On the truncated
principal region \(Q\setminus B_r(0)\), and on every shifted alias over
\(Q\), the reference part of the differentiated alias satisfies
\[
 \partial_\alpha A_{\bk,h}^{-\alpha}
 =-\log(A_{\bk,h})A_{\bk,h}^{-\alpha}.
\]
Hence
\[
\begin{aligned}
 &h^{2\delta}
   \{-\log(A_{\bk,h})A_{\bk,h}^{-\alpha}\}
 -\{-\log(A_{\bk,h})A_{\bk,h}^{-\alpha_0}\}\\
 &\quad=
 (h^{2\delta}-1)
 \{-\log(A_{\bk,h})A_{\bk,h}^{-\alpha}\}
 -\log(A_{\bk,h})
 \{A_{\bk,h}^{-\alpha}-A_{\bk,h}^{-\alpha_0}\}.
\end{aligned}
\]
Because \(|\delta|a_{\bn}\le\rho\),
\(|h^{2\delta}-1|\le e^\rho-1\).  For fixed \(r\), the truncated
principal alias is bounded away from zero, while the shifted aliases are
uniformly separated from zero and satisfy the summable logarithmic majorants
from Lemma~\ref{lem:uniform-moving-alias-estimates}.  Consequently the
first term, after summing the aliases, has \(L^2(\nu_Q)\)-norm at most
\(C_r(e^\rho-1)\), uniformly whenever \(|\delta|a_{\bn}\le\rho\).
For the second term, the mean-value theorem gives
\[
 |A_{\bk,h}^{-\alpha}-A_{\bk,h}^{-\alpha_0}|
 \le
 |\delta|\,|\log A_{\bk,h}|
 \sup_{\gamma\in[\alpha\wedge\alpha_0,
                      \alpha\vee\alpha_0]}
 A_{\bk,h}^{-\gamma}.
\]
The same fixed-\(r\) summable alias majorants, now with one additional
logarithmic factor, imply an \(L^2(\nu_Q)\) bound \(C_r|\delta|\).
Since \(a_{\bn}\ge1\) on all sufficiently fine grids,
\(|\delta|\le\rho\), and this contribution is at most \(C_r\rho\).

For the differentiated remainder, the triangle inequality and
Equation~\eqref{eq:uniform-moving-alias-remainder-bounds} give
\[
\begin{aligned}
 &\left|
 h^{2\delta}\partial_\alpha\eta_{\bk,h,\alpha}
 -\partial_\alpha\eta_{\bk,h,\alpha_0}
 \right|\\
 &\qquad\le
 e^\rho|\partial_\alpha\eta_{\bk,h,\alpha}|
 +|\partial_\alpha\eta_{\bk,h,\alpha_0}|
 \le
 C h^{2\varepsilon}(1+a_{\bn})
 \{A_{\bk,h}^{-\alpha-\varepsilon}
   +A_{\bk,h}^{-\alpha_0-\varepsilon}\},
\end{aligned}
\]
where \(e^\rho\le e\) has been absorbed into \(C\).  On the same
truncated alias domains the right-hand side is summable in \(\bk\) and
bounded in \(L^2(\nu_Q)\) by
\(C_rh^{2\varepsilon}(1+a_{\bn})\).  Summing the reference and remainder
contributions therefore proves
\[
 \left\|
 h^{2(\alpha-\alpha_0)}\dot b_{h,\alpha}^{r}
 -\dot b_{h,\alpha_0}^{r}
 \right\|_{L^2(\nu_Q)}
 \le C_r\{e^\rho-1+\rho+h^{2\varepsilon}(1+a_{\bn})\}
\]
whenever \(|\alpha-\alpha_0|a_{\bn}\le\rho\), as claimed.
}
The inequality
\(\|T_{\bn}(f)\|_{\mathrm{HS}}
 \le\sqrt{N_{\bn}}\|f\|_{L^2(\nu_Q)}\)
and boundedness of \(W_{\bn}\) therefore yield
\[
 \begin{aligned}
 \frac{\|\widetilde G_{\alpha,\bn}^{(r)}\|_{\mathrm{HS}}}{\sqrt{N_{\bn}}}
 &\le C_r\{e^\rho-1+\rho+h^{2\varepsilon}(1+a_{\bn})\},\\
 \frac{\rank\widetilde L_{\alpha,\bn}^{(r)}}{N_{\bn}}
 &\le C\left(r^d+\frac{a_{\bn}^d}{N_{\bn}}\right),\\
 \|\widetilde E_{\alpha,\bn}^{(r)}\|_2
 &\le Ch(1+a_{\bn}).
 \end{aligned}
\]
Here the constants do not depend on \(\alpha\) subject to the displayed
restriction.

Fix \(r>0\) and \(\eta>0\).  Since \(h(1+a_{\bn})\to0\), the error norm
is eventually at most \(\eta/2\).  The singular-value rank inequality
then gives
\[
 \frac{\#\{j:s_j(D_{\alpha,\bn})>\eta\}}{N_{\bn}}
 \le
 \frac{\rank\widetilde L_{\alpha,\bn}^{(r)}}{N_{\bn}}
 +\frac{4}{\eta^2}
 \frac{\|\widetilde G_{\alpha,\bn}^{(r)}\|_{\mathrm{HS}}^2}{N_{\bn}}.
\]
Consequently, for each fixed \(r\),
\[
 \lim_{\rho\downarrow0}\limsup_{\bn}
 \sup_{\substack{\alpha\in\mathcal I_\alpha\\
 |\alpha-\alpha_0|a_{\bn}\le\rho}}
 \frac{\#\{j:s_j(D_{\alpha,\bn})>\eta\}}{N_{\bn}}
 \le Cr^d.
\]
Letting \(r\downarrow0\) makes this count vanish.

{
Fix \(\rho_0>0\).  Lemma~\ref{lem:local-fixed-whitening-score-spectral-tails},
applied at \(\alpha\) and at \(\alpha_0\), and the triangle inequality for
the Schatten--\(4\) norm (as a unitarily invariant norm; see
\cite[Theorem~3.5.18, p.~210]{HornJohnson1991}), followed by
\((x+y)^4\le8(x^4+y^4)\) for \(x,y\ge0\), give
\[
 \frac{\tr|D_{\alpha,\bn}|^4}{N_{\bn}}
 \le 8\left\{
 \frac{\tr|\mathcal R_{\bn}(\alpha)|^4}{N_{\bn}}
 +\frac{\tr|R_{\alpha,\bn}|^4}{N_{\bn}}
 \right\}\le C
\]
for every \(\alpha\) satisfying
\(|\alpha-\alpha_0|a_{\bn}\le\rho_0\) and all sufficiently fine grids.
Cauchy--Schwarz therefore yields
\[
 \frac{\|D_{\alpha,\bn}\|_{\mathrm{HS}}^2}{N_{\bn}}
 \le \eta^2+
 \left(\frac{\tr|D_{\alpha,\bn}|^4}{N_{\bn}}\right)^{1/2}
 \left(\frac{\#\{j:s_j(D_{\alpha,\bn})>\eta\}}{N_{\bn}}\right)^{1/2}.
\]
Taking first \(\bn\to\infty\), then \(\rho\downarrow0\), then
\(r\downarrow0\), and finally \(\eta\downarrow0\), proves the claim.
}
\end{proof}

\begin{lemma}
\label{lem:local-fixed-whitening-secant-tails}
For every sequence \(\delta_{\bn}\to0\),
\[
\frac{\tr\mathsf K_{\delta_{\bn},\bn}}{N_{\bn}}
\rightarrow \int_Q r_{\alpha_0}(\bth)\,d\nu_Q(\bth),
\quad
\frac{\tr\mathsf K_{\delta_{\bn},\bn}^{\,2}}{N_{\bn}}
\rightarrow \int_Q r_{\alpha_0}(\bth)^2\,d\nu_Q(\bth),
\quad
\frac{\|\mathsf K_{\delta_{\bn},\bn}\|_2}{\sqrt{N_{\bn}}}
\rightarrow0.
\]
\end{lemma}
In particular, since
\(\mathsf K_{0,\bn}=R_{\alpha,\bn}\), taking
\(\delta_{\bn}\equiv0\) gives
\begingroup\small
\begin{equation}
 \frac{\operatorname{tr}R_{\alpha,\bn}}{N_{\bn}}
 \rightarrow
 \int_Q r_{\alpha_0}(\bth)\,d\nu_Q(\bth),\quad
 \frac{\operatorname{tr}(R_{\alpha,\bn}^2)}{N_{\bn}}
 \rightarrow
 \int_Q r_{\alpha_0}(\bth)^2\,d\nu_Q(\bth),\quad
 \frac{\|R_{\alpha,\bn}\|_2}{\sqrt{N_{\bn}}}
 \rightarrow0.
\label{eq:Ralpha-trace-limits}
\end{equation}
\endgroup
\begin{proof}
We first establish a polynomial bound for the spectral tails.
Fix \(p>2\), choose \(q>0\) such that \(qd>p\), and take
\(r=(1+t)^{-q}\).  Since \(\delta_{\bn}\to0\), for all sufficiently
large \(\bn\), \(2q|\delta_{\bn}|\le\frac{1}{4}\).  Lemma~
\ref{lem:local-fixed-whitening-secant-spectral-control} gives
\[
\|K_{\bn}^{\mathrm{reg},r}\|_2
\le
C\{1+q\log(1+t)\}(1+t)^{1/4}.
\]
We may therefore choose \(t_p\) independently of \(\bn\) so that
\(r\le r_0\) and
\[
\|K_{\bn}^{\mathrm{reg},r}\|_2\le\frac{t}{2},
\qquad t\ge t_p.
\]
For all sufficiently large \(\bn\), we also have
\(\|E_{\bn}^{(r)}\|_2\le t/2\).  Hence
\(\|K_{\bn}^{\mathrm{reg},r}+E_{\bn}^{(r)}\|_2\le t\),
and Hermitian rank perturbation therefore gives
\[
\begin{aligned}
\frac{1}{N_{\bn}}
\#\bigl\{j:
|\lambda_j(\mathsf K_{\delta_{\bn},\bn})|>t
\bigr\}
&\le
2\frac{\rank L_{\bn}^{(r)}}{N_{\bn}}\\
&\le
C(1+t)^{-qd}+C\frac{a_{\bn}^d}{N_{\bn}}\\
&\le
C_p(1+t)^{-p}+C\frac{a_{\bn}^d}{N_{\bn}},
\qquad t\ge t_p.
\end{aligned}
\]
After increasing \(C_p\), the finitely many initial indices can be
absorbed into the preceding bound.  This proves the polynomial tail estimate.

It remains to prove square uniform integrability.  Define
\[
H_{\bn}(t):=\frac{1}{N_{\bn}}
\#\bigl\{j:|\lambda_j(\mathsf K_{\delta_{\bn},\bn})|>t\bigr\}.
\]
Lemma~\ref{lem:local-fixed-whitening-path-control}, with
\(\rho=|\delta_{\bn}|a_{\bn}\), gives
\[
\|\mathsf K_{\delta_{\bn},\bn}\|_2
\le Ca_{\bn}e^{C|\delta_{\bn}|a_{\bn}}=:T_{\bn}.
\]
Hence
\(H_{\bn}(t)=0\) for \(t>T_{\bn}\). For \(t_p\le M<T_{\bn}\), the
layer-cake representation gives
\[
\begin{aligned}
&\frac{1}{N_{\bn}}
\tr\!\left\{
\mathsf K_{\delta_{\bn},\bn}^{\,2}
\mathbb I_{\{
|\mathsf K_{\delta_{\bn},\bn}|>M
\}}
\right\}\\
&\qquad
=
M^2H_{\bn}(M)
+2\int_M^{T_{\bn}}tH_{\bn}(t)\,dt\\
&\qquad
\le
C_p(1+M)^{2-p}
+C \frac{a_{\bn}^d}{N_{\bn}}T_{\bn}^2.
\end{aligned}
\]
Because \(\log N_{\bn}=(d/2)a_{\bn}+O(1)\) and
\(\delta_{\bn}\to0\), the second term in the preceding bound has the normalized rate
\[
 a_{\bn}^{d+2}N_{\bn}^{-1}
 \exp(2C|\delta_{\bn}|a_{\bn})
 =N_{\bn}^{-1+o(1)}\operatorname{polylog}(N_{\bn})\longrightarrow0.
\]
Consequently,
\[
\limsup_{\bn}
\frac{1}{N_{\bn}}
\sum_{|\lambda_j(\mathsf K_{\delta_{\bn},\bn})|>M}
\lambda_j(\mathsf K_{\delta_{\bn},\bn})^2
\le
C_p(1+M)^{2-p}.
\]
Because \(p>2\), the right-hand side tends to zero as \(M\to\infty\).
This proves square uniform integrability.  Also,
\[
\frac{\|\mathsf K_{\delta_{\bn},\bn}\|_2^2}{N_{\bn}}
\le
\frac{M^2}{N_{\bn}}
+
\frac{1}{N_{\bn}}
\tr\!\left\{
\mathsf K_{\delta_{\bn},\bn}^{\,2}
\mathbb I_{\{
|\mathsf K_{\delta_{\bn},\bn}|>M
\}}
\right\}.
\]
Thus \(\|\mathsf K_{\delta_{\bn},\bn}\|_2/
\sqrt{N_{\bn}}\longrightarrow0\).

{Finally, \(r_{\alpha_0}(\bth)\sim-\log\|\bth\|^2\) belongs to
\(L^2(\nu_Q)\).  Hermitian GLT spectral distribution gives convergence
against continuous compactly supported test functions
\cite[Theorem~5.2, p.~94]{GaroniSerraCapizzanoII}.  Applying this
to bounded continuous cutoffs of \(x^2\), then using the square uniform
integrability just proved to remove the cutoffs, gives the second trace limit.
Cauchy--Schwarz gives uniform integrability of the first moments; applying
the same cutoff argument to \(x\) gives the first trace limit.}
\end{proof}

\begingroup
\section{Likelihood under the fitted model and transfer to the
external data-generating model}
\label{app:working-likelihood-proofs}

Write
\[
 P_{\varphi,\bn}:=\mathcal N(0,\Sigma_{\varphi,\bn}),
 \qquad P^\star_{\bn}:=\mathcal N(0,\Sigma^\star_{\bn}).
\]
We call \(P_{\varphi,\bn}\) the fitted model and \(P^\star_{\bn}\) the
external data-generating model. Exact specification,
\(S^\star=S_{\varphi_0}\), implies
\(P^\star_{\bn}=P_{\varphi_0,\bn}\). The fitted-model arguments first use
exact specification and then transfer their conclusions under the
misspecification described below.
Define its whitening matrix by
\[
 \mathcal W_{\bn}:=(\Sigma_{\varphi_0,\bn})^{-1/2}.
\]
Let \(Y_{\bn}\) denote the sampled observed vector governed by the preceding Gaussian models.
Put \(Z_{\bn}:=\mathcal W_{\bn}Y_{\bn}\). Under exact specification,
\(Z_{\bn}\sim\mathcal N(0,I_{N_{\bn}})\).  Along the compensated
\(\alpha\)-direction, the relative covariance is
\[
 B_{\bn}(\alpha)
 =h^{-2(\alpha-\alpha_0)}
 \mathcal W_{\bn}
 \Sigma_{(\tau_0,\alpha),\bn}
 \mathcal W_{\bn}.
\]
Appendix~\ref{app:working-likelihood-proofs} is the detailed home for this whitened relative covariance, the restricted-score calculation, and the subsequent contiguity and transfer arguments.  Differentiating \(B_{\bn}(\alpha)\) at \(\alpha_0\) gives the centered quadratic score form with matrix \(B'_{\bn}(\alpha_0)\).

{We now allow the working model to be misspecified and retain the same
spectral-family notation.}
Under \(P^\star_{\bn}\), define
$
 C_{\bn}:=\mathcal W_{\bn}\Sigma^\star_{\bn}\mathcal W_{\bn}.
$
Then \(Z_{\bn}\sim\mathcal N(0,C_{\bn})\).

The comparison between the external and reference spectra is governed by
Assumption~\ref{ass:common-leading-tail-misspecification}; its difference
and two-sided comparison clauses are used below.

The score identity is unchanged as a function of the observations:
\[
 \widetilde\ell'_{\bn}(\alpha_0)
 =\frac{1}{2}\left\{
 Z_{\bn}^T(B_{\bn})'(\alpha_0)Z_{\bn}
 -\operatorname{tr}(B_{\bn})'(\alpha_0)
 \right\}.
\]
Under \(P^\star_{\bn}\), its mean and variance are
\begin{align*}
 \mathbb E_{P^\star_{\bn}}\{\widetilde\ell'_{\bn}(\alpha_0)\}
 &=\frac{1}{2}\operatorname{tr}\{(B_{\bn})'(\alpha_0)(C_{\bn}-I)\},\\
 \operatorname{Var}_{P^\star_{\bn}}\{\widetilde\ell'_{\bn}(\alpha_0)\}
 &=\frac{1}{2}\operatorname{tr}\{(B_{\bn})'(\alpha_0)C_{\bn}
 (B_{\bn})'(\alpha_0)C_{\bn}\}.
\end{align*}
Under exact specification, \(C_{\bn}=I_{N_{\bn}}\), so the score is
centered and has the fitted-model variance. Under misspecification,
\(C_{\bn}-I\) produces a finite-sample bias and changes the variance.
The transfer argument proved below shows that, under the common-leading-tail
misspecification condition, as
\(N_{\bn}\to\infty\),
\[
 \frac{\mathbb E_{P^\star_{\bn}}\{\widetilde\ell'_{\bn}(\alpha_0)\}}{\sqrt{N_{\bn}}}\longrightarrow0,
 \qquad
 \frac{1}{N_{\bn}}\left[
 \operatorname{Var}_{P^\star_{\bn}}\{\widetilde\ell'_{\bn}(\alpha_0)\}
 -\frac{1}{2}\operatorname{tr}\{[(B_{\bn})'(\alpha_0)]^2\}
 \right]\longrightarrow0.
\]
Thus the normalized score has asymptotically zero mean and the same limiting
variance under \(P^\star_{\bn}\) as under \(P_{\varphi_0,\bn}\).

The transfer has two steps.  First, we construct finite-rank approximations
to \(C_{\bn}-I\).  Second, we use them to establish mutual contiguity and to
transfer the normalized centering and covariance of the quadratic forms.  The
score proof then applies this transfer to the score matrices, including
\((B_{\bn})'(\alpha_0)\), and thereby transfers the score limit from
\(P_{\varphi_0,\bn}\) to \(P^\star_{\bn}\).

The separation argument compares two parameter values within the fitted
model and is separate from the transfer to the external model.
Equation~\eqref{eq:supp-gaussian-affinity} expresses their Hellinger affinity
through the generalized covariance eigenvalues, and
Lemma~\ref{lem:working-cayley-hellinger} supplies the lower bound used to prove
Proposition~\ref{prop:working-glt-hellinger-separation} in
Appendix~\ref{app:likelihood-separation-and-lan}.

To approximate \(C_{\bn}-I\) by finite-rank matrices, we replace
low-frequency Fourier modes by polynomial modes.  The resulting matrix bound
requires bounds for the residuals after they are evaluated on the grid and
whitened.  The following native-space argument obtains those bounds from
Sobolev norms of the continuous residual functions.

\emph{Native space and whitened samples.}
At \(\varphi_0\), Assumption~\ref{ass:differentiated-spectral-tail-family}
gives a real, even, continuous, strictly positive spectral density
\(S_{\varphi_0}\) in \(L^1(\R^d)\).
Wendland's Corollary~6.9 \citep[p.~73]{wendland2005scattered}
therefore shows that the kernel \(K_{\varphi_0}(x,y)=\rho_{\varphi_0}(x-y)\) is
strictly positive definite.  Wendland's Definition~10.9 and
Theorems~10.10--10.11 \citep[pp.~138--139]{wendland2005scattered} give the
unique native RKHS with reproducing kernel \(K_{\varphi_0}\), which we denote by
\(\mathcal H_{\varphi_0}\).  Thus, for \(g\in\mathcal H_{\varphi_0}\),
\(g(y)=\langle g,K_{\varphi_0}(\cdot,y)\rangle_{\mathcal H_{\varphi_0}}\).  The tail limit and
strict positivity on compact frequency sets give
\[
 S_{\varphi_0}(\bomega)
 \ge c(1+\|\bomega\|^2)^{-\alpha_0}.
\]

Set \(\mu_{\varphi_0}=(2\pi)^{-d}S_{\varphi_0}(\bomega)\,d\bomega\) and define
\(\Phi_x(\bomega)=(\cos(x\cdot\bomega),\sin(x\cdot\bomega))\) in
\(L^2(\mu_{\varphi_0};\R^2)\).  The trigonometric addition formula and evenness of
\(S_{\varphi_0}\) give
\[
 \langle\Phi_x,\Phi_y\rangle_{L^2(\mu_{\varphi_0};\R^2)}
 =\int_{\R^d}\cos\{(x-y)\cdot\bomega\}
 \mu_{\varphi_0}(d\bomega)
 =\rho_{\varphi_0}(x-y).
\]

Thus \(\Phi_x\) is a feature map for \(K_{\varphi_0}\).  For
\((a_c,a_s)\in L^2(\mu_{\varphi_0};\R^2)\), define the synthesis map
\[
 T(a_c,a_s)(x)
 :=\langle (a_c,a_s),\Phi_x\rangle_{L^2(\mu_{\varphi_0};\R^2)}
 =\int_{\R^d}\{a_c(\bomega)\cos(x\cdot\bomega)
 +a_s(\bomega)\sin(x\cdot\bomega)\}\,\mu_{\varphi_0}(d\bomega).
\]
The native-space construction identifies the range of \(T\) with
\(\mathcal H_{\varphi_0}\) and gives the norm
\[
 \|g\|_{\mathcal H_{\varphi_0}}^2=
 \inf\left\{\int(a_c^2+a_s^2)\,d\mu_{\varphi_0}:T(a_c,a_s)=g\right\}.
\]

We write \(C_c^\infty(\R^d)\) for the space of infinitely differentiable
functions on \(\R^d\) with compact support.  Following
\citep[p.~170]{wendland2005scattered}, for an integer \(q\ge0\), let
\(H^q=H^q(\R^d)\) denote the Sobolev space of functions whose weak derivatives
through order \(q\) belong to \(L^2(\R^d)\), equipped with the norm
\[
 \|g\|_{H^q}^2
 :=\sum_{|\beta|\le q}\|\partial^\beta g\|_{L^2(\R^d)}^2.
\]
Here \(\beta=(\beta_1,\ldots,\beta_d)\in\mathbb N_0^d\) is a multi-index,
\(|\beta|=\sum_{j=1}^d\beta_j\), and
\(\partial^\beta=\partial_{x_1}^{\beta_1}\cdots
\partial_{x_d}^{\beta_d}\).

For real \(g\in C_c^\infty(\R^d)\), write
\(\widehat g=u+\mathrm iv\).  Then \(u\) is even and \(v\) is odd.  The
coefficients
\[
 a_c(\bomega)=\frac{(2\pi)^{d/2}u(\bomega)}{S_{\varphi_0}(\bomega)},
 \qquad
 a_s(\bomega)=-\frac{(2\pi)^{d/2}v(\bomega)}{S_{\varphi_0}(\bomega)}
\]
satisfy \(T(a_c,a_s)=g\) by unitary Fourier inversion, and
\[
 \int(a_c^2+a_s^2)\,d\mu_{\varphi_0}
 =\int_{\R^d}\frac{|\widehat g(\bomega)|^2}{S_{\varphi_0}(\bomega)}\,d\bomega.
\]

Thus, for every integer \(q>\alpha_0\), the lower spectral envelope gives
\[
 \|g\|_{\mathcal H_{\varphi_0}}^2\le C\|g\|_{H^q}^2.
\]

Because \(K_{\varphi_0}\) is strictly positive definite and the grid points are
distinct, the covariance matrix \(\Sigma_{\varphi_0,\bn}\) is
positive definite.  For a function \(g\in\mathcal H_{\varphi_0}\), collect its values
at the grid points in
\[
 g_X:=\bigl(g(\bx_{\bi}^{(\bn)})\bigr)_{\bi\in\mathcal I_{\bn}}
 \in\R^{N_{\bn}}.
\]
Wendland's Theorem~13.1 \citep[pp.~223--224]{wendland2005scattered}
identifies kernel interpolation at the grid points as an orthogonal projection
in \(\mathcal H_{\varphi_0}\).  In the present notation, the fact that this projection
cannot increase the \(\mathcal H_{\varphi_0}\)-norm gives
\[
\|\mathcal W_{\bn}g_X\|_2^2
=g_X^T(\Sigma_{\varphi_0,\bn})^{-1}g_X
\le\|g\|_{\mathcal H_{\varphi_0}}^2.
\]
Combining this inequality with the Sobolev-to-native bound and taking square
roots gives the continuous-to-sampled bound used for the cutoff residuals:
\[
 \|\mathcal W_{\bn}g_X\|_2\le C\|g\|_{H^q},
 \qquad g\in C_c^\infty(\R^d),\quad q>\alpha_0.
\]
This inequality controls the values of \(g\) on the grid in \(\calD\), after
whitening, by its Sobolev norm over \(\R^d\).  We apply it to the errors made
by replacing low-frequency cosine and sine modes on \(\calD\) with Taylor
polynomials.

Fix an integer \(q>\alpha_0\), and choose \(\chi\in C_c^\infty(\R^d)\)
equal to one on a neighborhood of \(\calD=[0,L]^d\).  Define the truncated
Taylor polynomials for cosine and sine by
\[
 c_M(t)=\sum_{\substack{0\le j\le M\\j\ \mathrm{even}}}\frac{(-1)^{j/2}t^j}{j!},
 \qquad s_M(t)=\sum_{\substack{0\le j\le M\\j\ \mathrm{odd}}}\frac{(-1)^{(j-1)/2}t^j}{j!}.
\]
For \(\bn\), collect their values on the grid in
\[
 p^c_{\bomega,M}
 =\bigl(c_M(\bx_{\bi}^{(\bn)}\cdot\bomega)\bigr)_{\bi\in\mathcal I_{\bn}},
 \qquad
 p^s_{\bomega,M}
 =\bigl(s_M(\bx_{\bi}^{(\bn)}\cdot\bomega)\bigr)_{\bi\in\mathcal I_{\bn}}.
\]
For \(\|\bomega\|\le R\), define the corresponding cutoff residual functions by
\[
 g^c_{\bomega,M}(x)=\chi(x)\{\cos(x\cdot\bomega)-c_M(x\cdot\bomega)\},
 \qquad
 g^s_{\bomega,M}(x)=\chi(x)\{\sin(x\cdot\bomega)-s_M(x\cdot\bomega)\}.
\]
Because \(\chi=1\) on \(\calD\), \((g^c_{\bomega,M})_X\) is the difference
between the grid cosine values and \(p^c_{\bomega,M}\); the analogous identity
holds for sine.  The continuous-to-grid argument therefore gives the chain
\[
\begin{aligned}
 &\left\|\mathcal W_{\bn}
 \left\{
 \bigl(\cos(\bx_{\bi}^{(\bn)}\cdot\bomega)\bigr)_{\bi\in\mathcal I_{\bn}}
 -p^c_{\bomega,M}
 \right\}\right\|_2 \\
 &\qquad=\|\mathcal W_{\bn}(g^c_{\bomega,M})_X\|_2
 \le\|g^c_{\bomega,M}\|_{\mathcal H_{\varphi_0}}
 \le C\|g^c_{\bomega,M}\|_{H^q},
\end{aligned}
\]
with the same chain for the sine error.  Thus \(\chi\) preserves the error on
the grid while making each residual function compactly supported on \(\R^d\).
It remains to bound their Sobolev norms.  Both \(p^c_{\bomega,M}\) and
\(p^s_{\bomega,M}\) lie in the same degree-\(M\) polynomial evaluation
space, which makes the later low-frequency approximation finite rank.
Figure~\ref{fig:cutoff-taylor-localization} illustrates the construction for
\(\cos x\): \(c_2(x)\) supplies the polynomial part, and \(\chi\) makes the
residual compactly supported.
\begin{figure}[t]
\centering
\makebox[\textwidth][c]{%
 \includegraphics[
  width=\textwidth,
  trim=0 30bp 0 0,
  clip
]{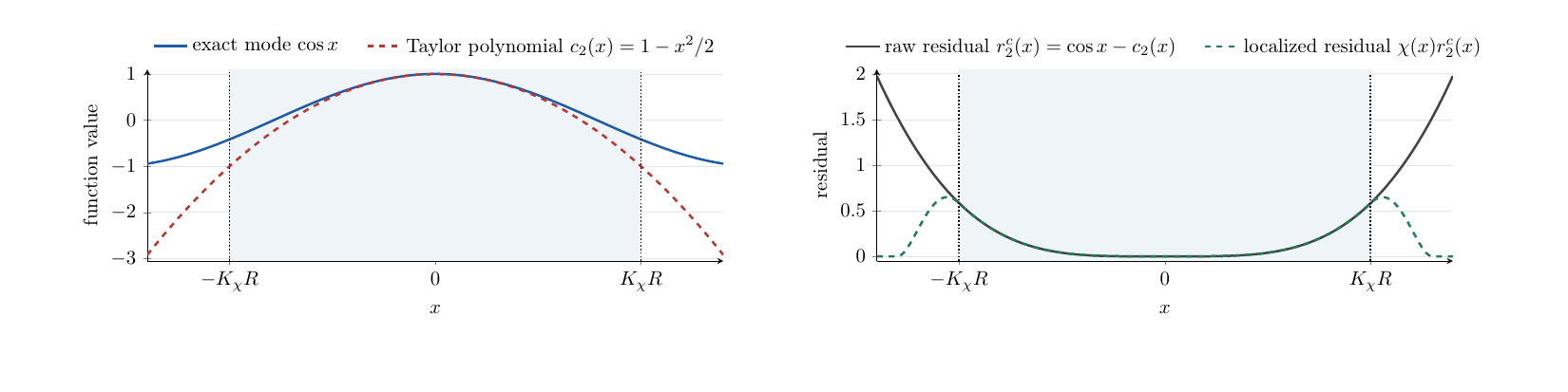} %
}
\vspace{-1.6em}
\caption{Taylor approximation and cutoff localization for \(\cos x\) with
\(c_2(x)=1-x^2/2\).  Left: the mode and approximation.  Right: the raw and
localized residuals.}
\label{fig:cutoff-taylor-localization}
\end{figure}
\begin{lemma}
\label{lem:cutoff-taylor-stirling}
There are constants \(C,c,C_0,C_1>0\), depending only on \(d,q,\chi\),
such that, for every \(R\ge1\), every Taylor truncation order
 \(M\ge C_1R+C_0\), and every \(\bomega\) with
\(\|\bomega\|\le R\),
\[
 \|g^c_{\bomega,M}\|_{H^q}+\|g^s_{\bomega,M}\|_{H^q}
 \le Ce^{-c(M-C_0R)}.
\]
\end{lemma}

\begin{proof}
Write \(r^c_M(t)=\cos(t)-c_M(t)\) and
\(r^s_M(t)=\sin(t)-s_M(t)\).  For \(\ell\in\{c,s\}\) and every
multi-index \(\beta\), the decisive Leibniz relation is
\[
 \partial^\beta\{\chi(x)r^\ell_M(x\cdot\bomega)\}
 =\sum_{\gamma\le\beta}\binom{\beta}{\gamma}
 \partial^{\beta-\gamma}\chi(x)\,\bomega^\gamma
 \bigl(r^\ell_M\bigr)^{(|\gamma|)}(x\cdot\bomega).
\]
Choose \(K_\chi\) so that \(|x\cdot\bomega|\le K_\chi R\) on
\(\operatorname{supp}\chi\) for \(\|\bomega\|\le R\).  Enlarging \(C_0\)
so that \(M\ge q\), Taylor's theorem gives, for \(\ell\in\{c,s\}\) and
\(0\le k\le q\),
\[
 \sup_{|t|\le K_\chi R}|\bigl(r^\ell_M\bigr)^{(k)}(t)|
 \le\frac{(K_\chi R)^{M+1-k}}{(M+1-k)!}.
\]
Consequently, uniformly for \(R\ge1\) and \(\|\bomega\|\le R\),
\[
 \max_{|\beta|\le q}\left\{\|\partial^\beta g^c_{\bomega,M}\|_{L^2}
 +\|\partial^\beta g^s_{\bomega,M}\|_{L^2}\right\}
 \le C(1+R)^q\max_{0\le k\le q}\frac{(K_\chi R)^{M+1-k}}{(M+1-k)!}.
\]
For \(m=M+1-k\), Stirling's lower bound \(m!\ge(m/e)^m\) yields
\[
 \frac{(K_\chi R)^m}{m!}\le\left(\frac{eK_\chi R}{m}\right)^m.
\]
After increasing \(C_1\) and \(C_0\), the condition
\(M\ge C_1R+C_0\) makes the base on the right-hand side uniformly smaller
than one for \(0\le k\le q\).  The same linear margin absorbs the factor
\((1+R)^q\), so the preceding aggregation bound is at most
\(Ce^{-c(M-C_0R)}\).  Summing the derivative bounds defining the
\(H^q\) norms proves the claim.
\end{proof}

For \(\bomega\in\R^d\), collect the feature map values on the grid by
$
 \begin{pmatrix}f^c_{\bomega}&f^s_{\bomega}\end{pmatrix}
 :=\bigl(\Phi_{\bx_{\bi}^{(\bn)}}(\bomega)^T\bigr)_{\bi\in\mathcal I_{\bn}}
 \in\R^{N_{\bn}\times2}.
$
Then \(v_{\bn}(\bomega)=f^c_{\bomega}-\mathrm i f^s_{\bomega}\).  For a
real integrable weight \(w\), recall the sampled Fourier outer-product operator
\[
 \mathcal P_{\bn}^{(\infty)}(w)
 :=(2\pi)^{-d}\int_{\R^d}w(\bomega)
 v_{\bn}(\bomega)v_{\bn}(\bomega)^*\,d\bomega,
\]
where \(v_{\bn}\) is the sampled mode defined in
Equation~\eqref{eq:supp-localized-fourier-mode}.  This is the
whole-space version of the local Fourier matrix in Appendix~\ref{app:preliminary-main-paper-proofs}.
The preceding continuous-to-grid construction uses only strict positivity and
a lower spectral envelope for its reference spectrum.  Here the reference
spectrum is the fitted spectrum \(S_{\varphi_0}\), so the construction
with the already defined \(\mathcal W_{\bn}\) gives
\[
 \|\mathcal W_{\bn}g_X\|_2\le C\|g\|_{H^q},
 \qquad g\in C_c^\infty(\R^d),
\]
uniformly in \(\bn\).  Since \(\chi=1\) on \(\calD\), the cutoff
residual identities give, for \(\ell\in\{c,s\}\),
\[
 \|\mathcal W_{\bn}\{f^\ell_{\bomega}-p^\ell_{\bomega,M}\}\|_2
 =\|\mathcal W_{\bn}(g^\ell_{\bomega,M})_X\|_2
 \le C\|g^\ell_{\bomega,M}\|_{H^q}.
\]
Lemma~\ref{lem:cutoff-taylor-stirling} therefore turns the Taylor residual
estimate into the whitened vector bound used below.

\begingroup
\begin{lemma}
\label{lem:whitened-localized-fourier-compression}
Under Assumption~\ref{ass:common-leading-tail-misspecification}, there are constants
\(C,c,C_0,C_1>0\), independent of \(\bn\), \(R\), and \(M\), such that,
for every sampling index \(\bn\), \(R\ge1\), and Taylor order
\(M\ge C_1R+C_0\), there is a self-adjoint matrix \(B_{\bn,R,M}\)
satisfying
\[
 \rank B_{\bn,R,M}\le CM^d,\qquad
 \|C_{\bn}-I-B_{\bn,R,M}\|_2
 \le C\{(1+R^2)^{-\eta}+e^{-c(M-C_0R)}\}.
\]
\end{lemma}

\begin{proof}
Put \(\Delta S:=S^\star-S_{\varphi_0}\). The common-leading-tail condition
gives
\[
 |\Delta S(\bomega)|\le
 C(1+\|\bomega\|^2)^{-\alpha_0-\eta}.
\]
We prove the asserted approximation of \(C_{\bn}-I\).  The exact identity
and its low/high-frequency split are
\[
 C_{\bn}-I
 =\mathcal W_{\bn}\mathcal P_{\bn}^{(\infty)}
   (\Delta S\mathbb I_{\{\|\bomega\|<R\}})\mathcal W_{\bn}
 +\mathcal W_{\bn}\mathcal P_{\bn}^{(\infty)}
   (\Delta S\mathbb I_{\{\|\bomega\|>R\}})\mathcal W_{\bn}.
\]
For the low-frequency term, evenness of \(\Delta S\) cancels the mixed
cosine--sine terms in the sampled complex Fourier outer product, giving
\[
 \mathcal P_{\bn}^{(\infty)}
 (\Delta S\mathbb I_{\{\|\bomega\|<R\}})
 =(2\pi)^{-d}\int_{\|\bomega\|<R}\Delta S(\bomega)
 \{f^c_{\bomega}(f^c_{\bomega})^T+
   f^s_{\bomega}(f^s_{\bomega})^T\}\,d\bomega.
\]
Put \(a^\ell_{\bomega}=\mathcal W_{\bn}f^\ell_{\bomega}\) and
\(b^\ell_{\bomega,M}=\mathcal W_{\bn}p^\ell_{\bomega,M}\) for
\(\ell\in\{c,s\}\), and define
\begin{align*}
 B_{\bn,R,M}:={}&(2\pi)^{-d}\int_{\|\bomega\|<R}\Delta S(\bomega)
 \{b^c_{\bomega,M}(b^c_{\bomega,M})^T+
   b^s_{\bomega,M}(b^s_{\bomega,M})^T\}\,d\bomega.
\end{align*}
This matrix is self-adjoint, and its range is contained in the image under
\(\mathcal W_{\bn}\) of the degree-\(M\) polynomial values on the grid.
This space has dimension at most \(\binom{M+d}{d}\), so
\(\rank B_{\bn,R,M}=O(M^d)\).
For \(\|\bomega\|\le R\), apply the estimate established before the lemma,
\[
 \|\mathcal W_{\bn}g_X\|_2\le C\|g\|_{H^q},
 \qquad g\in C_c^\infty(\R^d),
\]
to \(g=g^\ell_{\bomega,M}\).  The cutoff identity
\((g^\ell_{\bomega,M})_X=f^\ell_{\bomega}-p^\ell_{\bomega,M}\) and
Lemma~\ref{lem:cutoff-taylor-stirling} then give, uniformly in \(\bn\),
\[
 \|a^\ell_{\bomega}-b^\ell_{\bomega,M}\|_2
 \le C\|g^\ell_{\bomega,M}\|_{H^q}
 \le Ce^{-c(M-C_0R)},\qquad \ell\in\{c,s\}.
\]
Applying the same estimate to \(x\mapsto\chi(x)\cos(x\cdot\bomega)\) and
\(x\mapsto\chi(x)\sin(x\cdot\bomega)\) gives
\(\|a^\ell_{\bomega}\|_2\le C(1+R)^q\).  The preceding residual bound
then also gives \(\|b^\ell_{\bomega,M}\|_2\le C(1+R)^q\).  Therefore,
\[
 \begin{aligned}
 &\|a^\ell_{\bomega}(a^\ell_{\bomega})^T
   -b^\ell_{\bomega,M}(b^\ell_{\bomega,M})^T\|_2\\
 &\quad\le \|a^\ell_{\bomega}-b^\ell_{\bomega,M}\|_2
      \{\|a^\ell_{\bomega}\|_2+\|b^\ell_{\bomega,M}\|_2\}\\
 &\quad\le C(1+R)^q e^{-c(M-C_0R)}.
 \end{aligned}
\]
Integrating this outer-product bound against \(|\Delta S(\bomega)|\)
(which is integrable) gives
\[
 \left\|\mathcal W_{\bn}\mathcal P_{\bn}^{(\infty)}
 (\Delta S\mathbb I_{\{\|\bomega\|<R\}})\mathcal W_{\bn}
 -B_{\bn,R,M}\right\|_2
 \le C(1+R)^q e^{-c(M-C_0R)}.
\]
After decreasing \(c\) and increasing \(C_0\), the linear margin
\(M\ge C_1R+C_0\) absorbs the factor \(C(1+R)^q\).  Thus
the low-frequency contribution is at most \(Ce^{-c(M-C_0R)}\).

For the high-frequency term, the tail representation and strict positivity of
\(S_{\varphi_0}\) give a lower envelope
\(S_{\varphi_0}(\bomega)\ge c(1+\|\bomega\|^2)^{-\alpha_0}\).  Hence
\[
 \delta_R:=\sup_{\|\bomega\|>R}
 \frac{|\Delta S(\bomega)|}{S_{\varphi_0}(\bomega)}
 \le C(1+R^2)^{-\eta}.
\]
Thus \(|\Delta S(\bomega)|\le\delta_R S_{\varphi_0}(\bomega)\) for
\(\|\bomega\|>R\).  Quadratic-form domination consequently gives
\[
 -\delta_R\Sigma_{\varphi_0,\bn}\preceq
 \mathcal P_{\bn}^{(\infty)}
 (\Delta S\mathbb I_{\{\|\bomega\|>R\}})
 \preceq\delta_R\Sigma_{\varphi_0,\bn}.
\]
Whitening this Loewner bound gives the same inequality with
\(\Sigma_{\varphi_0,\bn}\) replaced by \(I\), so the high-frequency contribution
has operator norm at most \(C(1+R^2)^{-\eta}\). Combining the
two terms in the exact identity proves the stated bound.

\end{proof}
\endgroup

If \(S^\star=S_{\varphi_0}\), then
\(P^\star_{\bn}=P_{\varphi_0,\bn}\) and \(C_{\bn}=I_{N_{\bn}}\), so the
event and quadratic-form transfers below are immediate.  The following lemma
treats the misspecified case.

\begingroup
\begin{lemma}
\label{lem:relative-covariance-control}
Suppose that the external spectrum \(S^\star\) and the reference spectrum
\(S_{\varphi_0}\) satisfy
Assumption~\ref{ass:common-leading-tail-misspecification}. Then, for
constants \(0<c<C<\infty\), there is a \(K<\infty\), independent of
\(\bn\) and \(j\), such that
\[
 cI_{N_{\bn}}\preceq C_{\bn}\preceq CI_{N_{\bn}},
 \qquad s_j(C_{\bn}-I_{N_{\bn}})\le Kj^{-\beta},
 \quad \beta:=\frac{2\eta}{d}>\frac{1}{2},
\]
where singular values are ordered nonincreasingly and padded by zero after
dimension \(N_{\bn}\). Consequently,
\(\sup_{\bn}\|C_{\bn}-I_{N_{\bn}}\|_{\mathrm{HS}}<\infty\) and
\(\sup_{\bn}\sum_{j>m}s_j(C_{\bn}-I_{N_{\bn}})^2\le
Km^{1-2\beta}\).
\end{lemma}

\begin{proof}
With \(\Delta S=S^\star-S_{\varphi_0}\), the difference-bound clause of
Assumption~\ref{ass:common-leading-tail-misspecification}
and the lower spectral envelope for
\(S_{\varphi_0}\) give
\(\lvert\Delta S(\bomega)\rvert/S_{\varphi_0}(\bomega)
\le C(1+\|\bomega\|^2)^{-\eta}\).
For \(u\in\R^{N_{\bn}}\), the sampled Fourier covariance representation
therefore gives
\[
 u^T\Sigma^\star_{\bn}u
 =(2\pi)^{-d}\int_{\R^d}S^\star(\bomega)
 \left|\sum_{\bi\in\mathcal I_{\bn}}u_{\bi}
 e^{\mathrm i\bx_{\bi}^{(\bn)}\cdot\bomega}\right|^2d\bomega,
\]
and the two-sided spectral-comparison clause of
Assumption~\ref{ass:common-leading-tail-misspecification} yields
\(c\Sigma_{\varphi_0,\bn}\preceq\Sigma^\star_{\bn}
\preceq C\Sigma_{\varphi_0,\bn}\).  Congruence by
\(\mathcal W_{\bn}\) consequently gives
\(cI_{N_{\bn}}\preceq C_{\bn}\preceq CI_{N_{\bn}}\).
Choose \(M\asymp R\) in
Lemma~\ref{lem:whitened-localized-fourier-compression}.  After changing a
constant \(K\), it gives, for every \(R\ge1\), a matrix \(D_{\bn,R}\)
such that
\[
 \rank D_{\bn,R}\le KR^d,
 \qquad
 \|C_{\bn}-I_{N_{\bn}}-D_{\bn,R}\|_2
 \le KR^{-2\eta}.
\]
For every sufficiently large integer \(m\), choose \(R\asymp m^{1/d}\)
with the constants fixed so that \(\rank D_{\bn,R}\le m\).  If
\(m<N_{\bn}\), the singular-value perturbation/min--max inequality in
\cite[Theorem~3.3.16(c)]{HornJohnson1991}, applied to
\(C_{\bn}-I_{N_{\bn}}=D_{\bn,R}+
\{C_{\bn}-I_{N_{\bn}}-D_{\bn,R}\}\), gives
\[
 s_{m+1}(C_{\bn}-I_{N_{\bn}})
 \le \|C_{\bn}-I_{N_{\bn}}-D_{\bn,R}\|_2
 \le Km^{-2\eta/d}.
\]
If \(m\ge N_{\bn}\), take \(D=C_{\bn}-I_{N_{\bn}}\), so the
singular-value and Hilbert--Schmidt tails beyond \(m\) are empty.  Together
with the uniform operator bound from the covariance comparison, the preceding
inequality gives, with \(s_\ell=0\) for \(\ell>N_{\bn}\), for all \(\ell\ge1\),
\(s_\ell(C_{\bn}-I_{N_{\bn}})\le
K\ell^{-2\eta/d}\), after increasing \(K\). Consequently,
\[
 \sup_{\bn}\sum_{\ell>m}s_\ell(C_{\bn}-I_{N_{\bn}})^2
 \le K\sum_{\ell>m}\ell^{-4\eta/d}
 \le Km^{1-4\eta/d}.
\]
This tends to zero because \(4\eta>d\), proving the stated
Hilbert--Schmidt consequences.
\end{proof}
\endgroup

Mutual contiguity means that, for every sequence of events \(E_{\bn}\),
\(P_{\varphi_0,\bn}(E_{\bn})\to0\) if and only if
\(P^\star_{\bn}(E_{\bn})\to0\).  Thus it transfers probability statements,
whereas the next lemma separately transfers normalized quadratic forms.

\begin{lemma}
\label{lem:common-leading-tail-mutual-contiguity}
{Under Assumption~\ref{ass:common-leading-tail-misspecification}, \(P^\star_{\bn}\) and}
\(P_{\varphi_0,\bn}\) are mutually contiguous.
\end{lemma}

\begin{proof}
{Under Assumption~\ref{ass:common-leading-tail-misspecification}, mutual contiguity follows}
from the standard
\(L^p\) likelihood-ratio argument, which we give in both directions; see
\cite[Chapter~3]{LeCamYang2000} for general contiguity and likelihood-ratio
criteria.

Lemma~\ref{lem:relative-covariance-control} gives
\(cI\preceq C_{\bn}\preceq CI\) and a uniform Hilbert--Schmidt bound.
Define the likelihood ratio of the data-generating law with respect to
the fitted law by
$
        L_{\bn}
        :
        =
        \frac{dP^\star_{\bn}}{dP_{\varphi_0,\bn}}.
$
Following the Gaussian whitening and congruence-diagonalization argument
of \cite[Equations~(22)--(24), p.~613]{Hajek1958}, work in coordinates
in which \(P_{\varphi_0,\bn}\) is standard Gaussian and \(C_{\bn}\) is diagonal
with eigenvalues \(c_{\ell,\bn}\).
In these coordinates,
\[
 L_{\bn}
 =
 \prod_\ell c_{\ell,\bn}^{-1/2}
 \exp\left\{
 -\frac{1}{2}(c_{\ell,\bn}^{-1}-1)Z_{\ell,\bn}^2
 \right\},
 \qquad
 Z_{\ell,\bn}\stackrel{\mathrm{iid}}{\sim}\mathcal N(0,1).
\]
Choose \(p>1\) sufficiently close to \(1\) that
$
        1-p(1-x^{-1})>0,
         x\in[c,C].
$
Therefore, using
\(\mathbb E\exp(-a\xi^2/2)=(1+a)^{-1/2}\), we obtain
\[
\begin{aligned}
        \mathbb E_{P_{\varphi_0,\bn}}L_{\bn}^p
        &=
        \prod_{\ell}
        c_{\ell,\bn}^{-p/2}
        \{1-p(1-c_{\ell,\bn}^{-1})\}^{-1/2}  =
        \prod_{\ell}\exp\{g_p(c_{\ell,\bn})\},
\end{aligned}
\]
where
$
        g_p(x)
        :=
        -\frac{p}{2}\log x
        -\frac{1}{2}\log\{1-p(1-x^{-1})\}.
$
Since \(g_p(1)=g_p'(1)=0\) and \(g_p''\) is bounded on
\([c,C]\), Taylor's theorem gives
\[
        |g_p(x)|
        \le K_p(x-1)^2,
        \qquad x\in[c,C],
\]
for some finite constant \(K_p\). Consequently,
\[
\begin{aligned}
        \log\mathbb E_{P_{\varphi_0,\bn}}L_{\bn}^p
        =
        \sum_{\ell}g_p(c_{\ell,\bn}) 
        \le
        K_p\sum_{\ell}(c_{\ell,\bn}-1)^2 
        =
        K_p\|C_{\bn}-I\|_{\mathrm{HS}}^2.
\end{aligned}
\]
The uniform Hilbert--Schmidt bound therefore yields
\[
        \sup_{\bn}\mathbb E_{P_{\varphi_0,\bn}}L_{\bn}^p
        \le
        \exp\left\{
        K_p\sup_{\bn}\|C_{\bn}-I\|_{\mathrm{HS}}^2
        \right\}
        <\infty.
\]

For
\(M_{\bn}=dP_{\varphi_0,\bn}/dP^\star_{\bn}\), the same calculation, with
\(c_{\ell,\bn}\) replaced by \(c_{\ell,\bn}^{-1}\), gives
\(\sup_{\bn}\mathbb E_{P^\star_{\bn}}M_{\bn}^p<\infty\).  Hence, for every
event \(E_{\bn}\), H\"older's inequality gives
\[
 P^\star_{\bn}(E_{\bn})
 \le \{\mathbb E_{P_{\varphi_0,\bn}}L_{\bn}^p\}^{1/p}
 P_{\varphi_0,\bn}(E_{\bn})^{1-1/p},
 \qquad
 P_{\varphi_0,\bn}(E_{\bn})
 \le \{\mathbb E_{P^\star_{\bn}}M_{\bn}^p\}^{1/p}
 P^\star_{\bn}(E_{\bn})^{1-1/p}.
\]
By H\"older's inequality, each sequence of likelihood ratios yields the
corresponding vanishing-event implication.  By
\cite[Definition~1]{LeCamYang2000},
the two implications are precisely mutual contiguity.
\end{proof}

\begin{lemma}
\label{lem:centered-quadratic-form-transfer}
Suppose that \(C_{\bn}=I_{N_{\bn}}\) for every \(\bn\), or that
{Assumption~\ref{ass:common-leading-tail-misspecification} holds.} Let \(q\) be fixed,
and let \(A_{1,\bn},\ldots,A_{q,\bn}\) be a fixed finite collection of
sequences of real symmetric \(N_{\bn}\)-by-\(N_{\bn}\) matrices satisfying
\[
 \max_j\frac{\|A_{j,\bn}\|_2}{\sqrt{N_{\bn}}}\to0,
 \qquad
 \frac{\tr(A_{j,\bn}A_{k,\bn})}{N_{\bn}}\to\Gamma_{jk}
 \quad(1\le j,k\le q).
\]
Then, for \(1\le j,k\le q\),
\[
 \frac{\tr\{A_{j,\bn}(C_{\bn}-I)\}}{\sqrt{N_{\bn}}}\longrightarrow0,
 \qquad
 \frac{\tr(A_{j,\bn}C_{\bn}A_{k,\bn}C_{\bn})}{N_{\bn}}
 \longrightarrow\Gamma_{jk},
\]
and, under \(P^\star_{\bn}\),
\[
 \left\{\frac{Z_{\bn}^TA_{j,\bn}Z_{\bn}-\tr A_{j,\bn}}{\sqrt{N_{\bn}}}\right\}_{j=1}^q
 \Rightarrow\mathcal N(0,2\Gamma).
\]
\end{lemma}

\begin{proof}
When \(C_{\bn}=I_{N_{\bn}}\), the centering and covariance conclusions are
identities and the CLT is the standard-Gaussian case, so the result is
immediate. Otherwise {Assumption~\ref{ass:common-leading-tail-misspecification} holds,} and
Lemma~\ref{lem:relative-covariance-control} applies. \begingroup
Fix \(t\in\R^q\)
and define the one linear combination
\(A_{\bn}(t):=\sum_{j=1}^q t_jA_{j,\bn}\).  The trace limits give
\(\|A_{\bn}(t)\|_{\mathrm{HS}}/\sqrt{N_{\bn}}=O(1)\). The
singular-value bound in that lemma permits the self-adjoint rank-\(m\)
spectral truncation \(B_{m,\bn}\) of \(C_{\bn}-I\), with a
Hilbert--Schmidt remainder tending uniformly to zero. It gives
\[
 \frac{|\tr\{A_{\bn}(t)(C_{\bn}-I)\}|}{\sqrt{N_{\bn}}}
 \le \frac{mC\|A_{\bn}(t)\|_2}{\sqrt{N_{\bn}}}
 +\frac{\|A_{\bn}(t)\|_{\mathrm{HS}}}{\sqrt{N_{\bn}}}
 \|C_{\bn}-I-B_{m,\bn}\|_{\mathrm{HS}}.
\]
Letting \(\bn\to\infty\), then \(m\to\infty\), proves centering.  The
uniform Hilbert--Schmidt bound and
\[
 \begin{aligned}
 &\left|\tr\{A_{\bn}(t)C_{\bn}A_{\bn}(t)C_{\bn}\}
 -\tr\{A_{\bn}(t)^2\}\right|\\
 &\quad\le 2\|A_{\bn}(t)\|_2\|A_{\bn}(t)\|_{\mathrm{HS}}
 \|C_{\bn}-I\|_{\mathrm{HS}}
 +\|A_{\bn}(t)\|_2^2\|C_{\bn}-I\|_{\mathrm{HS}}^2.
 \end{aligned}
\]
show that the left-hand side is \(o(N_{\bn})\).
Thus its normalized limit is \(t^T\Gamma t\); polarization gives the
stated covariance limits.

\endgroup
Under \(P^\star_{\bn}\), write \(Z_{\bn}=C_{\bn}^{1/2}\xi_{\bn}\), where
\(\xi_{\bn}\sim\mathcal N(0,I_{N_{\bn}})\).  The identity
\[
 \begin{aligned}
 \frac{Z_{\bn}^TA_{\bn}(t)Z_{\bn}-\tr A_{\bn}(t)}{\sqrt{N_{\bn}}}
 &=
 \frac{\xi_{\bn}^TC_{\bn}^{1/2}A_{\bn}(t)C_{\bn}^{1/2}\xi_{\bn}
 -\tr\{C_{\bn}^{1/2}A_{\bn}(t)C_{\bn}^{1/2}\}}{\sqrt{N_{\bn}}}\\
 &\quad+\frac{\tr\{A_{\bn}(t)(C_{\bn}-I)\}}{\sqrt{N_{\bn}}}.
 \end{aligned}
\]
separates the centered quadratic form from its deterministic centering, whose
second term tends to zero by the preceding centering result.  Put
\(D_{\bn}(t)=C_{\bn}^{1/2}A_{\bn}(t)C_{\bn}^{1/2}\).  The covariance
comparison and the preceding covariance limit give
\[
 \frac{\|D_{\bn}(t)\|_2}{\sqrt{N_{\bn}}}
 \longrightarrow0,
 \qquad \frac{\tr\{D_{\bn}(t)^2\}}{N_{\bn}}
 \longrightarrow t^T\Gamma t.
\]
If \(\lambda_{\ell,\bn}\) are the eigenvalues of \(D_{\bn}(t)\),
orthogonal diagonalization gives the triangular-array representation
\[
 \frac{\xi_{\bn}^TD_{\bn}(t)\xi_{\bn}-\tr D_{\bn}(t)}{\sqrt{N_{\bn}}}
 \ \stackrel{d}{=}\
 \sum_{\ell=1}^{N_{\bn}}
 \frac{\lambda_{\ell,\bn}}{\sqrt{N_{\bn}}}
 \bigl(\xi_{\ell,\bn}^2-1\bigr),
\]
whose variance converges to \(2t^T\Gamma t\).  If \(t^T\Gamma t>0\),
the fourth-moment Lyapunov condition is
\[
 \frac{\sum_\ell
 \mathbb E\left|\lambda_{\ell,\bn}
 (\xi_{\ell,\bn}^2-1)/\sqrt{N_{\bn}}\right|^4}{\left\{2N_{\bn}^{-1}\sum_\ell\lambda_{\ell,\bn}^2\right\}^2}
 =\frac{\mathbb E(\xi^2-1)^4}{4}
 \frac{\sum_\ell\lambda_{\ell,\bn}^4}{\{\sum_\ell\lambda_{\ell,\bn}^2\}^2}
 \le\frac{\mathbb E(\xi^2-1)^4}{4}
 \frac{\|D_{\bn}(t)\|_2^2}{\tr\{D_{\bn}(t)^2\}}
 \longrightarrow0.
\]
The Lyapunov triangular-array central limit theorem
\cite[Theorem~27.3, p.~362]{Billingsley1995} therefore gives the scalar
\(\mathcal N(0,2t^T\Gamma t)\) limit.  If \(t^T\Gamma t=0\), its variance
tends to zero, so Chebyshev's inequality gives the same degenerate Gaussian
limit.  The deterministic centering already tends to zero, and hence, for
every \(t\in\R^q\), the scalar statistic
\[
 \sum_{j=1}^q t_j
 \frac{Z_{\bn}^TA_{j,\bn}Z_{\bn}-\tr A_{j,\bn}}{\sqrt{N_{\bn}}}
 \Rightarrow \mathcal N(0,2t^T\Gamma t).
\]
The Cram\'er--Wold device
\cite[page~16]{VanDerVaart1998} now yields the joint
\(\mathcal N_q(0,2\Gamma)\) limit.
\end{proof}

Write \(m_1:=\int_Q r_{\alpha_0}\,d\nu_Q\) and \(m_2:=\int_Q r_{\alpha_0}^2\,d\nu_Q\), so \(J=\frac{1}{2}\begin{pmatrix}1&m_1\\m_1&m_2\end{pmatrix}\).

\begin{proof}[Proof of Theorem~\ref{thm:tail-exponent-gradient-information}]
Under \(P_{\varphi_0,\bn}\), apply the case
\(C_{\bn}=I_{N_{\bn}}\) of
Lemma~\ref{lem:centered-quadratic-form-transfer} to
the two matrix sequences \(A_{1,\bn}=I_{N_{\bn}}\) and
\(A_{2,\bn}=R_{\alpha,\bn}\).
Equation \eqref{eq:Ralpha-trace-limits} gives
\[
 \frac{\|R_{\alpha,\bn}\|_2}{\sqrt{N_{\bn}}}\longrightarrow0,
 \qquad
 \frac{\tr R_{\alpha,\bn}}{N_{\bn}}\longrightarrow m_1,
 \qquad
 \frac{\tr(R_{\alpha,\bn}^2)}{N_{\bn}}\longrightarrow m_2.
\]
Together with \(\|I_{N_{\bn}}\|_2/\sqrt{N_{\bn}}\to0\), these are its
hypotheses with
\[
 \Gamma=
 \begin{pmatrix}1&m_1\\m_1&m_2\end{pmatrix}=2J.
\]
The vector CLT in that lemma shows that the quadratic-form vector displayed in
Theorem~\ref{thm:tail-exponent-gradient-information} converges, before
division by two, to \(\mathcal N(0,2\Gamma)=\mathcal N(0,4J)\).  The
displayed definition of \(\mathcal S_{\bn}\) in that theorem divides this
vector by two, and hence \(\mathcal S_{\bn}\Rightarrow\mathcal N(0,J)\).
\end{proof}
\endgroup

\begingroup
\section{Hellinger separation, consistency, and local likelihood expansion}
\label{app:likelihood-separation-and-lan}
Recall from Equation~\eqref{eq:tail-rate-matrix} the sample-size-dependent
local parameter chart: for \(u=(u_\tau,u_\alpha)^T\),
\[
        \varphi_{\bn}(u):=\varphi_0+\Delta_{\bn}u,
        \qquad
        \Delta_{\bn}=N_{\bn}^{-1/2}
        \begin{pmatrix}1&a_{\bn}\\0&1\end{pmatrix}.
\]
Here \(u_\tau\) is the pure scale coordinate, whereas \(u_\alpha\) is the
compensated exponent coordinate.  We retain the moments \(m_1,m_2\)
and information matrix \(J\) from the preceding proof.
{
For \(A\succ0\), define the Cayley transform
\(\mathfrak C(A):=(A-I)(A+I)^{-1}\).
To obtain uniform affinity separation, relative covariance
departure must be detected uniformly.  The Cayley transform provides a
bounded spectral certificate, while \(q_{\bn}\) records the compensated
scale--shape distance.
For \(\varphi=(\tau,\alpha)\), recall that we use
\[
        P_{\varphi,\bn}:=\mathcal N(0,\Sigma_{\varphi,\bn}),
        \qquad
        V_{\varphi,\bn}:=\Sigma_{\varphi_0,\bn}^{-1/2}
        \Sigma_{\varphi,\bn}\Sigma_{\varphi_0,\bn}^{-1/2},
        \qquad
        g_{\bn}(\varphi):=(\tau-\tau_0)-a_{\bn}(\alpha-\alpha_0).
\]
Equation~\eqref{eq:supp-normalized-tail-secant} gives
\begin{equation}
        V_{\varphi,\bn}
        =e^{g_{\bn}(\varphi)}
        \{I_{N_{\bn}}+(\alpha-\alpha_0)
        \mathsf K_{\alpha-\alpha_0,\bn}\}.
        \label{eq:supp-full-relative-covariance-secant}
\end{equation}
We write
$
        q_{\bn}(\varphi)
        :=\sqrt{g_{\bn}(\varphi)^2+(\alpha-\alpha_0)^2}.
$
If \(x_1,\ldots,x_{N_{\bn}}\) are the generalized covariance
eigenvalues, their Gaussian Hellinger affinity satisfies
\begin{equation}
        -\log\mathcal A(P_{\varphi_0,\bn},P_{\varphi,\bn})
        =\frac{1}{2}\sum_{j=1}^{N_{\bn}}
        \log\cosh\left(\frac{\log x_j}{2}\right).
        \label{eq:supp-gaussian-affinity}
\end{equation}

The Cayley transform turns eigenvalue distance from one into a
bounded squared signal accessible to GLT limits.
\begin{lemma}
\label{lem:working-cayley-hellinger}
Let \(P_0=\mathcal N(0,\Sigma_0)\) and \(P_1=\mathcal N(0,\Sigma_1)\) be Gaussian laws
on \(\mathbb R^N\), where \(\Sigma_0,\Sigma_1\succ0\). Set
\(
A:=\Sigma_0^{-1/2}\Sigma_1\Sigma_0^{-1/2}.
\)
Then
\[
-\log\mathcal A(P_0,P_1)
\ge \frac{1}{4}\tr\{\mathfrak C(A)^2\}.
\]
\end{lemma}

\begin{proof}[Proof of Lemma~\ref{lem:working-cayley-hellinger}]
If \(x>0\) is a generalized covariance eigenvalue and
\(z=(x-1)/(x+1)\), then
\[
        \frac{1}{2}\log\cosh\left(\frac{\log x}{2}\right)
        =-\frac{1}{4}\log(1-z^2)
        \ge\frac{1}{4}z^2.
\]
Sum this inequality over the eigenvalues and use
Equation~\eqref{eq:supp-gaussian-affinity}.
\end{proof}

The local quadratic affinity loss in \(q_{\bn}\) supplies
the local half of global separation.

\begin{lemma}
\label{lem:working-local-cayley-coercivity}
There are \(c_0,{\eta_{\mathrm{coer}}}>0\) such that, for all sufficiently large sample
sizes and every \(\varphi\in\Theta_{\mathrm{tail}}\) with
\(0<q_{\bn}(\varphi)\le{\eta_{\mathrm{coer}}}\),
\[
        -\log\mathcal A(P_{\varphi_0,\bn},P_{\varphi,\bn})
        \ge c_0N_{\bn}q_{\bn}(\varphi)^2.
\]
\end{lemma}

\begin{proof}
Suppose that the asserted uniform constants do not exist.  Then, after
passing to a subsequence, there is a sequence
\(\varphi_{\bn}=(\tau_{\bn},\alpha_{\bn})\) such that
\[
        s_{\bn}:=q_{\bn}(\varphi_{\bn})\longrightarrow0,
        \qquad s_{\bn}>0,
        \qquad
        \frac{-\log\mathcal A
        (P_{\varphi_0,\bn},P_{\varphi_{\bn},\bn})}{N_{\bn}s_{\bn}^2}\longrightarrow0.
\]
With \(\delta_{\alpha,\bn}:=\alpha_{\bn}-\alpha_0\), pass to a
further subsequence such that
\[
        \left(
        \frac{g_{\bn}(\varphi_{\bn})}{s_{\bn}},
        \frac{\delta_{\alpha,\bn}}{s_{\bn}}
        \right)
        \longrightarrow(u,v),
        \qquad u^2+v^2=1.
\]
Define
\[
        E_{\bn}:=\frac{V_{\varphi_{\bn},\bn}-I}{s_{\bn}},
        \qquad
        H_{\bn}:=E_{\bn}(V_{\varphi_{\bn},\bn}+I)^{-1}
        =s_{\bn}^{-1}\mathfrak C(V_{\varphi_{\bn},\bn}).
\]
By Equation~\eqref{eq:supp-full-relative-covariance-secant},
\[
        E_{\bn}
        =\frac{e^{g_{\bn}(\varphi_{\bn})}-1}{s_{\bn}}I
        +e^{g_{\bn}(\varphi_{\bn})}
        \frac{\delta_{\alpha,\bn}}{s_{\bn}}
        \mathsf K_{\delta_{\alpha,\bn},\bn}.
\]
Lemma~\ref{lem:local-fixed-whitening-secant-spectral-control} and
\(e^x-1=x+O(x^2)\) give
\[
        \{E_{\bn}\}_{\bn}\sim_{\GLT}u+vr_{\alpha_0}.
\]
Since \(V_{\varphi_{\bn},\bn}\sim_{\GLT}1\) and
\(V_{\varphi_{\bn},\bn}+I\succeq I\), inverse functional calculus gives
\[
        \{(V_{\varphi_{\bn},\bn}+I)^{-1}\}_{\bn}
        \sim_{\GLT}\frac{1}{2}.
\]
Thus the Hermitian GLT product rule yields
\[
        \{H_{\bn}\}_{\bn}\sim_{\GLT}g_{u,v},
        \qquad
        g_{u,v}:=\frac{u+vr_{\alpha_0}}{2}.
\]
The function \(g_{u,v}\) is not identically zero.  If \(v=0\), then
\(u=\pm1\); if \(v\ne0\), then
\(r_{\alpha_0}(\bth)=-2\log\|\bth\|+O(1)\) is not almost everywhere
constant.  Choose \(\chi\in C_c(\R)\) with
\(0\le \chi(x)\le x^2\) and
\(\int_Q\chi(g_{u,v})\,d\nu_Q>0\).  The Hermitian GLT distribution
formula then gives
\[
\begin{aligned}
        \liminf_{\bn}\frac{1}{N_{\bn}}\operatorname{tr}(H_{\bn}^2)
        &\ge \lim_{\bn}\frac{1}{N_{\bn}}\operatorname{tr}\chi(H_{\bn})\\
        &=\int_Q\chi(g_{u,v})\,d\nu_Q>0.
\end{aligned}
\]
Since \(\mathfrak C(V_{\varphi_{\bn},\bn})=s_{\bn}H_{\bn}\),
Lemma~\ref{lem:working-cayley-hellinger} gives
\[
        \liminf_{\bn}
        \frac{-\log\mathcal A
        (P_{\varphi_0,\bn},P_{\varphi_{\bn},\bn})}{N_{\bn}s_{\bn}^2}>0,
\]
contradicting the selected sequence.
\end{proof}
}

{Scalar rescaling can escape to zero or infinity; the Cayley
transform saturates at those endpoints and remains informative.}
{
\begin{lemma}
\label{lem:cayley-rescaling-glt}
Let \(M_{\bn}\succ0\) be Hermitian, with
\(\{M_{\bn}\}_{\bn}\sim_{\GLT}m\), where \(m\) is finite and positive
almost everywhere.  If
\(c_{\bn}>0\), then, whenever \(c_{\bn}\to c\in(0,\infty)\),
\begin{align}
        \label{eq:Caley_trace}
\frac{1}{N_{\bn}}\tr\{\mathfrak C(c_{\bn}M_{\bn})^2\}
 \longrightarrow\int_Q\left(\frac{cm-1}{cm+1}\right)^2\,d\nu_Q.
\end{align}
If \(c_{\bn}\to0\) or \(c_{\bn}\to\infty\), the same normalized trace
converges to \(1\).
\end{lemma}

\begin{proof}
For \(c_{\bn}\to c\in(0,\infty)\),
\(\{c_{\bn}M_{\bn}\}_{\bn}\sim_{\GLT}cm\). 
  Define
\[
\mu_{\bn}
:=
\frac{1}{N_{\bn}}\sum_{j=1}^{N_{\bn}}
\delta_{\lambda_j(c_{\bn}M_{\bn})},
\qquad
\mu:=(cm)_{\#}\nu_Q,
\]
i.e. 
$
\mu(B)
=
\nu_Q\{\bth\in Q:cm(\bth)\in B\},$
$
B\in\mathcal B([0,\infty)).
$
These are probability measures on \([0,\infty)\). By the Hermitian GLT spectral-distribution theorem
\cite[Theorem~5.2, p.~94]{GaroniSerraCapizzanoII},
\(\mu_{\bn}\to\mu\) vaguely. We show that this
convergence is weak.  Given \(\varepsilon>0\), choose
\(\chi\in C_c([0,\infty))\) such that \(0\le\chi\le1\) and
\(\int\chi\,d\mu\ge1-\varepsilon/2\).  If
\(K=\operatorname{supp}\chi\), then vague convergence implies that, for all
sufficiently large \(\bn\),
\[
\mu_{\bn}(K)
\ge \int\chi\,d\mu_{\bn}
\ge 1-\varepsilon.
\]
Enlarging \(K\) to include the finite supports of the finitely many initial
empirical measures, we obtain
$
\mu_{\bn}(K)\ge1-\varepsilon
$
for every \(\bn\).  Hence \(\{\mu_{\bn}\}_{\bn}\) is tight and thus
\(\mu_{\bn}\to\mu\) weakly.

Now set
$
h(x):=\left(\frac{x-1}{x+1}\right)^2, x\ge0.
$
The function \(h\) is bounded and continuous on \([0,\infty)\).  By the
spectral theorem and the definition of \(\mathfrak C\),
{\small
\[
\frac{1}{N_{\bn}}\tr\{\mathfrak C(c_{\bn}M_{\bn})^2\}
=\int_{[0,\infty)}h(x)\,d\mu_{\bn}(x)
\longrightarrow\int_{[0,\infty)}h(x)\,d\mu(x)
=\int_Q\left(\frac{cm-1}{cm+1}\right)^2\,d\nu_Q.
\]
}

If \(c_{\bn}\to0\), set
\(Q_{\bn}:=c_{\bn}M_{\bn}(c_{\bn}M_{\bn}+I)^{-1}\).  Then
\(\mathfrak C(c_{\bn}M_{\bn})=-I+2Q_{\bn}\).  
If \(\lambda_{j,\bn}\), \(1\le j\le N_{\bn}\), are the eigenvalues of
\(M_{\bn}\), then, for every \(K>0\), the corresponding eigenvalues of
\(Q_{\bn}\) satisfy
\[
0\le q_{j,\bn}
=\frac{c_{\bn}\lambda_{j,\bn}}{1+c_{\bn}\lambda_{j,\bn}}
\le
c_{\bn}K\,\mathbb I_{\{\lambda_{j,\bn}\le K\}}
+\mathbb I_{\{\lambda_{j,\bn}>K\}}.
\]
Consequently,
\[
\frac{\tr Q_{\bn}}{N_{\bn}}
\le c_{\bn}K+
\frac{1}{N_{\bn}}
\#\{j:\lambda_{j,\bn}>K\}.
\]

The Portmanteau theorem {\cite[Theorem~29.1(ii)]{Billingsley1995}}, applied to the
closed set \([K,\infty)\), gives
\[
\limsup_{\bn}\frac{1}{N_{\bn}}
\#\{j:\lambda_j(M_{\bn})\ge K\}
\le \nu_Q\{m\ge K\}.
\]
Combining this with the preceding bound and \(c_{\bn}K\to0\) yields
\[
\limsup_{\bn}\frac{\tr Q_{\bn}}{N_{\bn}}
\le \nu_Q\{m\ge K\}.
\]
Letting \(K\to\infty\) yields \(N_{\bn}^{-1}\tr Q_{\bn}\to0\), and hence
\(0\preceq Q_{\bn}\preceq I\) gives
\(\tr(Q_{\bn}^2)\le\tr Q_{\bn}\), hence
\[
 \frac{1}{N_{\bn}}\tr\{\mathfrak C(c_{\bn}M_{\bn})^2\}
 =1-\frac{4\tr Q_{\bn}}{N_{\bn}}
 +\frac{4\tr(Q_{\bn}^2)}{N_{\bn}} 
\longrightarrow1.
\]

If \(c_{\bn}\to\infty\), inversion gives
$
\{M_{\bn}^{-1}\}_{\bn}\sim_{\GLT}m^{-1},
$ 
where \(m^{-1}\) is finite and positive almost everywhere.  Now use 
$
\mathfrak C\bigl((c_{\bn}M_{\bn})^{-1}\bigr)
=-\mathfrak C(c_{\bn}M_{\bn}).
$ Apply the preceding argument with \(c_{\bn}^{-1}\) and \(M_{\bn}^{-1}\).
\end{proof}
}

{Applying saturation to every sequence with \(q_{\bn}\) bounded
away from zero yields an order-\(N_{\bn}\) Cayley signal.}
{
\begin{lemma}
\label{lem:working-noncritical-cayley-separation}
For every sequence \(\varphi_{\bn}\in\Theta_{\mathrm{tail}}\) satisfying
\(\liminf_{\bn}q_{\bn}(\varphi_{\bn})>0\),
\begin{align}
        \label{eq:Caley2}
 \liminf_{\bn}\frac{1}{N_{\bn}}\tr\{\mathfrak C(V_{\varphi_{\bn},\bn})^2\}>0.
\end{align}
\end{lemma}

\begin{proof}
Suppose, to the contrary, that the liminf in
Equation~\eqref{eq:Caley2} is zero.  Along a subsequence it then converges
to zero.  By compactness and extended-real sequential compactness, pass to a
further subsequence such that
\[
        \delta_{\alpha,\bn}:=\alpha_{\bn}-\alpha_0\longrightarrow\delta_*,
        \qquad
        g_{\bn}(\varphi_{\bn})\longrightarrow g_*
        \quad\text{in }[-\infty,\infty].
\]
Write \(V_{\varphi_{\bn},\bn}=e^{g_{\bn}(\varphi_{\bn})}\mathsf M_{\bn}\).
If \(\delta_*=0\), take
\[
        \mathsf M_{\bn}:=
        I+\delta_{\alpha,\bn}\mathsf K_{\delta_{\alpha,\bn},\bn}\succ0.
\]
Lemma~\ref{lem:local-fixed-whitening-secant-spectral-control} gives
\(\{\mathsf M_{\bn}\}_{\bn}\sim_{\GLT}1\).  If \(g_*=\pm\infty\),
Lemma~\ref{lem:cayley-rescaling-glt} gives a trace limit of one.  If
\(g_*\in\mathbb R\), then \(\liminf_{\bn}q_{\bn}(\varphi_{\bn})>0\) and
\(\delta_{\alpha,\bn}\to0\) imply \(\lvert g_*\rvert>0\).  Thus
\(c=e^{g_*}\ne1\), and the finite-\(c\) limit in
Lemma~\ref{lem:cayley-rescaling-glt} is
\[
        \left(\frac{c-1}{c+1}\right)^2
        =\tanh^2(g_*/2)>0.
\]

If \(\delta_*\ne0\), take
\[
        \mathsf M_{\bn}:=
        h^{-2\delta_{\alpha,\bn}}
        \Sigma_{\varphi_0,\bn}^{-1/2}
        \Sigma_{(\tau_0,\alpha_{\bn}),\bn}
        \Sigma_{\varphi_0,\bn}^{-1/2}.
\]
Lemma~\ref{lem:working-moving-rescaled-glt} and Hermitian GLT functional
calculus give
\[
        \{\mathsf M_{\bn}\}_{\bn}\sim_{\GLT}
        \frac{\psi_{\alpha_0+\delta_*}}{\psi_{\alpha_0}}>0.
\]
Because \(\tau_{\bn}\) is bounded and \(a_{\bn}\to\infty\),
\[
        g_{\bn}(\varphi_{\bn})\longrightarrow
        \begin{cases}
        -\infty, & \delta_*>0,\\
        +\infty, & \delta_*<0.
        \end{cases}
\]
Lemma~\ref{lem:cayley-rescaling-glt} again gives a trace limit of one.  Each
case contradicts the selected zero-trace subsequence.
\end{proof}
}

{

The proof now stitches the local coercivity bound to noncritical separation.
\begin{proof}[Proof of Proposition~\ref{prop:working-glt-hellinger-separation}]
We first consider \(q_{\bn}(\varphi)=0\), for which the result is
immediate.  When \(0<q_{\bn}(\varphi)\le{\eta_{\mathrm{coer}}}\),
Lemma~\ref{lem:working-local-cayley-coercivity} gives
\[
        -\log\mathcal A(P_{\varphi_0,\bn},P_{\varphi,\bn})
        \ge c_0N_{\bn}q_{\bn}(\varphi)^2.
\]

There is \(c_1>0\) such that, for all sufficiently large sample sizes,
\[
        \inf_{\substack{\varphi\in\Theta_{\mathrm{tail}}\\
        q_{\bn}(\varphi)\ge{\eta_{\mathrm{coer}}}}}
        \frac{-\log\mathcal A(P_{\varphi_0,\bn},P_{\varphi,\bn})}{N_{\bn}}
        \ge c_1.
\]
Otherwise, there would be a sequence
\(\varphi_{\bn}\in\Theta_{\mathrm{tail}}\) satisfying
\(q_{\bn}(\varphi_{\bn})\ge{\eta_{\mathrm{coer}}}\) and
\[
\frac{-\log\mathcal A
(P_{\varphi_0,\bn},P_{\varphi_{\bn},\bn})}{N_{\bn}}
\longrightarrow0.
\]
However, Lemmas~\ref{lem:working-noncritical-cayley-separation}
and~\ref{lem:working-cayley-hellinger} give
\[
        0<\frac{1}{4}\liminf_{\bn}\frac{1}{N_{\bn}}
        \tr\{\mathfrak C(V_{\varphi_{\bn},\bn})^2\}
        \le\liminf_{\bn}\frac{1}{N_{\bn}}
        \{-\log\mathcal A(P_{\varphi_0,\bn},P_{\varphi_{\bn},\bn})\},
\]
a contradiction.

Taking \(c:=\min\{c_0,c_1\}\), we obtain
\[
        -\log\mathcal A(P_{\varphi_0,\bn},P_{\varphi,\bn})
        \ge cN_{\bn}\min\{q_{\bn}(\varphi)^2,1\}
        \ge cN_{\bn}
        \frac{g_{\bn}(\varphi)^2+(\alpha-\alpha_0)^2}{1+g_{\bn}(\varphi)^2+(\alpha-\alpha_0)^2}.
\]
\end{proof}
\begin{corollary}
\label{cor:working-glt-fixed-neighborhood-separation}
For every relative open neighborhood \(U\) of \(\varphi_0\) in
\(\Theta_{\mathrm{tail}}\), there is \(c_U>0\) such that
\[
 \inf_{\varphi\in\Theta_{\mathrm{tail}}\setminus U}
 \{-\log\mathcal A(P_{\varphi_0,\bn},P_{\varphi,\bn})\}
 \ge c_U\frac{N_{\bn}}{a_{\bn}^2}
\]
for all sufficiently large sample sizes.
\end{corollary}

\begin{proof}
If \(\Theta_{\mathrm{tail}}\setminus U\) is empty, the claim is immediate.
Otherwise, with \(\delta_\tau:=\tau-\tau_0\) and
\(\delta_\alpha:=\alpha-\alpha_0\), compactness gives
\[
        d_U:=\inf_{\varphi\in\Theta_{\mathrm{tail}}\setminus U}
        \sqrt{\delta_\tau^2+\delta_\alpha^2}>0.
\]
The coordinate relation
\(\sqrt{\delta_\tau^2+\delta_\alpha^2}\le(a_{\bn}+2)q_{\bn}(\varphi)\)
gives \(q_{\bn}(\varphi)\ge d_U/(a_{\bn}+2)\) on the complement.
Proposition~\ref{prop:working-glt-hellinger-separation} and \(a_{\bn}\to\infty\)
now give the stated rate.
\end{proof}
The rate \(N_{\bn}/a_{\bn}^2\) is the fixed-complement input for global
consistency.
}
{
Deterministic affinity separation becomes a one-point likelihood probability
bound through the following inequality.
\begin{lemma}
\label{lem:likelihood-ratio-affinity-bound}
For every \(r\ge0\) and \(\varphi\in\Theta_{\mathrm{tail}}\),
\[
        \mathbb P_{\varphi_0}\{\ell_{\bn}(\varphi)
        -\ell_{\bn}(\varphi_0)\ge-r\}
        \le e^{r/2}\mathcal A(P_{\varphi_0,\bn},P_{\varphi,\bn}).
\]
\end{lemma}

\begin{proof}[Proof of Lemma~\ref{lem:likelihood-ratio-affinity-bound}]
Since
$
R_{\bn}(\varphi)
=\exp\{\ell_{\bn}(\varphi)-\ell_{\bn}(\varphi_0)\},
$
we have
\[
\{\ell_{\bn}(\varphi)-\ell_{\bn}(\varphi_0)\ge-r\}
=
\{R_{\bn}(\varphi)^{1/2}\ge e^{-r/2}\}.
\]
Markov's inequality therefore gives
\[
\begin{aligned}
e^{-r/2} \mathbb P_{\varphi_0}
\{\ell_{\bn}(\varphi)-\ell_{\bn}(\varphi_0)\ge-r\}
&\le
\mathbb E_{\varphi_0}R_{\bn}(\varphi)^{1/2}\\
&= \int
\left(\frac{dP_{\varphi,\bn}}{dP_{\varphi_0,\bn}}\right)^{1/2}
dP_{\varphi_0,\bn} \\
&=
\mathcal A
(P_{\varphi_0,\bn},P_{\varphi,\bn}) .
\end{aligned}
\]

\end{proof}

Define \(x_+:=\max(x,0)\).
Pointwise likelihood bounds require a modulus to transfer continuum
improvements to a finite net. We will use the following notation in the results below
\(
Z_{\bn}:=\Sigma_{\varphi_0,\bn}^{-1/2}Y_{\bn}
\)
and
\(
\mathcal A_{\bn}:=\{\|Z_{\bn}\|^2\le2N_{\bn}\}.
\)
\begin{lemma}
\label{lem:working-likelihood-modulus}
There is \(C>0\) such that, for all sufficiently large sample sizes, on
\(\mathcal A_{\bn}\),
\begin{equation}
\sup_{\varphi\in\Theta_{\mathrm{tail}}}
\|\nabla_\varphi\ell_{\bn}(\varphi)\|
\le
\mathfrak L_{\bn}
:=C N_{\bn}a_{\bn}
h^{-2(\alpha_+-\alpha_0)_+}.
\label{eq:working-likelihood-modulus}
\end{equation}
Moreover,
$
\mathbb P_{\varphi_0}(\mathcal A_{\bn})\longrightarrow1.
$
\end{lemma}

\begin{proof}
For \(i\in\{\tau,\alpha\}\), differentiation of the Gaussian likelihood
gives
\[
\partial_i\ell_{\bn}(\varphi)
=
-\frac{1}{2}\tr
(\Sigma_{\varphi,\bn}^{-1}\partial_i\Sigma_{\varphi,\bn})
+\frac{1}{2}Y_{\bn}^{T}\Sigma_{\varphi,\bn}^{-1}
(\partial_i\Sigma_{\varphi,\bn})
\Sigma_{\varphi,\bn}^{-1}Y_{\bn}.
\]
We control this expression uniformly over
\(\varphi\in\Theta_{\mathrm{tail}}\) by first bounding
\(\partial_i\Sigma_{\varphi,\bn}\) relative to
\(\Sigma_{\varphi,\bn}\), and then comparing
\(\Sigma_{\varphi,\bn}\) with \(\Sigma_{\varphi_0,\bn}\).

Equation~\eqref{eq:supp-aliased-likelihood-factorization} and
Lemma~\ref{lem:uniform-tail-family-comparison} give, uniformly over
\(\varphi\in\Theta_{\mathrm{tail}}\),
\[
|\partial_\alpha\phi_{h,\varphi}|
\le Ca_{\bn}\phi_{h,\varphi},
\qquad
\partial_\tau\phi_{h,\varphi}
=\phi_{h,\varphi}.
\]
The multilevel Toeplitz quadratic-form identity in
Equation~\eqref{eq:supp-toeplitz-quadratic-form} therefore gives
\begin{align*}
 -Ca_{\bn}\Sigma_{\varphi,\bn}
 &\preceq \partial_\alpha\Sigma_{\varphi,\bn}
 \preceq Ca_{\bn}\Sigma_{\varphi,\bn},\\
 \partial_\tau\Sigma_{\varphi,\bn}
 &=\Sigma_{\varphi,\bn}.
\end{align*}
Since \(a_{\bn}\to\infty\), these bounds and the likelihood derivative
formula imply, for \(i\in\{\tau,\alpha\}\),
\[
|\partial_i\ell_{\bn}(\varphi)|
\le
Ca_{\bn}
\left\{
N_{\bn}
+Y_{\bn}^{T}\Sigma_{\varphi,\bn}^{-1}Y_{\bn}
\right\}
\]
for all sufficiently large sample sizes.

It remains to control the quadratic term uniformly in \(\varphi\).
Equation~\eqref{eq:tail-alias-global-comparison} reduces the relevant
symbol ratio to the corresponding ratio of the alias series
\(\psi_{h,\alpha_0}\) and \(\psi_{h,\alpha}\).  
Let \(a_{\boldsymbol0}(\bth):=h^2+\|\bth\|^2\).  By definition,
\[
\psi_{h,u}(\bth)
=
a_{\boldsymbol0}(\bth)^{-u}
+
\sum_{\bk\ne\boldsymbol0}
\{h^2+\|\xi_{\bk}(\bth)\|^2\}^{-u}.
\]
The nonnegativity of the second term gives the lower bound.  Since
\(\alpha_->d/2\), the sum over \(\bk\ne\boldsymbol0\) is bounded uniformly
for \(\bth\in Q\), \(u\in[\alpha_-,\alpha_+]\), and all sufficiently
large sample sizes.  Therefore,
$
a_{\boldsymbol0}(\bth)^{-u}
\le
\psi_{h,u}(\bth)
\le
a_{\boldsymbol0}(\bth)^{-u}+C.
$ It follows that
\[
\frac{\psi_{h,\alpha_0}(\bth)}{\psi_{h,\alpha}(\bth)}
\le
\begin{cases}
C, & \alpha\ge\alpha_0,\\
C h^{-2(\alpha_0-\alpha)}, & \alpha<\alpha_0.
\end{cases}
\]
Combining these bounds with the factorization in
Equation~\eqref{eq:supp-aliased-likelihood-factorization} and the
comparison in Equation~\eqref{eq:tail-alias-global-comparison}, and using
\(\tau\in[\tau_-,\tau_+]\), gives
\[
\frac{\phi_{h,\varphi_0}(\bth)}{\phi_{h,\varphi}(\bth)}
\le
C e^{\tau_0-\tau}
h^{2(\alpha_0-\alpha)}
\frac{\psi_{h,\alpha_0}(\bth)}{\psi_{h,\alpha}(\bth)}
\le
C h^{-2(\alpha-\alpha_0)_+}
\le
C h^{-2(\alpha_+-\alpha_0)_+}.
\]

The Toeplitz quadratic-form identity now gives
$
\Sigma_{\varphi_0,\bn}
\preceq
C h^{-2(\alpha_+-\alpha_0)_+}
\Sigma_{\varphi,\bn}.
$
After inversion,
\[
Y_{\bn}^{T}\Sigma_{\varphi,\bn}^{-1}Y_{\bn}
\le
C h^{-2(\alpha_+-\alpha_0)_+}
Y_{\bn}^{T}\Sigma_{\varphi_0,\bn}^{-1}Y_{\bn} =
C h^{-2(\alpha_+-\alpha_0)_+}\|Z_{\bn}\|^2.
\]
Consequently, on \(\mathcal A_{\bn}\),
\[
|\partial_i\ell_{\bn}(\varphi)|
\le
C N_{\bn}a_{\bn}
h^{-2(\alpha_+-\alpha_0)_+},
\qquad i\in\{\tau,\alpha\},
\]
uniformly over \(\varphi\in\Theta_{\mathrm{tail}}\).  Absorbing the two
coordinate bounds into \(C\) proves
Equation~\eqref{eq:working-likelihood-modulus}.

Finally, under \(P_{\varphi_0,\bn}\),
\(Z_{\bn}\sim\mathcal N(0,I_{N_{\bn}})\).  Since
\(\|Z_{\bn}\|^2/N_{\bn}\to1\) in probability,
\(\mathbb P_{\varphi_0}(\mathcal A_{\bn})\to1\).
\end{proof}

The deterministic modulus also satisfies
\[
\log\mathfrak L_{\bn}
=
O(1)+\log N_{\bn}+\log a_{\bn}
+2(\alpha_+-\alpha_0)_+\log(1/h)
=
O(a_{\bn}).
\]
\begin{lemma}
\label{lem:working-likelihood-exclusion}
Let \(K_{\bn}\) be deterministic Borel subsets of
\(\Theta_{\mathrm{tail}}\).  With the conventions
\(\inf\varnothing=+\infty\) and \(\sup\varnothing=-\infty\), write
\[
        \beta_{\bn}
        :=\inf_{\varphi\in K_{\bn}}
        \{-\log\mathcal A(P_{\varphi_0,\bn},P_{\varphi,\bn})\}.
\]
There is \(C>0\) such that, for all sufficiently large sample sizes,
\[
\begin{aligned}
        P_{\varphi_0,\bn}
        \left\{
        \sup_{\varphi\in K_{\bn}}
        [\ell_{\bn}(\varphi)-\ell_{\bn}(\varphi_0)]
        \ge0
        \right\}
        &\le
        P_{\varphi_0,\bn}(\mathcal A_{\bn}^c)
        +C\mathfrak L_{\bn}^2e^{-\beta_{\bn}}.
\end{aligned}
\]
Consequently, this probability tends to zero whenever
$
        \beta_{\bn}-2\log\mathfrak L_{\bn}\longrightarrow\infty.
$
\end{lemma}

\begin{proof}
The assertion is immediate when \(K_{\bn}=\varnothing\).  Otherwise, the
likelihood increment is continuous, so its supremum over \(K_{\bn}\) equals
its supremum over a countable dense subset of \(K_{\bn}\); in particular,
the event in the statement is measurable.  Let
$
        \epsilon_{\bn}:=(16\mathfrak L_{\bn})^{-1}.
$
Because \(\Theta_{\mathrm{tail}}\) is a fixed compact rectangle in
\(\R^2\) and \(\mathfrak L_{\bn}\ge1\) eventually, \(K_{\bn}\) has an
internal \(\epsilon_{\bn}\)-net \(\mathcal T_{\bn}\subset K_{\bn}\)
satisfying
$
        |\mathcal T_{\bn}|\le C\mathfrak L_{\bn}^2.
$
If the event in the statement occurs, then the definition of the supremum
gives \(\varphi\in K_{\bn}\) such that
$
        \ell_{\bn}(\varphi)-\ell_{\bn}(\varphi_0)>-\frac{1}{16}.
$
Choose \(\widetilde\varphi\in\mathcal T_{\bn}\) within
\(\epsilon_{\bn}\) of \(\varphi\).  The segment joining these points lies
in the convex set \(\Theta_{\mathrm{tail}}\).  Hence, on
\(\mathcal A_{\bn}\), Equation~\eqref{eq:working-likelihood-modulus}
gives
\[
\begin{aligned}
        \ell_{\bn}(\widetilde\varphi)-\ell_{\bn}(\varphi_0)
        \ge
        \ell_{\bn}(\varphi)-\ell_{\bn}(\varphi_0)
        -\mathfrak L_{\bn}
        \|\widetilde\varphi-\varphi\|>-\frac{1}{8}.
\end{aligned}
\]
Thus Lemma~\ref{lem:likelihood-ratio-affinity-bound}, applied at every
center of \(\mathcal T_{\bn}\), gives
\[
\begin{aligned}
        P_{\varphi_0,\bn}
        \left\{
        \sup_{\varphi\in K_{\bn}}
        [\ell_{\bn}(\varphi)-\ell_{\bn}(\varphi_0)]
        \ge0
        \right\}\le
        P_{\varphi_0,\bn}(\mathcal A_{\bn}^c)
        +e^{1/16}|\mathcal T_{\bn}|e^{-\beta_{\bn}},
\end{aligned}
\]
which proves the bound.  The final assertion follows from
\(P_{\varphi_0,\bn}(\mathcal A_{\bn}^c)\to0\).
\end{proof}
}

{For consistency, fixed-neighborhood separation controls the grid
points, the likelihood modulus transfers continuum improvements to the grid,
and the net entropy is dominated by the affinity exponent.
\begin{proof}[Proof of Theorem~\ref{thm:tail-chart-consistency}]

{We use the neighborhood characterization of convergence in probability.
Fix a relative open neighborhood \(U\) of \(\varphi_0\) in
\(\Theta_{\mathrm{tail}}\), and set \(K:=\Theta_{\mathrm{tail}}\setminus U\).
If \(K=\varnothing\), the claim is immediate.  Suppose henceforth that
\(K\ne\varnothing\).  Then \(K\) is compact.  Since
\(\widehat\varphi_{\bn}\in\Theta_{\mathrm{tail}}\),
\[
 \{\widehat\varphi_{\bn}\in K\}
 \subseteq
 E_{\bn}:=\left\{
 \sup_{\varphi\in K}
 [\ell_{\bn}(\varphi)-\ell_{\bn}(\varphi_0)]\ge0
 \right\}.
\]
{
The conclusion of Theorem~\ref{thm:tail-chart-consistency} is under the
fitted law \(P_{\varphi_0,\bn}\), so it is enough to prove
\(P_{\varphi_0,\bn}(E_{\bn})\to0\).  For the rest of the proof, write
\(\mathbb P_{\varphi_0}:=P_{\varphi_0,\bn}\).}

We apply Lemma~\ref{lem:working-likelihood-exclusion} with
\(K_{\bn}=K\).  By the definition of \(\beta_{\bn}\) in that lemma,
Corollary~\ref{cor:working-glt-fixed-neighborhood-separation} gives
$
        \beta_{\bn}
        \ge c_U\frac{N_{\bn}}{a_{\bn}^2}.
$
Since \(\log\mathfrak L_{\bn}=O(a_{\bn})\) and
\(N_{\bn}/a_{\bn}^3\to\infty\),
\[
        \beta_{\bn}-2\log\mathfrak L_{\bn}
        \ge c_U\frac{N_{\bn}}{a_{\bn}^2}-O(a_{\bn})
        \longrightarrow\infty.
\]
Lemma~\ref{lem:working-likelihood-exclusion} therefore gives
$
        \mathbb P_{\varphi_0}(E_{\bn})\longrightarrow0.
$ Since \(U\) was
arbitrary, \((\widehat\tau_{\bn},\widehat\alpha_{\bn})\to
(\tau_0,\alpha_0)\) in probability.}
\end{proof}}

{Global localization is now complete.  The uniform LAN
remainder bounds use the following Gaussian quadratic-process tools.}

\begingroup
\begin{lemma}
\label{lem:gaussian-quadratic-process-chaining}
Fix \(p\in\mathbb N\).  There is a constant \(C_p<\infty\), depending
only on \(p\), such that the following holds:  Let \(T\subset\R^p\) be
nonempty and finite, let \(u_\star\in\R^p\), and suppose that
\(T\subset B(u_\star,R)\).  Let \(W(u)\), \(u\in T\cup\{u_\star\}\), be
deterministic symmetric \(N\times N\) matrices.  If
\(Z\sim\mathcal N(0,I_N)\), set
\(X(u):=Z^TW(u)Z-\operatorname{tr}W(u)\).  If, for constants \(L,b>0\),
\[
        \begin{aligned}
        \|W(u)-W(v)\|_{\mathrm{HS}}\le L\|u-v\|,
        \qquad\\
        \|W(u)-W(v)\|_2
        \le\min\{L\|u-v\|,b\},
        \qquad u,v\in T\cup\{u_\star\},
        \end{aligned}
\]
then, for every \(x\ge1\),
\begin{align*}
        \mathbb P\biggl\{
        \sup_{u\in T}|X(u)-X(u_\star)|
        &>C_p\biggl[
        LR+b\left\{1+\log\left(1+\frac{LR}{b}\right)\right\}\\
        &\quad+\sqrt{x}LR+xb\biggr]\biggr\}
        \le C_pe^{-x}.
\end{align*}
\end{lemma}

\begin{proof}
For \(u,v\in T\cup\{u_\star\}\),
\[
        X(u)-X(v)
        =Z^T\{W(u)-W(v)\}Z-\operatorname{tr}\{W(u)-W(v)\}.
\]
Because \(Z\) has standard Gaussian coordinates, Theorem~6.1 of
\citep{Dirksen2015}, with its universal rescaling, gives mixed-tail
increments for this difference with
\[
        d_2(u,v):=C\|W(u)-W(v)\|_{\mathrm{HS}},
        \qquad
        d_1(u,v):=C\|W(u)-W(v)\|_2.
\]
Set \(\widetilde T:=T\cup\{u_\star\}\).  Since
\(\operatorname{diam}(\widetilde T)\le2R\), the displayed hypotheses give
\(\operatorname{diam}(\widetilde T,d_2)\le CLR\) and
\(\operatorname{diam}(\widetilde T,d_1)\le Cb\), after absorbing fixed
factors into \(C\).  The Euclidean covering bound
{\cite[Corollary~4.2.13]{Vershynin2018}}, followed by Equation~(2.3) of
\citep{Dirksen2015}, gives
\[
        \gamma_2(\widetilde T,d_2)\le C_pLR,
        \qquad
        \gamma_1(\widetilde T,d_1)
        \le C_pb\left\{1+\log\left(1+\frac{LR}{b}\right)\right\}.
\]
Theorem~3.5 of \citep{Dirksen2015}, anchored at \(X(u_\star)\), now gives
the stated tail inequality, with the fixed factors absorbed into \(C_p\).
\end{proof}

\begin{corollary}
\label{cor:vanishing-gaussian-quadratic-process}
Let \(K\subset\R^p\) be fixed, compact, convex, and contain the origin.
Suppose that \(A_{\bn}(u)\) is a deterministic, continuously
differentiable field of symmetric \(N_{\bn}\times N_{\bn}\) matrices on a
neighborhood of \(K\), with \(A_{\bn}(0)=0\), and that
\[
        \gamma_{\bn}
        :=\max_{1\le j\le p}\sup_{u\in K}
        \|\partial_{u_j}A_{\bn}(u)\|_{\mathrm{HS}}
        \longrightarrow0.
\]
If \(Z_{\bn}\sim\mathcal N(0,I_{N_{\bn}})\), then
\[
        \sup_{u\in K}
        |Z_{\bn}^TA_{\bn}(u)Z_{\bn}
        -\operatorname{tr}A_{\bn}(u)|
        \xrightarrow{\mathbb P}0.
\]
\end{corollary}

\begin{proof}[Proof of Corollary~\ref{cor:vanishing-gaussian-quadratic-process}]
Choose \(R_K<\infty\) with \(K\subset B(0,R_K)\).  The derivative bound and
convexity give, for \(u,v\in K\),
\begin{align*}
        \|A_{\bn}(u)-A_{\bn}(v)\|_{\mathrm{HS}}
        &\le C_p\gamma_{\bn}\|u-v\|,\\
        \|A_{\bn}(u)-A_{\bn}(v)\|_2
        &\le\min\{C_p\gamma_{\bn}\|u-v\|,
        C_p\gamma_{\bn}R_K\}.
\end{align*}
If \(\gamma_{\bn}=0\), the conclusion is immediate.  Otherwise apply
Lemma~\ref{lem:gaussian-quadratic-process-chaining} to each finite subset
of \(K\), with external reference point \(u_{\star}=0\), \(L=C_p\gamma_{\bn}\), and
\(b=C_p\gamma_{\bn}R_K\).  Its centered scalar process is
\[
        X_{\bn}(u):=Z_{\bn}^TA_{\bn}(u)Z_{\bn}-\operatorname{tr}A_{\bn}(u),
\]
and \(A_{\bn}(0)=0\) gives \(X_{\bn}(0)=0\), so its anchored supremum is
the absolute supremum in the conclusion.  For fixed \(\bn\), let \(S_m\) be
a finite \(2^{-m}\)-net of \(K\) and set
\[
        T_m:=S_1\cup\cdots\cup S_m.
\]
Then \(T_m\subset T_{m+1}\), and \(\bigcup_{m\ge1}T_m\) is dense in
\(K\).  Since
\(u\mapsto Z_{\bn}^TA_{\bn}(u)Z_{\bn}-\operatorname{tr}A_{\bn}(u)\)
is continuous,
\[
        \max_{u\in T_m}
        |Z_{\bn}^TA_{\bn}(u)Z_{\bn}-\operatorname{tr}A_{\bn}(u)|
        \uparrow
        \sup_{u\in K}
        |Z_{\bn}^TA_{\bn}(u)Z_{\bn}-\operatorname{tr}A_{\bn}(u)|.
\]
The chaining bound is uniform in \(m\), so continuity from below of
probability extends it to \(K\).  Choose \(x\) first so that its tail
probability is arbitrarily small, and then take \(\bn\) large enough that
the displayed threshold is arbitrarily small.
\end{proof}

\begin{corollary}
\label{cor:smooth-transform-vanishing-gaussian-quadratic-process}
Let \(K\subset\R^p\) be fixed, compact, convex, and contain the origin.
Suppose that \(A_{\bn}(u)\) is a deterministic, continuously
differentiable field of symmetric \(N_{\bn}\times N_{\bn}\) matrices on a
neighborhood of \(K\), with \(A_{\bn}(0)=0\), and that
\[
        \sup_{u\in K}\|A_{\bn}(u)\|_2
        \longrightarrow0,
        \qquad
        \max_{1\le j\le p}\sup_{u\in K}
        \|\partial_{u_j}A_{\bn}(u)\|_{\mathrm{HS}}
        =O(1).
\]
Let \(I\subset\R\) be an open interval containing zero, and let
\(f\in C^1(I)\) satisfy
\[
        f(0)=f'(0)=0.
\]
For a symmetric matrix
\(A=U\operatorname{diag}(\lambda_1,\ldots,\lambda_N)U^T\)
whose spectrum is contained in \(I\), define
\[
        f(A)
        :=
        U\operatorname{diag}
        \{f(\lambda_1),\ldots,f(\lambda_N)\}U^T.
\]
If \(Z_{\bn}\sim\mathcal N(0,I_{N_{\bn}})\), then
\[
        \sup_{u\in K}
        \left|
        Z_{\bn}^Tf(A_{\bn}(u))Z_{\bn}
        -\operatorname{tr}f(A_{\bn}(u))
        \right|
        \xrightarrow{\mathbb P}0.
\]
\end{corollary}

\begin{proof}[
Proof of Corollary~\ref{cor:smooth-transform-vanishing-gaussian-quadratic-process}]
Choose \(\rho>0\) such that \([-\rho,\rho]\subset I\).  Symmetry implies
that every eigenvalue has absolute value at most the operator norm.  Hence,
the uniform operator-norm convergence eventually places every spectrum in
\([-\rho,\rho]\), uniformly over \(K\).  Set
\(B_{\bn}(u):=f(A_{\bn}(u))\), which is symmetric, with \(B_{\bn}(0)=0\).

For a symmetric matrix
\(A=U\operatorname{diag}(\lambda_1,\ldots,\lambda_N)U^T\)
and a symmetric direction \(H\), the Daleckii--Krein divided-difference formula
\cite[Theorem~5.3.1]{Bhatia2007PositiveDefinite}
gives
\[
        \bigl\{U^TDf(A)[H]U\bigr\}_{rs}
        =
        f^{[1]}(\lambda_r,\lambda_s)
        \bigl\{U^THU\bigr\}_{rs},
\]
where
\[
        f^{[1]}(x,y)
        :=
        \begin{cases}
        \dfrac{f(x)-f(y)}{x-y}, & x\ne y,\\[6pt]
        f'(x), & x=y.
        \end{cases}
\]
Consequently,
\[
        \|Df(A)[H]\|_{\mathrm{HS}}
        \le
        \sup_{|x|\le\|A\|_2}|f'(x)|
        \|H\|_{\mathrm{HS}}.
\]
The chain rule therefore yields
\begin{align*}
        \max_{1\le j\le p}\sup_{u\in K}
        \|\partial_{u_j}B_{\bn}(u)\|_{\mathrm{HS}}\le
        \sup_{\displaystyle |x|\le
        \sup_{u\in K}\|A_{\bn}(u)\|_2}
        |f'(x)|
        \max_{1\le j\le p}\sup_{u\in K}
        \|\partial_{u_j}A_{\bn}(u)\|_{\mathrm{HS}}
        \longrightarrow0.
\end{align*}
Indeed, the first factor tends to zero because \(f'\) is continuous and
\(f'(0)=0\), whereas the second factor is \(O(1)\).  Thus \(B_{\bn}\)
satisfies the hypotheses of
Corollary~\ref{cor:vanishing-gaussian-quadratic-process}, which gives the
claim.
\end{proof}

We now introduce the relative covariance increment and its first-order approximation.
Along the local chart, define its increment by
\begin{equation}
        D_{\bn}(u)
        :=V_{\varphi_{\bn}(u),\bn}-I_{N_{\bn}}.
        \label{eq:relative-covariance-increment}
\end{equation}
Thus \(D_{\bn}(0)=0\).  Its first-order term at the origin is
\begin{equation}
        L_{\bn}(u)
        :=\left.\frac{d}{dt}
        D_{\bn}(tu)\right|_{t=0}
        =
        N_{\bn}^{-1/2}
        \{u_\tau I_{N_{\bn}}+u_\alpha R_{\alpha,\bn}\},
        \label{eq:relative-covariance-linear-part}
\end{equation}
where $R_{\alpha,\bn}$ is defined in \eqref{eq:Ralpha}.

Recall from the relative covariance representation above that
\(g_{\bn}(\varphi)=(\tau-\tau_0)-a_{\bn}(\alpha-\alpha_0)\).  Write
\(\varphi_{\bn}(u)=(\tau_{\bn}(u),\alpha_{\bn}(u))\), and fix
\(\rho_{\mathrm{band}}>0\) so that the closed square
\[
 [\tau_0-\rho_{\mathrm{band}},\tau_0+\rho_{\mathrm{band}}]
 \times[\alpha_0-\rho_{\mathrm{band}},\alpha_0+\rho_{\mathrm{band}}]
 \subset\mathcal U
\]
has its \(\alpha\) interval contained in \(\mathcal I_\alpha\).  Define the
set
\[
 \mathcal B_{\bn}(\rho):=
 \left\{u:
 |g_{\bn}\{\varphi_{\bn}(u)\}|
 +a_{\bn}|\alpha_{\bn}(u)-\alpha_0|
 \le\rho\right\},
\]
for which we have
\[
 u\in\mathcal B_{\bn}(\rho)
 \quad\Longleftrightarrow\quad
 \frac{|u_\tau|+a_{\bn}|u_\alpha|}{\sqrt{N_{\bn}}}\le\rho.
\]
\begin{lemma}
\label{lem:first-order-relative-covariance-expansion}
There is a constant \(C>0\) such that, for all sufficiently large sample
sizes and every \(0<\rho\le\rho_{\mathrm{band}}\),
\begin{equation}
        \sup_{u\in\mathcal B_{\bn}(\rho)}\|D_{\bn}(u)\|_2\le C\rho.
        \label{eq:relative-increment-band-operator}
\end{equation}
Moreover,
\begin{equation}
 \lim_{\rho\downarrow0}\limsup_{\bn}
 \sup_{\substack{u\in\mathcal B_{\bn}(\rho)\\u\ne0}}
 \frac{\|D_{\bn}(u)-L_{\bn}(u)\|_{\mathrm{HS}}}{\|u\|}=0,
        \label{eq:relative-increment-shell-linearization}
\end{equation}
and
\[
 \lim_{\rho\downarrow0}\limsup_{\bn}
 \max_{j\in\{\tau,\alpha\}}
 \sup_{u\in\mathcal B_{\bn}(\rho)}
 \|\partial_{u_j}(D_{\bn}-L_{\bn})(u)\|_{\mathrm{HS}}=0.
\]
For \(u,v\in\mathcal B_{\bn}(\rho_{\mathrm{band}})\),
\begin{equation}
 \max_{j\in\{\tau,\alpha\}}
 \|\partial_{u_j}D_{\bn}(u)\|_{\mathrm{HS}}\le C,
 \qquad
 \|D_{\bn}(u)-D_{\bn}(v)\|_{\mathrm{HS}}
 \le C\|u-v\|.
        \label{eq:relative-increment-band-lipschitz}
\end{equation}
\end{lemma}

\begin{proof}[Proof of Lemma~\ref{lem:first-order-relative-covariance-expansion}]
For \(u\in\mathcal B_{\bn}(\rho_{\mathrm{band}})\), we now define a radial
path between \(0\) and \(u\).  We show that, along this path,
\(\frac{d}{dt}D_{\bn}(tu)\) is close to \(L_{\bn}(u)\); integration then
shows that \(D_{\bn}(u)\) is close to its linear approximation
\(L_{\bn}(u)\).
Write
\[
 \varphi_{\bn}(tu)
 =\bigl(\tau_{\bn}(t,u),\alpha_{\bn}(t,u)\bigr),
 \qquad 0\le t\le1.
\]
 Along the path,
\[
 |g_{\bn}\{\varphi_{\bn}(tu)\}|
 +a_{\bn}|\alpha_{\bn}(tu)-\alpha_0|
 =t\left[
 |g_{\bn}\{\varphi_{\bn}(u)\}|
 +a_{\bn}|\alpha_{\bn}(u)-\alpha_0|\right]
 \le\rho_{\mathrm{band}}.
\]
For all sufficiently large sample sizes, \(a_{\bn}\ge1\).  The displayed
identity then gives
\[
 |\tau_{\bn}(t,u)-\tau_0|\le\rho_{\mathrm{band}},
 \qquad
 |\alpha_{\bn}(t,u)-\alpha_0|\le\rho_{\mathrm{band}}.
\]
Thus the path lies in the selected closed square, and its exponent coordinate
lies in \(\mathcal I_\alpha\).  Put
\begin{align*}
 A_{\bn}(t,u)&:=V_{\varphi_{\bn}(tu),\bn}
 =e^{\tau_{\bn}(t,u)-\tau_0}
 V_{(\tau_0,\alpha_{\bn}(t,u)),\bn},\\
 K_{\bn}(t,u)&:=e^{\tau_{\bn}(t,u)-\tau_0}
 \mathcal R_{\bn}\{\alpha_{\bn}(t,u)\}.
\end{align*}
The fundamental theorem of calculus gives the decisive relation
\begin{equation}
 D_{\bn}(u)=\frac{1}{\sqrt{N_{\bn}}}\int_0^1
 \{u_\tau A_{\bn}(t,u)+u_\alpha K_{\bn}(t,u)\}\,dt.
        \label{eq:Dpath}
\end{equation}
Path control gives \(\sup\|A_{\bn}(t,u)\|_2\le C\).  By
Equation~\eqref{eq:tail-alias-global-comparison} and the Toeplitz
quadratic-form bound,
\[
 -Ca_{\bn}V_{(\tau_0,\alpha),\bn}
 \preceq\mathcal R_{\bn}(\alpha)
 \preceq Ca_{\bn}V_{(\tau_0,\alpha),\bn}.
\]
Path control bounds \(V_{(\tau_0,\alpha),\bn}\) on the selected rectangle,
and the exponential factor in \(K_{\bn}(t,u)\) is bounded there.  Thus
\(\sup\|K_{\bn}(t,u)\|_2\le Ca_{\bn}\), and \eqref{eq:Dpath} proves
Equation~\eqref{eq:relative-increment-band-operator}.

For fixed \(\rho\), set
\[
 \omega_{R,\bn}(\rho):=\sup_{\substack{\alpha\in\mathcal I_\alpha\\
 a_{\bn}|\alpha-\alpha_0|\le\rho}}
 \frac{\|\mathcal R_{\bn}(\alpha)-R_{\alpha,\bn}\|_{\mathrm{HS}}}{\sqrt{N_{\bn}}},
 \qquad
 \omega_R(\rho):=\limsup_{\bn}\omega_{R,\bn}(\rho),
\]
and \(\omega_A(\rho):=e^\rho(e^{C\rho}-1)+(e^\rho-1)\).
Lemma~\ref{lem:local-fixed-whitening-score-continuity} gives
\(\omega_R(\rho)\to0\) as \(\rho\downarrow0\), and, for fixed \(\rho\),
{after shrinking \(\rho_{\mathrm{band}}\), if necessary, while
preserving the selected rectangle, we may assume
\(\limsup_{\bn}\omega_{R,\bn}(\rho_{\mathrm{band}})\le1\).}
Moreover,
\(\omega_{R,\bn}(\rho)\le\omega_R(\rho)+{o_{\bn}(1)}\) for each
fixed \(\rho\).  Uniformly over the
band and the radial path,
\[
 \|A_{\bn}-I_{N_{\bn}}\|_2\le\omega_A(\rho)+{o_{\bn}(1)},
 \qquad
 \frac{\|\mathcal R_{\bn}(\alpha_{\bn})-R_{\alpha,\bn}\|_{\mathrm{HS}}}{\sqrt{N_{\bn}}}\le\omega_R(\rho)+{o_{\bn}(1)}.
\]
Since
\[
 K_{\bn}-R_{\alpha,\bn}
 =e^{\tau_{\bn}-\tau_0}
 \{\mathcal R_{\bn}(\alpha_{\bn})-R_{\alpha,\bn}\}
 +(e^{\tau_{\bn}-\tau_0}-1)R_{\alpha,\bn},
\]
and \(\|R_{\alpha,\bn}\|_{\mathrm{HS}}/\sqrt{N_{\bn}}\le C_R\), set
\(\omega_K(\rho):=e^\rho\omega_R(\rho)+(e^\rho-1)C_R\) and
\(\bar\omega(\rho):=\omega_A(\rho)+\omega_K(\rho)\).  These moduli vanish
as \(\rho\downarrow0\).  Subtracting
Equation~\eqref{eq:relative-covariance-linear-part} from \eqref{eq:Dpath}
therefore gives
\[
 \|D_{\bn}(u)-L_{\bn}(u)\|_{\mathrm{HS}}
 \le\{\bar\omega(\rho)+{o_{\bn}(1)}\}\|u\|,
 \qquad u\in\mathcal B_{\bn}(\rho),
\]
which proves Equation~\eqref{eq:relative-increment-shell-linearization}.
Using 
\[
\begin{aligned}
 \partial_{u_\tau}(D_{\bn}-L_{\bn})(u)
 &=
 N_{\bn}^{-1/2}\{A_{\bn}(1,u)-I_{N_{\bn}}\},\\
 \partial_{u_\alpha}(D_{\bn}-L_{\bn})(u)
 &=
 N_{\bn}^{-1/2}\{K_{\bn}(1,u)-R_{\alpha,\bn}\},
\end{aligned}
\]
the bounds established above for \(A_{\bn}-I_{N_{\bn}}\) and
\(K_{\bn}-R_{\alpha,\bn}\) therefore prove the derivative limit.  The
corresponding bounds for \(A_{\bn}\) and \(K_{\bn}\), together with
\(\|R_{\alpha,\bn}\|_{\mathrm{HS}}/\sqrt{N_{\bn}}\le C_R\), also give
\[
 \max_{j\in\{\tau,\alpha\}}
 \|\partial_{u_j}D_{\bn}(u)\|_{\mathrm{HS}}\le C,
\]
which is the first assertion of
Equation~\eqref{eq:relative-increment-band-lipschitz}.

The band
\(\mathcal B_{\bn}(\rho_{\mathrm{band}})\) is convex, so integration along
the segment from \(u\) to \(v\) gives its Lipschitz assertion.
\end{proof}
\begin{corollary}
\label{cor:relative-covariance-trace-comparison}
Uniformly for \(u,v\) in compact sets,
\begin{equation}
        \tr\{L_{\bn}(u)L_{\bn}(v)\}\longrightarrow2u^TJv.
        \label{eq:relative-linear-trace-limit}
\end{equation}

Moreover, since \(J\succ0\), there exist
\(0<\rho_{\mathrm{quad}}\le\rho_{\mathrm{band}}\) and constants
\(0<c<C<\infty\) such that, for all sufficiently large sample sizes,
\begin{equation}
 c\|u\|^2
 \le \operatorname{tr}\{D_{\bn}(u)^2\}
 \le C\|u\|^2,
 \qquad
 u\in\mathcal B_{\bn}(\rho),\quad
 0<\rho\le\rho_{\mathrm{quad}}.
 \label{eq:relative-increment-quadratic-bounds}
\end{equation}
\end{corollary}

\begin{proof}[Proof of Corollary~\ref{cor:relative-covariance-trace-comparison}]
     Using
\[
 L_{\bn}(u)
 =
 N_{\bn}^{-1/2}
 \{u_\tau I_{N_{\bn}}+u_\alpha R_{\alpha,\bn}\},
\]
expand \(\operatorname{tr}\{L_{\bn}(u)L_{\bn}(v)\}\).
Equation~\eqref{eq:Ralpha-trace-limits} then gives
Equation~\eqref{eq:relative-linear-trace-limit}, uniformly for \(u,v\) in
compact sets. 

Since \(L_{\bn}(u)\) is symmetric,
$
 \|L_{\bn}(u)\|_{\mathrm{HS}}^2
 =\operatorname{tr}\{L_{\bn}(u)^2\}.
$
Taking \(v=u\) in
Equation~\eqref{eq:relative-linear-trace-limit}, the convergence is uniform
on the unit sphere.  Since \(J\succ0\), and since \(L_{\bn}\) is linear,
there exist constants \(0<c_L<C_L<\infty\) such that, for all sufficiently
large sample sizes,
\begin{equation}
 c_L\|u\|^2
 \le\|L_{\bn}(u)\|_{\mathrm{HS}}^2
 \le C_L\|u\|^2,
 \qquad u\in\mathbb R^2.
 \label{eq:relative-linear-hs-bounds}
\end{equation}

For \(0<\rho\le\rho_{\mathrm{band}}\), define
\[
 e_{\bn}(\rho):=
 \sup_{\substack{u\in\mathcal B_{\bn}(\rho)\\u\ne0}}
 \frac{\|D_{\bn}(u)-L_{\bn}(u)\|_{\mathrm{HS}}}{\|u\|}.
\]
Equation~\eqref{eq:relative-increment-shell-linearization} gives
$
 \lim_{\rho\downarrow0}\limsup_{\bn}e_{\bn}(\rho)=0.
$

Let \(C_D\) denote the Lipschitz constant in
Equation~\eqref{eq:relative-increment-band-lipschitz}.  Since
\(0\in\mathcal B_{\bn}(\rho_{\mathrm{band}})\) and \(D_{\bn}(0)=0\),
that equation gives
\[
 \|D_{\bn}(u)\|_{\mathrm{HS}}
 \le C_D\|u\|,
 \qquad
 u\in\mathcal B_{\bn}(\rho_{\mathrm{band}}).
\]
Together with
\(\|L_{\bn}(u)\|_{\mathrm{HS}}\le\sqrt{C_L}\|u\|\), the
trace-difference inequality gives
\[
\begin{aligned}
 &\left|
 \tr\{D_{\bn}(u)^2\}
 -\tr\{L_{\bn}(u)^2\}
 \right|\\
 &\qquad\le
 \|D_{\bn}(u)-L_{\bn}(u)\|_{\mathrm{HS}}
 \{\|D_{\bn}(u)\|_{\mathrm{HS}}
   +\|L_{\bn}(u)\|_{\mathrm{HS}}\}\\
 &\qquad\le
 e_{\bn}(\rho)(C_D+\sqrt{C_L})\|u\|^2,
 \qquad
 u\in\mathcal B_{\bn}(\rho)\setminus\{0\}.
\end{aligned}
\]

Since \(e_{\bn}(\rho)\) is nondecreasing in \(\rho\), its ordered limit
allows us to choose
\(0<\rho_{\mathrm{quad}}\le\rho_{\mathrm{band}}\) such that, for all
sufficiently large sample sizes,
\[
 e_{\bn}(\rho)(C_D+\sqrt{C_L})
 \le\frac{c_L}{2},
 \qquad
 0<\rho\le\rho_{\mathrm{quad}}.
\]
The preceding trace-difference bound and
Equation~\eqref{eq:relative-linear-hs-bounds} then give
Equation~\eqref{eq:relative-increment-quadratic-bounds}.
\end{proof}

\begin{corollary}
\label{cor:compact-relative-covariance-expansion}
For every fixed \(M<\infty\),
\begin{align}
        \sup_{\|u\|\le M}\|D_{\bn}(u)\|_2&\longrightarrow0,
        \label{eq:relative-increment-operator-small}\\
        \sup_{\|u\|\le M}
        \|D_{\bn}(u)-L_{\bn}(u)\|_{\mathrm{HS}}&\longrightarrow0,
        \label{eq:relative-increment-hs-linearization}
\end{align}
and
\begin{equation}
\begin{split}
        \max_{j\in\{\tau,\alpha\}}\sup_{\|u\|\le M}
        \|\partial_{u_j}(D_{\bn}-L_{\bn})(u)\|_{\mathrm{HS}}&\to0,\\
        \max_{j\in\{\tau,\alpha\}}\sup_{\|u\|\le M}
        \|\partial_{u_j}D_{\bn}(u)\|_{\mathrm{HS}}&=O(1).
\end{split}
        \label{eq:relative-increment-derivative-linearization}
\end{equation}
\end{corollary}

\begin{proof}[Proof of Corollary~\ref{cor:compact-relative-covariance-expansion}]
For \(\|u\|\le M\), the physical chart quantity defining
\(\mathcal B_{\bn}(\rho)\) equals
\[
 \frac{|u_\tau|+a_{\bn}|u_\alpha|}{\sqrt{N_{\bn}}}
 \le\frac{M(1+a_{\bn})}{\sqrt{N_{\bn}}}\longrightarrow0.
\]
The three assertions therefore follow from
Lemma~\ref{lem:first-order-relative-covariance-expansion}.
\end{proof}

\endgroup
{
Recall that the data whitened under the fitted model are
$
        Z_{\bn}
        :=
        (\Sigma_{\varphi_0,\bn})^{-1/2}Y_{\bn}.
$
Hence, under \(P_{\varphi_0,\bn}\),
$
        Z_{\bn}\sim\mathcal N(0,I_{N_{\bn}}).
$ Define the normalized score and likelihood process by
\begin{equation}
        \mathcal S_{\bn}:=
        \frac{1}{2\sqrt{N_{\bn}}}
        \begin{pmatrix}
        Z_{\bn}^TZ_{\bn}-N_{\bn}\\
        Z_{\bn}^TR_{\alpha,\bn}Z_{\bn}
        -\operatorname{tr}R_{\alpha,\bn}
        \end{pmatrix},
        \qquad
        \Lambda_{\bn}(u):=
        \ell_{\bn}(\varphi_{\bn}(u))
        -\ell_{\bn}(\varphi_0).
\label{eq:supp-compensated-score-and-likelihood}
\end{equation}

\begin{proof}[Proof of Theorem~\ref{thm:compact-lan}]
Fix \(M<\infty\).  We prove
\[
\sup_{\|u\|\le M}\left|
 \Lambda_{\bn}(u)-u^T\mathcal S_{\bn}+\frac{1}{2}u^TJu\right|
 \longrightarrow0
 \quad\text{in }P_{\varphi_0,\bn}\text{-probability}.
\]
Corollary~\ref{cor:compact-relative-covariance-expansion} supplies the
fixed-\(M\) operator, Hilbert--Schmidt, and derivative controls.  The trace
curvature used below is supplied by
Corollary~\ref{cor:relative-covariance-trace-comparison}.
Here \(Z_{\bn}\) is standard Gaussian. Under
{Assumption~\ref{ass:common-leading-tail-misspecification},}
Lemma~\ref{lem:common-leading-tail-mutual-contiguity} transfers this
convergence to \(P^\star_{\bn}\)-probability; that external-law transfer is
the one used by Theorem~\ref{thm:misspecified-tail-mle-transfer}.

The proof has three steps.  First, we express the likelihood increment in
terms of the relative covariance change \(D_{\bn}(u)\).  Second, we show
that all higher-order random terms are negligible uniformly on the ball,
while the deterministic quadratic term converges to
\(-\frac{1}{2}u^TJu\).  Finally, we replace \(D_{\bn}(u)\) by its first-order
approximation \(L_{\bn}(u)\), the quadratic form associated with
\(L_{\bn}(u)\) is exactly \(u^T\mathcal S_{\bn}\).

By the definition of \(D_{\bn}(u)\) in
Equation~\eqref{eq:relative-covariance-increment}, the likelihood
increment has the exact representation
\[
        \Lambda_{\bn}(u)
        =-\frac{1}{2}\log\det\{I+D_{\bn}(u)\}
        -\frac{1}{2}Z_{\bn}^T
        [\{I+D_{\bn}(u)\}^{-1}-I]Z_{\bn}.
\]
Equation~\eqref{eq:relative-increment-operator-small} makes \(I+D_{\bn}(u)\)
invertible eventually, uniformly on the ball.  The log-determinant and
resolvent expansions are therefore uniform there.  Using
\begin{align*}
        \log\det(I+D)
        &=\operatorname{tr}D-\frac{1}{2}\operatorname{tr}(D^2)
        +O\{\|D\|_2\operatorname{tr}(D^2)\},\\
        (I+D)^{-1}
        &=I-D+D^2+E(D),
        \qquad
        E(D)=-D^3(I+D)^{-1},
\end{align*}
we get
\begin{align*}
        \Lambda_{\bn}(u)
        &=\frac{1}{2}\{Z_{\bn}^TD_{\bn}(u)Z_{\bn}
        -\operatorname{tr}D_{\bn}(u)\}
        -\frac{1}{4}\operatorname{tr}\{D_{\bn}(u)^2\}
        \\
        &\quad-\frac{1}{2}\left[
        Z_{\bn}^TD_{\bn}(u)^2Z_{\bn}
        -\operatorname{tr}\{D_{\bn}(u)^2\}
        \right]
        \\
        &\quad-\frac{1}{2}\left[
        Z_{\bn}^TE\{D_{\bn}(u)\}Z_{\bn}
        -\operatorname{tr}E\{D_{\bn}(u)\}
        \right]
        \\
        &\quad
        +O\!\left(
        \|D_{\bn}(u)\|_2
        \operatorname{tr}\{D_{\bn}(u)^2\}
        \right).
\end{align*}
The first line contains the leading centered term and deterministic
curvature, the next two are transformed stochastic remainders, and the final
term is the deterministic Taylor remainder.  We now show that the last three
terms vanish in probability.

We start with the two stochastic remainders (the middle lines).
Apply Corollary~\ref{cor:smooth-transform-vanishing-gaussian-quadratic-process}
with \(A_{\bn}=D_{\bn}\) on the ball.  Its hypotheses are
Equation~\eqref{eq:relative-increment-operator-small} and the \(O(1)\)
derivative assertion in Equation~\eqref{eq:relative-increment-derivative-linearization}.
For the two remainders above, take
\(f_1(x)=x^2\) and \(f_2(x)=-x^3/(1+x)\).  The corollary then shows that
both are uniformly \(o_{P_{\varphi_0,\bn}}(1)\).

For the deterministic Taylor remainder, use that
\[
 |\operatorname{tr}\{D_{\bn}(u)^2\}-\operatorname{tr}\{L_{\bn}(u)^2\}|
 \le\|D_{\bn}(u)-L_{\bn}(u)\|_{\mathrm{HS}}
 \{\|D_{\bn}(u)\|_{\mathrm{HS}}+\|L_{\bn}(u)\|_{\mathrm{HS}}\}.
\]
The trace limit gives \(\sup_{\|u\|\le M}\|L_{\bn}(u)\|_{\mathrm{HS}}=O(1)\);
the Hilbert--Schmidt linearization then gives the same bound for \(D_{\bn}\).
Hence \(\sup_{\|u\|\le M}|\operatorname{tr}D_{\bn}(u)^2-2u^TJu|\to0\).
Together with operator smallness, this makes the
remainder uniform \(o(1)\).

We now replace \(D_{\bn}(u)\) by \(L_{\bn}(u)\) for the remaining term.
Apply Corollary~\ref{cor:vanishing-gaussian-quadratic-process} to
\(A_{\bn}=D_{\bn}-L_{\bn}\).  Here \(A_{\bn}(0)=0\), and
Equation~\eqref{eq:relative-increment-derivative-linearization} supplies
the derivative condition.  Thus its centered form is uniformly
\(o_{P_{\varphi_0,\bn}}(1)\), which replaces the leading centered \(D_{\bn}\) term
by \(L_{\bn}\).  Finally,
\[
 \frac{1}{2}\{Z_{\bn}^TL_{\bn}(u)Z_{\bn}-\operatorname{tr}L_{\bn}(u)\}
 =u^T\mathcal S_{\bn}.
\]
Together with the preceding limit for \(\operatorname{tr}\{D_{\bn}(u)^2\}\),
this proves the result.
\end{proof}
}

Recall from Assumption~\ref{ass:differentiated-spectral-tail-family} that
\(\mathcal U\) is the open parameter neighborhood on which the differentiated
{spectral family is defined. The tail chart in}
Section~\ref{sec:maximum-likelihood} gives
\[
 \Theta_{\mathrm{tail}}=[\tau_-,\tau_+]\times[\alpha_-,\alpha_+]
 \Subset\mathcal U,
 \qquad \varphi_0\in\operatorname{int}(\Theta_{\mathrm{tail}}).
\]
\begin{lemma}
\label{lem:spectral-comparison-localization}
Fix
\(\delta_U\in(0,\operatorname{dist}\{\varphi_0,\partial\Theta_{\mathrm{tail}}\})\)
and define \(U:=\{\varphi:\|\varphi-\varphi_0\|<\delta_U\}\).  Then \(U\)
is open and convex, with
\(\overline U\subset\operatorname{int}(\Theta_{\mathrm{tail}})\).  For
\begin{align}
 E_{\bn,M}(U)
 &:=
 \left\{
 \sup_{\substack{\varphi_{\bn}(u)\in U\\ \|u\|\ge M}}
 \Lambda_{\bn}(u)\ge0
 \right\},
 \label{eq:spectral-comparison-localization-event}
\end{align}
\begin{equation}
 \lim_{M\to\infty}\limsup_{\bn}
 P_{\varphi_0,\bn}\{E_{\bn,M}(U)\}
 =0.
 \label{eq:spectral-comparison-localization-fitted-limit}
\end{equation}
If {Assumption~\ref{ass:common-leading-tail-misspecification} holds,} then
\begin{equation}
 \lim_{M\to\infty}\limsup_{\bn}
 P^\star_{\bn}\{E_{\bn,M}(U)\}
 =0.
 \label{eq:spectral-comparison-localization-limit}
\end{equation}
\end{lemma}

\begin{proof}[Proof of Lemma~\ref{lem:spectral-comparison-localization}]
We first prove the result under \(P_{\varphi_0,\bn}\),
then use Lemma~\ref{lem:common-leading-tail-mutual-contiguity}
to establish the result under \(P^\star_{\bn}\).

The exact likelihood identity is
\begin{align*}
 \Lambda_{\bn}(u)
 &=-\frac{1}{2}\log\det\{I_{N_{\bn}}+D_{\bn}(u)\}
 -\frac{1}{2}Z_{\bn}^T\mathsf W_{\bn}(u)Z_{\bn},\\
 \mathsf W_{\bn}(u)
 &:=\{I_{N_{\bn}}+D_{\bn}(u)\}^{-1}-I_{N_{\bn}},
\end{align*}
where \(Z_{\bn}:=(\Sigma_{\varphi_0,\bn})^{-1/2}Y_{\bn}\sim
\mathcal N(0,I_{N_{\bn}})\) under \(P_{\varphi_0,\bn}\).
By Lemma~\ref{lem:first-order-relative-covariance-expansion} and
Corollary~\ref{cor:relative-covariance-trace-comparison}, there is a fixed
\(\rho>0\) and constants \(0<c<C<\infty\) such that, for all sufficiently
large sample sizes,
\[
 \|D_{\bn}(u)\|_2\le\frac{1}{2},
 \qquad
 c\|u\|^2\le\tr\{D_{\bn}(u)^2\}\le C\|u\|^2,
 \qquad u\in\mathcal B_{\bn}(\rho).
\]
Here and below, \(c\) and \(C\) are positive, independent of \(\bn,u\),
and \(R\), and may change from line to line.  On the fixed interval
\([-1/2,1/2]\), for some \(c>0\),
\[
 \frac{1}{2}\left\{\log(1+x)-\frac{x}{1+x}\right\}\ge cx^2.
\]
Split \(U\) into
\[
 \mathcal L_{\bn}(\rho):=\{\varphi\in U:
 |g_{\bn}(\varphi)|+a_{\bn}|\alpha-\alpha_0|\le\rho\}
 \quad\text{and}\quad
 \mathcal L_{\bn}^{c}(\rho):=U\setminus\mathcal L_{\bn}(\rho).
\]

For \(\varphi\in\mathcal L_{\bn}^{c}(\rho)\), Cauchy--Schwarz gives
\(g_{\bn}(\varphi)^2+(\alpha-\alpha_0)^2>
\rho^2/(1+a_{\bn}^2)\).  Define directly
\[
 \beta_{\bn}:=\inf_{\varphi\in\mathcal L_{\bn}^{c}(\rho)}
 \{-\log\mathcal A(P_{\varphi_0,\bn},P_{\varphi,\bn})\}.
\]
Proposition~\ref{prop:working-glt-hellinger-separation} yields
\[
 \beta_{\bn}\ge cN_{\bn}\frac{\rho^2}{1+a_{\bn}^2+\rho^2}
 \ge c\rho^2N_{\bn}/a_{\bn}^2.
\]
Thus \(\beta_{\bn}-2\log\mathfrak L_{\bn}\to\infty\), since
\(\log\mathfrak L_{\bn}=O(a_{\bn})\) and \(N_{\bn}/a_{\bn}^3\to\infty\).
Because {\(\Delta_{\bn}\)} is nonsingular and
\(\Lambda_{\bn}(u)=\ell_{\bn}(\varphi_{\bn}(u))-\ell_{\bn}(\varphi_0)\),
\[
\begin{aligned}
 &P_{\varphi_0,\bn}\left\{
 \sup_{\varphi_{\bn}(u)\in\mathcal L_{\bn}^{c}(\rho)}
 \Lambda_{\bn}(u)\ge0\right\}\\
 &\qquad=
 P_{\varphi_0,\bn}\left\{
 \sup_{\varphi\in\mathcal L_{\bn}^{c}(\rho)}
 [\ell_{\bn}(\varphi)-\ell_{\bn}(\varphi_0)]\ge0\right\}
 \longrightarrow0
\end{aligned}
\]
by Lemma~\ref{lem:working-likelihood-exclusion}.  For \(R\ge M\), let
\[
 \mathcal S_{\bn}(R,\rho):=\{u:\varphi_{\bn}(u)\in
 \mathcal L_{\bn}(\rho),\ R\le\|u\|<2R\}.
\]
Set
\[
 \mathbb G_{\bn}(u):=-\{Z_{\bn}^T\mathsf W_{\bn}(u)Z_{\bn}
 -\tr\mathsf W_{\bn}(u)\}/2,
 \qquad
 \Lambda_{\bn}(u)=\mathbb E_{P_{\varphi_0,\bn}}\Lambda_{\bn}(u)
 +\mathbb G_{\bn}(u).
\]
The first term supplies the deterministic drift on
\(\mathcal S_{\bn}(R,\rho)\), and \(\mathbb G_{\bn}\) is the centered
fluctuation.

For \(\varphi_{\bn}(u)\in\mathcal L_{\bn}(\rho)\), the definitions of
\(\mathcal L_{\bn}(\rho)\) and \(\mathcal B_{\bn}(\rho)\) give
\(u\in\mathcal B_{\bn}(\rho)\).  The preceding bounds and the eigenvalue
identity give the drift
\[
\begin{aligned}
 \mathbb E_{P_{\varphi_0,\bn}}\Lambda_{\bn}(u)
 &=-\frac{1}{2}\sum_j\left\{\log(1+d_j)-\frac{d_j}{1+d_j}\right\}\\
 &\le-c\tr\{D_{\bn}(u)^2\}
 \le-cR^2,
 \qquad u\in\mathcal S_{\bn}(R,\rho).
\end{aligned}
\]
where \(d_j\) are the eigenvalues of \(D_{\bn}(u)\).
On \(\mathcal B_{\bn}(\rho)\),
\[
 \mathsf W_{\bn}(u)=-(I+D_{\bn}(u))^{-1}D_{\bn}(u),
 \qquad
 \|(I+D_{\bn}(u))^{-1}\|_2\le2.
\]
The trace upper bound supplies Hilbert--Schmidt control, and the operator
bound supplies the second estimate:
\[
 \|\mathsf W_{\bn}(u)\|_{\mathrm{HS}}
 \le2\|D_{\bn}(u)\|_{\mathrm{HS}}
 \le C\|u\|,
 \qquad
 \|\mathsf W_{\bn}(u)\|_2
 \le2\|D_{\bn}(u)\|_2
 \le C\rho.
\]
For \(u,v\in\mathcal S_{\bn}(R,\rho)\), use
\[
 \mathsf W_{\bn}(u)-\mathsf W_{\bn}(v)
 =-(I+D_{\bn}(u))^{-1}\{D_{\bn}(u)-D_{\bn}(v)\}
 (I+D_{\bn}(v))^{-1}.
\]

Since \(\|(I+D_{\bn}(u))^{-1}\|_2\le2\),
Equation~\eqref{eq:relative-increment-band-lipschitz} gives
\[
\begin{aligned}
 \|\mathsf W_{\bn}(u)-\mathsf W_{\bn}(v)\|_{\mathrm{HS}}
 &\le
 4\|D_{\bn}(u)-D_{\bn}(v)\|_{\mathrm{HS}}
 \le C\|u-v\|,\\
 \|\mathsf W_{\bn}(u)-\mathsf W_{\bn}(v)\|_2
 &\le
 4\|D_{\bn}(u)-D_{\bn}(v)\|_2\\
 &\le
 4\|D_{\bn}(u)-D_{\bn}(v)\|_{\mathrm{HS}}
 \le C\|u-v\|.
\end{aligned}
\]
Since
\(\|\mathsf W_{\bn}(u)-\mathsf W_{\bn}(v)\|_2
\le\|\mathsf W_{\bn}(u)\|_2+\|\mathsf W_{\bn}(v)\|_2\le C\rho\),
we have
\[
 \|\mathsf W_{\bn}(u)-\mathsf W_{\bn}(v)\|_{\mathrm{HS}}
 \le C\|u-v\|, \qquad
 \|\mathsf W_{\bn}(u)-\mathsf W_{\bn}(v)\|_2
 \le\min\{C\|u-v\|,C\rho\}.
\]
We now apply
Lemma~\ref{lem:gaussian-quadratic-process-chaining}.  The lemma is stated
for finite index sets, so fix an arbitrary finite set
\(T\subset\mathcal S_{\bn}(R,\rho)\).  We first derive a bound that is
uniform in \(T\) and then pass to the full shell by continuity.

The two bounds in the preceding display verify the required increment
conditions when \(u,v\in T\).  It remains to verify them when one point is
the external anchor \(0\).  Because \(D_{\bn}(0)=0\), we have
\(\mathsf W_{\bn}(0)=0\) and \(\mathbb G_{\bn}(0)=0\).  Hence, for
\(u\in T\),
\[
\begin{aligned}
 \|\mathsf W_{\bn}(u)-\mathsf W_{\bn}(0)\|_{\mathrm{HS}}
 &=\|\mathsf W_{\bn}(u)\|_{\mathrm{HS}}
 \le C\|u\|
 =C\|u-0\|,\\
 \|\mathsf W_{\bn}(u)-\mathsf W_{\bn}(0)\|_2
 &=\|\mathsf W_{\bn}(u)\|_2\\
 &\le
 \min\{\|\mathsf W_{\bn}(u)\|_{\mathrm{HS}},C\rho\}
 \le\min\{C\|u-0\|,C\rho\}.
\end{aligned}
\]
The same estimates hold with \(u\) and \(0\) interchanged, while the case
in which both points equal \(0\) is immediate.  Consequently, for all
\(u,v\in T\cup\{0\}\),
\[
 \|\mathsf W_{\bn}(u)-\mathsf W_{\bn}(v)\|_{\mathrm{HS}}
 \le C\|u-v\|,
 \qquad
 \|\mathsf W_{\bn}(u)-\mathsf W_{\bn}(v)\|_2
 \le\min\{C\|u-v\|,C\rho\}.
\]
Thus the two increment hypotheses of
Lemma~\ref{lem:gaussian-quadratic-process-chaining} hold on
\(T\cup\{0\}\).  Moreover,
$
 T\subset\mathcal S_{\bn}(R,\rho)\subset B(0,2R).
$
We may therefore apply the lemma with external anchor \(u_{\star}=0\), radius \(2R\),
\(L=C\), and \(b=C\rho\).  We obtain, uniformly over finite
\(T\subset\mathcal S_{\bn}(R,\rho)\), for every \(x\ge1\),
\[
 P_{\varphi_0,\bn}\left\{
 \sup_{u\in T}|\mathbb G_{\bn}(u)|
 >
 C\left[R+\rho\{1+\log(1+R/\rho)\}
 +\sqrt{x}R+x\rho\right]\right\}
 \le Ce^{-x}.
\]

If \(\mathcal S_{\bn}(R,\rho)\neq\varnothing\), choose increasing finite
sets \(T_m\subset\mathcal S_{\bn}(R,\rho)\) with dense union.  By
continuity,
\[
 \sup_{u\in T_m}|\mathbb G_{\bn}(u)|
 \uparrow
 \sup_{u\in\mathcal S_{\bn}(R,\rho)}
 |\mathbb G_{\bn}(u)|.
\]
Applying the preceding bound to \(T_m\) and letting \(m\to\infty\) gives
the same bound with \(T=\mathcal S_{\bn}(R,\rho)\).  For an empty shell,
the bound is trivial.
After fixing \(\rho\),
choose \(\epsilon>0\) so that
\(C(\sqrt{\epsilon}+\epsilon\rho)<c/2\).  Then choose \(M\) so that
\(\epsilon M^2\ge1\) and
\[
 C[R+\rho\{1+\log(1+R/\rho)\}]\le cR^2/2,
 \qquad R\ge M.
\]
Finally take \(\bn\) sufficiently large; the resulting threshold is
independent of \(R\).  With \(x=\epsilon R^2\), the stochastic bound is
below the drift for every \(R\ge M\), so
\begin{equation}
 P_{\varphi_0,\bn}\left\{
 \sup_{u\in\mathcal S_{\bn}(R,\rho)}\Lambda_{\bn}(u)\ge0\right\}
 \le Ce^{-\epsilon R^2}.
        \label{eq:local-likelihood-shell-bound}
\end{equation}
Every \(u\) in the local band with \(\|u\|\ge M\) belongs to one of the
half-open shells \(\mathcal S_{\bn}(2^kM,\rho)\), \(k\ge0\).  Thus the
exterior event and the dyadic shells satisfy
\[
 E_{\bn,M}(U)\subseteq
 \left\{\sup_{\varphi_{\bn}(u)\in\mathcal L_{\bn}^{c}(\rho)}
 \Lambda_{\bn}(u)\ge0\right\}
 \cup\bigcup_{k\ge0}
 \left\{\sup_{u\in\mathcal S_{\bn}(2^kM,\rho)}
 \Lambda_{\bn}(u)\ge0\right\}.
\]
Consequently,
\[
 P_{\varphi_0,\bn}\{E_{\bn,M}(U)\}
 \le o(1)+\sum_{k\ge0}Ce^{-\epsilon4^kM^2},
 \]
which tends to zero as \(M\to\infty\).  Hence
\(\lim_{M\to\infty}\limsup_{\bn}
P_{\varphi_0,\bn}\{E_{\bn,M}(U)\}=0\).

{
For the external law, define
\[
 q_0(M):=\limsup_{\bn}P_{\varphi_0,\bn}\{E_{\bn,M}(U)\},
 \qquad
 q_\star(M):=\limsup_{\bn}P^\star_{\bn}\{E_{\bn,M}(U)\}.
\]
Because \(E_{\bn,M}(U)\) decreases in \(M\), both functions are
nonincreasing.  We have already proved \(q_0(M)\downarrow0\).  Suppose,
for contradiction, that
\(q_\star(M)\downarrow q_\star^\infty>0\).  Put
\(\delta:=q_\star^\infty/2\).  Choose an increasing sequence
\(M_k\uparrow\infty\) such that \(q_0(M_k)\le k^{-1}\).  Monotonicity gives
\(q_\star(M_k)\ge q_\star^\infty=2\delta\) for every \(k\).

For each \(k\), the definition of limsup implies that there are arbitrarily
large sampling indices \(\bn\) with
\(P^\star_{\bn}\{E_{\bn,M_k}(U)\}\ge\delta\).  On the other hand,
\(q_0(M_k)\le k^{-1}\) implies that, for all sufficiently large \(\bn\),
\(P_{\varphi_0,\bn}\{E_{\bn,M_k}(U)\}\le2/k\).  We may therefore choose
an increasing sequence \(\bn_k\) such that simultaneously
\[
 P^\star_{\bn_k}\{E_{\bn_k,M_k}(U)\}\ge\delta,
 \qquad
 P_{\varphi_0,\bn_k}\{E_{\bn_k,M_k}(U)\}\le\frac{2}k.
\]
The second probability tends to zero.  Mutual contiguity from
Lemma~\ref{lem:common-leading-tail-mutual-contiguity}, applied to the event
sequence \(E_{\bn_k,M_k}(U)\), therefore forces the first probability to
tend to zero as well, contradicting its lower bound by \(\delta\).  Hence
\(q_\star(M)\downarrow0\), which proves
Equation~\eqref{eq:spectral-comparison-localization-limit}.}
\end{proof}

\begin{proof}[Proof of Theorem~\ref{thm:working-tail-mle-normality}]
\begingroup We first prove the working-model statement by
localizing the restricted maximizer and comparing it with the compact
quadratic expansion. The external-law score and argmax transfer is proved
separately in the proof of
Theorem~\ref{thm:misspecified-tail-mle-transfer} below.\endgroup

First, we show that $
        \widehat u_{\bn}
        :=\Delta_{\bn}^{-1}
        (\widehat\varphi_{\bn}-\varphi_0),
$ is bounded in probability. We then show that \(\widehat u_{\bn}\) converges to the solution of the quadratic
approximation in Corollary~\ref{cor:local-likelihood-geometry}.

Theorem~\ref{thm:tail-chart-consistency} gives
\(\widehat\varphi_{\bn}\to\varphi_0\) in probability. {Choose the
neighborhood \(U\) in Lemma~\ref{lem:spectral-comparison-localization}.
Because \(\widehat\varphi_{\bn}\) maximizes the likelihood,
\[
        \ell_{\bn}(\widehat\varphi_{\bn})
        -\ell_{\bn}(\varphi_0)\ge0
\]
Equation~\eqref{eq:spectral-comparison-localization-event} then gives
$
 \{\|\widehat u_{\bn}\|\ge M\}
 \cap
 \{\widehat\varphi_{\bn}\in U\}
 \subseteq E_{\bn,M}(U).
$}
Consequently, for every fixed \(M>0\), since \(\widehat\varphi_{\bn}\to\varphi_0\) in probability
\begin{align*}
 \limsup_{\bn}
 P_{\varphi_0,\bn}\{\|\widehat u_{\bn}\|\ge M\}
 &\le
 \limsup_{\bn}
 P_{\varphi_0,\bn}\{\widehat\varphi_{\bn}\notin U\}
 +
 \limsup_{\bn}
 P_{\varphi_0,\bn}\{E_{\bn,M}(U)\}\\
 &=
 \limsup_{\bn}
 P_{\varphi_0,\bn}\{E_{\bn,M}(U)\}.
\end{align*}
  Letting \(M\to\infty\)
and applying
Equation~\eqref{eq:spectral-comparison-localization-fitted-limit} now gives
\[
 \lim_{M\to\infty}\limsup_{\bn}
 P_{\varphi_0,\bn}\{\|\widehat u_{\bn}\|\ge M\}=0.
\]
Thus
$
        \widehat u_{\bn}=O_{\mathbb P}(1).
$

Write
$
        U_{\bn}
        :=\{u:\varphi_{\bn}(u)\in\Theta_{\mathrm{tail}}\}.
$
Since {\(\Delta_{\bn}\)} is nonsingular, the map
\(u\mapsto\varphi_{\bn}(u)=\varphi_0+\Delta_{\bn}u\)
is a bijection from \(U_{\bn}\) onto
\(\Theta_{\mathrm{tail}}\). 
The definition of \(\widehat u_{\bn}\) therefore gives
\[
        \widehat u_{\bn}
        \in\operatorname*{arg\,max}_{u\in U_{\bn}}
        \Lambda_{\bn}(u)
        =
        \operatorname*{arg\,max}_{u\in U_{\bn}}
        \left\{
        \ell_{\bn}(\varphi_{\bn}(u))
        -\ell_{\bn}(\varphi_0)
        \right\}.
\]
By Theorem~\ref{thm:compact-lan}, the quadratic approximation to
\(\Lambda_{\bn}\) is
\[
        {\mathcal Q_{\bn}^{\mathrm{LAN}}(u)}
        :=u^T\mathcal S_{\bn}-\frac{1}{2}u^TJu.
\]
Its unconstrained maximizer is
\(u_{\bn}^{\mathrm q}:=J^{-1}\mathcal S_{\bn}\), as in
Corollary~\ref{cor:local-likelihood-geometry}.
Theorem~\ref{thm:tail-exponent-gradient-information} gives
\(\mathcal S_{\bn}\Rightarrow\mathcal N(0,J)\), and hence
\(u_{\bn}^{\mathrm q}=O_{\mathbb P}(1)\).

Because
$
        d_0
        :=\operatorname{dist}
          \{\varphi_0,\partial\Theta_{\mathrm{tail}}\}>0,
$
the event \(u_{\bn}^{\mathrm q}\notin U_{\bn}\) implies
\(\|\Delta_{\bn}u_{\bn}^{\mathrm q}\|\ge d_0\).  Hence
\[
 P_{\varphi_0,\bn}\{u_{\bn}^{\mathrm q}\notin U_{\bn}\}
 \le
 P_{\varphi_0,\bn}\{
   \|\Delta_{\bn}u_{\bn}^{\mathrm q}\|\ge d_0
 \}
 \longrightarrow0.
\]
Thus the unconstrained maximizer of the quadratic approximation is feasible
with probability tending to one.  Since
\(\widehat u_{\bn}=O_{\mathbb P}(1)\) and
$
        \widehat u_{\bn}
        \in\operatorname*{arg\,max}_{u\in U_{\bn}}\Lambda_{\bn}(u),
$
Corollary~\ref{cor:local-likelihood-geometry} gives
\[
        \widehat u_{\bn}
        =J^{-1}\mathcal S_{\bn}+o_{\mathbb P}(1).
\]
Finally, Theorem~\ref{thm:tail-exponent-gradient-information} and
Slutsky's theorem yield
\[
        \widehat u_{\bn}
        \Rightarrow\mathcal N(0,J^{-1}).
\]
\begingroup The corresponding external-law conclusion does not
follow from contiguity alone; it is established by the centered-score and
argmax bridge in the proof of
Theorem~\ref{thm:misspecified-tail-mle-transfer} below.\endgroup

\end{proof}
We next justify the replacement of the logarithmic tail coordinate by the raw
tail constant in
Equation~\eqref{eq:working-tail-mle-readable-limit}.  Recall that
$
        \widehat c_{\infty,\bn}:=e^{\widehat\tau_{\bn}},
        c_{\infty,0}:=e^{\tau_0}.
$
The representation established above gives
\[
        \widehat\tau_{\bn}-\tau_0
        =O_{\mathbb P}
          \left(\frac{a_{\bn}}{\sqrt{N_{\bn}}}\right).
\]
{
More generally, on the isotropic grid in any fixed dimension \(d\ge1\),
\(N_{\bn}=(L/h+1)^d\asymp h^{-d}\) and
\(a_{\bn}=2\log(1/h)\).  Hence
\[
        \frac{a_{\bn}^2}{\sqrt{N_{\bn}}}
        =O\!\left(h^{d/2}\log^2(1/h)\right)
        \longrightarrow0.
\]
Taylor expansion of the exponential map therefore gives}
\[
        \sqrt{N_{\bn}}
        \left\{
        \frac{\widehat c_{\infty,\bn}-c_{\infty,0}}{c_{\infty,0}}
        -(\widehat\tau_{\bn}-\tau_0)
        \right\}
        =o_{\mathbb P}(1).
\]
Thus replacing \(\widehat\tau_{\bn}-\tau_0\) by the relative error in
\(\widehat c_{\infty,\bn}\) does not affect the limit in
Equation~\eqref{eq:working-tail-mle-readable-limit}.

\begin{proof}[Proof of Corollary~\ref{cor:cubic-grid-componentwise-rates}]
Set
\(U_{\bn}:=\Delta_{\bn}^{-1}(\widehat\varphi_{\bn}-\varphi_0)
=(U_{\tau,\bn},U_{\alpha,\bn})^T\).  By
Theorem~\ref{thm:working-tail-mle-normality},
\(U_{\bn}\Rightarrow\mathcal N(0,J^{-1})\).  The definition of
\(\Delta_{\bn}\) gives the two coordinate identities
\[
 \sqrt{N_{\bn}}(\widehat\alpha_{\bn}-\alpha_0)=U_{\alpha,\bn},
 \qquad
 \sqrt{N_{\bn}}(\widehat\tau_{\bn}-\tau_0)
 =U_{\tau,\bn}+a_{\bn}U_{\alpha,\bn}.
\]
On an isotropically refined cubic grid, \(N_{\bn}=n^d\) and
\(h=L/(n-1)\), hence
\(a_{\bn}=2\log((n-1)/L)=(2/d)\log N_{\bn}-2\log L+o(1)\).  In particular,
\(a_{\bn}\to\infty\) and \(d a_{\bn}/(2\log N_{\bn})\to1\).  Since
\(U_{\bn}=O_{\mathbb P}(1)\), the preceding coordinate identity therefore
implies
\[
 \frac{d\sqrt{N_{\bn}}}{2\log N_{\bn}}
 (\widehat\tau_{\bn}-\tau_0)
 =U_{\alpha,\bn}+o_{\mathbb P}(1).
\]
It also gives
\(\widehat\tau_{\bn}-\tau_0=O_{\mathbb P}(a_{\bn}/\sqrt{N_{\bn}})\).
Taylor's theorem for the exponential map then yields
\((\widehat c_{\infty,\bn}-c_{\infty,0})/c_{\infty,0}
=(\widehat\tau_{\bn}-\tau_0)
+O_{\mathbb P}(a_{\bn}^2/N_{\bn})\).  Because
\(a_{\bn}/\sqrt{N_{\bn}}\to0\), this remainder becomes
\(o_{\mathbb P}(1)\) after multiplication by
\(d\sqrt{N_{\bn}}/(2\log N_{\bn})\asymp\sqrt{N_{\bn}}/a_{\bn}\).
Thus the normalized relative error in \(c_\infty\) is also
\(U_{\alpha,\bn}+o_{\mathbb P}(1)\).  Combining this with
\(\sqrt{N_{\bn}}(\widehat\alpha_{\bn}-\alpha_0)=U_{\alpha,\bn}\) and
Slutsky's theorem gives the asserted rank-one limit.  The common limiting
variance is
\((J^{-1})_{22}=2/(m_2-m_1^2)=2/v_{\alpha_0}=\sigma_{\alpha_0}^2\).
\end{proof}

\begin{proof}[Proof of Corollary~\ref{cor:known-tail-parameter-mles}]
\noindent\emph{(1) Known exponent \(\alpha_0\) (scale assertion).}\par
At \(\alpha=\alpha_0\), the exact fitted scale identity is
\[
 \Sigma_{(\tau,\alpha_0),\bn}
 =e^{\tau-\tau_0}\Sigma_{\varphi_0,\bn}.
\]
Thus, with \(\mathsf q_{\bn}^{(\alpha_0)}:=Y_{\bn}^T\Sigma_{\varphi_0,\bn}^{-1}Y_{\bn}\),
\[
 \ell_{\bn}(\tau,\alpha_0)-\ell_{\bn}(\tau_0,\alpha_0)
 =-\frac{N_{\bn}}{2}(\tau-\tau_0)
 -\frac{\mathsf q_{\bn}^{(\alpha_0)}}{2}\{e^{-(\tau-\tau_0)}-1\}.
\]
The unrestricted maximizer is \(\tau_0+\log(\mathsf q_{\bn}^{(\alpha_0)}/N_{\bn})\). Under
\(P_{\varphi_0,\bn}\), \(\mathsf q_{\bn}^{(\alpha_0)}\sim\chi^2_{N_{\bn}}\), and the
interiority of \(\tau_0\) makes the constrained and unrestricted maximizers
agree with probability tending to one. The chi-square central limit theorem
and the logarithmic delta method therefore give
\(\sqrt{N_{\bn}}(\widehat\tau_{\bn}^{(\alpha_0)}-\tau_0)
\Rightarrow\mathcal N(0,2)\).

\par\medskip\noindent\emph{(2) Known coefficient \(\tau_0\) (exponent assertion).}\par
Write \(\mathcal S_{\bn}=(\mathcal S_{\tau,\bn},\mathcal S_{\alpha,\bn})^T\)
for the score from Theorem~\ref{thm:tail-exponent-gradient-information}.
For every \(\varepsilon>0\), the separation established in the proof of
Theorem~\ref{thm:tail-chart-consistency} rules out, with probability tending
to one, a maximizer of \(\ell_{\bn}(\tau_0,\cdot)\) satisfying
\(\lvert\alpha-\alpha_0\rvert\ge\varepsilon\).
Thus \(\widehat\alpha_{\bn}^{(\tau_0)}\to\alpha_0\) in probability.  

For \(t\in\mathbb R\), write
\[
        \alpha_{\bn}(t)
        :=\alpha_0+\frac{t}{\sqrt{N_{\bn}}a_{\bn}},
        \qquad
        u_{\bn}(t):=
        \begin{pmatrix}
        -t\\ t/a_{\bn}
        \end{pmatrix}.
\]
{Equation~\eqref{eq:tail-rate-matrix}} gives the exact map
\[
        \varphi_0+\Delta_{\bn}u_{\bn}(t)
        =(\tau_0,\alpha_{\bn}(t)).
\]
Consequently,
\(\widehat t_{\bn}:=\sqrt{N_{\bn}}a_{\bn}
(\widehat\alpha_{\bn}^{(\tau_0)}-\alpha_0)\) maximizes
\(\Lambda_{\bn}\{u_{\bn}(t)\}\) over its feasible values.

Fix the neighborhood \(U\) from
Lemma~\ref{lem:spectral-comparison-localization}.  Since
\(\|u_{\bn}(t)\|=\lvert t\rvert\sqrt{1+a_{\bn}^{-2}}\), the likelihood
increment at the maximizer is nonnegative, and
\[
 \left\{
 \lvert\widehat t_{\bn}\rvert\sqrt{1+a_{\bn}^{-2}}\ge M
 \right\}
 \cap
 \left\{
 (\tau_0,\widehat\alpha_{\bn}^{(\tau_0)})\in U
 \right\}
 \subseteq E_{\bn,M}(U).
\]
The consistency above and
Equation~\eqref{eq:spectral-comparison-localization-fitted-limit} show that
\(\widehat t_{\bn}\) is bounded in probability.

For every \(T<\infty\), Theorem~\ref{thm:compact-lan} and the displayed map
give, uniformly for \(\lvert t\rvert\le T\),
\[
 \Lambda_{\bn}\{u_{\bn}(t)\}
 =
 -t\mathcal S_{\tau,\bn}
 -\frac{t^2}{4}
 +o_{\mathbb P}(1).
\]
Indeed,
\[
 u_{\bn}(t)^T\mathcal S_{\bn}
 =-t\mathcal S_{\tau,\bn}
   +\frac{t}{a_{\bn}}\mathcal S_{\alpha,\bn},
 \qquad
 \frac{1}{2}u_{\bn}(t)^TJu_{\bn}(t)
 =\frac{t^2}{4}
 \left(1-\frac{2m_1}{a_{\bn}}+\frac{m_2}{a_{\bn}^2}\right).
\]
Here \(\mathcal S_{\bn}=O_{\mathbb P}(1)\), and the quadratic on the
right-hand side is uniquely maximized at \(-2\mathcal S_{\tau,\bn}\).

Let \(d_\alpha:=\operatorname{dist}\{\alpha_0,\{\alpha_-,\alpha_+\}\}>0\).
Because \(\sqrt{N_{\bn}}a_{\bn}\to\infty\) and
\(\mathcal S_{\tau,\bn}=O_{\mathbb P}(1)\),
\[
 P_{\varphi_0,\bn}\!\left\{
 \alpha_{\bn}(-2\mathcal S_{\tau,\bn})
 \notin[\alpha_-,\alpha_+]
 \right\}
 \le
 P_{\varphi_0,\bn}\!\left\{
 2\lvert\mathcal S_{\tau,\bn}\rvert
 \ge d_\alpha\sqrt{N_{\bn}}a_{\bn}
 \right\}
 \longrightarrow0.
\]
{Thus the quadratic maximizer is feasible with probability tending
to one.  Choose \(T<\infty\) so that, with probability arbitrarily close to
one, both \(\widehat t_{\bn}\) and \(-2\mathcal S_{\tau,\bn}\) belong to
\([-T,T]\) and the latter is feasible.  On this event, set
\[
 r_{\bn,T}(t)
 :=\Lambda_{\bn}\{u_{\bn}(t)\}
 +t\mathcal S_{\tau,\bn}+\frac{t^2}{4},
 \qquad |t|\le T.
\]
The uniform approximation gives \(\sup_{|t|\le T}|r_{\bn,T}(t)|
=o_{\mathbb P}(1)\).  Comparing the actual and quadratic maximizers gives
\[
 \frac{1}{4}\{\widehat t_{\bn}+2\mathcal S_{\tau,\bn}\}^2
 \le2\sup_{|t|\le T}|r_{\bn,T}(t)|
 =o_{\mathbb P}(1).
\]
Hence \(\widehat t_{\bn}=-2\mathcal S_{\tau,\bn}+o_{\mathbb P}(1).\)}
Theorem~\ref{thm:tail-exponent-gradient-information} gives
\(\mathcal S_{\tau,\bn}\Rightarrow\mathcal N(0,J_{11})
=\mathcal N(0,1/2)\).  Slutsky's theorem therefore yields
\(\widehat t_{\bn}\Rightarrow\mathcal N(0,2)\), which is
Equation~\eqref{eq:known-tail-coefficient-alpha-mle}.
\end{proof}
\begingroup
\begin{proof}[proof of Theorem~\ref{thm:misspecified-tail-mle-transfer}]
Write \(\widehat u_{\bn}:=\Delta_{\bn}^{-1}
(\widehat\varphi_{\bn}-\varphi_0)\).
The score transfer is not a consequence of contiguity alone. Apply
Lemma~\ref{lem:centered-quadratic-form-transfer} with
\(A_{1,\bn}=I_{N_{\bn}}\) and \(A_{2,\bn}=R_{\alpha,\bn}\). The trace
and operator limits in Equation~\eqref{eq:Ralpha-trace-limits} give
\(\Gamma=2J\), and therefore
\[\mathcal S_{\bn}\Rightarrow\mathcal N(0,J)
\quad\text{under \(P^\star_{\bn}\).}\]

Mutual contiguity is used only for vanishing-probability statements: tail-chart
consistency, the compact working-likelihood quasi-LAN remainder in
Theorem~\ref{thm:compact-lan}, and localization. On those events the
working-model quadratic argmax argument yields
\[\widehat u_{\bn}=J^{-1}\mathcal S_{\bn}+o_{\mathbb P}(1)
\quad\text{under \(P^\star_{\bn}\)},\]
which, with the score limit, proves the joint conclusion.

For known \(\alpha_0\), put \(\mathsf q_{\bn}^{(\alpha_0)}:=Z_{\bn}^TZ_{\bn}\).
The \(A_{1,\bn}=I\) transfer gives
\[
 \frac{\mathsf q_{\bn}^{(\alpha_0)}-N_{\bn}}{\sqrt{N_{\bn}}}
 \Rightarrow\mathcal N(0,2)
 \quad\text{under \(P^\star_{\bn}\)}.
\]
The exact unrestricted maximizer is
\(\tau_0+\log\{\mathsf q_{\bn}^{(\alpha_0)}/N_{\bn}\}\). Interiority
makes it feasible with probability tending to one, and the logarithmic delta
method gives the stated scale limit. For known \(\tau_0\), transferred restricted
consistency and localization permit the compact quasi-LAN expansion along
\(u_{\bn}(t)=(-t,t/a_{\bn})\), uniformly for \(t\) in every compact set:
\[
 \Lambda_{\bn}\{u_{\bn}(t)\}=-t\mathcal S_{\tau,\bn}-\frac{t^2}{4}
 +o_{\mathbb P}(1)\quad\text{under \(P^\star_{\bn}\)}.
\]
Its one-dimensional argmax gives
\[
 \sqrt{N_{\bn}}a_{\bn}
 \bigl(\widehat\alpha_{\bn}^{(\tau_0)}-\alpha_0\bigr)
 =-2\mathcal S_{\tau,\bn}+o_{\mathbb P}(1)
 \quad\text{under \(P^\star_{\bn}\)}.
\]
It therefore converges in distribution to \(\mathcal N(0,2)\).
Thus both restricted conclusions hold under \(P^\star_{\bn}\).
\end{proof}
\endgroup
\endgroup

\section{Covariance-family verification}
\label{app:wendland-differentiated-tail}

{This appendix verifies the named covariance families against
Assumption~\ref{ass:block-tail-detailed} and
Assumption~\ref{ass:differentiated-spectral-tail-family}: the former
supplies the scalar envelope and unit-tail limit, and the latter the
scale, remainder, and derivative interface.}

\subsection*{{Mat\'ern verification}}

{We first verify the fixed-range Mat\'ern fitted chart. The proof identifies the normalized spectrum, compares it with \((1+\|\bomega\|^2)^{-\alpha}\), and bounds the resulting remainder and its \(\alpha\)-derivative.}

\begin{proof}[Proof of Proposition~\ref{prop:matern-tail-family-verification}]
{We first establish the uniform scalar envelope and unit-tail limit, then identify the reference remainder and its \(\alpha\)-derivative.}
Fix \(\lambda=\log\kappa\).  Write the Mat\'ern family in the tail
coordinates \(\alpha=\nu+d/2\) and \(\tau=\log c_{\infty,\mathrm M}\), and
let \(\mathcal K_\alpha\Subset(d/2,\infty)\). The fitted normalized
spectrum is
$
        \overline S_{\alpha}(\bomega)
        =\bigl(e^{2\lambda}+\|\bomega\|^2\bigr)^{-\alpha}.
$

We first verify the scalar envelope and unit-tail limit in
Equations~\eqref{eq:block-envelope-detailed}
and~\eqref{eq:block-tail-limit-detailed}.  {Because \(\lambda\) is
fixed, there are constants \(0<c<C<\infty\) such that
\[
        c(1+\|\bomega\|^2)
        \le e^{2\lambda}+\|\bomega\|^2
        \le C(1+\|\bomega\|^2)
\]
for every \(\bomega\in\R^d\).  Raising this fixed-\(\lambda\) comparison
to the variable exponent \(-\alpha\) uses compactness of
\(\mathcal K_\alpha\) and gives, uniformly in
\(\alpha\in\mathcal K_\alpha\),
\[
        0<\overline S_{\alpha}(\bomega)
        \le C(1+\|\bomega\|^2)^{-\alpha}.
\]}
The spectrum is also continuous and even.  Moreover, for
\(\|\bomega\|\ge1\), {compactness of \(\mathcal K_\alpha\)
again makes the following tail estimate uniform:}
\[
        \left|
        \|\bomega\|^{2\alpha}
        \overline S_{\alpha}(\bomega)-1
        \right|
        =
        \left|
        \left(1+\frac{e^{2\lambda}}{\|\bomega\|^2}\right)^{-\alpha}-1
        \right|
        \le C\|\bomega\|^{-2}.
\]
Thus Equations~\eqref{eq:block-envelope-detailed}
and~\eqref{eq:block-tail-limit-detailed} hold uniformly on \(\mathcal K_\alpha\), with
leading constant one.

{For the differentiated-tail clauses, we now verify the reference decomposition and both remainder bounds.}
We now verify the reference decomposition in
Assumption~\ref{ass:differentiated-spectral-tail-family} by defining
\[
        S^{\mathrm{rem}}_{\alpha}(\bomega)
        :=\bigl(e^{2\lambda}+\|\bomega\|^2\bigr)^{-\alpha}
        -(1+\|\bomega\|^2)^{-\alpha}.
\]
To bound this remainder without an implicit intermediate point, let
\(q_t(\lambda):=1+t(e^{2\lambda}-1)\), \(0\le t\le1\).  The fundamental
theorem of calculus gives
\[
        S^{\mathrm{rem}}_{\alpha}(\bomega)
        =-\alpha(e^{2\lambda}-1)
        \int_0^1
        \{q_t(\lambda)+\|\bomega\|^2\}^{-\alpha-1}\,dt.
\]
The quantity \(q_t(\lambda)\) lies on the line segment joining \(1\) and
\(e^{2\lambda}\).  Hence, uniformly on \(\mathcal K_\alpha\times[0,1]\),
\[
        c(1+\|\bomega\|^2)
        \le q_t(\lambda)+\|\bomega\|^2
        \le C(1+\|\bomega\|^2).
\]
Differentiating the integral identity in \(\alpha\) gives
\begin{align*}
        \partial_\alpha
        S^{\mathrm{rem}}_{\alpha}(\bomega)
        &=(e^{2\lambda}-1)
        \int_0^1
        \bigl[\alpha\log\{q_t(\lambda)+\|\bomega\|^2\}-1\bigr]
        \{q_t(\lambda)+\|\bomega\|^2\}^{-\alpha-1}\,dt.
\end{align*}
The preceding comparison now yields
\begin{align*}
        \left|S^{\mathrm{rem}}_{\alpha}(\bomega)\right|
        &\le C(1+\|\bomega\|^2)^{-\alpha-1},\\
        \left|\partial_\alpha
        S^{\mathrm{rem}}_{\alpha}(\bomega)\right|
        &\le C(1+\|\bomega\|^2)^{-\alpha-1}
        \{1+\log(2+\|\bomega\|)\}.
\end{align*}
These are Equations~\eqref{eq:spectral-tail-remainder-bound} and
\eqref{eq:spectral-tail-alpha-remainder-bound} with \(\varepsilon=1\), which
completes the verification.
\end{proof}

\subsection*{{Generalized-Wendland verification}}

{We next verify the fixed-shape generalized-Wendland spectral family. Its boundary and support parameters \((\mu_0,\ell_0)\) remain fixed while \((\tau,\alpha)\) vary.}

Fix the boundary parameter \(\mu_0>(d+1)/2\) and a logarithmic support radius
\(\ell_0\in\R\).  The tail
coefficient is given by
\[
        c_{\infty,\mathrm W}(\sigma^2,\alpha;\mu_0,\ell_0)
        =2^d\pi^{d/2}\sigma^2
        \frac{\Gamma(\alpha)\Gamma(\mu_0+2\alpha-d)}{\Gamma(\alpha-d/2)\Gamma(\mu_0)}
        e^{-(2\alpha-d)\ell_0}.
\]
We therefore use the tail coordinates
\[
        \left( \tau=\log c_{\infty,\mathrm W},\alpha \right)
        \qquad
        {\Theta_{\mathrm W}^{\mathrm{fix}}}
        =\R\times\left(\frac{d+1}{2},\mu_0\right).
\]
All derivatives below are taken at fixed
\((\tau,\mu_0,\ell_0)\).

In \cite{bevilacqua2019gw}, their \(\kappa\) is
\(\alpha-(d+1)/2\), their \(\lambda(d,\kappa)\) equals our tail exponent
\(\alpha\), and their \(\beta\) is \(e^{\ell_0}\).  Their
Equation~(6) and Theorem~1 use the convention
\(\rho(\bx)=\int_{\R^d}e^{\mathrm i\bomega\cdot\bx}
\widehat\rho(\bomega)\,d\bomega\), whereas the present paper uses
\(S=(2\pi)^d\widehat\rho\).  Their printed radial-transform
Equation~(5), including in the published version, appears to have a
typographical normalization error: it displays \((2\pi)^{-d}\), where this
radial transform under their convention requires \((2\pi)^{-d/2}\).
After the Fourier-convention conversion and division by
\(c_{\infty,\mathrm W}\),
\cite[Point~1 of Theorem~1]{bevilacqua2019gw} gives
\begin{equation}
\begin{split}
        \overline S_{\alpha}(\bomega)
        &=\frac{\Gamma(\mu_0)}{\Gamma(\mu_0+2\alpha)}
        e^{2\alpha\ell_0} 
        {}_1F_2\left(
        \alpha;\alpha+\frac{\mu_0}{2},
        \alpha+\frac{\mu_0+1}{2};
        -\frac{e^{2\ell_0}\|\bomega\|^2}{4}
        \right).
        \label{eq:wendland-normalized-spectrum}
\end{split}
\end{equation}
Here \({}_1F_2\) is the generalized hypergeometric function. 
The normalized spectrum in
Equation~\eqref{eq:wendland-normalized-spectrum} is continuous and even; its
smooth dependence on \(\alpha\) will be established below from the
compactly supported covariance representation.  Moreover,
\cite[Theorem~2.2]{ChernihHubbert2014} gives strict positivity when
\(\mu_0>\alpha\).

For each fixed
\(\alpha\in((d+1)/2,\mu_0)\),
\cite[Point~2 of Theorem~1]{bevilacqua2019gw} gives
\begin{align}
        \overline S_{\alpha}(\bomega)
        &=\|\bomega\|^{-2\alpha}
        \{1+O(\|\bomega\|^{-2})\}\notag\\
        &\quad+D_{\alpha}
        \|\bomega\|^{-(\mu_0+\alpha)}
        \{\cos(\Phi_{\alpha}(\bomega))
        +O(\|\bomega\|^{-1})\},
        \label{eq:wendland-fixed-parameter-tail-expansion}
\end{align}
as \(\|\bomega\|\to\infty\), where
$
        D_{\alpha}
        =2^{1-\alpha}\frac{\Gamma(\mu_0)}{\Gamma(\alpha)}
        e^{(\alpha-\mu_0)\ell_0},
  $ and $
        \Phi_{\alpha}(\bomega)
        =e^{\ell_0}\|\bomega\|
        -\frac{\pi}{2}(\mu_0+\alpha)$.
This expansion identifies the algebraic and boundary contributions for each
fixed \(\alpha\). This is not sufficient for the uniform bounds in
Equations~\eqref{eq:spectral-tail-remainder-bound}, but we will now establish that it holds uniformly in \(\alpha\) on compact subsets of \(((d+1)/2,\mu_0)\).

For \(\|\bomega\|>0\), set
$
        q=e^{\ell_0}\|\bomega\|,
        $
        $ \nu:=\alpha-\frac{3}{2}.
$
 Define the finite Hankel integral
\begin{align}
        \mathcal H_\alpha(q)
        :=\int_0^1t^{\nu+1}(1-t)^{\mu_0}
        J_{\nu}(qt)\,dt,
        \label{eq:H}
\end{align}
so that their formula gives the normalization bridge
\[ {}_1F_2\left(
        \alpha;\alpha+\frac{\mu_0}{2},
        \alpha+\frac{\mu_0+1}{2};-\frac{q^2}{4}
        \right)=\frac{\Gamma(\alpha-1/2)}{B(2\alpha-1,\mu_0+1)}
        \left(\frac{2}q\right)^{\nu}
        \mathcal H_\alpha(q).
\]
Substitution into Equation~\eqref{eq:wendland-normalized-spectrum}, followed
by the gamma duplication formula, therefore gives the exact identity
\[
        \overline S_\alpha(\bomega)
        =A_\alpha q^{-\nu}\mathcal H_\alpha(q),
        \qquad
        A_\alpha
        =\frac{\sqrt\pi\,2^{1/2-\alpha}}{\mu_0\Gamma(\alpha)}e^{2\alpha\ell_0}.
\]
The two endpoints in the integral representation \eqref{eq:H} have distinct
roles: \(t=0\) produces the algebraic tail, whereas \(t=1\) produces the
oscillatory support-boundary term.

\begin{lemma}
\label{lem:wendland-fixed-real-endpoint-remainder}
Let
\(\mathcal J_\alpha\Subset((d+1)/2,\mu_0)\).
There are constants \(C>0\) and \(M\ge1\) such that, uniformly for
\(\alpha\in\mathcal J_\alpha\) and \(\|\bomega\|\ge M\),
\[
        \left|\overline S_\alpha(\bomega)
        -\|\bomega\|^{-2\alpha}\right|
        \le C\left\{
        \|\bomega\|^{-2\alpha-2}
        +\|\bomega\|^{-(\mu_0+\alpha)}\right\}.
\]
\end{lemma}

\begin{proof}
Set
$
        V:=\left\{\alpha-\frac{3}{2}:\alpha\in\mathcal J_\alpha\right\}.
$
Then \(\nu=\alpha-3/2\) ranges over
\(V\Subset((d-2)/2,\mu_0-3/2)\subset(-1/2,\infty)\).  Recall that
\(q=e^{\ell_0}\|\bomega\|\), so \(q\to\infty\) exactly when
\(\|\bomega\|\to\infty\).  We work below with \(q\ge1\), and all constants
and \(O(\cdot)\) terms are uniform for \(\nu\in V\).  Define
\[
        \mathcal T_\nu f(q)
        :=\int_0^1t^{\nu+1}f(t)J_\nu(qt)\,dt.
\]

Choose a fixed smooth cutoff \(\chi_0=1\) near zero whose support is separated
from one, and put
\[
\begin{gathered}
        g_{\mathrm{org}}(t)=\chi_0(t)(1-t)^{\mu_0},\qquad
        g_{\mathrm{end}}(t)=\{1-\chi_0(t)\}(1-t)^{\mu_0},
\end{gathered}
\]
thus
$       \mathcal H_\alpha(q)
        =\mathcal T_\nu g_{\mathrm{org}}(q)
        +\mathcal T_\nu g_{\mathrm{end}}(q).
$

The first term is localized near \(t=0\).  The second vanishes near \(t=0\)
and contains the contribution from the support endpoint \(t=1\).  We estimate
these two terms separately.

Since \(\operatorname{supp}\chi_0\subset[0,b]\) for some \(b<1\), extension
by zero gives \(g_{\mathrm{org}}\in C_c^\infty([0,\infty))\).  For the first
term, Wong's construction
\cite[Equations~(4.5)--(4.9)]{Wong1980IntegralError} gives a
four-term expansion of \(\mathcal T_\nu g_{\mathrm{org}}(q)\) as
\(q\to\infty\), together with an exact integral formula for its remainder. Set \(f_\nu(t):=t^{\nu+1}g_{\mathrm{org}}(t)\).  Since
Wong's statement is for fixed \(\nu\), we then bound the exact remainder
directly to obtain an \(O(q^{-\nu-5})\) bound uniform for \(\nu\in V\).
In Wong's notation, the required expansion at \(t=0\) is
\[
        f_\nu(t)\sim\sum_{s\ge0}c_st^{s+\nu+1},
        \qquad t\downarrow0,
\]
so the source exponent in his notation is \(\nu+2\).
Take \(n=4\) in that construction and put
\[
        c_s=\frac{g_{\mathrm{org}}^{(s)}(0)}{s!},\qquad
        R_4(t)=g_{\mathrm{org}}(t)-\sum_{s=0}^3c_st^s,\qquad
        m=\lceil\nu+5\rceil.
\]
{For Wong's conditions~(i)--(iii) in
\cite[Equation~(4.10)]{Wong1980IntegralError}, the source exponent is
\(\nu+2\), and we take \(n=4\).  The smooth compact support gives the
required \(C^m\) regularity and differentiable expansion, while
\(5\le m-\nu<6\) gives Wong's integer window
\(\nu+9/2<m<\nu+6\) uniformly on \(V\), with only finitely many possible
values of \(m\).  Also \(2\nu+2>0\), and the condition~(iii) integrals
vanish because \(f_\nu\) and its derivatives vanish on \([1,\infty)\).
Thus Wong's finite-sum formula in
\cite[Equation~(4.5)]{Wong1980IntegralError}, with the exact remainder
interface in Equations~(4.7)--(4.9), gives}
\[
        \mathcal T_\nu g_{\mathrm{org}}(q)
        =\sum_{s=0}^3c_s
        \frac{2^{s+\nu+1}\Gamma(\nu+s/2+1)}{\Gamma(-s/2)}
        q^{-s-\nu-2}+\delta_4(q).
\]
Here \(\delta_4(q)\) is Wong's remainder. For \(s=0,2\), the corresponding terms in the finite sum are zero because
the reciprocal gamma function vanishes at the nonpositive integers.  For
\(s=1\), using \(c_1=g_{\mathrm{org}}'(0)\) and
\(\Gamma(-1/2)=-2\sqrt\pi\) gives the term displayed below.  The \(s=3\)
term is \(O(q^{-\nu-5})\), uniformly for \(\nu\in V\).  Hence
\[
        \mathcal T_\nu g_{\mathrm{org}}(q)
        =-\frac{2^{\nu+1}\Gamma(\nu+3/2)}{\sqrt\pi}
        g_{\mathrm{org}}'(0)q^{-\nu-3}
        +O(q^{-\nu-5})+\delta_4(q).
\]

Let \(\Pi_m(z)=\prod_{j=0}^{m-1}(z-2j)\).  Iterating the remainder recurrence
in \cite[Equations~(4.7)--(4.9)]{Wong1980IntegralError} gives, initially in
the same Gauss-summability sense,
\[
\begin{aligned}
        \delta_4(q)
        &=(-q^{-1})^m\int_0^\infty
          J_{\nu+m}(qt)\varphi_{m,4}^{(\nu)}(t)\,dt,\\
        \varphi_{m,4}^{(\nu)}(t)
        &=t^{\nu+1-m}\Pi_m(t\,d/dt)R_4(t).
\end{aligned}
\]
{Write \(E=t\,d/dt\).  Taylor's formula at zero,
\(\Pi_m(0)=\Pi_m(2)=\Pi_m(4)=0\), and the polynomial form of \(R_4\)
off the support of \(g_{\mathrm{org}}\) give, using
\(5\le m-\nu<6\), the following bound.  Its constant is uniform in
\(\nu\) because \(V\) is compact and only finitely many values of \(m\)
occur:}
\[
        \left|t^{-1/2}\varphi_{m,4}^{(\nu)}(t)\right|
        \le C
        \begin{cases}
        t^{-1/2},&0<t\le1,\\
        t^{-3/2},&t\ge1.
        \end{cases}
\]
{The orders \(\nu+m\) lie in a compact interval
\(V_*\Subset(-1/2,\infty)\).  The Bessel power series in
\cite[Equation~(10.2.2)]{NISTDLMF} for \(0<z\le1\) and the large-argument
bound in \cite[Equation~(1.5), Theorem~1.5, and Proposition~B.1]{Nemes2017}
for \(z\ge1\) give
\(\sup_{\kappa\in V_*,z>0}z^{1/2}|J_\kappa(z)|<\infty\).  The displayed
majorant is therefore uniform and integrable, so Wong's remainder integral
is absolutely convergent and}
\[
        |\delta_4(q)|
        \le Cq^{-m-1/2}\le Cq^{-\nu-11/2},\qquad q\ge1.
\]
Since \(g_{\mathrm{org}}'(0)=-\mu_0\), this gives uniformly
\[
        \mathcal T_\nu g_{\mathrm{org}}(q)
        =\frac{\mu_0 2^{\nu+1}\Gamma(\nu+3/2)}{\sqrt\pi}
        q^{-\nu-3}+O(q^{-\nu-5}).
\]
This proves the result for the first term.

We now focus on the second term $\mathcal T_\nu g_{\mathrm{end}}(q)$.
We use a first-order Taylor expansion at $t=1$ to separate two terms that
can be evaluated explicitly.  We then show by repeated integration by parts
that the remainder decays at the required rate.
Set
\[
\begin{aligned}
        a_0&=2^{-\mu_0},\qquad
        a_1=\mu_0 2^{-\mu_0-2},\qquad
        N=\lceil\mu_0+1\rceil,\\
        R_2(t)&=\{1-\chi_0(\sqrt t)\}(1-\sqrt t)^{\mu_0}
        -a_0(1-t)^{\mu_0}-a_1(1-t)^{\mu_0+1}.
\end{aligned}
\]
{The first term in \(R_2\) vanishes near zero, and Taylor's
formula for \((1+\sqrt t)^{-\mu_0}\) shows that near one \(R_2(t)\) is
\((1-t)^{\mu_0+2}\) times a smooth function.  Hence
\(R_2\in C^N[0,1]\) and \(R_2^{(k)}(1)=0\) for \(k<N\), because
\(\mu_0+2-N\in(0,1]\), and}
\[
        g_{\mathrm{end}}(t)
        =\sum_{j=0}^1a_j(1-t^2)^{\mu_0+j}+R_2(t^2).
\]
{Sonine's finite integral
\cite[\href{https://dlmf.nist.gov/10.22.E19}{Equation~(10.22.19)}]{NISTDLMF},
after \(t=\sin\theta\), gives}
\[
\begin{aligned}
        \mathcal T_\nu g_{\mathrm{end}}(q)
        ={}&\sum_{j=0}^1a_j2^{\mu_0+j}\Gamma(\mu_0+j+1)
        q^{-\mu_0-j-1}J_{\nu+\mu_0+j+1}(q)+I_{R_2}(q),\\
        I_{R_2}(q)&:=\int_0^1t^{\nu+1}R_2(t^2)J_\nu(qt)\,dt.
\end{aligned}
\]
{Here \(\nu>-1/2\) and \(\mu_0+j>-1\), so Sonine's hypotheses
hold.  The orders \(\nu+\mu_0+j+1\) lie in a compact set, and Nemes'
large-argument expansion and remainder bound
\cite[Equation~(1.5), Theorem~1.5, and Proposition~B.1]{Nemes2017} give}
\[
        \sup_{\nu\in V,\,q\ge1}
        q^{1/2}|J_{\nu+\mu_0+j+1}(q)|<\infty,
        \qquad j=0,1.
\]
{Thus the two explicit terms together are
\(O(q^{-\mu_0-3/2})\).}

{The derivative identity
\cite[\href{https://dlmf.nist.gov/10.6.E6}{Equation~(10.6.6)}]{NISTDLMF}
and \(N\) integrations by parts give}
\[
        I_{R_2}(q)
        =\left(-\frac{2}q\right)^N
        \int_0^1t^{\nu+N+1}R_2^{(N)}(t^2)J_{\nu+N}(qt)\,dt.
\]
{The boundary terms vanish at one by the endpoint conditions and
at zero because \(R_2^{(k)}\) is bounded and the power series for
\(J_{\nu+k+1}\) supplies \(t^{2\nu+2k+2}=o(1)\).  The final orders lie in
the compact interval \(V+N\), so the same Nemes bound gives, for \(q\ge1\),}
\[
        |I_{R_2}(q)|\le Cq^{-N-1/2}
        \le Cq^{-\mu_0-3/2}.
\]
{Together with the explicit terms, this proves uniformly that
\(\mathcal T_\nu g_{\mathrm{end}}(q)=O(q^{-\mu_0-3/2})\).}

Combining the two terms gives
\[
        \mathcal H_\alpha(q)
        =\frac{\mu_0 2^{\nu+1}\Gamma(\nu+3/2)}{\sqrt\pi}
        q^{-\nu-3}+O(q^{-\nu-5})+O(q^{-\mu_0-3/2}).
\]
Since \(\nu+3/2=\alpha\) and \(q=e^{\ell_0}\|\bomega\|\), the leading term
has the exact normalization
\[
        A_\alpha q^{-\nu}
        \frac{\mu_0 2^{\nu+1}\Gamma(\nu+3/2)}{\sqrt\pi}
        q^{-\nu-3}
        =\|\bomega\|^{-2\alpha}.
\]
The two remainders become, uniformly on \(\mathcal J_\alpha\),
\(O(\|\bomega\|^{-2\alpha-2})\) and
\(O(\|\bomega\|^{-(\mu_0+\alpha)})\), respectively.  All estimates above
hold for \(q\ge1\), so the choice
\(M=\max\{1,e^{-\ell_0}\}\) proves the claim.
\end{proof}

\subsubsection*{Parameter regularity and real interpolation}

The interpolation argument requires crude high-order parameter derivatives,
not differentiated endpoint asymptotics.  We obtain them from the covariance
in physical space.

\begin{lemma}
\label{lem:wendland-fixed-real-parameter-regularity}
For every compact interval
\(\mathcal J_\alpha\Subset((d+1)/2,\mu_0)\), the tail-normalized covariance
is a \(C^m\) map from \(\mathcal J_\alpha\) into \(L^1(\R^d)\) for every
integer \(m\ge1\).  Consequently, for each \(j\le m\),
\[
        \sup_{\alpha\in\mathcal J_\alpha}
        \sup_{\bomega\in\R^d}
        |\partial_\alpha^j\overline S_\alpha(\bomega)|<\infty.
\]
The derivatives are jointly continuous in \((\alpha,\bomega)\).
\end{lemma}

\begin{proof}
Let \(v_\alpha\) be the raw variance for which the Wendland tail coefficient
equals one.  The displayed tail coefficient above gives
\[
        v_\alpha
        =\frac{\Gamma(\alpha-d/2)\Gamma(\mu_0)}{2^d\pi^{d/2}\Gamma(\alpha)
        \Gamma(\mu_0+2\alpha-d)}
        e^{(2\alpha-d)\ell_0}.
\]
This is a smooth positive function of \(\alpha\) on
\(\mathcal J_\alpha\).  With
\(s(\bx)=e^{-\ell_0}\|\bx\|\), the corresponding tail-normalized
covariance is
\[
        \overline\rho_\alpha(\bx)
        =\frac{v_\alpha\mathbb I_{\{s(\bx)<1\}}}{B(2\alpha-d-1,\mu_0+1)}
        \times\int_{s(\bx)}^1
        u\{u^2-s(\bx)^2\}^{\alpha-(d+3)/2}
        (1-u)^{\mu_0}\,du.
\]

Set
\[
        \kappa_-:=\inf_{\alpha\in\mathcal J_\alpha}
        \left\{\alpha-\frac{d+1}{2}\right\}>0.
\]
Differentiating the inner kernel \(j\) times in \(\alpha\) produces powers
of \(\log(u^2-s^2)\) of degree at most \(j\).  Up to a constant, all such
derivatives for \(j\le m\) are bounded by
\[
        u(u^2-s^2)^{\kappa_--1}(1-u)^{\mu_0}
        \{1+|\log(u^2-s^2)|^m\}.
\]
The resulting bound is integrable.  Indeed, using polar coordinates and
setting \(s=ut\), its integral is bounded by a constant times
\[
\begin{split}
        &\int_0^1u^{2\kappa_-+d-1}(1-u)^{\mu_0}
        \{1+|\log u|^m\}\,du\\
        &\quad\times\int_0^1t^{d-1}(1-t^2)^{\kappa_--1}
        \{1+|\log(1-t^2)|^m\}\,dt,
\end{split}
\]
and both factors are finite.  The derivatives of \(v_\alpha\) and of the
beta-function factor are bounded on \(\mathcal J_\alpha\).

This common majorant justifies differentiation under both integrals and
continuity in \(L^1(\R^d)\) by the dominated convergence theorem. Under the paper's Fourier convention,
\[
        \partial_\alpha^j\overline S_\alpha(\bomega)
        =\int_{\R^d}e^{-\mathrm i\bomega\cdot\bx}
        \partial_\alpha^j\overline\rho_\alpha(\bx)\,d\bx.
\]
The asserted uniform bound and joint continuity now follow from the
\(L^1\) bound and dominated convergence.
\end{proof}

We also use the following one-dimensional interpolation bound.

\begin{lemma}
\label{lem:wendland-fixed-real-interpolation}
Let \(\mathcal K\) and \(\mathcal J\) be compact intervals with
\(\mathcal K\subset\operatorname{int}(\mathcal J)\), and let \(m\ge2\).
There is a constant \(C\), depending only on the two intervals and \(m\),
such that every \(f\in C^m(\mathcal J)\) satisfies
\[
        \|f'\|_{L^\infty(\mathcal K)}
        \le C\left[
        \|f\|_{L^\infty(\mathcal J)}
        +\|f\|_{L^\infty(\mathcal J)}^{1-1/m}
        \|f^{(m)}\|_{L^\infty(\mathcal J)}^{1/m}
        \right].
\]
\end{lemma}

\begin{proof}
{Apply the bounded-domain Gagliardo--Nirenberg inequality \cite[Theorem~1.3]{LiZhang2022} on \(\operatorname{int}(\mathcal J)\) with \(n=1\), \(k=m\), \(j=1\), \(p=q=r=\infty\), and \(\theta=1/m\); its additive bounded-domain term is the first term on the right-hand side, and restriction from \(\mathcal J\) to \(\mathcal K\) proves the stated bound.}
\end{proof}

For the remaining bounds, write
$
        L(\bomega):=1+\log(2+\|\bomega\|).
$

\begin{lemma}
\label{lem:wendland-fixed-differentiated-tail-bounds}
Let
\(\mathcal K_\alpha\Subset((d+1)/2,\mu_0)\).  Suppose that
\[
        0<\varepsilon<
        \min\left\{1,\frac{1}{2}
        \inf_{\alpha\in\mathcal K_\alpha}(\mu_0-\alpha)\right\}.
\]
There are constants \(C>0\) and \(M\ge1\) such that, uniformly for
\(\alpha\in\mathcal K_\alpha\) and \(\|\bomega\|\ge M\),
\begin{subequations}
\label{eq:wendland-uniform-parameter-estimates}
\begin{align}
        \left|\overline S_{\alpha}(\bomega)
        -\|\bomega\|^{-2\alpha}\right|
        &\le C\left\{
        \|\bomega\|^{-2\alpha-2}
        +\|\bomega\|^{-(\mu_0+\alpha)}\right\},
        \label{eq:wendland-uniform-tail-remainder}\\
        \left|\partial_\alpha
        \left\{\overline S_{\alpha}(\bomega)
        -\|\bomega\|^{-2\alpha}\right\}\right|
        &\le CL(\bomega)
        \|\bomega\|^{-2\alpha-2\varepsilon}.
        \label{eq:wendland-uniform-alpha-derivative}
\end{align}
\end{subequations}
\end{lemma}

\begin{proof}
{Since the upper bound on \(\varepsilon\) is strict, compactness of
\(\mathcal K_\alpha\) allows us to choose a compact interval
\(\mathcal J_\alpha\) such that
\begin{align*}
 \mathcal K_\alpha&\subset\operatorname{int}(\mathcal J_\alpha),
 \quad \mathcal J_\alpha\Subset\left(\frac{d+1}{2},\mu_0\right),
\end{align*}
and $2\varepsilon<
 \delta:=\min\left\{2,
 \inf_{\alpha\in\mathcal J_\alpha}(\mu_0-\alpha)\right\}.$}
Lemma~\ref{lem:wendland-fixed-real-endpoint-remainder}, applied on
\(\mathcal J_\alpha\), proves
Equation~\eqref{eq:wendland-uniform-tail-remainder} after restriction to
\(\mathcal K_\alpha\).

{For \(\|\bomega\|\ge M\),} weight the same remainder by setting
\[
        E_\alpha(\bomega)
        :=\overline S_\alpha(\bomega)-\|\bomega\|^{-2\alpha},
        \qquad
        G_\alpha(\bomega)
        :=\|\bomega\|^{2\alpha}E_\alpha(\bomega).
\]
The endpoint lemma and, for every \(m\ge2\), parameter regularity plus
Leibniz's rule give, with \(\alpha_+=\sup\mathcal J_\alpha\),
\[
        \|G_\bullet(\bomega)\|_{L^\infty(\mathcal J_\alpha)}
        \le C\|\bomega\|^{-\delta},
        \qquad
        \|\partial_\alpha^mG_\bullet(\bomega)\|_
        {L^\infty(\mathcal J_\alpha)}
        \le C\|\bomega\|^{2\alpha_+}L(\bomega)^m.
\]
Lemma~\ref{lem:wendland-fixed-real-interpolation} now yields
\[
        \sup_{\alpha\in\mathcal K_\alpha}
        |\partial_\alpha G_\alpha(\bomega)|
        \le CL(\bomega)\|\bomega\|^{-\gamma_m},
        \qquad
        \gamma_m
        :=\delta-\frac{\delta+2\alpha_+}{m}.
\]
Fix \(m\) so large that \(\gamma_m>2\varepsilon\).  Returning to the
unweighted remainder \(E_\alpha\), we have
\[
        \partial_\alpha E_\alpha(\bomega)
        =\|\bomega\|^{-2\alpha}
        \{\partial_\alpha G_\alpha(\bomega)
        -2\log\|\bomega\|G_\alpha(\bomega)\}.
\]
Since \(\gamma_m,\delta>2\varepsilon\), this proves
Equation~\eqref{eq:wendland-uniform-alpha-derivative}.
\end{proof}

\begin{proof}[Proof of Proposition~\ref{prop:wendland-tail-family-verification}]
Fix
\(\mathcal K_\alpha\Subset((d+1)/2,\mu_0)\), set
$
        \Delta_\alpha
        :=\inf_{\alpha\in\mathcal K_\alpha}(\mu_0-\alpha)>0,
$
and fix
$
        0<\varepsilon<
        \min\left\{1,\frac{\Delta_\alpha}{2}\right\}.
$
Let \(C,M\) be the constants supplied by
Lemma~\ref{lem:wendland-fixed-differentiated-tail-bounds} for this
\(\varepsilon\).  We verify every clause of
Assumptions~\ref{ass:block-tail-detailed}
and~\ref{ass:differentiated-spectral-tail-family}, beginning with the scale
representation and basic properties.

\emph{Scale and basic properties.}
The tail-coordinate construction gives
\[
        S_{\tau,\alpha}(\bomega)
        =e^\tau\overline S_{\alpha}(\bomega),
\]
which is the scale-family representation in
Assumption~\ref{ass:differentiated-spectral-tail-family} for the
two-coordinate fitted chart \(\varphi=(\tau,\alpha)\).
It also shows that differentiation in \(\alpha\) holds
\((\tau,\mu_0,\ell_0)\) fixed.  Continuity and evenness follow from
Equation~\eqref{eq:wendland-normalized-spectrum}, while
Lemma~\ref{lem:wendland-fixed-real-parameter-regularity} gives \(C^1\)
regularity in \(\alpha\). {Strict positivity follows from \cite[Theorem~2.2]{ChernihHubbert2014}, as used at the normalized-spectrum interface.}

Since \(\mu_0-\alpha\ge\Delta_\alpha>0\), both error terms in
Equation~\eqref{eq:wendland-uniform-tail-remainder} are bounded by a constant
multiple of \(\|\bomega\|^{-2\alpha}\) for \(\|\bomega\|\ge M\).
Continuity and compactness extend the resulting upper envelope to the bounded
frequency region.  Hence
\[
        0<\overline S_{\alpha}(\bomega)
        \le C(1+\|\bomega\|^2)^{-\alpha}
\]
uniformly on \(\mathcal K_\alpha\times\R^d\).  The same estimate gives
\[
        \sup_{\alpha\in\mathcal K_\alpha}
        \left|
        \|\bomega\|^{2\alpha}
        \overline S_{\alpha}(\bomega)-1
        \right|
        \le C\left\{
        \|\bomega\|^{-2}
        +\|\bomega\|^{-\Delta_\alpha}\right\}
        \longrightarrow0,
        \qquad \|\bomega\|\to\infty.
\]
These are Equations~\eqref{eq:block-envelope-detailed}
and~\eqref{eq:block-tail-limit-detailed}, uniformly with leading constant one.

We have 
$
        S^{\mathrm{rem}}_{\alpha}(\bomega)
        :=\overline S_{\alpha}(\bomega)
        -(1+\|\bomega\|^2)^{-\alpha}.
$ For
\(j=0,1\) and \(\|\bomega\|\ge M\), direct differentiation of the
reference-spectrum difference gives
\[
        \left|\partial_\alpha^j
        \left\{\|\bomega\|^{-2\alpha}
        -(1+\|\bomega\|^2)^{-\alpha}\right\}\right|
        \le CL(\bomega)^j\|\bomega\|^{-2\alpha-2}.
\]
Uniformly for \(\alpha\in\mathcal K_\alpha\) and
\(\|\bomega\|\ge M\),
\begin{align*}
        \|\bomega\|^{-2\alpha-2}
        +\|\bomega\|^{-(\mu_0+\alpha)}
        &\le C\|\bomega\|^{-2\alpha-2\varepsilon},\\
        L(\bomega)\|\bomega\|^{-2\alpha-2}
        &\le CL(\bomega)
        \|\bomega\|^{-2\alpha-2\varepsilon}.
\end{align*}
Combining these comparisons with
Equations~\eqref{eq:wendland-uniform-tail-remainder}
and~\eqref{eq:wendland-uniform-alpha-derivative} proves
Equations~\eqref{eq:spectral-tail-remainder-bound}
and~\eqref{eq:spectral-tail-alpha-remainder-bound} for
\(\|\bomega\|\ge M\).

The normalized spectrum in
Equation~\eqref{eq:wendland-normalized-spectrum} is jointly continuous in
\((\alpha,\bomega)\), and its \(\alpha\)-derivative is continuous on
$
        \mathcal K_\alpha
        \times\{\bomega:\|\bomega\|\le M\}.
$
Compactness extends the two bounds to this remaining frequency region.
\end{proof}

\begingroup
\begin{proof}[Proof of Proposition~\ref{prop:common-leading-tail-covariance-pairs}]
We verify the two displayed clauses of
Assumption~\ref{ass:common-leading-tail-misspecification}, for all the four combinations of Mat\'ern and generalized Wendland spectra. 

The required \(\eta>d/4\) can be chosen as follows:
\begin{enumerate}
\item if both spectra are Mat\'ern, take \(\eta=1\);
\item if \(S^\star\) is generalized Wendland and \(S_{\varphi_0}\) is
Mat\'ern, take
\(d/4<\eta<\min\{1,(\mu^\star-\alpha_0)/2\}\);
\item if \(S^\star\) is Mat\'ern and \(S_{\varphi_0}\) is generalized
Wendland, take
\(d/4<\eta<\min\{1,(\mu_0-\alpha_0)/2\}\);
\item if both spectra are generalized Wendland, take
\[
 d/4<\eta<\min\left\{1,\frac{\mu^\star-\alpha_0}{2},
 \frac{\mu_0-\alpha_0}{2}\right\}.
\]
\end{enumerate}
These intervals are nonempty: the gaps in the proposition exceed \(d/2\), while
\(d/4<1\) because \(d\le3\).

{In case 2, the generalized Wendland spectrum \(S^\star\) and
its remainder are fixed external objects; only the fitted Mat\'ern family
varies over the neighborhood.  Since \(2\eta<\mu^\star-\alpha_0\), a
sufficiently small compact fitted-\(\alpha\) interval about \(\alpha_0\)
retains this strict gap.}

{We will now build compact intervals for $\alpha$ and \(\tau\) that contain the fitted parameters in their interiors
and satisfy Assumption~\ref{ass:differentiated-spectral-tail-family}.}
{This is done by selecting intervals so that Proposition~\ref{prop:matern-tail-family-verification}
and Proposition~\ref{prop:wendland-tail-family-verification} are satisfied.}
{For $\tau$, choose any compact interval around $\tau_0$.}
{The same holds for $\alpha$ for the Mat\'ern family.}
{For the generalized Wendland family, the selected \(\eta\) satisfies
\(2\eta<\mu_0-\alpha_0\).}
{Thus we choose a small interval around $\alpha_0$ where the inequality is satisfied.}
{As the propositions are satisfied, we may therefore choose an open neighborhood \(\mathcal U\)
of \(\varphi_0\) whose closure lies in that product.}
{Applying the relevant proposition at the fixed external parameters gives continuity,
evenness, and strict positivity of \(S^\star\).}

We will now prove  the difference-bound clause of
Assumption~\ref{ass:common-leading-tail-misspecification}.
First we assume that at least one spectrum is generalized Wendland.  For the
chosen \(\eta\), the family verifications give the two separate bounds for
\(\|\bomega\|\) sufficiently large,
\begin{align*}
 \left|S^\star(\bomega)-e^{\tau_0}
 (1+\|\bomega\|^2)^{-\alpha_0}\right|
 &\le C(1+\|\bomega\|^2)^{-\alpha_0-\eta},\\
 \left|S_{\varphi_0}(\bomega)-e^{\tau_0}
 (1+\|\bomega\|^2)^{-\alpha_0}\right|
 &\le C(1+\|\bomega\|^2)^{-\alpha_0-\eta}.
\end{align*}
The common term in the two bounds cancels, so the
triangle inequality proves the difference-bound clause of
Assumption~\ref{ass:common-leading-tail-misspecification} at
high frequency.  Continuity extends that bound to all \(\bomega\in\R^d\)
after increasing its constant.

Second, if both spectra are Mat\'ern, let \(\kappa_0\) and \(\kappa^\star\) be
the fixed ranges of \(S_{\varphi_0}\) and \(S^\star\), respectively.  The
mean-value theorem gives
\[
 \lvert S^\star(\bomega)-S_{\varphi_0}(\bomega)\rvert
 \le e^{\tau_0}\alpha_0\lvert\kappa^{\star2}-\kappa_0^2\rvert
 \{\|\bomega\|^2+\min(\kappa^{\star2},\kappa_0^2)\}^{-\alpha_0-1}
 \lesssim(1+\|\bomega\|^2)^{-\alpha_0-1},
\]
which proves the difference bound with \(\eta=1\).

We now show that the two-sided spectral-comparison clause of
Assumption~\ref{ass:common-leading-tail-misspecification} holds.
 {For all spectra, the difference bound and the common leading tail give
\[
 \frac{S^\star(\bomega)}{S_{\varphi_0}(\bomega)}\longrightarrow1,
 \qquad \|\bomega\|\longrightarrow\infty.
\]}
Hence this ratio is bounded above and below by positive constants outside a
sufficiently large ball.  On the closed ball,
continuity and strict positivity give positive minimum and finite maximum.
Combining the two regions proves
the two-sided spectral-comparison clause of
Assumption~\ref{ass:common-leading-tail-misspecification}.  Thus all clauses of
Assumption~\ref{ass:common-leading-tail-misspecification} hold. 
\end{proof}
\endgroup

\bibliographystyle{plainnat}
\bibliography{likelihood_asymptotics}

@article{CandesDemanetYing2007,
  author  = {Cand{\`e}s, Emmanuel and Demanet, Laurent and Ying, Lexing},
  title   = {Fast Computation of {Fourier} Integral Operators},
  journal = {SIAM Journal on Scientific Computing},
  volume  = {29},
  number  = {6},
  pages   = {2464--2493},
  year    = {2007},
}

@book{Bhatia2007PositiveDefinite,
  author    = {Bhatia, Rajendra},
  title     = {Positive Definite Matrices},
  series    = {Princeton Series in Applied Mathematics},
  publisher = {Princeton University Press},
  address   = {Princeton, NJ},
  year      = {2007},
  isbn      = {978-0-691-12918-1}
}

@article{Hajek1958,
  author  = {H{\'a}jek, Jaroslav},
  title   = {On a Property of Normal Distributions of Any Stochastic Process},
  journal = {Czechoslovak Mathematical Journal},
  volume  = {8},
  number  = {4},
  pages   = {610--618},
  year    = {1958},
  note    = {In Russian}
}

@book{Vershynin2018,
  author    = {Vershynin, Roman},
  title     = {High-Dimensional Probability:
               An Introduction with Applications in Data Science},
  publisher = {Cambridge University Press},
  year      = {2018},
}

@book{HornJohnson1991,
  author    = {Horn, Roger A. and Johnson, Charles R.},
  title     = {Topics in Matrix Analysis},
  publisher = {Cambridge University Press},
  address   = {Cambridge},
  year      = {1991},
}

@article{bevilacqua2019gw,
  author  = {Moreno Bevilacqua and Tarik Faouzi and Reinhard Furrer and Emilio Porcu},
  title   = {Estimation and Prediction Using Generalized {W}endland Covariance Functions under Fixed Domain Asymptotics},
  journal = {The Annals of Statistics},
  year    = {2019},
  volume  = {47},
  number  = {2},
  pages   = {828--856},
}

@book{Matern1986,
  author    = {Mat{\'e}rn, Bertil},
  title     = {Spatial Variation},
  edition   = {Second},
  series    = {Lecture Notes in Statistics},
  volume    = {36},
  publisher = {Springer},
  address   = {New York},
  year      = {1986},
}

@article{Wendland1995,
  author  = {Wendland, Holger},
  title   = {Piecewise Polynomial, Positive Definite and Compactly Supported
             Radial Functions of Minimal Degree},
  journal = {Advances in Computational Mathematics},
  volume  = {4},
  pages   = {389--396},
  year    = {1995},
}

@book{Stein1999,
  author    = {Stein, Michael L.},
  title     = {Interpolation of Spatial Data: Some Theory for Kriging},
  series    = {Springer Series in Statistics},
  publisher = {Springer},
  address   = {New York},
  year      = {1999},
  isbn      = {978-0-387-98629-6},
}

@book{GrenanderSzego1958,
  author    = {Grenander, Ulf and Szeg{\H{o}}, G{\'a}bor},
  title     = {{Toeplitz} Forms and Their Applications},
  publisher = {University of California Press},
  address   = {Berkeley},
  year      = {1958},
}

@article{TyrtyshnikovZamarashkin1998,
  author  = {Tyrtyshnikov, E. E. and Zamarashkin, N. L.},
  title   = {Spectra of Multilevel {Toeplitz} Matrices: Advanced Theory via Simple Matrix Relationships},
  journal = {Linear Algebra and its Applications},
  volume  = {270},
  number  = {1--3},
  pages   = {15--27},
  year    = {1998},
}

@book{GaroniSerraCapizzanoI,
  author    = {Carlo Garoni and Stefano Serra{-}Capizzano},
  title     = {Generalized Locally Toeplitz Sequences: Theory and Applications. Volume I},
  series    = {Springer Monographs in Mathematics},
  publisher = {Springer},
  address   = {Cham},
  year      = {2017},
  isbn      = {978-3-319-53679-8},
}

@book{GaroniSerraCapizzanoII,
  author    = {Carlo Garoni and Stefano Serra{-}Capizzano},
  title     = {Generalized Locally Toeplitz Sequences: Theory and Applications. Volume II},
  series    = {Springer Monographs in Mathematics},
  publisher = {Springer},
  address   = {Cham},
  year      = {2018},
  isbn      = {978-3-030-02232-7},
}

@article {Zhang2004,
	AUTHOR = {Zhang, Hao},
	TITLE = {Inconsistent estimation and asymptotically equal
	interpolations in model-based geostatistics},
	JOURNAL = {J.\ Amer.\ Statist.\ Assoc.},
	FJOURNAL = {Journal of the American Statistical Association},
	VOLUME = {99},
	YEAR = {2004},
	NUMBER = {465},
	PAGES = {250--261},
}

@misc{ChongMies2025,
  author       = {Chong, Carsten H. and Mies, Fabian},
  title        = {Likelihood Asymptotics of Stationary {G}aussian Arrays},
  year         = {2025},
  eprint       = {2502.09229},
  archiveprefix = {arXiv},
  primaryclass = {math.ST},
  note         = {Preprint, arXiv v2, November 2025},
}

@article{Dirksen2015,
  author  = {Dirksen, Sjoerd},
  title   = {Tail Bounds via Generic Chaining},
  journal = {Electronic Journal of Probability},
  volume  = {20},
  number  = {53},
  pages   = {1--29},
  year    = {2015},
}

@article{Kailath1967,
  author    = {Kailath, T.},
  title     = {The Divergence and {Bhattacharyya} Distance Measures in Signal Selection},
  journal   = {{IEEE} Transactions on Communication Technology},
  volume    = {15},
  number    = {1},
  pages     = {52--60},
  year      = {1967},
}

@article{BrownPurves1973,
  author    = {Brown, L. D. and Purves, R.},
  title     = {Measurable Selections of Extrema},
  journal   = {The Annals of Statistics},
  volume    = {1},
  number    = {5},
  pages     = {902--912},
  year      = {1973},
}

@book{LeCamYang2000,
  author    = {Le Cam, Lucien and Yang, Grace Lo},
  title     = {Asymptotics in Statistics: Some Basic Concepts},
  edition   = {Second},
  series    = {Springer Series in Statistics},
  publisher = {Springer},
  address   = {New York},
  year      = {2000},
}

@book{Billingsley1995,
  author    = {Billingsley, Patrick},
  title     = {Probability and Measure},
  edition   = {Third},
  publisher = {Wiley},
  address   = {New York},
  year      = {1995},
  isbn      = {978-0-471-00710-4}
}

@book{VanDerVaart1998,
  author    = {van der Vaart, A. W.},
  title     = {Asymptotic Statistics},
  series    = {Cambridge Series in Statistical and Probabilistic Mathematics},
  volume    = {3},
  publisher = {Cambridge University Press},
  address   = {Cambridge},
  year      = {1998},
}

@article{SerraCapizzano2003,
  author  = {Serra-Capizzano, Stefano},
  title   = {Generalized Locally {Toeplitz} Sequences: Spectral Analysis and Applications to Discretized Partial Differential Equations},
  journal = {Linear Algebra and its Applications},
  volume  = {366},
  pages   = {371--402},
  year    = {2003},
}

@article{Ying1991,
  author  = {Ying, Zhiliang},
  title   = {Asymptotic Properties of a Maximum Likelihood Estimator with
             Data from a {G}aussian Process},
  journal = {Journal of Multivariate Analysis},
  volume  = {36},
  number  = {2},
  pages   = {280--296},
  year    = {1991},
}

@article{Ying1993,
  author  = {Ying, Zhiliang},
  title   = {Maximum Likelihood Estimation of Parameters under a Spatial
             Sampling Scheme},
  journal = {The Annals of Statistics},
  volume  = {21},
  number  = {3},
  pages   = {1567--1590},
  year    = {1993},
}

@article{DuZhangMandrekar2009,
  author  = {Du, Juan and Zhang, Hao and Mandrekar, V. S.},
  title   = {Fixed-Domain Asymptotic Properties of Tapered Maximum
             Likelihood Estimators},
  journal = {The Annals of Statistics},
  volume  = {37},
  number  = {6A},
  pages   = {3330--3361},
  year    = {2009},
}

@article{KaufmanShaby2013,
  author  = {Kaufman, Cari G. and Shaby, Benjamin A.},
  title   = {The Role of the Range Parameter for Estimation and Prediction
             in Geostatistics},
  journal = {Biometrika},
  volume  = {100},
  number  = {2},
  pages   = {473--484},
  year    = {2013},
}

@article{WangLoh2011,
  author  = {Wang, Daqing and Loh, Wei-Liem},
  title   = {On fixed-domain asymptotics and covariance tapering in {G}aussian random field models},
  journal = {Electronic Journal of Statistics},
  volume  = {5},
  pages   = {238--269},
  year    = {2011},
}

@article{LiuFengXiao2026,
  author  = {Liu, Nian and Feng, Haoxiang and Xiao, Yimin},
  title   = {Infill Asymptotics for Closed-Form Alternatives to Maximum
             Likelihood Estimators of Covariance Parameters in the
             {Ornstein--Uhlenbeck} Fields},
  journal = {Annals of the Institute of Statistical Mathematics},
  year    = {2026},
}

@article{WuLim2016,
  author  = {Wu, Wei-Ying and Lim, Chae Young},
  title   = {Estimation of Smoothness of a Stationary {G}aussian Random Field},
  journal = {Statistica Sinica},
  volume  = {26},
  number  = {4},
  pages   = {1729--1745},
  year    = {2016},
}

@article{LohSunWen2021,
  author  = {Loh, Wei-Liem and Sun, Saifei and Wen, Jun},
  title   = {On Fixed-Domain Asymptotics, Parameter Estimation and Isotropic
             {G}aussian Random Fields with {M}at\'ern Covariance Functions},
  journal = {The Annals of Statistics},
  volume  = {49},
  number  = {6},
  pages   = {3127--3152},
  year    = {2021},
}

@article{Loh2015,
  author  = {Loh, Wei-Liem},
  title   = {Estimating the Smoothness of a {G}aussian Random Field from
             Irregularly Spaced Data via Higher-Order Quadratic Variations},
  journal = {The Annals of Statistics},
  volume  = {43},
  number  = {6},
  pages   = {2766--2794},
  year    = {2015},
}

@article{ChenOwhadiStuart2021,
  author  = {Chen, Yifan and Owhadi, Houman and Stuart, Andrew M.},
  title   = {Consistency of Empirical {B}ayes and Kernel Flow for Hierarchical
             Parameter Estimation},
  journal = {Mathematics of Computation},
  volume  = {90},
  number  = {332},
  pages   = {2527--2578},
  year    = {2021},
}

@article{Petit2025,
  author  = {Petit, S{\'e}bastien J.},
  title   = {An Asymptotic Study of the Joint Maximum Likelihood Estimation
             of the Regularity and the Amplitude Parameters of a Periodized
             {Mat\'ern} Model},
  journal = {Electronic Journal of Statistics},
  volume  = {19},
  number  = {1},
  pages   = {2052--2094},
  year    = {2025},
}

@article{Karvonen2023,
  author  = {Karvonen, Toni},
  title   = {Asymptotic Bounds for Smoothness Parameter Estimates in
             {G}aussian Process Interpolation},
  journal = {SIAM/ASA Journal on Uncertainty Quantification},
  volume  = {11},
  number  = {4},
  pages   = {1225--1257},
  year    = {2023},
}

@article{KorteStapffKarvonenMoulines2025,
  author  = {Korte-Stapff, Moritz and Karvonen, Toni and Moulines, {\'E}ric},
  title   = {Smoothness Estimation for {Whittle--Mat\'ern} Processes on
             Closed {R}iemannian Manifolds},
  journal = {Stochastic Processes and their Applications},
  volume  = {189},
  pages   = {104685},
  year    = {2025},
}

@misc{KotharkarStein2025,
  author       = {Kotharkar, Varun and Stein, Michael L.},
  title        = {Fixed and Increasing Domain Asymptotics for the Roughness
                  and Scale of Isotropic {G}aussian Random Fields},
  year         = {2025},
  eprint       = {2510.26045},
  archiveprefix = {arXiv},
  primaryclass = {math.ST},
  note         = {Preprint},
}

@article{BrousteFukasawa2018,
  author  = {Brouste, Alexandre and Fukasawa, Masaaki},
  title   = {Local Asymptotic Normality Property for Fractional {G}aussian
             Noise under High-Frequency Observations},
  journal = {The Annals of Statistics},
  volume  = {46},
  number  = {5},
  pages   = {2045--2061},
  year    = {2018},
}

@article{ChernihHubbert2014,
  author  = {Chernih, Andrew and Hubbert, Simon},
  title   = {Closed Form Representations and Properties of the Generalised
             {W}endland Functions},
  journal = {Journal of Approximation Theory},
  volume  = {177},
  pages   = {17--33},
  year    = {2014},
}

@article{Wong1980IntegralError,
  author  = {Wong, R.},
  title   = {Error Bounds for Asymptotic Expansions of Integrals},
  journal = {SIAM Review},
  volume  = {22},
  number  = {4},
  pages   = {401--435},
  year    = {1980},
}

@article{Nemes2017,
  author  = {Nemes, Gerg{\H{o}}},
  title   = {Error Bounds for the Large-Argument Asymptotic Expansions of
             the {H}ankel and {B}essel Functions},
  journal = {Acta Applicandae Mathematicae},
  volume  = {150},
  number  = {1},
  pages   = {141--177},
  year    = {2017},
}

@article{LiZhang2022,
  author  = {Li, Congming and Zhang, Kai},
  title   = {A Note on the {G}agliardo--{N}irenberg Inequality in a
             Bounded Domain},
  journal = {Communications on Pure and Applied Analysis},
  volume  = {21},
  number  = {12},
  pages   = {4013--4017},
  year    = {2022},
}

@book{wendland2005scattered,
  author    = {Wendland, Holger},
  title     = {Scattered Data Approximation},
  series    = {Cambridge Monographs on Applied and Computational Mathematics},
  number    = {17},
  publisher = {Cambridge University Press},
  address   = {Cambridge},
  year      = {2005},
  isbn      = {978-0-521-84335-5},
  note      = {Copyright page states ``First published in print format 2004''}
}

@misc{NISTDLMF,
  author       = {{NIST Digital Library of Mathematical Functions}},
  title        = {{NIST Digital Library of Mathematical Functions}},
  note         = {Release 1.2.7 of 15 June 2026, accessed 1 August 2026},
  year         = {2026}
}

@article{LimStein2008,
  author  = {Lim, Chae Young and Stein, Michael L.},
  title   = {Properties of Spatial Cross-Periodograms Using Fixed-Domain
             Asymptotics},
  journal = {Journal of Multivariate Analysis},
  volume  = {99},
  number  = {9},
  pages   = {1962--1984},
  year    = {2008},
}

@article{BevilacquaCaamanoPorcu2022,
  author  = {Bevilacqua, Moreno and Caama\~no-Carrillo, Christian and Porcu,
             Emilio},
  title   = {Unifying Compactly Supported and {Mat\'ern} Covariance Functions
             in Spatial Statistics},
  journal = {Journal of Multivariate Analysis},
  volume  = {189},
  pages   = {104949},
  year    = {2022},
}

@misc{KarvonenBachoc2025,
  author       = {Karvonen, Toni and Bachoc, Fran\c{c}ois},
  title        = {Scale Estimation and Rate-Unbiasedness for {Gaussian}
                  Processes under Smoothness Misspecification},
  year         = {2025},
  eprint       = {2110.02810},
  archiveprefix = {arXiv},
  primaryclass = {math.ST},
  note         = {Preprint, arXiv v3, October 2025},
}

@misc{Pillai2026,
  author       = {Pillai, Natesh S.},
  title        = {Microergodicity Implies Orthogonality of {Mat\'ern} Fields
                  on Bounded Domains in {$\mathbb{R}^4$}},
  year         = {2026},
  eprint       = {2603.23959},
  archiveprefix = {arXiv},
  primaryclass = {math.ST},
  note         = {Preprint, arXiv v1, March 2026},
}

@article{PorcuEtAl2024,
  author  = {Porcu, Emilio and Bevilacqua, Moreno and Schaback, Robert and Oates, Chris J.},
  title   = {The {Mat\'ern} Model: A Journey Through Statistics, Numerical Analysis and Machine Learning},
  journal = {Statistical Science},
  year    = {2024},
  volume  = {39},
  number  = {3},
  pages   = {469--492},
}
\end{document}